\documentclass[12pt]{amsart}

\usepackage[utf8]{inputenc}
\usepackage{amssymb,amsmath, amsthm, mathabx}
\usepackage{color, enumitem}
\usepackage{hyperref}
\usepackage{caption}
\usepackage{subcaption}
\usepackage{a4wide}
\usepackage{accents}
\usepackage{tikz}
\usepackage{graphicx}

\let\counterwithin\relax
\usepackage{chngcntr}
\usepackage{chngcntr}
\usepackage{datetime2}
\usepackage{scalerel}    %%%%  used first for scaling \bullet  
\usepackage{float}

\usepackage{xargs}
\usepackage[colorinlistoftodos,prependcaption,textsize=footnotesize]{todonotes}
\newcommandx{\addmath}[2][1=]{\todo[linecolor=red,backgroundcolor=red!25,bordercolor=red,#1]{#2}}
\newcommandx{\fixtext}[2][1=]{\todo[linecolor=blue,backgroundcolor=blue!25,bordercolor=blue,#1]{#2}}
\newcommandx{\note}[2][1=]{\todo[linecolor=yellow,backgroundcolor=yellow!25,bordercolor=yellow,#1]{#2}}

\usetikzlibrary{decorations.pathreplacing}

\usepackage{xifthen}% provides \isempty test

\newcommand{\bbL}{\mathbb{L}}

\newcommand{\bbN}{\mathbb{N}}

\newcommand{\bbQ}{\mathbb{Q}}
\newcommand{\bbR}{\mathbb{R}}

\newcommand{\bbV}{\mathbb{V}}

\newcommand{\bbZ}{\mathbb{Z}}

\newcommand{\bfA}{\mathbf{A}}
\newcommand{\bfB}{\mathbf{B}}

\newcommand{\bfP}{\mathbf{P}}

\newcommand{\mfa}{\mathfrak{a}}
\newcommand{\mfb}{\mathfrak{b}}
\newcommand{\mfc}{\mathfrak{c}}

\newcommand{\sF}{\mathcal{F}}
\newcommand{\sG}{\mathcal{G}}

\newcommand{\sT}{\mathcal{T}}
\newcommand{\sU}{\mathcal{U}}
\newcommand{\sV}{\mathcal{V}}

\newcommand{\one}{\mathbf{1}}

\newcommand{\lf}{\lfloor}
\newcommand{\rf}{\rfloor}

\DeclareMathOperator\Exp{Exp}

\def\dd{\mathrm d}

\def\Lpt{\mathrm L}

\def \G{\mathrm G}
\def \B{\mathrm B}

\def\ww{\mathrm{w}}

\def\hor{\mathrm{hor}}
\def\ver{\mathrm{ver}}

\newcommand{\ter}[1]{\overline{#1}}
\newcommand{\init}[1]{\underline{#1}}

\newcommand{\ir}[2]{\mathbf{i}_{#1}^{#2}}
\newcommand{\jr}[2]{\mathbf{j}_{#1}^{#2}}

\definecolor{darkblue}{rgb}{.1,.1,.6}

\makeatletter
\def\widebreve{\mathpalette\wide@breve}
\def\wide@breve#1#2{\sbox\z@{$#1#2$}%
     \mathop{\vbox{\m@th\ialign{##\crcr
\kern0.08em\brevefill#1{0.8\wd\z@}\crcr\noalign{\nointerlineskip}%
                    $\hss#1#2\hss$\crcr}}}\limits}
\def\brevefill#1#2{$\m@th\sbox\tw@{$#1($}%
  \hss\resizebox{#2}{\wd\tw@}{\rotatebox[origin=c]{90}{\upshape(}}\hss$}
\makeatletter

\newcommand{\wt}[1]{\widetilde{#1}}
\newcommand{\wh}[1]{\widehat{#1}}
\newcommand{\wc}[1]{\widecheck{#1}}
\newcommand{\wb}[1]{\Grave{#1}}

\newcommand{\Fabs}[4]
{
\ifthenelse{\isempty{#2}}
{
{#1}_{#3}^{#4}
}
{
{#1}_{#3}^{#4}(#2)
}
}

\def\P{\bfP}

\def\shor{\shp_{\operatorname{hor}}}
\def\sver{\shp_{\operatorname{ver}}}

 \def\shp{\gamma}

\newcommand\geo[3]{\pi_{#1}^{#2,#3}} %Geodesic. Arguments are 1.) the index, 2.) the starting point, 3.) the direction
\newcommand\geod[1]{\pi_{#1}} %A generic geodesic without a root or direction. The argument is the index.

\newcommand\LPP[2]{\G_{#1,#2}} %The passage time from #1 to #2. Requires #1 \leq #2

\def \I{\mathrm I}
\def \J{\mathrm J}

\newcommand\zmin[1]{\chi^{#1}}

\newcommand\Path[2]{\Pi_{#1}^{#2}}

\def\w{\omega}
\newcommand\ex{\tau} %i.i.d. exponential(1) weights

\DeclareMathOperator\Rect{\mathrm{R}}

\newcommand\eset{\Lambda}

\definecolor{darkgreen}{rgb}{.1,.6,.1}

\def\Z{\mathbb Z} 
\newcommand{\be}{\begin{equation}}   
\newcommand{\ee}{\end{equation}}   
\def\Cif{U}
\def\Cifs{V}
\def\Cifa{\varphi}  %another way to write cif 
\def\Cifb{\psi}    %the cont. time cif 
\def\cif{\xi_*}    %limiting direction of cif 

\newcommand{\Buse}[3]{\B_{#1,#2}^{#3}}    % G-increments.  #1 = initial point,  #2 = 1 or 2 (initial increment),  #3  end point of LPP
\newcommand{\buse}[3]{\B_{#1,#2}^{#3}}    % G-increments.  #1 = initial point,  #2 = 1 or 2 (initial increment),  #3  end point of LPP
\def\vela{v^*}  %velocity a -- limiting speed of 2 class particle 

\def\ddd{\displaystyle} 

\newcommand{\smin}[1]{{#1}^{\text{\rm min}}}   %%%% minimum of a sequence 
\newcommand{\sinf}[1]{{#1}^{\text{\rm inf}}}   %%%% minimum of a sequence 

\newcommand\bbullet{{\raisebox{0.5pt}{\scaleobj{0.6}{\bullet}}}} %% looks good for paths x_\bbullet 
\newcommand*{\cl}[1]{\overline{#1}}

\newtheorem{theorem}{Theorem}[section]
\newtheorem{proposition}[theorem]{Proposition}
\newtheorem{lemma}[theorem]{Lemma}

\theoremstyle{remark}
\newtheorem{remark}[theorem]{Remark}
\AtBeginEnvironment{remark}{%
  \pushQED{\qed}%
}
\AtEndEnvironment{remark}{\popQED\endremark}
\newtheorem{example}[theorem]{Example}
\AtBeginEnvironment{example}{%
  \pushQED{\qed}%
}
\AtEndEnvironment{example}{\popQED\endremark}
\newtheorem{condition}[theorem]{Condition}
\AtBeginEnvironment{condition}{%
  \pushQED{\qed}%
}
\AtEndEnvironment{condition}{\popQED\endremark}

\counterwithin{figure}{section}
\counterwithin{equation}{section}

\allowdisplaybreaks

\begin{document}

\title[Anomalous geodesics in the iCGM]{Anomalous geodesics in the \\ inhomogeneous corner growth model}
 
%\title[Busemann functions and geodesics in inhomogeneous CGM]{Busemann functions and geodesics in the solvable inhomogeneous exponential corner growth model} 

\author[E.~Emrah]{Elnur Emrah}
\address{Elnur Emrah\\ University of Bristol \\ School of Mathematics \\ Bristol \\ United Kingdom}
\email{e.emrah@bristol.ac.uk}
\urladdr{https://sites.google.com/view/elnur-emrah}
\thanks{E.\ Emrah was partially supported by the EPSRC grant EP/W032112/1, by grant KAW 2015.0270 from the Knut and Alice Wallenberg Foundation, and by the Mathematical Sciences Departments at Carnegie Mellon University through a postdoctoral position.}

\author[C.~Janjigian]{Christopher Janjigian}
\address{Christopher Janjigian\\ Purdue University\\  Department of Mathematics \\ 150 N. University St.\\   West Lafayette, IN 47907\\ USA.}
\email{cjanjigi@math.purdue.edu}
\urladdr{http://www.math.purdue.edu/~cjanjigi}
\thanks{C.\ Janjigian was partially supported by a postdoctoral grant from the Fondation Sciences Math\'ematiques de Paris while working at Universit\'e Paris Diderot, a postdoctoral position at the University of Utah, and by National Science Foundation grant DMS-2125961.}

\author[T.~Sepp\"al\"ainen]{Timo Sepp\"al\"ainen}
\address{Timo Sepp\"al\"ainen\\ University of Wisconsin-Madison\\  Mathematics Department\\ Van Vleck Hall\\ 480 Lincoln Dr.\\   Madison WI 53706-1388\\ USA.}
\email{seppalai@math.wisc.edu}
\urladdr{http://www.math.wisc.edu/~seppalai}
\thanks{T.\ Sepp\"al\"ainen was partially supported by  National Science Foundation grants  DMS-1602486, DMS-1854619, DMS-2152362,  and DMS-2448375, by Simons Foundation grant 1019133, and by the Wisconsin Alumni Research Foundation.} 

\thanks{}

\begin{abstract}   
We study Busemann functions, semi-infinite geodesics, and competition interfaces in the exactly solvable last-passage percolation with inhomogeneous exponential weights. New phenomena  concerning geodesics arise due to inhomogeneity. These include novel Busemann functions associated with  flat regions of the limit shape and thin rectangles, semi-infinite geodesics with intervals of asymptotic directions, non-trivial axis-directed geodesics, intervals with no geodesic directions, and isolated geodesic directions. We further observe a new dichotomy for competition interfaces and second-class customers in a series of memoryless continuous-time queues with inhomogeneous service rates: a second-class customer either becomes trapped or proceeds through the service stations at strictly positive speed.
\end{abstract}
\maketitle

\setcounter{tocdepth}{1}
\tableofcontents

\section{Introduction}

%\subsection{Homogeneous and inhomogeneous corner growth model in the KPZ class}
The {\it corner growth model} (CGM), which is also known as {\it directed last-passage percolation} (LPP), is one of the best-studied models in the Kardar-Parisi-Zhang (KPZ) universality class. Viewed as a growth model, it describes a growing random set of infected sites on the first quadrant of the integer lattice  $\bbZ_{\geq 0}^2$, which begins with the origin infected and then evolves by infecting north and east neighbors of already infected sites. The input to the model is the 
\textit{environment} which is a collection of random weights, typically taken to be i.i.d.~or ergodic. If they are non-negative, these represent the time it takes for a site to be added to the cluster once its neighbors to the left and below have already joined. In its LPP formulation, the model can be thought of as a directed version of first-passage percolation (FPP). When the weights are exponentially distributed, the evolution is Markovian and the model is closely linked to other extensively studied stochastic models such as TASEP and series of M/M/1 queues. 

The CGM with i.i.d.~exponential weights is {\it exactly solvable}, meaning that the model has structure which allows explicit computation of statistics of interest. This exact solvability lies behind Johansson's seminal result \cite{Joh-06} showing Tracy-Widom fluctuations of the passage times, confirming rigorously that the model lies in the KPZ class, and subsequent works extending this to process-level convergence to the KPZ fixed point/directed landscape \cite{Dau-Ort-Vir-22-,Dau-Vir-21-, Mat-Qua-Rem-21}. 

The present paper studies the solvable inhomogeneous extension of the exponential CGM. In the LPP formulation, this means varying rates  along rows and columns. In TASEP language, this corresponds to particles and holes each carrying their own exponential clocks with different rates. Some aspects of the inhomogeneous model have been studied previously: hydrodynamics and shape theorems \cite{Emr-16, Emr-Jan-Sep-21, Sep-Kru-99}, some limiting statistics \cite{Bor-Pec-08, Dim-24-,Joh-Rah-22}, and large deviations \cite{Emr-Jan-17}. There has been recent interest in further exact formulas in this model, as well as in its discrete-time and continuous-space counterparts \cite{Bis-etal-23,Joh-Rah-22, Kni-Pet-Sae-19}. There has also been recent work  in the physics literature on fluctuations and connections to localization in the Brownian analogue of the inhomogeneous model we study and its positive temperature counterpart \cite{Kra-LeD-Oco-21}. The localization phenomena observed there are related to some of the novel behavior of infinite geodesics we outline momentarily. 

\subsection{Highlights of main results} 
In the i.i.d.~exponential CGM, considerable work has been devoted to the study of interrelated questions concerning semi-infinite geodesics, competition interfaces, and Busemann functions (directional limits of passage-time increments)  \cite{Cat-Pim-12,Cat-Pim-13,Fan-Sep-20,Fer-Mar-Pim-06, Fer-Pim-05}. These include distributional structure of Busemann functions, directedness, uniqueness,  and coalescence of semi-infinite geodesics. Such problems are also connected to asymptotics of second-class particles  in TASEP and second-class customers in series of queues. 

Our interest is in the impact of inhomogeneity on Busemann functions, semi-infinite geodesics and competition interfaces. %These include directedness, uniqueness,  and coalescence of semi-infinite geodesics. asymptotic behavior of  competition interfaces, and asymptotics of second-class particles. 
\textit{We establish new phenomena that arise from inhomogeneity and are not present in the i.i.d.~ setting.} These  include the following:
\begin{enumerate}
\item In Theorem \ref{thm:buslim}, we show that there are (potentially infinitely many) non-trivial Busemann functions obtained as the limit of passage-time increments along fixed rows or columns. Limits of these \emph{thin rectangle} Busemann functions give the (unique) Busemann function associated to each flat segment of the limit shape.  
\item In Theorem \ref{thm:geo}, we show that environments exist with infinitely many non-trivial non-coalescing semi-infinite geodesics rooted at zero which have the same \textit{fixed} asymptotic direction. See Example \ref{ex:suffcond}\eqref{ex:suffcond:iidunif} for a concrete example. %Our examples of these geodesics are axis (i.e.\ $\{e_1,e_2\}$) directed.
\item Theorem \ref{thm:geo} also shows that axis directed geodesics exist which do not become trapped on a row or column. Again, see Example \ref{ex:suffcond}\eqref{ex:suffcond:iidunif} for a concrete example.
\item Non-empty intervals of directions exist such that no semi-infinite geodesic anywhere on the lattice has a subsequential limit direction in these intervals. Example \ref{ex:suffcond}\eqref{ex:suffcond:empty} gives an extreme case where the set of such directions is the entire linear segment other than the boundary direction. See also Examples \ref{ex:suffcond}\eqref{ex:suffcond:interval}, \eqref{ex:suffcond:isolated}, and \eqref{ex:suffcond:critical}.
\item In Example \ref{ex:suffcond}\eqref{ex:suffcond:interval}, we show existence of a semi-infinite geodesic with a prescribed interval as its set of subsequential limit directions. In particular, there are geodesics in a continuous independent environment without an asymptotic direction.
\item In Example \ref{ex:suffcond}\eqref{ex:suffcond:isolated}, we show existence of isolated directions of geodesics, meaning that a semi-infinite geodesic has direction $\zeta$ but a neighborhood around $\zeta$ contains no other subsequential limit direction of any other  semi-infinite geodesic.
\item In Theorems  \ref{thm:cif1} and \ref{th:2cl}, we prove a sharp dichotomy for competition interfaces and the asymptotic behavior of a second-class customer in a series of inhomogeneous memoryless queues: the competition interface either becomes trapped on a row or column or else converges to a direction in the strictly concave region of the limit shape; similarly, a second-class customer either moves at a strictly positive speed or is eventually  trapped at a single service station.
\end{enumerate}

Theorem \ref{thm:geo} shows that in general each  semi-infinite geodesic from a fixed initial point falls in exactly one of three types:  %\\[-5pt] 
\begin{itemize} 
\item 
directed  into the strictly concave region of the limit shape,
\item divergent $e_1$ and $e_2$ coordinates and subsequential limit directions contained in one of the two (possibly degenerate) linear segments, or 
\item trapped on a row or column. % \\[-21pt] 
\end{itemize}  The first and last types always exist, with the first type behaving largely similarly to geodesics in the homogeneous model. We show that there exist geodesics which are directed into the linear segments and which do not become trapped on a row or column of the lattice if and only if there is no most favorable row or column (in the sense of weight means). These are the most novel (and subtle) of the semi-infinite geodesics we observe and they generate most of the anomalous examples mentioned above. The behavior of the geodesics in the linear region depends strongly on the precise form of the inhomogeneity, as described in Theorem \ref{thm:lindir}.

\subsection{Background: geodesics and regularity of the limit shape}
In metric-like stochastic growth models, convexity and differentiability  of the limit shape are closely connected to the geometry of geodesics. In FPP and LPP, the limit shape in i.i.d.~models can have flat regions if the minimum (resp.\ maximum) of the vertex weight is attained frequently enough to create an infinite cluster.  When this happens, the shape function is   affine in a cone symmetric about the diagonal of the plane. In FPP,  this phenomenon traces back to the classic paper of Durrett and Liggett \cite{Dur-Lig-81}, and was subsequently studied by Marchand \cite{Mar-02} and Auffinger-Damron \cite{Auf-Dam-13}.  The phenomenon is the same in LPP, as recorded in Section 3.2 of \cite{Geo-Ras-Sep-17-ptrf-1}. In ergodic FPP, it is known that any compact convex subset of $\bbR^2$ with the symmetries of $\bbZ^2$ arises  as a limit shape \cite{Hag-Mee-95}. The proof of this fact, as well as the construction of the polygonal shapes in \cite{Ale-Ber-18, Bri-Hof-21}, rely on  random favorable paths in a sea of unfavorable weights, carefully constructed to preserve ergodicity. 

In models like ours, linear segments arise from a related but different source, where favorable regions are created by independent weights with different distributions. A particular phenomenon leading to linear segments in this model has previously been studied under the name of  \textit{mesoscopic clustering}.
It is perhaps easiest to understand in TASEP where the jump rates of the particles are chosen randomly from an ergodic distribution and the jump rates of the holes are constant.  Denote by $c>0$ the infimum of the support of the random rate and assume that the left tail of the distribution is sufficiently thin near $c$. In this case, particles with rates close to $c$ occur infinitely often, but relatively rarely. Because of the exclusion rule, faster moving particles become trapped behind slow particles, forming {\it platoons}. Ahead of each such slow particle, however, is another even slower particle and so over time platoons merge and move at speeds approaching $c$. This merging and the subsequent slow-down of the model occurs on a mesoscopic scale, below the hydrodynamic scale. Consequently, at densities below a certain critical density, one sees only rigid transport at speed $c$. This fixed-speed evolution manifests itself as a flat segment on the limit shape of the growth model. See \cite{And-etal-00,Bah-Bod-18, Gri-Kan-Sep-04, Kru-Fer-96, Sep-Kru-99,Sly-16-} for previous work on this TASEP formulation. Linear segments similar to the ones we observe have also appeared in an inhomogeneous FPP \cite{Ahl-Dam-Sid-16}. 

When the limit shape exhibits linear segments, the standard convexity and curvature considerations which enforce directedness of semi-infinite geodesics no longer apply. In \cite{Ale-Ber-18}, Alexander and Berger gave an example of an ergodic FPP model with a polygonal limit shape, where the mechanism creating linear segments enforces that all semi-infinite geodesics are directed into the corners of the shape. Brito and Hoffman \cite{Bri-Hof-21} subsequently produced another ergodic FPP model where a different mechanism results in a polygonal shape. In that model, there is one semi-infinite geodesic directed into each of the linear segments and this geodesic has the full linear segment as its set of subsequential limit directions. In our inhomogeneous but independent setting,  a richer structure is possible, with essentially arbitrary sub-intervals of the linear segment arising as the set of directions of unique semi-infinite geodesics.

\subsection{Methods}

As alluded to previously, the limits of passage-time increments along a given direction, row, or column define the corresponding Busemann function. Our approach begins with establishing the existence of Busemann functions and accessing their distributional structure. %to a sufficient degree. 
Specifically, for each Busemann function, we compute the marginal distributions along each nearest-neighbor edge and show that these are independent along any down-right path. The latter feature is an aspect of the \emph{Burke property} discussed further in Section \ref{sec:Burke}. In our model, %we find that 
the edge marginals are exponentially distributed with certain inhomogeneous rates. The form of the inhomogeneity of the rates is chosen to preserve this notion of exact solvability. 

%Our results concerning 
The properties of the Busemann functions are collected in Theorem \ref{thm:buslim}. The proof of this result relies on various couplings with the stationary version of the inhomogeneous exponential CGM and its Burke property. In the strictly concave regions, the argument proceeds similarly to the homogeneous case \cite{Geo-etal-15, Sep-18}, through squeezing the Busemann functions by increments of the stationary models. To implement this approach for the inhomogeneous CGM, we utilize shape theorems developed in our previous work \cite{Emr-Jan-Sep-21}. The flat regions, being adjacent to the axes, cannot be treated in the same fashion because the squeezing argument breaks down from one side (the axis direction).

To overcome the preceding difficulty, we introduce thin-rectangle Busemann functions. Once again, it is unclear \textit{a priori} how to squeeze from the axis direction. Since one works with a fixed number of rows or columns at this level, there is always a first most favorable row or column. Our technical innovation is to interpret this as a boundary coming from a stationary model. This observation enables us to execute a version of the squeezing argument for the thin-rectangle case. We subsequently show that the Busemann functions of the flat regions can be squeezed from the axis direction via the limits of the thin-rectangle Busemann functions. Agreement of the limits of thin rectangle Busemann functions with limits coming from the strictly concave region underlies the uniqueness of the Busemann functions in flat regions as well as the dichotomy we prove for competition interfaces. %in Theorem \ref{thm:cif1}.  

% Our approach utilizes %focuses on 
% Busemann functions %in the LPP interpretation of the model, 
% constructed  through a coupling argument which relies on the ``Burke property" of the inhomogeneous exponential CGM. 

Busemann functions in lattice growth models trace back to the seminal work of Newman \cite{New-95}, with subsequent work including \cite{Cat-Pim-12, Cat-Pim-13, Dam-Han-14, Geo-Ras-Sep-17-ptrf-2, Geo-Ras-Sep-17-ptrf-1, Hof-05, Hof-08, Jan-Ras-Sep-22-}. The aforementioned Burke property was first observed in a quadrant growth model by Cator and Groeneboom \cite{Cat-Gro-05,Cat-Gro-06}. Shortly thereafter, Bal\'azs, Cator, and the last author extended this to the homogeneous CGM \cite{Bal-Cat-Sep-06}. %Couplings like ours were first used to construct Busemann functions by Cator and Pimentel in  \cite{Cat-Pim-12}, with applications in \cite{Cat-Pim-13}. % The closest coupling arguments to ours in the past literature are those of \cite{Geo-etal-15, Sep-18}. 

%Our proof of the Busemann limits in the strictly concave regions follow arguments which are similar to those in \cite{Geo-etal-15, Sep-18}. 
%These coupling arguments become more delicate and require some non-trivial technical innovations when studying both the thin rectangle Busemann functions and the critical Busemann functions corresponding to linear segments. 
%Our results concerning Busemann functions are collected in Theorem \ref{thm:buslim}.

%Through the Busemann process we give a partial description of the global structure of semi-infinite geodesics.   
%Our results on geodesics are closest to those developed in the homogeneous model in \cite{Geo-Ras-Sep-17-ptrf-2}. The present  work provides the foundation for the kind of complete description of the global structure of semi-infinite geodesics achieved in \cite{Jan-Ras-Sep-22-}. The main inputs needed to complete this program are a tractable description the distribution of the Busemann process in the inhomogeneous model (likely attainable following the methods of \cite{Fan-Sep-20}) and a more complete description of the behavior of semi-infinite geodesics directed into the linear regions. Implementing the geometric arguments in \cite{Jan-Ras-Sep-22-} in this setting is relatively straightforward, though made cleaner if one knows the distribution of the Busemann process beforehand, which is why it is not done in this paper.

Modulo some technical differences, most of our results concerning geodesics follow from the distributional structure of Busemann functions similar to arguments in \cite{Geo-Ras-Sep-17-ptrf-2}, with two exceptions: coalescence and our result giving control over the linear segment geodesics. Theorem \ref{thm:geo} describes the %collects our basic results on the 
general structure of semi-infinite geodesics (part \eqref{thm:geo:coal} covers coalescence), while the result controlling linear segment geodesics is Theorem \ref{thm:lindir}. In both of these results, we once again utilize estimates from our previous work \cite{Emr-Jan-Sep-21}. 

The much-used  Licea-Newman \cite{Lic-new-96} coalescence argument is not available to us  because the environment is no longer shift-invariant. Recent years have seen a variety of arguments for coalescence which bypass this argument in various solvable models \cite{Rah-Vir-21-,Sep-18, Sep-Sor-23}. Our techniques to prove coalescence are a variant of the approach introduced in \cite{Sep-18, Sep-20}. 

Our argument for controlling linear segment geodesics is new and somewhat counter-intuitive. We develop bounds similar to the classical curvature bounds that have been used previously to control geodesics in the strictly concave region of models in ergodic environments, but apply these in the linear segment. The reason this is possible despite studying directions where the shape is flat is that the finite volume passage time in this setting is naturally concentrated not on the true limit shape, but rather on the limit shape that would have appeared had the inhomogeneity been periodic. Such limit shapes always have curvature which is bounded from below, but of course these bounds break down as one takes limits. Nevertheless, under mild hypotheses, we are able to retain enough uniform control over the passage times to govern the sets of limit directions of the geodesics. 

%Theorem \ref{thm:geo} collects our basic results on the general structure of geodesics (part \eqref{thm:geo:coal} covers coalescence), while the result controlling linear segment geodesics is Theorem \ref{thm:lindir}.

Our main theorem on competition interfaces is Theorem \ref{thm:cif1}, which follows from arguments similar to those in \cite{Fer-Pim-05,Geo-Ras-Sep-17-ptrf-2}. 
Using a coupling due to Ferrari and Pimentel \cite{Fer-Pim-05}, these have consequences for second-class customers in the inhomogeneous M/M/1 queue, as mentioned above. These are recorded as Theorem \ref{th:2cl}.

\subsection{Extensions and applications}
The first natural direction of extension would be to study a more general inhomogeneity structure under which the environment still homogenizes. Extending beyond column-row inhomogeneity or to non-exponential (or geometric) distributions may be challenging because these changes would break exact solvability. Our use of solvability begins with our reliance on a detailed understanding of the structure of the limit shape that appears for any collection of inhomogeneity parameters satisfying our mild regularity assumptions. The product-form structure of passage-time increments of models with appropriate boundaries coming from the Burke property and uniform tail estimates for exponential random variables with rates bounded away from zero also play an important role in several of our proofs. Some of our coupling arguments also rely on the full Burke structure of the model, including the dual weights. In particular, we highlight the coalescence argument in Section \ref{sec:dualcoal}.

A second natural direction would be to other models which admit the same inhomogeneity structure while remaining solvable. This would include, for example, the inhomogeneous log-gamma polymer studied in \cite{Cor-etal-14} and the inhomogeneous Brownian last-passage percolation and O'Connell-Yor polymers studied in \cite{Kra-LeD-Oco-21}. Without having written out the details carefully, we expect that results similar to ours can be obtained in these settings with similar methods, though there may be some additional technical challenges.

Finally, we note that the results of this work have seen recent application in \cite{Bat-etal-25-} as part of a novel description of the joint distribution of Busemann functions in the homogeneous model.

\subsection{Organization of the paper}
%The structure of the paper is as follows.
Section \ref{sec:CGM} introduces the model we study. Section \ref{sec:res} contains the statements of our main results. Section \ref{sec:Burke} introduces our main tool, the Burke property.  We prove existence and some key properties of Busemann functions in Section \ref{sec:Bus}. These Busemann functions are then used as tools to study the structure of semi-infinite geodesics in Section \ref{sec:geo}, competition interfaces in Section \ref{sec:cif}, and the interacting particle system interpretation of the model in Section \ref{sec:particles}.

\subsection{Notation and conventions}
\label{SsNotCon}

 $\bbZ$, $\bbQ$ and $\bbR$ stand for the sets of integers, rational numbers and real numbers, respectively.  $\cl{\bbR}$ denotes the extended reals $\bbR \cup \{-\infty, \infty\}$. Restricted subsets are indicated  with subscripts, such as  $\bbZ_{>k}=\{k+1, k+2,  k+3, \dots\}$ and $\bbZ_{\ge k}=\{k, k+1, k+2,\dotsc\}$. For $n \in \bbZ_{\ge 0}$,  $[n] = \{i \in \bbZ_{>0}: i \le n\}$. In particular, $[0]$ is the empty set $\emptyset$. For $x \in \bbR$,  $x^+ = \max(x, 0)$. Given $a,b \in \bbR$, we will denote $\min(a,b) = a \wedge b$ and $\max(a,b) = a \vee b$. 

 The standard basis vectors of $\bbR^2$ are $e_1 = (1,0)$ and $e_2 = (0,1)$. We denote by $[e_2,e_1] = \{t e_2 + (1-t) e_1 : 0 \leq t \leq 1\}$. For $\zeta, \eta \in [e_2,e_1], ]\zeta,\eta[ = \{t \zeta + (1-t) \eta : 0 < t < 1\}$. The half-open intervals $[\zeta,\eta[$ and $]\zeta,\eta]$ are defined analogously. 

We call a path (a sequence) $\geod{} = (\pi_i)$ on $\bbZ^2$ up-right if $\geod{i} - \geod{i-1} \in \{e_1,e_2\}$ and down-right if $\geod{i} - \geod{i-1} \in\{e_1,-e_2\}$. It will be convenient at times to identify an up-right or down-right path $\pi$ with its set $\{\pi_{i}\}$ of vertices.

We write $\le$ for the coordinatewise partial order on $\bbZ^2$. Thus, for $x, y \in \bbZ^2$, the inequality $x \le y$ means that $x\cdot e_1 \leq y \cdot e_1$ and $x \cdot e_2 \leq y \cdot e_2$. For $x,y\in\bbZ^2$, we define the coordinate-wise %minimum and 
maximum $x \vee y$ via $(x\vee y)\cdot e_i = (x\cdot e_i)\vee (y\cdot e_i)$ for $i\in\{1,2\}.$  The minimum $x \wedge y$ is defined analogously. 
For $x, y \in \bbZ$, let
\begin{align}
\label{eq:Rect}
\Rect_{x}^{y} = \{v \in \bbZ^2: x \le v \le y\}
%= \{v \in \bbZ^2 : x \cdot e_1 \leq v \cdot e_1 \leq y \cdot e_1\text{ and }x \cdot e_2 \leq v\cdot e_2 \leq y \cdot e_2 \}
\end{align}
denote the rectangle (rectangular grid) of lattice sites  bounded from below by $x$ and above by $y$. By definition, $\Rect_x^y = \emptyset$ unless $x \le y$. A down-right path $\pi$ from the upper left corner $(x \cdot e_1, y \cdot e_2)$ to the lower right corner  $(y \cdot e_1, x \cdot e_2)$ (necessarily $x \le y$) partitions $\Rect_{x}^y \smallsetminus \pi$ into the two sets  
\begin{align}
\label{eq:Gpm}
\sG_{x, y, \pi}^\pm = \{p \in \Rect_{x}^y: p \mp k(e_1 +e_2) \in \pi \text{ for some } k \in \bbZ_{>0}\}. 
\end{align}
Equivalently, $p \in \Rect_{x}^y$ satisfies  $p \in \sG_{x, y, \pi}^+$ if and only if $p > q$ for some $q \in \pi$, and  satisfies $p \in \sG_{x, y, \pi}^-$ if and only if $p < q$ for some $q \in \pi$. 

For $r \in \bbR$, the $r$ level in $\bbR^2$ is $\bbV_r = \{x \in \bbR^2 : x \cdot (e_1+e_2) = r\}$.  An up-right path 
$\geod{}$ on $\bbZ^2$  is indexed so that $\geod{n} \in \bbV_n$. The dual lattice will be denoted by $\bbZ^{2*}= \bbZ^2 + (1/2,1/2)$. We take the notational convention that if $\geod{}$ is a path in $\bbZ^{2*}$, $\geod{n} \cdot (e_1+e_2) = n + 1$.

 We define an ordering on $\bbV_r$ by $\zeta \preceq \eta$ if $\zeta,\eta \in \bbV_r$ and $\zeta\cdot e_1 \leq \eta \cdot e_1$. Similarly, $\zeta \prec \eta$ if $\zeta,\eta \in \bbV_r$ and $\zeta \cdot e_1< \eta \cdot e_1$.  Given a sequence of sites $v_n \in \bbV_1$, we define limsup and liminf using this ordering: $\varlimsup v_n = (\varlimsup v_n \cdot e_1, 1 - \varlimsup v_n \cdot e_1)$ and $\varlimsup v_n = (\varliminf v_n \cdot e_1, 1 - \varliminf v_n \cdot e_1)$.

For $0<\lambda<\infty$, $X \sim \Exp(\lambda)$  means that random variable $X$  has exponential distribution with rate $\lambda$:  $P(X>x)=e^{-\lambda x^+}$ for $x \in \bbR$. $X \sim \Exp(0)$ means that $X = \infty$ almost surely. We use the notational conventions $1/0 = \infty$,  $\infty/\infty = 1$, $x/\infty = 0$ for $x \in \bbR$. 

$a_{-\infty:\infty} = (a_i)_{i\in \bbZ}$ and the restriction of $a_{-\infty:\infty}$ to indices between $m$ and $n$ is denoted by $a_{m:n}$. We denote by $\smin{c}_{k: n} = \min_{k \le i \le n}c_i$ and $\sinf{c}_{n:\infty} = \inf_{i: i\ge n} c_i.$ The minimum of an empty sequence is infinity.

A Borel measure is non-zero if it is not the zero measure. Given a non-zero Borel measure $\mu$ on $\bbR$, the essential infimum under $\mu$ is denoted $\underline{\mu}$. The vague topology on Borel measures on $\bbR$ is the weak$^*$ topology generated by integrating against continuous functions that vanish at infinity. 

\subsection{Acknowledgements}
The authors thank Pierre Le Doussal for pointing out the connection to the phenomena observed in \cite{Kra-LeD-Oco-21} at the R\'enyi Institute's workshop on stochastic interacting particle systems and random matrices in 2025. We also thank two anonymous referees for helpful comments.

\section{Last-passage percolation with inhomogeneous exponential weights} \label{sec:CGM}

\subsection{Last-passage times}
Given a \textit{weight configuration} $\ww \in \bbR^{\bbZ^2}$, the associated last-passage times are defined by 
\begin{align}
\Lpt_{x, y} = \Lpt_{x, y}(\ww) = \max_{\pi \in \Pi_{x}^{y}} \bigg\{\sum_{p \in \pi} \ww_p\bigg\} \quad \text{ for } x, y \in \bbZ^2  \label{eq:Lpt}
\end{align}
where $\Pi_x^y$ is the set of all up-right paths (see Subsection \ref{SsNotCon}) $\pi$ on $\bbZ^2$ with $\min \pi = x$ and $\max \pi = y$. We define $\Lpt_{x, y} = -\infty$ if $x \le y$ fails. Last passage times can be computed through the following recursions, which are immediate from \eqref{eq:Lpt}. For $x,y$ with $x \le y$ and %hold for all 
$\ww \in\bbR^{\bbZ^2}$, 
\begin{align}
\label{eq:LptRec}
\begin{split}
\Lpt_{x, y} &= \ww_x + (\Lpt_{x+e_1, y} \vee \Lpt_{x+e_2, y})^+ = \ww_y + (\Lpt_{x, y-e_1} \vee \Lpt_{x, y-e_2})^+.
\end{split}
\end{align}
Throughout the paper, we  consider several different choices of the weights $\ww$ in coupling arguments. We will phrase results which hold for all $\ww \in \bbR^{\bbZ^2}$ in terms of $\Lpt$ and then introduce new notation for the process evaluated at randomly sampled $\ww$ as the paper progresses.
It will at times be important to note that $\Lpt_{x, y}$ only depends on the entries of $\ww$ indexed by the rectangle $\Rect_x^y$.

\subsection{Last-passage increments}

For $x \le y$, define the last-passage increments with respect to the initial point by 
\begin{align}
\init{\I}_{x, y} = \Lpt_{x, y}-\Lpt_{x+e_1, y} \quad \text{ and } \quad \init{\J}_{x, y} = \Lpt_{x, y}-\Lpt_{x+e_2, y}, \label{eq:incint}
\end{align}
and with respect to the terminal point by 
\be\begin{aligned}
\ter{\I}_{x, y} = \Lpt_{x, y}-\Lpt_{x, y-e_1} \quad \text{ and } \quad \ter{\J}_{x, y} = \Lpt_{x, y}-\Lpt_{x, y-e_2}.
\end{aligned}\label{eq:incbw}
\ee
Note that $\init{\I}_{x, y} = \infty = \ter{\I}_{x,y}$ and $\init{\J}_{x, y} = \infty = \ter{\J}_{x,y}$, respectively, when the inequalities $x+e_1 \le y$ and $x+e_2 \le y$ do not hold. 
From \eqref{eq:LptRec}, \eqref{eq:incint} and \eqref{eq:incbw}, one obtains the following increment recursions for $x+e_1+e_2 \le y$: 
\begin{align}
\init{\I}_{x, y} = \ww_x + (\init{\I}_{x+e_2, y}-\init{\J}_{x+e_1, y})^+, &\qquad  \init{\J}_{x, y} = \ww_x + (\init{\J}_{x+e_1, y}-\init{\I}_{x+e_2, y})^+, \quad \text{ and } \label{eq:inrec} \\ 
\ter{\I}_{x, y} = \ww_y + (\ter{\I}_{x, y-e_2}-\ter{\J}_{x, y-e_1})^+, &\qquad \ter{\J}_{x, y} = \ww_y + (\ter{\J}_{x, y-e_1}-\ter{\I}_{x, y-e_2})^+. \label{eq:bwrec}
\end{align}
One can also recover the initial and terminal weights from the increments as follows: 
\begin{align}
\label{eq:wrec}
\init{\I}_{x, y} \wedge \init{\J}_{x, y}  = \ww_x \quad \text{ and } \quad \ter{\I}_{x, y} \wedge \ter{\J}_{x, y} = \ww_y \quad \text{ for } x < y. 
\end{align}

\subsection{Inhomogeneous exponential LPP}
\label{sec:iExpLPP}

Consider bi-infinite sequences of real numbers, $a_{-\infty:\infty}$ and $b_{-\infty:\infty}$, which satisfy
\begin{align}
& \sinf{a}_{i:\infty} + \sinf{b}_{j:\infty} > 0 \quad \text{ for every } i,j \in \bbZ, \label{as:inf} \\
&\lim_{n \rightarrow \infty}\frac{1}{n} \sum_{k=1}^n \delta_{a_k} = \alpha, \quad \text{ and } \quad \lim_{n \rightarrow \infty}\frac{1}{n} \sum_{k=1}^n \delta_{b_k} = \beta, \label{as:weakcon}
\end{align}
where $\alpha$ and $\beta$ are non-zero subprobability measures on $\bbR$ and the limits hold in the vague topology. Note that \eqref{as:inf} is slightly weaker than the inequality $a_{-\infty:\infty}^{\inf} + b_{-\infty:\infty}^{\inf} > 0$. These assumptions are essentially minimal. We refer the reader to   \cite{Emr-Jan-Sep-21} for a discussion of the (stronger) hypotheses which have appeared previously in the literature.

Let $\{\ex_x: x \in \bbZ^2\}$ be independent $\Exp(1)$ random variables defined on a probability space $(\Omega,\sF,\bfP)$. For $x = (i,j) \in \bbZ^2$, write $\ex_x = \ex_{i,j}$ and define the \emph{weights}  $\w \in \bbR^{\bbZ^2}$ via
\begin{align}
\w_x = \w_{i,j} = \frac{\ex_{i,j}}{a_i+b_j}.  \label{eq:wdef}
\end{align}
Then the weights are independent and $\omega_{i,j} \sim \Exp(a_i+b_j)$. For $x,y \in \bbZ^2$, 
we introduce the 
{\it last passage times} by 
\begin{align}
\label{eq:LPP}
\LPP{x}{y} &= \Lpt_{x,y}(\w) = \max_{\pi \in \Path{x}{y}}\left\{\sum_{p \in \pi} \w_{p}\right\}, 
\end{align}
and their increments with respect to the initial point by 
\begin{align}
\label{eq:LPPInc}
\begin{split}
\I_{x, y} &= \init{\I}_{x, y}(\w) = \G_{x, y}-\G_{x+e_1, y}, \\ 
\J_{x, y} &= \init{\J}_{x, y}(\w) = \G_{x, y}-\G_{x+e_2, y}. 
\end{split}
\end{align}

\subsection{Limit shape and direction duality}
\label{sec:LimShp}
Note that \eqref{as:inf} implies that the sum of the essential infima of measures $\alpha$ and $\beta$ is positive: 
$\underline{\alpha} + \underline{\beta}>0$. For $z \in [-\underline{\alpha}, \underline{\beta}]$
and $\xi = (\xi_1,\xi_2) \in [e_2,e_1]$, we define 
\be\label{eq:statshape}
\begin{aligned}
\shp_z(\xi) &= \xi_1 \int_0^\infty \frac{\alpha(\dd a)}{a+z} +  \xi_2 \int_0^\infty \frac{\beta(\dd b)}{b-z} = \xi_1 \shor(z) + \xi_2 \sver(z), \qquad \text{ where }\\
\shor(z) &= \shp_z(e_1)  \qquad \text{and}\qquad \sver(z) = \shp_z(e_2).
\end{aligned}
\ee
Note that one of the integrals above may be infinite at each of the boundary points $z \in \{- \underline{\alpha}, \underline{\beta}\}$. For $x = (i,j) \in \bbZ^2$ and $\xi \in [e_2,e_1]$, we define the \textit{limit shape} to be
\begin{align}
\label{eq:shape}
\shp^x(\xi) = \inf_{-\sinf{a}_{i:\infty} <\, z\, < \sinf{b}_{j:\infty}}\left\{ \shp_z(\xi) \right\}= \shp_{\zmin{x}(\xi)}(\xi),
\end{align}
where $\zmin{x}(\xi)$ is the unique value of $z \in [- \sinf{a}_{i:\infty},  \sinf{b}_{j:\infty} ]$ for which the equality $\shp^x(\xi) = \shp_z(\xi)$ holds. A detailed study of the structure of this function and its appearance as the shape function in the last-passage percolation model described above appears in our previous paper \cite{Emr-Jan-Sep-21}. We summarize a handful of key properties which are important in this project. The main connection is the following shape theorem, which follows from Theorems 3.6 and 3.7 in \cite{Emr-Jan-Sep-21}.
\begin{proposition}\label{prop:shape}
The following holds $\bfP$-almost surely.  For all $x  \in \bbZ^2$ and all sequences $v_n \in \bbZ^2$ satisfying that\begin{align*}
\lim_{n\to\infty}v_n/n = \xi \in [e_2,e_1] \quad \text{ and } \quad \lim_{n\to \infty} v_n \cdot e_1 = \lim_{n\to \infty} v_n \cdot e_2 = \infty,
\end{align*}   
we have
\begin{align*}
\lim_{n\to\infty }\frac{\LPP{x}{v_n}}{n} = \shp^x(\xi).
\end{align*}
Moreover, if $x= (i,j)$ then for $m \geq i$ and $n \geq j$ fixed, 
\begin{align*}
\lim_{\ell\to\infty}\frac{\LPP{x}{(m,\ell)}}{\ell} = \int \frac{\beta(\dd b)}{b + \smin{a}_{i:m}},\qquad \text{ and } \qquad \lim_{k\to\infty}\frac{\LPP{x}{(k,n)}}{k} = \int \frac{\alpha(\dd a)}{a + \smin{b}_{j:n}}.
\end{align*}
\end{proposition}

In general, $\shp^x$ is homogeneous of degree one: for $c>0$, $\shp^x(c\xi)=c\shp^x(\xi)$. The variational expression for $\gamma^x$ in \eqref{eq:shape} defines a duality between  $z \in [- \sinf{a}_{i:\infty},  \sinf{b}_{j:\infty} ]$ and directions $\xi \in [e_2,e_1]$. Recall our convention that $1/\infty=0$; the form of the minimizer in the variational problem \eqref{eq:shape} leads us to define for %$x = (i,j) \in \bbZ^2$, and 
$z \in [-\underline{\alpha}, \underline{\beta}]$, a direction $\rho(z)\in [e_2,e_1]$ via 
\be
\begin{aligned} \label{eq:chardir}
\rho(z) \cdot e_1 &= \frac{\sver'(z)}{\sver'(z)-\shor'(z)} =  \frac{\int_0^\infty (b-z)^{-2}\beta(\dd b)}{\int_0^\infty (a+z)^{-2}\alpha(\dd a) + \int_0^\infty (b-z)^{-2}\beta(\dd b)}. 
\end{aligned}
\ee
Assumptions \eqref{as:inf} and \eqref{as:weakcon} and the hypothesis that neither $\alpha$ nor $\beta$ is the zero measure imply that the integrals in \eqref{eq:chardir} are positive and finite.

Calculus shows that for each $x=(i,j) \in \bbZ^2$, $\rho$ defines a differentiable bijection between  $(- \sinf{a}_{i:\infty}, \sinf{b}_{j:\infty})$ and  $]\mfc_{1}^x, \mfc_{2}^x[$, where $e_2 \preceq \mfc_{1}^x \prec \mfc_{2}^x \preceq e_1$ and the \textit{critical directions} $\mfc_1^x, \mfc_2^x \in [e_2,e_1]$ are given by
\be \label{eq:crit}
\begin{aligned}
\mfc_{1}^x \cdot e_1 &= \frac{\sver'(-\sinf{a}_{i:\infty})}{\sver'(-\sinf{a}_{i:\infty})-\shor'(-\sinf{a}_{i:\infty})} = \frac{\int_0^\infty (b+ \sinf{a}_{i:\infty})^{-2}\beta(\dd b) }{\int_0^\infty (a- \sinf{a}_{i:\infty})^{-2}\alpha(\dd a) +\int_0^\infty (b+ \sinf{a}_{i:\infty})^{-2}\beta(\dd b)  }, \\
 \mfc_{2}^x \cdot e_1 &= \frac{\sver'(\sinf{b}_{j:\infty})}{\sver'(\sinf{b}_{j:\infty})-\shor'(\sinf{b}_{j:\infty})} = \frac{\int_0^\infty (b- \sinf{b}_{j:\infty})^{-2}\beta(\dd b)}{\int_0^\infty (a+ \sinf{b}_{j:\infty})^{-2}\alpha(\dd a) + \int_0^\infty (b- \sinf{b}_{j:\infty})^{-2}\beta(\dd b)}. 
\end{aligned}
\ee
Recalling the notational convention $1/0=\infty$, the cases  $\mfc_{1}^x = e_2$ and $\mfc_{2}^x = e_1$ are equivalent to the conditions $\int_0^\infty (a- \sinf{a}_{i:\infty})^{-2}\alpha(\dd a) = \infty$ and $\int_0^\infty (b- \sinf{b}_{j:\infty})^{-2}\beta(\dd b) = \infty$, respectively. We see immediately from \eqref{eq:crit} that if $x_1 \cdot e_1 \leq x_2 \cdot e_1$ and $y_1 \cdot e_2 \leq y_2 \cdot e_2$ then 
\begin{align}
\mfc_{1}^{x_2}\preceq \mfc_{1}^{x_1} \quad \text{ and } \quad \mfc_{2}^{y_1} \preceq\mfc_{2}^{y_2}. \label{eq:critineq}
\end{align} 
With the notation $\zmin{x}(\xi)$ from \eqref{eq:shape} and the fact that $\rho$ is an invertible map on $]\mfc_{1}^x,\mfc_{2}^x[$, some calculus gives that
\begin{align}
\label{eq:zetaext}
\zmin{x}(\xi) = 
\begin{cases}
- \sinf{a}_{i:\infty} \quad &\text{ for } \xi \in[e_2, \mfc_{1}^x]  \\
(\rho)^{-1}(\xi) & \xi \in ]\mfc_{1}^x,\mfc_{2}^x[ \\
 \sinf{b}_{j:\infty} \quad &\text{ for } \xi \in [\mfc_{2}^x, e_1]
\end{cases}.
\end{align}
From the above observations, it is straightforward to see that $\shp^x(\bbullet)$ is strictly concave on the (non-degenerate) interval $]\mfc_1^x,\mfc_2^x[$ and linear on the (possibly degenerate) intervals $[e_2,\mfc_1^x]$ and $[\mfc_2^x,e_1]$. See Figure \ref{fig:shapex} for an example. 

\begin{figure}
\includegraphics[scale=.75]{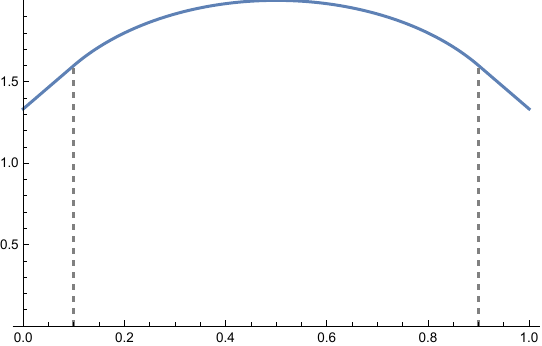}
\caption{$\xi_1 \mapsto \shp^{(0,0)}(\xi_1,1-\xi_1)$  for $\xi_1\in[0,1]$, where $\alpha(da) = \delta_{1/2}(da)$ and $\beta(db) = \delta_{1/2}(db)$ are both Dirac masses at $1/2$ and $\sinf{a}_{0:\infty} = \sinf{b}_{0:\infty} = 1/4$. Here, $\mfc_1^{(0,0)}=(1/10,9/10)$ and $\mfc_2^{(0,0)}=(9/10,1/10)$, so the depicted shape is linear for $\xi_1 \in [0,1/10]\cup [9/10,1]$ and strictly concave for $\xi_1 \in (1/10,9/10)$. This is the shape function from every lattice site if $a_n=b_n=1/2$ except for an infinite forward density zero set of $n\in\bbZ$, on which both are equal to $1/4$.}\label{fig:shapex}
\end{figure}

\subsection{Geodesics and competition interfaces}
Given a fixed $\ww \in \bbR^{\bbZ^2}$, a path $\geod{} \in \Path{x}{y}$ is called a (finite) \emph{geodesic} from $x \in \bbZ^2$ to $y \in \bbZ^2$ if $\geod{}$ is a maximizer in \eqref{eq:Lpt}.  The models we study have weights which are independent and have continuous distributions and therefore there is an event of full probability on which there is a unique geodesic between $x$ and $y$ for each pair $x,y \in \bbZ^2$ with $x \le y$. Some of our results concern \textit{semi-infinite geodesics}, which are up-right paths which have a first site but no last site and have the property that every finite subpath is a (finite) geodesic between its endpoints.

Given $x,y \in \bbZ^2$ with $x \cdot (e_1+e_2) = k$, $y \cdot (e_1+e_2) = n$, and $x \leq y$, if geodesics are unique in the environment $\ww$, the unique geodesic $\geod{}$ from $x$ to $y$ evolves according to the following local rules: $\geod{k} = x$ and for $\ell < n$,
\be
\geod{\ell+1} = \begin{cases}
\geod{\ell} + e_1 & \text{ if } \Lpt_{\geod{\ell}+e_1,y} > \Lpt_{\geod{\ell}+e_2,y}\\
\geod{\ell} + e_2 & \text{ if } \Lpt_{\geod{\ell}+e_1,y} < \Lpt_{\geod{\ell}+e_2,y}
\end{cases} = \begin{cases}
\geod{\ell} + e_1 & \text{ if } \init{\I}_{\geod{\ell}, y} < \init{\J}_{\geod{\ell}, y} \\
\geod{\ell} + e_2 & \text{ if } \init{\J}_{\geod{\ell}, y} < \init{\I}_{\geod{\ell}, y} \\
\end{cases}.\label{eq:geoloc}\ee

Similarly, it follows from the uniqueness of finite geodesics in the environment $\ww$ that for each site $x \in \bbZ^2$, the collection of geodesics from $x$ to the sites $y \in \bbZ^2$ with $y \geq x$ forms a tree, which we denote by $\sT_x$. Each such geodesic rooted at $x$ either passes through $x+e_1$ or $x+e_2$. This splits $\sT_x$ into two subtrees, $\sT_{x,x+e_1}$ and $\sT_{x,x+e_2}$, which can be thought of as competing infections. The {\it competition interface} is a dual lattice  path $\Cifa^{x}$ (living on the dual lattice $\bbZ^2 + (1/2,1/2))$ which separates them. It is defined by setting $\Cifa_k^{x}=x + (1/2,1/2)$ and then evolving according to the following rules for $n \geq k$: 
\be   
\Cifa_{n+1}^x=  
\begin{cases}  
\Cifa_n^x+e_1, &\Lpt_{x,\Cifa_n^x -(1/2,1/2)+e_1} < \Lpt_{x, \Cifa_n^x -(1/2,1/2)+e_2} \\
\Cifa_n^x+e_2, &\Lpt_{x,\Cifa_n^x -(1/2,1/2) +e_1} > \Lpt_{x,\Cifa_n^x -(1/2,1/2)+e_2} . 
 \end{cases}  \label{eq:cifdef}
\ee

%See Figure \ref{fig:CIF} for a simulation. 
From this definition, one checks inductively that 
$\Cifa_n^x$ is the unique point $x$ on the line segment $\bbL^x_n=\{y \geq x : y \cdot (e_1+e_2) = n\}$ such that 
\[  
x+  (1/2,1/2) + e_1\Z_{>0}\subset\sT_{x,x+e_1}\quad\text{and}\quad 
x +  (1/2,1/2)+e_2\Z_{>0}\subset\sT_{x,x+e_2}. \] 
\begin{figure}[H]
\includegraphics[scale=1]{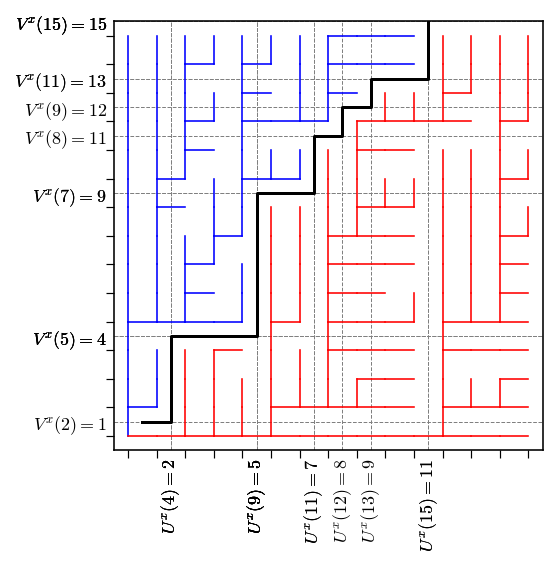}
\caption{A simulation of the geodesic tree rooted at $x=(1,1)$, separated into two subtrees $\sT_{x,x+e_1}$ (red) and $\sT_{x,x+e_2}$ (blue) on a $15 \times 15$ grid. The competition interface $\Cifa^{x}$ (black) is plotted on the dual lattice and the locations $\Cif^x$ and $\Cifs^x$ at which levels are first reached are labeled at the upper endpoint of each interval on which they are constant. In this simulation, $\Cif^x$ is %constant and 
equal to $2$ on $[2,4]$, $5$ on $[5,9]$, $7$ on $[10,11]$, and  $11$ on $[14,15]$. $\Cifs^x$ is %constant and 
equal to $4$ on $[2,5]$, $9$ on $[6,7]$, and $15$ on $[14,15]$.}\label{fig:CIF}
\end{figure}

It will be convenient to track competition interfaces through the locations where they pass horizontal and vertical lines. We define, for $n> x\cdot e_2$, 
\be\label{eq:cifslope}  \begin{aligned}
\Cif^x(n)&=\sup\{m\in\Z_{\ge x \cdot e_1}:  \Lpt_{x+e_2,(m,n)} > \Lpt_{x+e_1,(m,n)}\} \\
&=\min\{ m: (m+1/2,n+1/2)\in\Cifa^x \} =\max\{ m: (m+1/2,n-1/2)\in\Cifa^x \}.
\end{aligned}  \ee
$\Cif^x(n)$ tracks the first coordinate of the point at which the competition interface first reaches the horizontal level of index $n$. The symmetric counterpart is defined for $m> x \cdot e_1$ by
\be\label{def:Cifs}   \begin{aligned}
\Cifs^{x}(m)&=\sup\{n\in\Z_{\ge x \cdot e_2}: \Lpt_{x+e_1,(m,n)} > \Lpt_{x+e_2,(m,n)} \} \\
&=\min\{ n: (m+1/2,n+1/2)\in\Cifa^x\} =\max\{ n: (m-1/2,n+1/2)\in\Cifa^x \}.
\end{aligned}  \ee
$\Cifs^{x}(m)$ tracks the point at which the competition interface first reaches the vertical level of index $m$.

It follows from Lemma \ref{lem:incineq} below that $\Cif^x(n)$ and  $\Cifs^{x}(m)$ are both monotone non-decreasing in $n > x \cdot e_1$ and $m > x \cdot e_2$ respectively. We denote the limits by
\begin{align}
\lim_{n \to \infty}\Cif^x(n) = \Cif^x(\infty) \qquad \text{ and } \qquad \lim_{m \to \infty} \Cifs^{x}(m) = \Cifs^{x}(\infty).
\end{align}

\subsection{Inhomogeneous TASEP}\label{sec:inTASEP}
With certain initial conditions, there is a bijective correspondence between the inhomogeneous exponential CGM discussed above and an inhomogeneous generalization of the totally asymmetric simple exclusion process (TASEP). This correspondence comes from the seminal work of Rost \cite{Ros-81}. 

TASEP is a model typically defined on the state space $\{0,1\}^{\bbZ}$, describing the evolution of infinitely many particles, represented by $1s$, and holes, represented by $0s$, on the lattice $\Z$.  Particles always march to the right and holes to the left. We restrict attention to initial conditions with infinitely many particles and holes, where there is a rightmost particle and a leftmost hole.  For such initial conditions, we index  particles and holes by $\Z_{>0}$. At time $t\in\bbR_{\ge0}$,  $H_i(t)$ is the position of hole $i$ and  $P_j(t)$  the position of particle $j$,  for $i,j\in\Z_{>0}$.   Holes are labeled from left to right, so that for all $i \in \bbZ_{>0}$ and for all $t \in \bbR_{\geq0}$, $H_i(t)<H_{i+1}(t)$.  Particles move from right to left and we have $P_{j+1}(t)<P_j(t)$ for all $j \in \bbZ_{>0}$ and $t \in \bbR_{\geq0}$. The system evolves according to the following rules: once hole $i$ lies immediately to the right of particle $j$, i.e.~$P_j=H_i-1$,   they switch positions at exponential rate $a_i+b_j$ to become $H_i=P_j-1$. The process can be realized through a Harris-type construction by attaching to hole $i$ a Poisson clock with rate $a_i$ and to particle $j$ a Poisson clock with rate $b_j$. In this construction, whenever a particle is immediately to the left of a hole, they interchange places if either of their Poisson clocks rings. The hypothesis above on the initial condition ensures that this construction is well-defined; at any given time, one only needs to keep track of finitely many Poisson clocks to determine the next jump.

Consider the initial  configuration 
\be\label{t:67}   \begin{cases}  P_1(0)=1\\ P_j(0)=1-j\text{ for } j\ge 2  \end{cases} 
\quad\text{and}\quad 
 \begin{cases}  H_1(0)=0\\ H_i(0)=i\text{ for } i\ge 2.  \end{cases}  \ee
 
 If at time $t$ hole $i$ and particle $j$ are adjacent in either order, they occupy sites $i-j$ and $i-j+1$, i.e.~$\{H_i(t), P_j(t)\}=\{ i-j, i-j+1\}$.  One can check inductively that this property is preserved by every particle-hole interchange.  
  Each particle-hole pair $(P_j,H_i)$  exchanges  positions exactly once   
  during the evolution to become a  hole-particle pair  $(H_i,P_j)$. 

 The {\it *pair} (``star pair'') is a hole-particle pair in the process whose  moves are dictated by the underlying particle evolution. It was introduced in \cite{Fer-Pim-05} to encode the evolution of a second-class particle in TASEP.  At time $t\in\bbR_{\ge0}$, we denote by  $(H^*(t), P^*(t))$ the position of the *pair and let $I(t)$ and $J(t)$ denote  the hole and particle  indices of the *pair. Initially $(H^*(0), P^*(0))=(0,1)$ and  $(I(0),J(0))=(1,1)$. The underlying particle dynamics are as described above and the *pair evolves within these dynamics as follows: whenever a particle interchanges with the hole in the *pair, the *pair moves one unit to the left and whenever the particle in the *pair interchanges with a hole, the *pair moves one unit to the right. These moves can be represented schematically as below, where 0 denotes a hole, 1 denotes a particle, and $(0 \  1)^*$ denotes the *pair: 
 \begin{align}
\label{star1}  \text{*pair moves left:} \  \ &\text{from}  \quad  1 \ \ (0\ \ 1)^*  \quad \text{to}\quad (0\ \ 1)^*\ 1 \\
\label{star2}    \text{*pair moves right:} \ \  &\text{from}  \quad   (0\ \ 1)^*\ 0  \quad \text{to}\quad  0 \ \ (0\  \ 1)^* .
 \end{align}
We see that for all $t \in \bbR_{\geq0},$ $(H^*(t), P^*(t))=(H_{I(t)}(t), P_{J(t)}(t))=(I(t)-J(t), I(t)-J(t)+1)$.

In two-class TASEP, particles are either labelled as first-class or second-class. Whenever a first-class particle is to the immediate left of a second-class particle, the pair interchange as if the second-class particle were a hole in the discussion above. Otherwise, the dynamics proceed exactly as above. A mapping from the *pair to a second-class particle which is valid in our setting is given in  \cite[Lemma 6]{Fer-Pim-05}. 
 
 \begin{lemma}\label{lem:fpcouple} \cite[Lemma 6]{Fer-Pim-05} There is a coupling of two-type TASEP with initial condition where all of the sites $x \leq -1$ are occupied by first-class particles and the particle at $0$ is a second-class particle to the process described above in which the location of the second-class particle $X(t)$ is equal to the difference $I(t) - J(t)$ for all $t \geq 0$.
 \end{lemma}

Now, let $\Cifa=\Cifa^{(1,1)}$ denote the competition interface rooted at $(1,1)$ and let $\Cifa^* = \Cifa - (1/2,1/2)$, so that $\Cifa^*_2 = (1,1)$.  Call $\tau_n=\LPP{(1,1)}{\Cifa_n^*}-\w_{(1,1)}$ (i.e., the passage time with the first weight removed so that $\tau_2 = 0$) and define a continuous-time extension of the competition interface by 
\be\label{cif-456} \Cifb_t
=\Cifa_n^* \quad\text{for }  t\in[\tau_n, \tau_{n+1}), \ n\in\Z_{\ge2}. \ee
As discussed around \cite[(21)]{Fer-Pim-05}, a consequence of the coupling in \cite[Lemma 6]{Fer-Pim-05} is the following lemma.
\begin{lemma}\label{fp-lem} 
$\bigl((I(t),J(t)):t\in\bbR_{\ge0}\bigr)$  has the same distribution as $\bigl(\Cifb(t):t\in\bbR_{\ge0}\bigr)$.   
\end{lemma} 
Thus, our results on competition interfaces will have immediate consequences for the behavior of second-class particles in TASEP.

\subsection{TAZRP and inhomogeneous queues}\label{sec:TAZRP}
We now specialize the rates to $a_i=0$ for all $i$, so that holes become indistinguishable. This extra hypothesis is needed only to simplify the interpretation of model we now introduce. The totally asymmetric zero range process (TAZRP) can be interpreted as a series of memoryless continuous-time queues, with service stations labelled $j=1,2,3,\dots$ that  carry service rates $b_1,b_2,b_3,\dots$. In the model we study, customers come in two types: first-class and second-class. First-class customers are always served before second-class customers who are waiting in the same queue.   Customers of the same type are indistinguishable. 

Rigorously, we define the inhomogeneous TAZRP $\eta(t)$ through a coupling as a function of the inhomogeneous TASEP described above by letting $\eta_j(t)=P_{j-1}(t)-P_j(t)-1$ for $j\in\Z_{>0}$. In words, the number of holes between TASEP  particle locations $P_j(t)<P_{j-1}(t)$ is the number of customers at station $j$ at time $t$. A jump of TASEP particle $j$ at time $t$  ($P_j(t)=P_j(t-)+1$) is a departure from server $j$ at time $t$ and a simultaneous arrival at server $j+1$. We add an extra TASEP particle $P_0(t)\equiv\infty$ at infinity to have $\eta_1(t)\equiv\infty$, corresponding to the assumption that there are initially infinitely many customers in the queue at station $1$.

The initial condition \eqref{t:67} corresponds to one where there is a single second-class customer in the queue at station 2 and, as noted above, infinitely many first-class customers in line at station 1. We denote the location of the second-class customer at time $t$ by $Z(t)$. The location of this second-class customer is tracked by the *pair. This is recorded in the next lemma, which can be verified straightforwardly jump-by-jump in the coupling.

\begin{lemma} \label{2cl-star-lm}  At time $t\in\bbR_{\ge0}$,
 $Z(t)=J(t)+1$ and $I(t)-1$ is the number of  first-class customers that have passed the second-class customer by time $t$.   
\end{lemma} 
%\begin{proof} 
% The lemma is true for the initial values  $(I(0),J(0))=(1,1)$.   Suppose $(I(t-),J(t-))=(i,j)$.   Then $(H^*(t-), P^*(t-))=(H_i(t-), P_j(t-))=(i-j, i-j+1)$.  Assume inductively that  the lemma holds at time $t-$: namely,   the second-class customer is at station $j+1$ and has been passed by $i-1$ first-class customers.   Consider the two possible  jumps of the   *pair at time $t$.  
%\begin{itemize}
%\item The *pair jumps right at rate $b_j$ provided $H_{i+1}(t-)=i-j+2$ (this hole is immediately to the right of  the *particle $P_j(t-)$).   After the jump  $(I(t),J(t))=(i+1,j)$.   This means that a customer departed server $j$ and arrived at station $j+1$.  This new customer passed the second-class customer and correspondingly $I(t)$ increased by 1. 
%\item  The *pair jumps left at rate $b_{j+1}$ provided $P_{j+1}(t-)=i-j-1$ (this particle is immediately to the left of  the *hole  $H_i(t-)$).   This means that  at time $t-$ the second-class customer was the only customer at station $j+1$.  The jump means that the second-class customer got served at time $t$, and moved to station $j+2$.   After the jump  $(I(t),J(t))=(i,j+1)$.   
%\end{itemize} 
%Thus the lemma  is verified inductively jump by jump.  \end{proof} 
%
For some intuition on how the *pair tracks the second-class customer, note that the holes in the range $P_j+1,\dotsc, P_{j-1}-1$  represent the customers at station $j$ in the order in which they will be served.   
%The hole $H^*(t)=H_{I(t)}(t)$ of the *pair   represents the  second-class customer. 
The fact that $H^*(t)=H_{I(t)}(t)$ is always adjacent to $P^*(t)=P_{J(t)}(t)$ implies that the  second-class customer is always the {\it last} customer in the queue at station $J(t)+1$.
   
%Because $I(t)$ evolves among the hole indices,    we cannot link   the  TASEP holes with customers in a  fixed manner.
%If $H_i(t-) = P_j(t-)+1$ and $P_j$ and $H_i$ exchange positions  at time $t$, it does make sense to think of hole $i$ as the customer  served by station $j$ at time $t$.  But if also $P_j(t-)=P^*(t-)$,   then after the jump hole $i$ is part of the *pair and hence now represents the second-class customer. The customer that just arrived is now represented by hole $i-1$.  In other words,  we can identify  each first-class customer with a particular hole with these rules:
%\begin{itemize}  
%\item   at time zero first-class customers are attached to holes $H_2, H_3, H_4,\dotsc$. 
%\item  when a first-class  customer passes the second-class customer, the  hole index of the first-class  customer decreases by 1. 
%\end{itemize} 
%At time $t$ the  first-class customers are attached to holes $H_1, \dotsc, H_{I(t)-1}, H_{I(t)+1}, \dotsc$. 

Through a combination of Lemmas \ref{fp-lem} and \ref{2cl-star-lm},   results on the  competition interfaces   have immediate consequences for the behavior of second-class customers in a series inhomogeneous queues. 

\section{Results} \label{sec:res}
\subsection{Busemann functions}
\label{sub:buslim}

Our main tool throughout the paper will be the stochastic process of Busemann functions, called the \textit{Busemann process}. Inhomogeneity leads to two different types of Busemann function: those coming from the bulk of the model (i.e.~strictly concave directions), which are similar to the Busemann functions in the homogeneous model studied in \cite{Cat-Pim-12,Cat-Pim-13,Geo-Ras-Sep-17-ptrf-2,Geo-Ras-Sep-17-ptrf-1,Jan-Ras-Sep-22-,Sep-18} and those coming from the coordinate (i.e.~$\{e_1,e_2\}$) boundaries. The need to separate out the coordinate boundaries is a consequence of the fact that in the coordinate directions $e_1$ and $e_2$, the Busemann limit, recorded below as Theorem \ref{thm:buslim}\eqref{thm:buslim:buslim}, is no longer independent of the sequence of terminal points approximating the direction. The same dependence on the approximating sequence of terminal points is also present in the shape theorem in this setting, as can be seen from Proposition \ref{prop:shape}.

It is convenient to introduce the following notation, which keeps track of the indices where the running minimum of the parameter sequences $a_{-\infty:\infty}$ and $b_{-\infty:\infty}$ change ahead of a site $x=(i,j)\in\bbZ^2$:  
\begin{align}
\label{eq:rec}
\begin{split}
\ir{k}{x} &= \inf\{i \in \bbZ : x \cdot e_1 \leq i < k+1, a_{i} = \sinf{a}_{x \cdot e_1:k}\} \\ 
\jr{\ell}{x} &= \inf\{j \in \bbZ : x \cdot e_2 \leq j < \ell+1, b_{j} = \sinf{b}_{(x \cdot e_2): \ell}\}.
\end{split}
\end{align}
As suggested by working with infima rather than minima, will use this notation when $k=\infty$ as well, in cases where the global minimum of the parameter sequence is (first) achieved. Some statements below use the observation that $\ir{\infty}{x}\notin \bbZ$ means that the running minimum ahead of $x$ changes infinitely often, with a similar statement for $\jr{\infty}{x} \notin\bbZ$. We also recall the notation $\mfc_i^x$ for the critical directions from \eqref{eq:crit} and the optimizer in the variational expression for the limit shape $\zmin{x}(\xi)$ from \eqref{eq:zetaext}.

With this notion in mind, our next result collects the main properties of the Busemann process. 

\begin{theorem}
\label{thm:buslim}
There exists an %the following %sets of 
$\cl{\bbR}$-valued stochastic process%es
\begin{align*}
&\{\buse{x}{y}{(k,\infty)}:  k \in \bbZ, x,y \in \bbZ_{\le k} \times \bbZ\} \cup \{\buse{x}{y}{(\infty,\ell)}: \ell \in \bbZ, x,y \in \bbZ \times \bbZ_{\le \ell}\} \\ 
&\qquad \qquad \cup \{\buse{x}{y}{\xi\pm}: \xi \in [e_2,e_1], x,y \in \bbZ^2\}
\end{align*}
with the following properties on a single event of $\P$-probability one. Let %$x, y, z \in \bbZ^2$, 
$\xi \in [e_2, e_1]$, $k, \ell \in \bbZ$ and $\square \in \{\xi-, \xi+, (k, \infty), (\infty, \ell)\}$. In the statements below, expressions of the form $\B_{x, y}^{(k, \infty)}$ and $\B_{x, y}^{(\infty, \ell)}$ tacitly assume that $(x \vee y) \cdot e_1\le k$ and $(x \cdot y) \cdot e_2 \le \ell$, respectively. Also, when $\buse{x}{y}{\xi+}=\buse{x}{y}{\xi-}$, we write $\buse{x}{y}{\xi}$ for this common value. 
\begin{enumerate}[label={\rm(\alph*)}, ref={\rm\alph*}] \itemsep=3pt
\item \label{thm:buslim:pos} {\rm{(Positivity)}}. The following statements hold for $x, y \in \bbZ^2$ with $x \le y$. 
\begin{enumerate}[label={\rm(\roman*)}, ref={\rm \roman*}] \itemsep=3pt
\item \label{thm:buslim:pos:Bxx=0} $\buse{x}{x}{\square} = 0$. 
\item \label{thm:buslim:pos:Bxyneq0} $\buse{x}{y}{\square} > 0$ if $x \neq y$. 
\item \label{thm:buslim:pos:Bxy=infty} $\buse{x}{y}{\square} = \infty$ if and only if  
\begin{align*}
&\square = (k, \infty) \ \text{ and } \ \ir{k}{x} < y \cdot e_1,  \ \text{ or } \\
&\square = (\infty, \ell) \ \text{ and } \ \jr{\ell}{x} < y \cdot e_2, \ \text{ or } \\ 
&\square \in \{\xi+, \xi-\}, \ \xi \in [e_2, \mfc_1^x] \ \text{ and } \ \ir{\infty}{x} < y \cdot e_1, \ \text{ or } \\  
&\square \in \{\xi+, \xi-\}, \ \xi \in [\mfc_2^x, e_1] \ \text{ and } \jr{\infty}{x} < y \cdot e_2. 
\end{align*}
\end{enumerate}

\item\label{thm:buslim:decomp} {\rm(Cocycle)}. The following statements hold for $x, y \in \bbZ^2$. 
\begin{enumerate}[label={\rm(\roman*)}, ref={\rm \roman*}] \itemsep=3pt
\item \label{thm:buslim:decomp:upright} If $x \le y$ then, for any up-right path $\pi \in \Pi_x^y$, 

\begin{align*}
\buse{x}{y}{\square} = \sum_{p \in \pi: p+e_1 \in \pi} \buse{p}{ p+e_1}{\square} + \sum_{p \in \pi: p+e_2 \in \pi} \buse{p}{p+e_2}{\square}. 
\end{align*}
\item \label{thm:buslim:decomp:ord} $\buse{x}{y}{\square} = \buse{x \wedge y}{y}{\square}-\buse{x \wedge y}{x}{\square}$.
\item \label{thm:buslim:decomp:sym} $\buse{y}{x}{\square} = -\buse{x}{y}{\square}$. 
\end{enumerate}
\item\label{thm:buslim:recovery} {\rm(Recovery)}.   
$\w_x =  \buse{x}{x+e_1}{\square} \wedge \buse{x}{x+e_2}{\square}$ for $x \in \bbZ^2$.  
\item \label{thm:buslim:recursion} {\rm{(Recursion)}}. For $x \in \bbZ^2$, 
\begin{align*}
\buse{x}{x+e_1}{\square} &= \w_x + (\buse{x+e_2}{x+e_1+e_2}{\square}-\buse{x+e_1}{x+e_1+e_2}{\square})^+, \\ %\quad \text{ and } \\
\buse{x}{x+e_2}{\square} &= \w_x + (\buse{x+e_1}{x+e_1+e_2}{\square}-\buse{x+e_2}{x+e_1+e_2}{\square})^+. 
\end{align*}
\item \label{thm:buslim:exdir}{\rm(Exceptional directions)}. 
For each $x \in \bbZ^2$, the random set
\begin{align*} \eset_{x}%^{\omega} 
= \big\{\eta \in [e_2,e_1]\, :\, \buse{x}{x+e_1}{\eta+} \neq \buse{x}{x+e_1}{\eta-} \, \text{ or }\,\buse{x}{x+e_2}{\eta+} \neq \buse{x}{x+e_2}{\eta-}\big\} \end{align*}
satisfies the following properties. 
\begin{enumerate}[label={\rm(\roman*)}, ref={\rm \roman*}] \itemsep=3pt
\item $\Lambda_x$ is countable. 
\item %For all $x \in \bbZ^2$, 
$\eset_{x} \subseteq\, ]\mfc_1^x,\mfc_2^x[$.  
\item $\bfP(\eta \in \eset_{x})=0$ for each $\eta \in [e_2, e_1]$. 
% for each $\xi \in [e_2,e_1], \bfP(\xi \in \eset_{x})=0$.
\end{enumerate}  
\item \label{thm:buslim:linconst} {\rm(Constant on linear segments)}. For $x, y \in \bbZ^2$ with $x \le y$, $\buse{x}{y}{\xi} = \buse{x}{y}{\mfc_1^x}$ if $\xi \in [e_2, \mfc_1^x]$ and $\buse{x}{y}{\xi} = \buse{x}{y}{\mfc_2^x}$ if $\xi \in [\mfc_2^x, e_1]$.
%\begin{align*}
%\buse{x}{y}{\xi} &= \buse{x}{y}{\mfc_1^x} \quad \text{ if } \xi \in [e_2, \mfc_1^x], \\
%\buse{x}{y}{\xi} &= \buse{x}{y}{\mfc_2^x} \quad \text{ if } \xi \in [\mfc_2^x, e_1]. 
%\end{align*}
\item \label{thm:buslim:buslim} {\rm(Busemann limits away from the axes)}. Recall definition \eqref{eq:LPPInc} of the increment variables. The following statements hold for $x, y \in \bbZ^2$ and any sequence $(v_n)$ on $\bbZ^2$ such that $n^{-1}v_n \to \xi$ and $\min \{v_n \cdot e_1, v_n \cdot e_2\} \to \infty$ as $n \to \infty$.   
\begin{enumerate}[label={\rm(\roman*)}, ref={\rm \roman*}] \itemsep=3pt
\item 
$
\begin{aligned}[t]
\varliminf_{n \to \infty}\I_{x, v_n} &\ge \buse{x}{x+e_1}{\xi+}, \quad \quad \quad &&\varlimsup_{n \to \infty}\I_{x, v_n} \le \buse{x}{x+e_1}{\xi-}, \\
\varliminf_{n \to \infty} \J_{x, v_n} &\ge \buse{x}{x+e_2}{\xi-}, \quad \text{ and } \quad &&\varlimsup_{n \to \infty}\J_{x, v_n} \le \buse{x}{x+e_2}{\xi+}.  
\end{aligned}
$
\item If $\xi \not \in \Lambda_p$ for $p \in \Rect_{x \wedge y}^{x \vee y}$ then  
\begin{align*}
\lim_{n \to \infty}\{\G_{x, v_n} - \G_{y, v_n}\} = \buse{x}{y}{\xi}.
\end{align*}
\end{enumerate}
\item \label{thm:buslim:thinbuslim} {\rm{(Busemann limits near the axes)}}. 
For $x, y \in \bbZ^2$, 
\begin{align*}
\lim_{n\to\infty} \{\G_{x, (k, n)}-\G_{y, (k, n)}\} &= \buse{x}{y}{(k,\infty)} \quad \text{ if } \min \{x \cdot e_1, y \cdot e_1\} \le k, \\ 
\lim_{n\to\infty} \{\G_{x, (n, \ell)}-\G_{y, (n, \ell)}\} &= \buse{x}{y}{(\infty,\ell)} \quad \text{ if } \min\{x \cdot e_2, y \cdot e_2\} \le \ell. 
\end{align*}
\item \label{thm:buslim:mono} {\rm(Monotonicity)}. 
For $x \in \bbZ^2$, $k', \ell' \in \bbZ$ and $\zeta, \eta \in [e_2, e_1]$ such that $x \le (k, \ell) \le (k', \ell')$ and $\zeta \prec \eta$, 
\begin{align*}
\buse{x}{x+e_1}{(\infty, \ell)} \le \buse{x}{x+e_1}{(\infty, \ell')} \leq \buse{x}{x+e_1}{\eta+} \leq \buse{x}{x+e_1}{\eta-} \leq \buse{x}{x+e_1}{\zeta+} \leq \buse{x}{x+e_1}{\zeta-} \leq \buse{x}{x+e_1}{(k', \infty)} \leq  \buse{x}{x+e_1}{(k, \infty)} , \\
\buse{x}{x+e_2}{(\infty, \ell)} \ge \buse{x}{x+e_2}{(\infty, \ell')} \ge \buse{x}{x+e_2}{\eta+} \geq \buse{x}{x+e_2}{\eta-} \geq \buse{x}{x+e_2}{\zeta+} \geq \buse{x}{x+e_2}{\zeta-} \ge \buse{x}{x+e_2}{(k', \infty)} \ge \buse{x}{x+e_2}{(k, \infty)}.  
\end{align*}
\item \label{thm:buslim:thinconst} {\rm(Case of equality for thin Busemann functions)}. The following statements hold for $x \in \bbZ^2$. %and $i \in \{1, 2\}$. 
\begin{enumerate}[label={\rm(\roman*)}, ref={\rm \roman*}] \itemsep=3pt
\item $\buse{x}{x+e_1}{(k, \infty)} = \buse{x}{x+e_1}{(\ir{k}{x}, \infty)}$ if $x \cdot e_1 \le k-1$, and $\buse{x}{x+e_2}{(k, \infty)} = \buse{x}{x+e_2}{(\ir{k}{x}, \infty)}$.
\item $\buse{x}{x+e_2}{(\infty, \ell)} = \buse{x}{x+e_2}{(\infty, \jr{\ell}{x})}$ if $x \cdot e_2 \le \ell-1$, and $\buse{x}{x+e_1}{(\infty, \ell)} = \buse{x}{x+e_1}{(\infty, \jr{\ell}{x})}$.  
\end{enumerate}
\item \label{thm:buslim:cont} {\rm(Directional continuity)}. The following statements hold for $x, y \in \bbZ^2$. 
\vspace{0.1in}
\begin{enumerate}[label={\rm(\roman*)}, ref={\rm \roman*}] \itemsep=3pt
\item  
$\lim \limits_{\eta \uparrow \xi}\buse{x}{y}{\eta+}=\lim \limits_{\eta \uparrow \xi}\buse{x}{y}{\eta-} = \buse{x}{y}{\xi-}$ \quad and \quad $\lim \limits_{\eta \downarrow \xi}\buse{x}{y}{\eta+}  =  \lim \limits_{\eta \downarrow \xi}\buse{x}{y}{\eta-}  = \buse{x}{y}{\xi+}$. 
\item 
If $x \le y$ then $\lim \limits_{k \to \infty} \buse{x}{y}{(k, \infty)} 
= \buse{x}{y}{\mfc_1^x}$ \quad and \quad 
$\lim \limits_{\ell \to \infty} \buse{x}{y}{(\infty, \ell)}
=\buse{x}{y}{\mfc_2^x}$. 
\end{enumerate}
\item \label{thm:buslim:marg} {\rm(Marginals)}
For  $x = (i, j) \in \bbZ^2$,
\begin{align}
\label{eq:busmar}
\begin{split}
&\buse{x}{x+e_1}{\square} \sim 
\begin{cases}
\Exp\{a_{i}-\smin{a}_{i: k}\} \quad &\text{ if } \square = (k, \infty), \\  
\Exp\{a_i- \sinf{a}_{i: \infty}\} \quad &\text{ if } \square = \xi\pm \text{ and } \xi \in [e_2, \mfc_1^x], \\
\Exp\{a_i + \zmin{x}(\xi)\} \quad &\text{ if } \square = \xi\pm \text{ and } \xi \in ]\mfc_1^x, \mfc_2^x[, \\
\Exp\{a_i+ \sinf{b}_{j: \infty}\} \quad &\text{ if } \square = \xi\pm \text{ and } \xi \in [\mfc_2^x, e_1], \\
\Exp\{a_i + \smin{b}_{j: \ell}\} \quad &\text{ if } \square = (\infty, \ell),  
\end{cases} \\
&\buse{x}{x+e_2}{\square} \sim 
\begin{cases}
\Exp\{b_{j}+\smin{a}_{i: k}\} \quad &\text{ if } \square = (k, \infty), \\%\text{ and } i \le k, \\ %k \geq i+1 \\ 
\Exp\{b_j+ \sinf{a}_{i: \infty}\} \quad &\text{ if } \square = \xi\pm \text{ and } \xi \in [e_2, \mfc_1^x], \\
\Exp\{b_j - \zmin{x}(\xi)\} \quad &\text{ if } \square = \xi\pm \text{ and } \xi \in ]\mfc_1^x, \mfc_2^x[, \\
\Exp\{b_j- \sinf{b}_{j: \infty}\} \quad &\text{ if } \square = \xi\pm \text{ and } \xi \in [\mfc_2^x, e_1], \\
\Exp\{b_j - \smin{b}_{j: \ell}\} \quad &\text{ if } \square = (\infty, \ell).  
\end{cases}
\end{split}
\end{align}
\item \label{thm:buslim:indp} \rm{(Independence along down-right paths)}. For $x, y \in \bbZ^2$ with $x \le y$ and such that 
\begin{align*}
&y \cdot e_1 \le k \text{ and } \ir{k}{x} = y \cdot e_1 \quad \text{ if } \square = (k, \infty), \\ 
&y \cdot e_2 \le \ell \text{ and } \jr{\ell}{x} = y \cdot e_2 \quad \text{ if } \square = (\infty, \ell), \\ 
&\ir{\infty}{x} = \ir{\infty}{y} \quad \text{ if } \square \in \{\xi-, \xi+\} \text{ and } \xi \in [e_2, \mfc_1^x], \\ 
&\jr{\infty}{x} = \jr{\infty}{y} \quad \text{ if } \square \in \{\xi-, \xi+\} \text{ and } \xi \in [\mfc_2^x, e_2],
\end{align*}
and any down-right path $\pi$ from $(x \cdot e_1, y \cdot e_2)$ to $(y \cdot e_1, x \cdot e_2)$, the collection 
\begin{align*}
&\{\w_p: p \in \sG_{x, y, \pi}^-\} \cup \{\buse{p}{p+e_1}{\square}: p, p+e_1 \in \pi\} \\
&\cup \{\buse{p}{p+e_2}{\square}: p, p+e_2 \in \pi\} \cup \{\buse{p-e_1}{p}{\square} \wedge \buse{p-e_2}{p}{\square}: p \in \sG_{x, y, \pi}^+\}
\end{align*}
is independent. 
 
\end{enumerate} 
\end{theorem}
\begin{remark}
In part \eqref{thm:buslim:decomp}, the cocycle property is phrased only along up-right paths. This is only to avoid expressions of the form $\infty - \infty$, due to the potential of infinite values of the Busemann functions as noted in part \eqref{thm:buslim:pos}\eqref{thm:buslim:pos:Bxy=infty}.
\end{remark} 
\subsection{Semi-infinite geodesics} \label{sub:geo}
Our basic tools in our study of the global structure of infinite geodesics are the Busemann geodesics, which are semi-infinite geodesics generated from the Busemann functions of Theorem \ref{thm:buslim} according to the following local rules.

For $k,\ell\in\bbZ$, $\xi \in [e_2,e_1]$, and $\square \in \{\xi+,\xi-,(k,\infty),(\infty,\ell)\}$, define
\be\begin{aligned}
e_{\square} &= \begin{cases}
e_1 & \square = {\xi+} \text{ or } (\infty,\ell) \\
e_2 & \square = {\xi-} \text{ or } (k,\infty)
\end{cases}.\label{eq:tiedef}
\end{aligned}\ee
For $x \in \bbZ^2$ with $x \leq (k,\ell)$, set $\geo{x\cdot(e_1+e_2)}{x}{\square} = x$ and recursively for $n \geq x\cdot(e_1+e_2)$, define
\begin{align}
\geo{n+1}{x}{\square} &= \begin{cases}
\geo{n}{x}{\square} + e_1 & \text{ if } \buse{\geo{n}{x}{\square}}{\geo{n}{x}{\square}+e_1}{\square}  <   \buse{\geo{n}{x}{\square}}{\geo{n}{x}{\square}+e_2}{\square}  \\
\geo{n}{x}{\square} + e_2 & \text{ if }  \buse{\geo{n}{x}{\square}}{\geo{n}{x}{\square}+e_1}{\square}  >   \buse{\geo{n}{x}{\square}}{\geo{n}{x}{\square}+e_2}{\square}\\
\geo{n}{x}{\square} + e_{\square} & \text{ if }  \buse{\geo{n}{x}{\square}}{\geo{n}{x}{\square}+e_1}{\square}  =   \buse{\geo{n}{x}{\square}}{\geo{n}{x}{\square}+e_2}{\square}
\end{cases}, \label{eq:busgeodef}
\end{align}
This recursion says that Busemann geodesics follow the minimum of the Busemann increments and, in the event of a tie, the geodesic goes in direction $e_{\square}$. Our next lemma records the key fact that the Busemann geodesics defined in this way are in fact semi-infinite geodesics and that along a Busemann geodesic, the associated Busemann increment is the passage time. We omit the proof as this is a well-known consequence of the cocycle and recovery properties (Theorem \ref{thm:buslim}\eqref{thm:buslim:decomp} and \eqref{thm:buslim:recovery}). The proof of \cite[Lemma 4.1]{Geo-Ras-Sep-17-ptrf-2}, for example, applies in our setting line-by-line.
\begin{lemma}\label{lem:BusRGeo}
The following holds $\bfP$-almost surely. For all $x=(i,j) \in \bbZ^2$ and all $\square \in \{\xi+,\xi-,(k,\infty), (\infty,\ell) : \xi \in [e_2,e_1], k \in \bbZ_{\geq i}, \ell \in \bbZ_{\geq j}\}$,
\begin{enumerate} [label={\rm(\alph*)}, ref={\rm\alph*}] \itemsep=3pt
\item   $\geo{}{x}{\square}$ is a semi-infinite geodesic.
\item For all $n\geq i+j$, $\LPP{x}{\geo{n}{x}{\square}} = \buse{x}{\geo{n}{x}{\square}}{\square}$.
\end{enumerate}
\end{lemma}

The next theorem collects our main results concerning the global structure of geodesics. Recall the notation $\mfc_i^x$ for the critical directions from \eqref{eq:crit}. Also recall from \eqref{eq:rec} the notation $\ir{k}{x}$ and $\jr{\ell}{x}$ for the locations where the parameter sequences change ahead of a site $x\in\bbZ^2$.

\begin{theorem}
\label{thm:geo}
The following statements hold $\bfP$-almost surely.
\begin{enumerate}[label={\rm(\alph*)}, ref={\rm\alph*}] \itemsep=3pt
\item \textup{(Directedness)}. \label{thm:geo:dir}
For all $x \in \bbZ^2$ and all semi-infinite geodesics $\geod{}$ with $x \in \geod{}$, exactly one of the following three possibilities holds:
\begin{enumerate}[label={\rm(\roman*)}, ref={\rm\roman*}] \itemsep=3pt
\item \textup{(Concave segment directed)} 
There exists $\xi \in \,]\mfc_1^x,\mfc_2^x[$ such that
\begin{align*}
\lim_{n\to\infty} \frac{\geod{n} }{n} = \xi.
\end{align*}
\item \textup{(Row/column constrained)} Exactly one of the following two conditions holds:
\begin{enumerate}[label={\rm(\arabic*)}, ref={\rm\arabic*}] \itemsep=3pt
\item There exists $k \in \mathbb{N}$ such that for all sufficiently large $n$, $\geod{n}\cdot e_1 = \ir{k}{x}$.
\item  There exists $\ell \in \mathbb{N}$ such that for all sufficiently large $n$, $\geod{n}\cdot e_2 = \jr{\ell}{x}$.
\end{enumerate}
\item \textup{(Linear segment directed)} Exactly one of the following two conditions holds: 
\begin{enumerate}[label={\rm(\arabic*)}, ref={\rm\arabic*}] \itemsep=3pt
\item $\geod{n} \cdot e_1 \to \infty$ and $e_2 \preceq \varlimsup_{n\to\infty} \frac{\geod{n}}{n} \preceq \mfc_1^x$.
\item $\geod{n} \cdot e_2 \to \infty$ and  $\mfc_2^x \preceq \varliminf_{n\to\infty} \frac{\geod{n} }{n} \preceq e_1.$
\end{enumerate}
\end{enumerate}
\item \label{thm:geo:ex}\textup{(Busemann geodesic directions).} The Busemann geodesics satisfy the following.
\begin{enumerate}[label={\rm(\roman*)}, ref={\rm\roman*}] \itemsep=3pt
\item \label{thm:geo:ex:con} \textup{(Concave segment directed)} For all $x \in \bbZ^2$ and all $\xi \in ]\mfc_1^x,\mfc_2^x[$ and $\square \in \{+,-\}$,
\begin{align*}
\lim_{n\to\infty} \frac{\geo{n}{x}{\xi\square}}{n} = \xi
\end{align*}
\item \label{thm:geo:ex:bdy} \textup{(Boundary trapped)} For all $x = (i,j) \in \bbZ^2$ and all $(k,\ell) \geq x$,
\begin{enumerate}[label={\rm(\arabic*)}, ref={\rm\arabic*}] \itemsep=3pt
\item For all sufficiently large $n$, $\geo{n}{x}{(k,\infty )} \cdot e_1 = \ir{k}{x}$.
\item For all sufficiently large $n$, $\geo{n}{x}{(\infty,\ell)}\cdot e_2 = \jr{\ell}{x}$.
\end{enumerate}

\item \label{thm:geo:ex:lin} \textup{(Linear segment directed)}. For all $x \in \bbZ^2$,
\be\begin{aligned}
e_2 \preceq\varlimsup_{n\to\infty}   \frac{\geo{n}{x}{\mfc_1^x} }{n} \preceq \mfc_1^x \qquad \text{ and }\qquad \mfc_2^x \preceq \varliminf_{n\to\infty} \frac{\geo{n}{x}{\mfc_2^x}}{n} \preceq e_1
\end{aligned}\label{eq:thm:geo:ex:lin-1}
\ee
Moreover, 
\begin{enumerate}[label={\rm(\arabic*)}, ref={\rm\arabic*}] \itemsep=3pt
\item $\geo{n}{x}{\mfc_1^x}\cdot e_1 \to \infty$ if and only if $\ir{\infty}{x}\notin \bbZ$.
\item  $\geo{n}{x}{\mfc_2^x}\cdot e_2 \to \infty$ if and only if $\jr{\infty}{x} \notin\bbZ$ 
\end{enumerate}

\end{enumerate}
\item \textup{(Uniqueness and extremality).} \label{thm:geo:uniq}
The following properties hold for all $x = (i,j) \in \mathbb{Z}^2$ and all semi-infinite geodesics $\geod{}$ with $x \in \geod{}$: 
\begin{enumerate}[label={\rm(\roman*)}, ref={\rm\roman*}] \itemsep=3pt
\item \textup{(Concave segment directed)} If $\xi \in ]\mfc_1^x,\mfc_2^x[$ and 
\begin{align*}
\lim_{n\to\infty} \frac{\geod{n} }{n} = \xi,
\end{align*}
then for all $n \geq i+j$, $\geo{n}{x}{\xi-} \preceq \geod{n} \preceq \geo{n}{x}{\xi+}$.
\item \textup{(Boundary trapped)} \label{thm:geo:uniq:nr}
\begin{enumerate}[label={\rm(\arabic*)}, ref={\rm\arabic*}] \itemsep=3pt
\item If $k \in \bbN$ is such that for all sufficiently large $n$, $\ir{k}{x} \leq \geod{n}\cdot e_1 \leq k$, then for all $n \geq i+j$, $\geod{n} = \geo{n}{x}{(k,\infty)} = \geo{n}{x}{(\ir{k}{x},\infty)}$.
\item  If $\ell \in \bbN$ is such that for all sufficiently large $n$, $ \jr{\ell}{x} \leq \geod{n}\cdot e_2 \leq \ell$, then for all $n \geq i+j$, $\geod{n} = \geo{n}{x}{(\infty,\ell)}=\geo{n}{x}{(\infty,\,\jr{\ell}{x})}$.
\end{enumerate}
\item \label{thm:geo:uniq:linun} \textup{(Linear segment directed away from the boundary)} 
\begin{enumerate}[label={\rm(\arabic*)}, ref={\rm\arabic*}] \itemsep=3pt
\item If 
\begin{align*}
\varlimsup_{n\to\infty} \frac{\geod{n} }{n} \preceq \mfc_1^x  \text{ and } \pi_n \cdot e_1 \to \infty, 
\end{align*}
then for all $n  \geq i+j,$ $\geod{n} = \geo{n}{x}{\mfc_1^x}$.
\item  If
\begin{align*}
\mfc_2^x \preceq \varliminf_{n\to\infty} \frac{\geod{n} }{n} \text{ and } \pi_n \cdot e_2 \to \infty, 
\end{align*}
then for all $n \geq i+j$, $\geod{n} = \geo{n}{x}{\mfc_2^x}$.
\end{enumerate}

\end{enumerate}
\item\label{thm:geo:coal} \textup{(Concave segment coalescence).} For each $x,y\in\bbZ^2$ and $\xi \in ]\mfc_1^{x\wedge y},\mfc_2^{x\wedge y}[$,  \begin{align*}
\bfP(\geo{}{x}{\xi}\text{ and }\geo{}{y}{\xi}\text{ coalesce}) := 
\bfP(\geo{n}{x}{\xi}=\geo{n}{y}{\xi} \text{ for all } n \ge N \text{ for some } N \in \bbZ) = 1. 
\end{align*}
 \end{enumerate}

\end{theorem}

The most interesting and novel behavior of geodesics in our setting occurs in the linear segments $[e_2,\mfc_1^x]$ and $[\mfc_2^x,e_1]$. The uniqueness in Theorem \ref{thm:geo}\eqref{thm:geo:uniq}\eqref{thm:geo:uniq:linun} implies that there is at most one geodesic which is directed into each of these segments which does not become trapped on a row or column. If such a geodesic exists, it is necessarily one of the Busemann geodesics $\geo{}{x}{\mfc_1^x}$ or $\geo{}{x}{\mfc_2^x}$, so we focus our attention on these. If $x=(i,j)$, a necessary and sufficient condition for these geodesics to not become trapped is that $a_n > \sinf{a}_{i:\infty}$ for all $n \geq i$ or $b_m > \sinf{b}_{j:\infty}$ for all $m \geq j,$ respectively.

Our main interest lies in exploring the range of possible phenomena concerning asymptotic directions. To avoid some technical issues  in the linear region, we restrict attention to sequences satisfying certain  mild simplifying hypotheses.

In the statement of the next result, we will write for $x=(i,j)$,
\be\label{eq:bdypar}\begin{aligned}
\varlimsup_{n \to \infty} \frac{1}{n} \sum_{k=i}^{n} \frac{1}{(a_k-\sinf{a}_{i:\infty})^2} &= \overline{\mfa}_x \qquad \text{ and } \qquad  \varliminf_{n \to \infty} \frac{1}{n} \sum_{k=i}^{n} \frac{1}{(a_k-\sinf{a}_{i:\infty})^2} = \underline{\mfa}_x \\
\varlimsup_{n \to \infty} \frac{1}{n} \sum_{k=j}^{n} \frac{1}{(b_k-\sinf{b}_{j:\infty})^2} &= \overline{\mfb}_x \qquad \text{ and } \qquad \varliminf_{n \to \infty} \frac{1}{n} \sum_{k=j}^{n} \frac{1}{(b_k-\sinf{b}_{j:\infty})^2} = \underline{\mfb}_x
\end{aligned}.\ee
We will also use the following notation:
\begin{align*}
    \bfA_x = \int \frac{1}{(a+\sinf{b}_{j:\infty})^2}\alpha(da)\qquad \text{ and } \qquad \bfB_x = \int \frac{1}{(b+\sinf{a}_{i:\infty})^2}\beta(db).
\end{align*}
Under the hypothesis that $\overline{\mfa}_x<\infty$ and $\overline{\mfb}_x<\infty$ in addition to the following condition, we show that any closed subintervals of $]e_2,\mfc_1^x]$ and
$[\mfc_2^x,e_1[$ can be achieved as the set of subsequential limits of $\geo{n}{x}{\mfc_1^x}/n$ and $\geo{n}{x}{\mfc_2^x}/n$. 
\begin{condition}\label{cond:lincond}
For each $x =(i,j) \in \bbZ^2$, there exists $\epsilon\in(0,1/2)$ and $N \in \bbN$ for which 
\begin{align}
\smin{a}_{i: n} - \sinf{a}_{i:\infty} \geq n^{-1/2+\epsilon} \quad \text{ for } n \ge N, \label{eq:lincond-1}\\
\smin{b}_{j: n} - \sinf{b}_{j:\infty} \geq n^{-1/2+\epsilon} \quad \text{ for } n \ge N.   \label{eq:lincond-2}
\end{align}
\let\qed\relax
\end{condition}
Note if $\overline{\mfa}_x<\infty$, then we must have $\smin{a}_{i:n}-\sinf{a}_{i:\infty}\geq c n^{-1/2}$ for some $c>0$, so \eqref{eq:lincond-1} is not far from optimal under that hypothesis. That the collection of limit points of $\geo{n}{x}{\mfc_1^x}/n$ must be an interval follows from the path structure. The exclusion of the endpoints $e_1$ and $e_2$ in our next result is almost certainly a purely technical point: our proof relies on concentration estimates which break down if $\overline{\mfa}_x=\infty$ or $\overline{\mfb}_x=\infty$ are permitted. With this caveat, this means that all other possible collections of potential limit points of geodesics directed into the linear region consistent with nearest-neighbor paths are possible. See Example \ref{ex:suffcond} for concrete examples.

\begin{theorem}\label{thm:lindir}
The following holds $\bfP$ almost surely for all $x =(i,j) \in \bbZ^2$.
\begin{enumerate}[label={\rm(\alph*)}, ref={\rm\alph*}] \itemsep=3pt
\item If \eqref{eq:lincond-1} holds and $\overline{\mfa}_x <\infty$, then the set limit points of $\geo{n}{x}{\mfc_1^x}/n$ is precisely the collection of vectors $\xi \in [e_2,e_1]$ with 
\begin{align}
\xi\cdot e_1 \in  \bigg[\frac{\bfB_x}{\overline{\mfa}_x + \bfB_x}, \frac{\bfB_x}{\underline{\mfa}_x + \bfB_x}\bigg].
\end{align}
\item If \eqref{eq:lincond-2} holds and $\overline{\mfb}_x <\infty$, then the set limit points of $\geo{n}{x}{\mfc_2^x}/n$ is precisely the collection of vectors $\xi \in [e_2,e_1]$ with
\begin{align}
\xi\cdot e_1 \in  \bigg[\frac{\underline{\mfb}_x}{\bfA_x + \underline{\mfb}_x}, \frac{\overline{\mfb}_x}{\bfA_x + \overline{\mfb}_x}\bigg].
\end{align}

\end{enumerate}
\end{theorem}

\begin{example}\label{ex:suffcond}
We record here sufficient conditions for the novel behaviors of geodesics which are not seen in the i.i.d.~setting which were described in the introduction.
\begin{enumerate}  \itemsep=2pt
\item \label{ex:suffcond:empty} \textit{No geodesics exist with limit points in a non-empty interval of directions and infinitely many non-trivial non-coalescing geodesics.}  For $x=(i,j)$, it follows immediately from \eqref{eq:crit} that a sufficient condition for $]e_2,\mfc_1^x]$ to be non-empty is $\textup{ess inf}\{\alpha\}=\underline{\alpha} >  \sinf{a}_{i:\infty}$. In particular, if $\underline{\alpha} >  \sinf{a}_{-\infty:\infty}$ and $\sinf{a}_{-\infty:\infty}=a_k$ for infinitely many $k \in \bbN$, but the density of such indices is zero, then $\mfc_1^x = \mfc_1$ does not depend on $x$ and there is no infinite geodesic $\pi$ anywhere on the lattice which satisfies that $\geod{n}\cdot e_1 \to \infty$ and that $\geod{n} /n$ has a subsequential limit in $[e_2,\mfc_1]$. From each site $x$, the $\mfc_1$ geodesic becomes trapped on the first column ahead of $x$ where the global minimum of the parameter sequence is realized, which implies existence of infinitely many non-coalescing non-trivial geodesics in the $e_2$ direction.

For a concrete example, take $a_i = b_i$ for all $i\in\bbZ$ to be defined as follows: if $i$ is not equal to $n^2$ for any $n\in\bbN$, let $a_i=b_i=1/2$; if $i=n^2$ for some $n\in\bbN$, then set $a_i=b_i=1/4$. The limit shape for this model is plotted in Figure \ref{fig:shapex}. In this case, for all $x\in\bbZ^2$, $\mfc_1^x:=\mfc_1 = (1/10,9/10)$ and $\mfc_2^x:=\mfc_2=(9/10,1/10)$. We also have that $\alpha=\beta = \delta_{1/2}$ so $\underline{\alpha}=\underline{\beta}=1/2$. In this example, $]e_2,\mfc_1]$ and $[\mfc_2,e_1[$ contain no asymptotic directions of semi-infinite geodesics from any site of the lattice.
\item \label{ex:suffcond:iidunif} \textit{Non-trapped axis-directed geodesic.}  If $\alpha(dx) = 1_{(0,1)}(x)dx$ and if $a_1,a_2,\dots$ is an i.i.d.~sequence drawn from $\alpha$, then $\sinf{a}_{1:\infty}=0 =\underline{\alpha}$. By \eqref{eq:crit}, $\mfc_1^{(1,1)}=e_2$. By Theorem \ref{thm:geo}\eqref{thm:geo:ex}, $e_2= \lim_{n\to\infty}\geo{n}{(1,1)}{e_2}/n$ and $\geo{n}{(1,1)}{e_2}\cdot e_1\to\infty$. In this case, there are also infinitely many $e_1$ directed geodesics rooted at $(1,1).$
\item \label{ex:suffcond:interval} \textit{A geodesic which wanders inside a specified interval.}  Fix $t>1$, $p\in (0,1/2)$, and $r>0$. Let $b_j = 1$ for all $j$. If $k,i \in \bbZ_{>0}$ are such that $t^k \leq i < t^{k}+t^{(1-2p)k}<t^{k+1}$, set $a_i = \sqrt{r} t^{-pk}$. For all other $i$, set $a_i=1.$ Then $\alpha=\beta=\delta_1$, $\sinf{a}_{1:\infty}=0$, and $\mfc_1=\mfc_1^{(1,1)}=(1/2,1/2)$ is the critical direction. $\smin{a}_{1:n} \sim n^{-p}$, so condition \eqref{eq:lincond-1} is satisfied.
Computation shows that the set limit points of $\geo{n}{(1,1)}{\mfc_1}/n$ is the collection of vectors $\xi \in [e_2,e_1]$ with 

\[
\xi \cdot e_1 \in \bigg[\frac{1}{2 + \frac{t}{r(t-1)}}, \frac{1}{2+\frac{1}{r(t-1)}}\bigg].
\]
If $0 < a < b < 1/2,$ we may choose  \[\frac{1}{r} = \frac{1}{a}-\frac{1}{b}\text{ and }t= \frac{b}{a} \frac{1-2a}{1-2b},\] in which case the limit points are all vectors with $\xi\cdot e_1\in[a,b].$ 

This example can be modified to allow for the critical direction as a limit point by making minor changes. If $k,i \in \bbZ_{>0}$ are such that $2^{k^2} \leq i < 2^{k^2}+2^{(1-2p)k^2}$, instead set $a_i = \sqrt{r} 2^{-pk^2}$ with $a\equiv 1$ otherwise. Again, $\alpha=\beta=\delta_1$, $\sinf{a}_{1:\infty}=0$, and the critical direction $\mfc_1^{(1,1)}=(1/2,1/2)$.
% , so the diagonal is the critical direction. 
% As above, \eqref{eq:lincond-1} is satisfied and the liminf and limsup of $n^{-1}\sum_{i=1}^n a_i^{-2}$ are achieved along the subsequences $m_k = 2^{(k+1)^2}-1$ and $\ell_k=2^{k^2} + 2^{(1-2p)k^2}$, $k \in \bbN$, respectively. The same 
%%%%%%%%%%
% \underline{Further details}. Arguing as in the first hidden display above, one obtains that 
% \begin{align*}
% \begin{split}
% \frac{1}{m_k} \sum_{i=1}^{m_k} \frac{1}{a_i^2} &= 1 + \frac{1}{2^{(k+1)^2}-1} \sum_{\ell=1}^k \frac{2^{2p\ell^2}}{r} \cdot 2^{(1-2p) \ell^2} + o(1) \\ 
% &= 1 + \frac{1}{2^{(k+1)^2}-1} \sum_{\ell=1}^k \frac{2^{\ell^2}}{r} + o(1) \\ 
% &= 1 + O\bigg(\frac{k2^{k^2}}{2^{(k+1)^2}}\bigg) + o(1) \stackrel{k \to \infty}{\to} 1 = \underline{\mfa}_x, \text{ and } \\ 
% \frac{1}{\ell_k} \sum_{i=1}^{\ell_k} \frac{1}{a_i^2} &= 1 + \frac{1}{2^{k^2}+2^{(1-2p)k^2}} \sum_{u=1}^{k} \frac{2^{u^2}}{r}+ o(1) \to 1 + \frac{1}{r} = \overline{\mfa}_x. 
% \end{split}
% \end{align*}
% %end of further details.
%%%%%%%%%%
Computation shows that the limit points are vectors with $\xi \cdot e_1 \in [1/(2+1/r),1/2]$.

\item \label{ex:suffcond:isolated} \textit{A geodesic with an isolated interior asymptotic direction.} The structure of the previous example can also be modified to allow for a fixed asymptotic direction in $]e_2,\mfc_1^{(1,1)}[$. Fix any $r>0$, $p \in (0,1/2)$, and let $q = (1-p)/2 \in (1/4,1/2).$ If $k,i \in \bbZ_{>0}$ are such that $k^2 \leq i < k^2 + k^p$, set $a_i = \sqrt{r/2} k^{-q}$ and let $a_i = 1$ otherwise. Again, let $b_j = 1$ for all $j$. As above, $\alpha=\beta=\delta_1$, $\sinf{a}_{1:\infty}=0$, and $\mfc_1=\mfc_1^{(1,1)}=(1/2,1/2)$.
% , so the diagonal is the critical direction. 
Since $\smin{a}_{1:n} \sim n^{-q/2}$, condition \eqref{eq:lincond-1} is satisfied. 
% Using $2q+p=1$, 
% \[
% \frac{1}{n}\sum_{i=1}^n a_i^{-2} = 1 + \frac{1}{n} \sum_{i=1}^{\lfloor n^{1/2}\rfloor} 2/r i^{2q+p} + o(1) = 1 + 1/r + o(1).
% \]
% Consequently, 
%%%
%%%
Computation checks that $\geo{n}{(1,1)}{\mfc_1}/n$ converges to $\zeta = (\frac{1}{2+1/r}, 1- \frac{1}{2+1/r})$. Therefore, there are no geodesics rooted at $(1,1)$ with limit points in either $]e_2, \zeta[$ or $]\zeta,(1/2,1/2)[$.
%\item {\color{purple} E: Preliminary:} Fix $p \in (0, 1/2)$. Let $m_j = \lceil j^{1/(2p)}\rceil$ for $j 
%\in \bbZ_{\ge 0}$. Then, for $i \in \bbZ_{>0}$, set $$a_i = \begin{cases}m_j^{-1/2+p} &\quad \text{ if } i = m_j \\ 1 &\quad \text{ otherwise. }\end{cases}$$ 
%Since $\hash\{j: m_j \le n\} \sim n^{2p}$, the limit measure is $\alpha = \delta_1$ as before. 
%For each $n \in \bbZ_{>0}$, there is a unique $j \in \bbZ_{>0}$ with $m_{j-1} < n \le m_j$.  
\item \label{ex:suffcond:critical} \textit{A geodesic with a critical asymptotic direction.} Let $a_1,a_2,\dots$ be an i.i.d.~sequence drawn from the measure $\alpha(da) = 7a^6\one_{(0,1)}(a)da$ and let $b_j = 1$ for all $j$. Borel-Cantelli checks that $\smin{a}_{1:n}$ satisfies \eqref{eq:lincond-1} almost surely with $\epsilon =1/6$. We have $\sinf{a}_{1:\infty} = \textup{ess inf}\{\alpha\} = \underline{\alpha}=0$ and $\mfc_1=\mfc_1^{(1,1)}=(5/12,7/12)$. Because $a_1^{-2}$ is integrable, it follows from the law of large numbers and Theorem \ref{thm:lindir} that the limit of $\geo{n}{(1,1)}{\mfc_1}/n$ is $\mfc_1$. \qedhere
%%%%%
% \underline{Borel-Cantelli argument above.} With $\epsilon \in (0, 1/2)$, one has 
% \begin{align*}
% \begin{split}
% \P\{a_{1, n}^{\min} < n^{-1/2+\epsilon}\} &= 1 - \P\{a_{1, n}^{\min} \ge n^{-1/2+\epsilon}\} = 1-\P\{a_1 \ge n^{-1/2+\epsilon}\}^n \\ 
% &= 1 - (1-n^{-7/2 + 7\epsilon})^n \stackrel{\text{MVT}}{\le}  n^{-7/2+7\epsilon} \cdot n = n^{-5/2+7\epsilon}, 
% \end{split}
% \end{align*}
% which is summable in $n$ if $5/2 - 7\epsilon > 1$, that is, $\epsilon < 3/14$. Hence, $\epsilon = 1/6$ works as claimed.  
%%%%%

\end{enumerate}
\end{example}

\subsection{Asymptotic directions of the competition interfaces}
Recall the definition \eqref{eq:cifdef} of the competition interface $\Cifa^x$ at $x=(i,j)$, as well as the locations $\Cif^x(n)$ and $\Cifs^x(m)$ where it crosses horizontal and vertical lines, given in \eqref{eq:cifslope} and \eqref{def:Cifs}. The next theorem  collects our main results about the asymptotic directions of competition interfaces. By Theorem \ref{thm:buslim} we have the following representation of the limits of these quantities: 
\be
\label{ERinfty}
\begin{aligned}
\Cif^{x}(\infty) &= \sup_{n } \Cif^{x}(n) = \lim_{n \rightarrow \infty} \Cif^{x}(n) = \sup \{m \geq i : \buse{x}{x+e_1}{(m,\infty)} > \buse{x}{x+e_2}{(m,\infty)}\}.  \\
\Cifs^{x}(\infty) &= \sup_m \Cifs^x(m) = \lim_{m\to\infty} \Cifs^x(m) = \sup\{n \geq j : \buse{x}{x+e_2}{(\infty,n)} > \buse{x}{x+e_1}{(\infty,n)}\}
\end{aligned}
\ee
\begin{theorem}\label{thm:cif1}  
Fix $x = (i,j) \in \bbZ^2$.  
\begin{enumerate}[label={\rm(\alph*)}, ref={\rm\alph*}] \itemsep=3pt 
\item \label{thm:cif1:1} $\Cif^x(\infty)$ has distribution given for  $m\in\Z_{\ge i}$ by
\be\label{cif29} \begin{aligned}
\bfP\left(\Cif^x(\infty) = m\right) &=  \frac{\smin{a}_{i:m}-\smin{a}_{i:m+1}}{a_i+b_j}
\quad
\text{and}\quad   \bfP\left(\Cif^x(\infty) = \infty\right) &=  \frac{ \sinf{a}_{i:\infty}+b_j}{a_i+b_j} . 
\end{aligned}\ee
\item \label{thm:cif1:2} $\Cifs^x(\infty)$ has distribution given for $n \in \Z_{\geq j}$ by 
\be\label{cif31} \begin{aligned}
\bfP\left(\Cifs^{x}(\infty) = n\right) &=  \frac{\smin{b}_{j:n}-\smin{b}_{j:n+1}}{a_i+b_j}\quad
\text{and}\quad   
\bfP\left(\Cifs^x(\infty) = \infty\right) &=  \frac{a_i+ \sinf{b}_{j:\infty}}{a_i+b_j} . 
\end{aligned}\ee

\item \label{thm:cif1:3}
The $[e_2,e_1]$-valued limit $\ddd\cif^{x}=\lim_{n\to\infty} \Cifa_n^x/n$ exists $\bfP$-almost surely. Its distribution is given for $\xi \in [e_2,e_1[$ by 
\be\label{cif5} \begin{aligned}
\bfP\left(\cif^{x}=e_2\right)=\frac{a_{i}- \sinf{a}_{i:\infty}}{a_{i}+b_{j}}, \hspace{1pc}
\bfP\left(\cif^{x} \preceq \xi\right)=\frac{a_{i}+\zmin{x}(\xi)}{a_{i}+b_{j}}, \hspace{1pc} 
\bfP\left(\cif^{x}=e_1\right)=\frac{b_{j}- \sinf{b}_{j:\infty}}{a_{i}+b_{j}},
\end{aligned}\ee
\end{enumerate} 
where $\zmin{x}(\xi)$ is defined in equation \eqref{eq:zetaext}.
\end{theorem} 
%Combining (i)-(iii) above, the equalities of events
%\[   \{\Cif^{x}(\infty) = \infty\} = \{e_2 \prec \cif^{x}\}
%\quad\text{and}\quad 
%\{\Cifs_{x}(\infty) = \infty\} = \{\cif^{x}\prec e_1\}
% \]  
%hold, and hence w
%Recall that $\Cif^{x}(n)$ denotes the first coordinate of the point at which the competition interface first reaches the horizontal level of index $n$. 
The previous result implies the following perhaps unexpected dichotomy:  either   $\Cif^{x}(n)$ remains bounded (in which case, the competition interface becomes trapped on a horizontal level) as $n$ grows or else it grows ballistically. %at least at rate  $n$. 
The analogous statement also holds for $\Cifs^{x}(m)$ as $m$ grows. From \eqref{cif31} we see that  $\Cifs^{x}(\infty)$ has an atom at $n \in \bbZ_{\ge j}$  if and only if $b_{n+1} < \smin{b}_{j,n}$. That is to say, the  rows or columns where  the competition interface can become stuck are exactly those at which the running minimum of the parameter sequence decreases. If 
 $\Cifs^{x}(\infty)=n$, the entire tree  $\sT_{x,x+e_1}$  is confined to $\Z_{\ge i}\times[j,n]$.

From \eqref{cif5} we see that  the only possible atoms of $\cif^{x}$ are the coordinate directions $e_1$ and $e_2$.  Furthermore, the flat segments with the boundary removed are not included in the support of the random variable $\cif^x$: 
\be\label{cif6} 
\bfP\left(\cif^{x}  \in]e_2, \mfc_1^x] \right) =0 
\quad\text{and}\quad 
\bfP\left(\cif^{x}  \in[\mfc_2^x,e_1[ \right) =0.
\ee

\subsection{Asymptotics of the second-class particle}
Through the couplings in Sections \ref{sec:inTASEP} and \ref{sec:TAZRP}, the results above have immediate consequences for the asymptotics of second-class particles in the inhomogeneous TASEP and second-class customers in the inhomogeneous TAZRP. Recall that we denote the location of the second-class particle in the inhomogeneous TASEP by $X(t)$ and in 
the inhomogeneous TAZRP by $Z(t)$. Because of the distributional identities in Lemmas \ref{lem:fpcouple} and \ref{fp-lem}, the asymptotic behavior of $X(t)$ is already explained by Theorem \ref{thm:cif1}, so we omit the statement. The following is our main result on the long-term behavior of the second-class customer $Z(t)$.

\begin{theorem}\label{th:2cl}   
\begin{enumerate}[label={\rm(\alph*)}, ref={\rm\alph*}] \itemsep=3pt
Suppose that $a_i = 0$ for all $i$, then for the TAZRP described in Section \ref{sec:TAZRP},
\item \label{th:2cl:1} The  $\Z_{\ge2}\cup\{\infty\}$-valued  almost sure limit  $\ddd Z(\infty)=\lim_{t\to\infty} Z(t)$ exists and has the following distribution:
\be\label{2cl31} \begin{aligned}
\bfP(Z(\infty) = n) &=  \frac{\smin{b}_{1:n-1}-\smin{b}_{1:n}}{b_1}
\quad\text{for } n\in\Z_{\ge2} \ \text{and} \\ 
\bfP(Z(\infty) = \infty) &= \frac{\sinf{b}_{1:\infty}}{b_1} . 
\end{aligned}\ee

\item \label{th:2cl:2} The  limiting speed  $\vela=\lim_{t\to\infty} t^{-1}Z(t)$  exists and satisfies
\begin{align*}
\vela \in \left[0, \left(\int_0^\infty b^{-1} \beta(db)\right)^{-1}\right]
\end{align*}
$\bfP$-almost surely. The distribution of $\vela$ is given by
\be\label{2cl5} \begin{aligned}
 \bfP(\vela=0)&=1- \frac{ \sinf{b}_{1:\infty}}{b_{1}},\\
%\text{and}\quad 
\bfP(\vela\le s)&=1-\frac{(\gamma^{-1})'(1/s,1)}{b_{1}} 
%&=1-\frac{\zeta(\gamma^{-1}(1/s))}{b_{1}} 
\quad \text{for }  s\in \left(0, \left(\int_0^\infty b^{-1} \beta(db)\right)^{-1}\right],
\end{aligned}\ee
where $(\shp^{-1})'(s,1)$ is the derivative of the inverse of the function $s \mapsto \shp^{(1,1)}((s,1))$ defined through \eqref{eq:shape}. \qedhere
\end{enumerate}
\end{theorem} 
We have again an almost sure  dichotomy.  If  $b_1 \neq  \sinf{b}_{1:\infty}$ then  with probability $1- { \sinf{b}_{1:\infty}}/{b_{1}}$    the second-class customer becomes stuck at some  station $j = \jr{\ell}{(1,1)}$ for some $\ell \geq 2$.  With the complementary  probability $\sinf{b}_{1:\infty}/{b_{1}}$ the second-class customer escapes with  positive speed.

\section{Busemann function proofs} \label{sec:Bus}
This section establishes Theorem \ref{thm:buslim}. The proof is carried out in stages, treating the strictly concave region, the boundary thin rectangle regions, and flat regions separately. 

\subsection{Deterministic preliminaries} 
We begin by recording some deterministic structure of last-passage percolation. During this discussion, we also recall a notion of duality which plays a key role in what follows.

Given weights $\ww = \{\ww_x \in \bbR: x \in \Rect_u^v\}$ on a rectangle $\Rect_u^v$ and $p \in \bbZ^2$, one obtains weights $\ww_{\bullet + p} = \{\ww_{x+p}: x \in \Rect_{u-p}^{v-p}\}$ on the rectangle $\Rect_{u-p}^{v-p}$ via translation by $p$.  The following is clear from definitions \eqref{eq:Lpt}, \eqref{eq:incint} and \eqref{eq:incbw}.
\begin{lemma}
\label{lem:LppShf}
The following statements hold for $p \in \bbZ^2$ and $x, y \in \Rect_{u-p}^{v-p}$. 
\begin{enumerate}[label={\rm(\alph*)}] 
\itemsep=3pt
\item $\Lpt_{x, y}(\ww_{\bullet+p}) = \Lpt_{x+p, y+p}(\ww)$.  
\item If $x \le y$ then 
\begin{align*}
\begin{split}
&\init{\I}_{x, y}(\ww_{\bullet + p}) = \init{\I}_{x+p, y+p}(\ww) \quad \text{ and } \quad \init{\J}_{x, y}(\ww_{\bullet + p}) = \init{\J}_{x+p, y+p}(\ww), \\
&\ter{\I}_{x, y}(\ww_{\bullet + p}) = \ter{\I}_{x+p, y+p}(\ww) \quad \text{ and } \quad \ter{\J}_{x, y}(\ww_{\bullet + p}) = \ter{\J}_{x+p, y+p}(\ww). 
\end{split}
\end{align*}
\end{enumerate}
\end{lemma}

Our next lemma records a key monotonicity property for the last-passage increments defined at \eqref{eq:incint} and \eqref{eq:incbw}. For two different proofs of this result, known as the \emph{comparison} or \emph{path crossing lemma}, we refer the reader to \cite[Lemma 6.2]{Ras-18} and \cite[Lemma 4.6]{Sep-18}.  
\begin{lemma}
\label{lem:incineq}
Let $x, y \in \Rect_{u}^v$ with $x \le y$. 
\begin{enumerate}[label={\rm(\alph*)}, ref={\rm\alph*}] \itemsep=3pt
\item \label{lem:incineq:1} If $y+e_1 \le v$ then $\init{\I}_{x, y} \ge \init{\I}_{x, y+e_1}$ and $\init{\J}_{x, y} \le \init{\J}_{x, y+e_1}$. 
\item If $y+e_2 \le v$ then $\init{\I}_{x, y} \le \init{\I}_{x, y+e_2}$ and $\init{\J}_{x, y} \ge \init{\J}_{x, y+e_2}$. 
\item If $x-e_1 \geq u$ then $\ter{\I}_{x,y} \geq \ter{\I}_{x-e_1,y}$ and $\ter{\J}_{x,y}\leq \ter{\J}_{x-e_1,y}$.
\item If $x-e_2 \geq u$ then $\ter{\I}_{x,y}\leq \ter{\I}_{x-e_2,y}$ and $\ter{\J}_{x,y}\geq \ter{\J}_{x-e_2,y}.$
\end{enumerate}
\end{lemma}
%%%%
%Omitted proof
%%%%
% \begin{proof}
% We prove the first half of \eqref{lem:incineq:1}, with the remaining cases being similar. The first inequality in \eqref{lem:incineq:1} is equivalent to $\Lpt_{x,y} + \Lpt_{x+e_1,y+e_1} \geq \Lpt_{x+e_1,y}+\Lpt_{x,y+e_1}$. Note that if $\pi_1$ is any geodesic from $x+e_1$ to $y$ and if $\pi_2$ is any geodesic from $x$ to $y+e_1$, then $\pi_1$ and $\pi_2$ must cross. Switching the paths after the first crossing point gives admissible paths from $x$ to $y$ and $x+e_1$ to $y+e_1$, from which the result follows.
% \end{proof}
%%%%
Another basic planarity argument concerning geodesics will come up several times in our proofs. This argument has been used previously in the proof of \cite[Lemma 6.1]{Emr-Jan-Sep-23}, for example. The proof is illustrated in Figure \ref{fig:geosqueeze}.

\begin{lemma}
\label{lem:plan}
The following statements hold for $x, y \in \Rect_{u}^v$ with $x \le y$. 
\begin{enumerate}[label={\rm(\alph*)}, ref={\rm\alph*}] \itemsep=3pt
\item \label{lem:plan:e1} If $\Lpt_{x, y} = \Lpt_{x, y-e_1} + \ww_y$ then 
$\Lpt_{p, y} = \Lpt_{p, y-e_1} + \ww_y$ for $p \in \Rect_{(u \cdot e_1, x \cdot e_2)}^{(x \cdot e_1, y \cdot e_2)}$. 
\item \label{lem:plan:e2} If $\Lpt_{x, y} = \Lpt_{x, y-e_2} + \ww_y$ then $\Lpt_{p, y} = \Lpt_{p, y-e_2} + \ww_y$ for $p \in \Rect_{(x \cdot e_1, u \cdot e_2)}^{(y \cdot e_1, x \cdot e_2)}$. 
\end{enumerate}
\end{lemma}

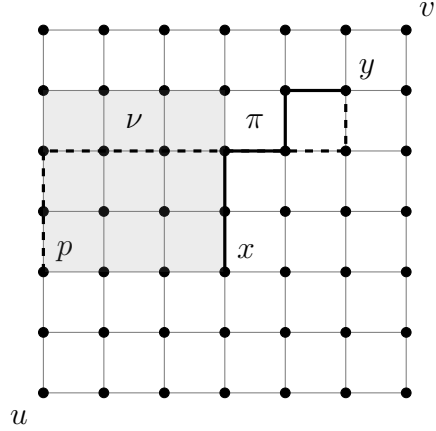
\begin{figure}[h]
    \begin{tikzpicture}[scale = 0.8]
                \draw (4.5,4.5) node[above]{$\qquad$};
		\draw[gray,thin] (-2,-2) grid (4,4);
		\foreach \x in {-2,-1,0,1,2,3,4}
			\foreach \y in {-2,-1,0,1,2,3,4}
				\filldraw (\x,\y) circle [radius=.08];
            \node at (3.35,3.35){$y$};
            \node at (1.35,.35){$x$};
            \node at (4.35,4.35){$v$};
            \node at (-2.4,-2.4){$u$};
            \node at (-1.65,.35){$p$};
            \fill[gray,opacity=.15,]  (-2,0) rectangle (1,3);
            \draw[very thick] (1,0) -- (1,1) -- (1,2) -- (2,2) --(2,3) -- (3,3);
            \node at (-.5,2.5){$\nu$};
            \node at (1.5,2.5){$\pi$};
            \draw[very thick, dashed] (-2,0) -- (-2,2) -- (3,2) -- (3,3);
    \end{tikzpicture}
  \caption{If a geodesic $\pi$ (solid) from $x$ to $y$ passes through $y-e_1$, then for any $p$ with $p \cdot e_1 \leq x \cdot e_1$ and $p  \cdot e_2 \leq y \cdot e_2$ (shaded), there is a geodesic from $p$ to $y$ passing through $y-e_1$ obtained by taking any geodesic $\nu$ (dashed) from $p$ to $y$ and concatenating the segment until the first intersection of $\nu$ and $\pi$ with the remaining segment of $\pi$.}\label{fig:geosqueeze}
\end{figure}
%%%%%%%%%%%%%%
%Omitted proof. 
%%%%%%%%%%%%%%
%\begin{proof
%To verify (a), assume that $\Lpt_{x, y} = \Lpt_{x, y-e_1} + \ww_y$. In particular, $x \le y-e_1$ because the right-hand side would equal $-\infty$ otherwise. By our assumption, there exists a geodesic $\pi \in \Pi_x^y$ such that $y-e_1 \in \pi$. Now let $p \in \Rect_{(u \cdot e_1, x \cdot e_2)}^{(x \cdot e_1, y \cdot e_2)}$. Any path $\nu \in \Pi_{p}^y$ necessarily shares some vertex $z$ with $\pi \smallsetminus \{y\}$. If $\nu$ is a geodesic then one can build another geodesic $\nu' \in \Pi_{p}^y$ by following $\nu$ until $z$ and then switching to $\pi$. Then $y-e_1 \in \nu'$ because the shared vertex $z$ is at or before $y-e_1$ along $\pi$. Consequently, $\Lpt_{p, y} = \Lpt_{p, y-e_1} + \ww_y$ as claimed. 
%
%The proof of (b) is similar. 
%\end{proof}
%%%%%%%%%%%%%
%%%%%%%%%%%%%

%We continue with some basic increment identities. 
Continuing with basic identities, let $\ww^{\leftarrow} = \{\ww_x^{\leftarrow} \in \bbR: x \in \Rect_u^v\}$ denote the \emph{reflected weights} given by 
\begin{align}
\label{eq:wref}
\ww_x^{\leftarrow} = \ww_{u+v-x} \quad \text{ for } x \in \Rect_{u}^v. 
\end{align}
It can be seen from definitions \eqref{eq:Lpt} and \eqref{eq:wref} that 
\begin{align}
\label{eq:Lref}
\begin{split}
\Lpt_{x, y}(\ww^{\leftarrow}) &= \max_{\pi \in \Pi_{x}^y} \sum_{p \in \pi} \ww_p^{\leftarrow} = \max_{\pi \in \Pi_{x}^y} \sum_{p \in \pi} \ww_{u+v-p} = \max_{\pi \in \Pi_{u+v-y}^{u+v-x}} \sum_{p \in \pi} \ww_{p} \\ 
&= \Lpt_{u+v-y, u+v-x}(\ww) \quad \text{ for } x, y \in \Rect_u^v. 
\end{split}
\end{align}
The following lemma is an immediate consequence of identity \eqref{eq:Lref} and the definitions of the increments. Since the map $\ww \mapsto \ww^{\leftarrow}$ is clearly an involution (a bijection that is its own inverse) on $\Rect_u^v$, the identities below also hold after interchanging $\ww$ and $\ww^{\leftarrow}$. 
\begin{lemma}
\label{lem:incid-2}
The following statetements hold for $x, y \in \Rect_u^v$ with $x \le y$. 
\begin{enumerate}[label={\rm(\alph*)}, ref={\rm\alph*}] \itemsep=3pt
\item $\ter{\I}_{x, y}(\ww^{\leftarrow}) = \init{\I}_{u+v-y, u+v-x}(\ww)$.  
\item $\ter{\J}_{x, y}(\ww^{\leftarrow}) = \init{\J}_{u+v-y, u+v-x}(\ww)$.  
\end{enumerate}
\end{lemma}
% %%%%%%%%%%
% Omitted proof below. 
% %%%%%%%%%
% \begin{proof}
% The following verifies part (a). Using \eqref{eq:Lref}, one obtains that 
% \begin{align*}
% \ter{\I}_{x, y}(\ww^{\leftarrow}) &= \Lpt_{x, y}(\ww^{\leftarrow})-\Lpt_{x, y-e_1}(\ww^{\leftarrow}) \\ 
% &= \Lpt_{u+v-y, u+v-x}(\ww)-\Lpt_{u+v-y+e_1, u+v-x}(\ww)= \init{\I}_{u+v-y, u+v-x}(\ww). 
% \end{align*}
% The proof of part (b) is similar. 
% \end{proof}
% %%%%%%%%%%%
% %%%%%%%%%%%

We next introduce the function $F = (F_1, F_2, F_3):\bbR^3 \to \bbR^3$ by   
\begin{align}
F(I,J,W) = (W+(I-J)^+,W+(J-I)^+,I\wedge J).\label{eq:Lindley}
\end{align}
The first two components of $F$ capture the increment recursion in \eqref{eq:inrec} and \eqref{eq:bwrec}. These recursions can now be expressed concisely as  
\begin{align}
\label{eq:IncRec}
\begin{split}
(\init{\I}_{x, y}, \init{\J}_{x, y}) &= (F_1, F_2)(\init{\I}_{x+e_2, y}, \init{\J}_{x+e_1, y}, \w_x), \\
(\ter{\I}_{x, y}, \ter{\J}_{x, y}) &= (F_1, F_2)(\ter{\I}_{x, y-e_2}, \ter{\J}_{x, y-e_1}, \w_y).
\end{split}
\end{align}
With the third component, $F$ becomes an involution. Consequently, one can write the recursions in \eqref{eq:IncRec} as well as the recovery property \eqref{eq:wrec} also in the form  
\begin{align}
\label{eq:IncRec-2}
\begin{split}
(\init{\I}_{x+e_2, y}, \init{\J}_{x+e_1, y}, \w_x) &= F(\init{\I}_{x, y}, \init{\J}_{x, y}, \init{\I}_{x+e_2, y} \wedge \init{\J}_{x+e_1, y}), \\
(\ter{\I}_{x, y-e_2}, \ter{\J}_{x, y-e_1}, \w_y) &= F(\ter{\I}_{x, y}, \ter{\J}_{x, y}, \ter{\I}_{x, y-e_2} \wedge \ter{\J}_{x, y-e_1}). 
\end{split}
\end{align}

We now extend the involution $F$ to rectangles. From the given $\ww$-weights on $\Rect_u^v$, define the \emph{dual weights} $\ww^* = \{\ww^*_x \in \bbR: x \in \Rect_{u}^v\}$ by
\begin{align}
\label{Edwop}
\begin{split}
\ww^*_x &=  (\ter{\I}_{u, x+e_1} \wedge \ter{\J}_{u, x+e_2}) \one_{\{x \le v-e_1-e_2\}} \\
&+ \ter{\I}_{u, x+e_1}\one_{\{x \cdot e_2 = v \cdot e_2,\, x < v\}} + \ter{\J}_{u, x+e_2}\one_{\{x \cdot e_1 = v \cdot e_1,\, x < v\}} \quad \text{ for } x \in \Rect_{u}^v. 
\end{split}
\end{align}
In particular, $\ww^*_v = 0$. Also, since $\ww_u$ is irrelevant to definition \eqref{Edwop}, one may assume here that $\ww_u = 0$ without any loss. By \eqref{eq:IncRec-2} and definition \eqref{Edwop}, in the special case $v = u + e_1 + e_2$ of a unit square, the three nontrivial $\ww^*$-weights form the vector  
\begin{align*}
(\ww^*_{u+e_2}, \ww^*_{u+e_1}, \ww^*_{u}) &= (\ter{\I}_{u, u+e_1+e_2}, \ter{\J}_{u, u+e_1+e_2}, \ter{\I}_{u, u+e_1} \wedge \ter{\J}_{u, u+e_2}) \\ 
&= F(\ter{\I}_{u, u+e_1}, \ter{\J}_{u, u+e_2}, \ww_{u+e_1+e_2}) = F(\ww_{u+e_1}, \ww_{u+e_2}, \ww_{u+e_1+e_2}). 
\end{align*}
Therefore, the $*$-map restricted to the weights on $\Rect_{u}^{u+e_1+e_2} \smallsetminus \{u\}$ coincides with the involution $F$ up to permuting and re-indexing the components. The following lemma observes that the involutive property of the $*$-map (composed with the reflection map $\leftarrow$) holds for an arbitrary rectangle. This gives a sense in which $\ww$ and $\ww^*$ are dual to each other. 
\begin{lemma}
\label{lem:dwop}
The map $\ww \mapsto (\ww^*)^{\leftarrow}$ is an involution on the space $\{\ww \in \bbR^{\Rect_u^v}: \ww_u = 0\}$.
\end{lemma}
One can verify Lemma \ref{lem:dwop} by computation using Lemmas \ref{lem:incid-2} and \ref{lem:incid}.  We will not appeal to Lemma \ref{lem:dwop} except for the purpose of motivation, so we omit its proof. 

We close this subsection with another set of increment identities which say that certain increments in the primal weights $\ww$ are equal to other increments in the dual weights $\ww^*$. Their proofs can be found in \cite[Lemma 4.7]{Sep-18}. 
\begin{lemma}
\label{lem:incid}
The following statements hold for $x \in \Rect_u^v$. 
\begin{enumerate}[label={\rm(\alph*)}, ref={\rm\alph*}] \itemsep=3pt
\item \label{lem:incid:1} If $x + e_1 \leq v$ then $\init{\I}_{x, v}(\ww^*) = \ter{\I}_{u, x+e_1}(\ww)$. 
\item If $x + e_2 \leq v$ then $\init{\J}_{x, v}(\ww^*) = \ter{\J}_{u, x+e_2}(\ww)$. 
\end{enumerate}
\end{lemma}

\subsection{Increment-stationary exponential LPP}
\label{sec:Burke}
Another crucial ingredient for the present work is that, even with inhomogeneity, one can create versions of the exponential LPP with stationary increments by introducing suitable boundary weights \cite{Emr-16}. Through various couplings with such processes, we will be able to perform exact calculations and in particular identify the distributions of the Busemann functions in Theorem \ref{thm:buslim}. Throughout this section, we recommend consulting with Figure \ref{fig:Burke}, which illustrates the first increment-stationary coupling we study.

\begin{figure}[h]
                \centering
                \begin{tikzpicture}[scale = 0.8]
		\draw[gray, very thin] (-2,-2) grid (4,4);
%		\draw[gray,thin] (-2,-2) grid (4,-1);
%		\draw[gray,thin] (-2,-2) grid (-1,4);
		\draw[](-0.6, 2.5) node [right]{{$\pi$}};
		% \draw[](-1.1, -.5) node [right]{{$\sG_\pi^-$}};
		% \draw[](1.9, 2.5) node [right]{{$\sG_\pi^+$}};
		\draw[<->](-2, 4.5)node[above]{}--(-2, -2)--(4.5, -2)node[right]{};
		\draw[dashed, very thick] (-2,4) -- (-1,4) -- (-1,3) -- (0,3) -- (0,2) -- (1,2) -- (1,1) -- (2,1) -- (2,0) -- (3,0) -- (3,-1) --(4,-1) -- (4,-2);
		\foreach \x in {0,...,4}
			\foreach \y in {\x,...,4}
				\filldraw (4-\x,\y) circle [radius=.08];		
		\foreach \x in {-2,...,3}
			\foreach \y in {-2,...,\x}
				\draw[draw=black,fill=white]  (1-\x,\y) circle [radius=.08];	
        \draw[](-1.3, -1.3) node{$u$};
		\draw[](4.3, 4.3) node{$v$};
        \draw[](4.85-7.6,4) node{$b_{n}$};
        \draw[](4.85-7.6,2) node{$b_{j}$};
        \draw[](1, 4.55-7.1) node{$a_{i}$};
		%\draw[](5.15-8,3) node{$b_{n-1}$};
		\draw[](4.85-7.6,-1) node{$b_{\ell}$};
%		\draw[](-2, 4.55) node{$\lambda_{k}$};
		\draw[](-1, 4.55-7.1) node{$a_{k}$};
		\draw[](2.5, 4.55-7.1) node{$\dots$};
        \draw[](0, 4.55-7.1) node{$\dots$};
        \draw[](4.85-7.6,0.8) node{$\vdots$};
        \draw[](4.85-7.6,3.2) node{$\vdots$};
		\draw[](3.85+0.3, 4.55-7.1) node{$a_{m}$};
        \foreach \x in {-2, ..., 3}
        {\draw[thick,->] (\x, 2-\x+0.1) -- (\x + 1,2-\x+0.1);
        \draw[thick,->] (2-\x+0.1, \x) -- (2-\x+0.1, \x + 1);
        }
        \node at (2.5,2.5){$\sG_{u-e_1-e_2, v, \pi}^+$};
        \node at (.5,.5){$\sG_{u-e_1-e_2, v, \pi}^-$};        
        %\draw[thick,->] (0,0) -- (0.5,0);
		\end{tikzpicture}
	    \caption{Illustrates the notation in Proposition \ref{prop:Burke} on $\Rect_{u-e_1-e_2}^v$ with $u = (k, \ell) < (m, n) = v$. A down-right path $\pi$ (dashed) from $(k-1, n)$ to $(m, \ell-1)$, the bulk weights $\w_x$ (black) strictly above $\pi$ (in $\sG_{u-e_1-e_2, v, \pi}^+$), and the dual weights $(\wh{\w}^{u, v, z})^*_x$ (hollow) strictly below $\pi$ (in $\sG_{u-e_1-e_2, v, \pi}^-$) are shown. Right and up arrows into $x \in \pi$ represent the increments $\wh{\I}_{u-e_1-e_2, x}^{u, v, z}$ and $\wh{\J}_{u-e_1-e_2, x}^{u, v, z}$, respectively. If $z \in (-a_{k:m}^{\min}, b_{\ell:n}^{\min})$, the drawn random variables are independent with marginals $\w_x \sim \Exp(a_i+b_j)$, $(\wh{\w}^{u, v, z})^*_{x} \sim \Exp(a_{i+1}+b_{j+1})$, $\wh{\I}_{u-e_1-e_2, x}^{u, v, z} \sim \Exp(a_{i}+z)$ and $\wh{\J}_{u-e_1-e_2, x}^{u, v, z} \sim \Exp(b_{j}-z)$ at $x = (i, j)$. 
     }\label{fig:Burke}
            \end{figure}
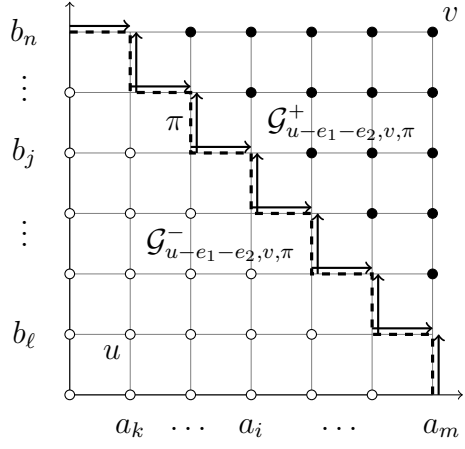

Let $u, v \in \bbZ^2$ satisfy $u \le v$, and pick a boundary parameter $z \in (-\smin{a}_{(u \cdot e_1):(v \cdot e_1)}, \smin{b}_{(u \cdot e_2):(v \cdot e_2)})$. Introduce a family of weights $\wh{\w}^{u, v, z} = \{\wh{\w}_x^{u, v, z}: x \in \Rect_{u-e_1-e_2}^v\}$ by 
\begin{align}
\label{Ewh}
\begin{split}
\wh{\w}^{u, v, z}_x &= \frac{\tau_x}{a_{x \cdot e_1} + z}\cdot \one_{\{x \cdot e_1 \ge u \cdot e_1,\, x \cdot e_2 = u \cdot e_2-1\}} + \frac{\tau_x}{b_{x \cdot e_2} - z} \cdot \one_{\{x \cdot e_2 \ge u\cdot e_2,\, x \cdot e_1 = u \cdot e_1-1\}} \\
&+\frac{\tau_x}{a_{x \cdot e_1} + b_{x \cdot e_2}} \cdot \one_{\{x \ge u\}} \quad \text{ for } x \in \Rect_{u-e_1-e_2}^v 
\end{split}
\end{align}
reusing the independent $\Exp(1)$-distributed weights $\{\tau_x: x \in \bbZ^2\}$ from Subsection \ref{sec:iExpLPP}. By definition, the weights $\wh{\w}^{u, v, z}$ are mutually independent, and the marginal distributions on the \emph{south} and \emph{west} boundaries are given by $\wh{\w}^{u, v, z}_{u-e_1-e_2} = 0$, 
\begin{align}
\label{ESWBd}
\begin{split}
\wh{\w}^{u, v, z}_{i, u \cdot e_2-1} &\sim \Exp\{a_i + z\} \quad \text{ for } u \cdot e_1 \le i \le v \cdot e_1 \quad \text{ and } \\
\wh{\w}^{u, v, z}_{u \cdot e_1-1, j} &\sim \Exp\{b_j-z\} \quad \text{ for } u \cdot e_2 \le j \le v \cdot e_2. 
\end{split}
\end{align}
Furthermore, due to \eqref{eq:wdef}, we have that $\wh{\w}^{u, v, z}_x = \w_x$ for $x \in \Rect_u^v$. %See Figure \ref{fig:wh}. 

The LPP process corresponding to the weights \eqref{Ewh} is given by 
\begin{align}
\label{ELpph}
\wh{\G}^{u, v, z}_{x, y} = \Lpt_{x, y}(\wh{\w}^{u, v, z}) \quad \text{ for } x, y \in \Rect_{u-e_1-e_2}^v. 
\end{align}
Denote the increments of this process with respect to the terminal points by 
\begin{align}
\label{EInch}
\begin{split}
\wh{\I}^{u, v, z}_{x, y} &= \ter{\I}_{x, y}(\wh{\w}^{u, v, z}) = \wh{\G}^{u, v, z}_{x,y}-\wh{\G}^{u, v, z}_{x, y-e_1} \quad \text{ and } \\ 
\wh{\J}^{u, v, z}_{x, y} &= \ter{\J}_{x, y}(\wh{\w}^{u, v, z})= \wh{\G}^{u, v, z}_{x,y}-\wh{\G}^{u, v, z}_{x, y-e_2} \quad \text{ for } x, y \in \Rect_{u-e_1-e_2}^v \text{ with } x \le y.
% \wh{\I}^{u, v, z}_{y} &= \ter{\I}_{u-e_1-e_2, y}(\wh{\w}^{u, v, z}) = \wh{\G}^{u, v, z}_{u-e_1-e_2,y}-\wh{\G}^{u, v, z}_{u-e_1-e_2, y-e_1} \quad \text{ and } \\ 
% \wh{\J}^{u, v, z}_{y} &= \ter{\J}_{u-e_1-e_2, y}(\wh{\w}^{u, v, z})= \wh{\G}^{u, v, z}_{u-e_1-e_2,y}-\wh{\G}^{u, v, z}_{u-e_1-e_2, y-e_2} \quad \text{ for } y \in \Rect_{u-e_1-e_2}^v. 
\end{split}
\end{align}
%Note that if $y\cdot e_1 = u \cdot e_1 -1$, then $\wh{\I}^{u, v, z}_{y}=\infty$ and if $y \cdot e_2 = u \cdot e_2 -1$, then $\wh{\J}^{u, v, z}_{y}=\infty$. 

As the next proposition shows, the increments in \eqref{EInch} enjoy a tractable distributional structure, which can be termed the \emph{Burke property} in analogy with earlier works \cite{Bal-Cat-Sep-06, Sep-12-corr}. 
\begin{proposition}[Burke property]
\label{prop:Burke}
The following statements hold.  
\begin{enumerate}[label={\rm(\alph*)}] 
\itemsep=3pt
\item $\wh{\I}^{u, v, z}_{u-e_1-e_2, x} \sim \Exp(a_{x \cdot e_1} + z)$ for $x \in \Rect_{u-e_2}^v$. 
\item $\wh{\J}^{u, v, z}_{u-e_1-e_2, x} \sim \Exp(b_{x \cdot e_2}-z)$ for $x \in \Rect_{u-e_1}^v$.  
\item $(\wh{\w}^{u, v, z})^*_x \sim \Exp(a_{x \cdot e_1+1} + b_{x \cdot e_2+1})$ for $x \in \Rect_{u-e_1-e_2}^{v-e_1-e_2}$. 
\item For any down-right path $\pi$ from $(u \cdot e_1-1, v \cdot e_2)$ to $(v \cdot e_1, u \cdot e_2-1)$, the collection 
\begin{align*}
&\{(\wh{\w}^{u, v, z})^*_x: x \in \sG_{u-e_1-e_2, v, \pi}^-\} \cup \{\wh{\I}_{u-e_1-e_2, x}^{u, v, z}: x, x-e_1 \in \pi\} \\ 
&\cup \{\wh{\J}_{u-e_1-e_2, x}^{u, v, z}: x, x-e_2 \in \pi\} \cup 
%\{\wh{\w}^{u, v, z}_x: x \in \sG_{u-e_1-e_2, v, \pi}^+\}
\{\w_x: x \in \sG_{u-e_1-e_2, v, \pi}^+\}
\end{align*}
is independent. 
\end{enumerate}
\end{proposition}
See Figure \ref{fig:Burke} below for an illustration. In part (c) of the proposition, $(\wh{\w}^{u, v, z})^*$ denotes the dual weights associated with the $\wh{\w}^{u, v, z}$-weights according to \eqref{Edwop}. Recall also from \eqref{eq:Gpm} that $\sG^\pm_{u-e_1-e_2, v, \pi}$ in part (d) are the two subsets of $\Rect_{u-e_1-e_2}^{v}$ strictly above and strictly below a given down-right path $\pi$. Parts (a), (b) and (d) together with definition \eqref{Ewh} imply that 
\begin{align}
\label{eq:hvinc}
\begin{split}
\{\wh{\I}_{u-e_1-e_2, (i, \ell)}^{u, v, z}: u \cdot e_1 \le i \le v \cdot e_1\} \stackrel{\rm{dist.}}{=} \{\wh{\w}_{(i, 0)}^{u, v, z}: u \cdot e_1 \le i \le v \cdot e_1\}, \\ 
\{\wh{\J}_{u-e_1-e_2, (k, j)}^{u, v, z}: u \cdot e_2 \le j \le v \cdot e_2\} \stackrel{\rm{dist.}}{=} \{\wh{\w}_{(0, j)}^{u, v, z}: u \cdot e_2 \le j \le v \cdot e_2\}
\end{split}
\end{align}
for any horizontal level $\ell \in \{u \cdot e_2-1, \dotsc, v \cdot e_2\}$ and vertical level $k \in \{u \cdot e_1-1, \dotsc, v \cdot e_1\}$. In particular, the $\wh{\G}^{u, v, z}$-process with the initial point fixed at $u-e_1-e_2$ has stationary increments in the sense that the joint distributions of the left-hand sides in \eqref{eq:hvinc} are not level-dependent.

Parts (a) and (b) as well as a slightly stronger version of \eqref{eq:hvinc} previously appeared in \cite[Proposition 4.1]{Emr-16}. %To our knowledge, the full content of Proposition \ref{prop:Burke} has not been reported in the literature. For completeness, we include the proof using a variation of the 
The proof is standard following the inductive argument in the proof of the homogeneous case in \cite[Lemma 4.2]{Bal-Cat-Sep-06}, so we omit it. %The basic input to the proof is the following distributional invariance of exponential random variables under the involution $F$ defined at \eqref{eq:Lindley}. %This has previously been shown by computing moment generating functions \cite[Lemma 4.1]{Bal-Cat-Sep-06} and densities \cite[Lemma 4.2]{Emr-16}. 

\subsection{Summary of coupled environments}
With the previous section in mind and before proceeding to the proofs, we now collect some of the notation for the various environments that will appear below to make the exposition easier to follow. Four types of environments, other than the bulk environment $\w$ defined in \eqref{eq:wdef}, appear in this section. The different environments are distinguished by a few features, which we now summarize. We outline the environments in the order that they will appear in the discussion to follow.

As we have just seen, we can construct an increment-stationary model by placing appropriate independent exponential weights on the south-west (SW) boundary of a rectangle and computing appropriate increments. The resulting field of increments and dual weights is illustrated in Figure \ref{fig:Burke}. This results in the environment $\wh{\w}^{u,v,z}$ defined in \eqref{Ewh} that we have just encountered. 

We can also build an increment stationary model by placing boundary conditions on the north-east (NE) boundary and computing appropriate increments. This results in the weights $\wt{\w}^{u,v,z}$ defined below at \eqref{Ewt}. Taking advantage of the involution recorded in Lemma \ref{lem:dwop} and the structure of the Burke property in Proposition \ref{prop:Burke}, these can be connected back to an environment of the type we have just seen. Because the reflection map in \eqref{eq:wref} reverses the order of parameters, to connect the south-west and north-east boundary models, it is convenient to introduce a south-west boundary model with reversed parameters, which we denote by $\wh{\w}^{u, v, z, \leftarrow}$ and define in \eqref{Ewhrev} below.

Finally, the main object of study in this section are the Busemann increments themselves. The cocycle and recovery properties imply that all Busemann functions in a rectangle can be recovered from the values of the Busemann functions on the north-east boundary and the bulk weights. It is thus natural to place the Busemann increments as north-east boundary conditions similar to the $\wt{\w}$ environments. This results in the environments $\wb{\w}^{x, v, \square}$ where $\square \in \{\xi, (k,\infty), (\infty,\ell) : \xi \in ]\mfc_1^x, \mfc_2^x[, k \geq x \cdot e_1, \ell \geq x \cdot e_2\}$ defined below at \eqref{eq:wb} and \eqref{eq:wb-thin}.\\

\begin{center}
\begin{tabular}{|c|c|c|c|}
\hline
Environment & Boundary Type  & Parameters & Definition\\
\hline
$\wh{\w}^{u,v,z}$ & SW & Normal & \eqref{Ewh}\\
\hline
$\wt{\w}^{u,v,z}$ & NE & Normal & \eqref{Ewt}\\
\hline
$\wh{\w}^{u,v,z,\leftarrow}$ & SW & Reversed & \eqref{Ewhrev}\\
\hline
$\wb{\w}^{x, v, \square}$ & NE, Busemann & Normal & \eqref{eq:wb} and \eqref{eq:wb-thin}\\
\hline
\end{tabular}
\end{center}
Each of these cases come equipped with passage times similar to $\wh{\G}^{u, v, z}$ as defined in \eqref{ELpph} and increments similar to $\wh{\I}^{u, v, z}$ and $\wh{\J}^{u, v, z}$ as defined in \eqref{EInch}.

\subsection{Northeast boundary and reversed parameters}
\label{sec:RevLPP}

Our argument will utilize several variations of the $\wh{\G}$-process defined at \eqref{ELpph}. These processes come in two basic types, one with \emph{northeast boundary} and another with \emph{reversed} inhomogeneity parameters.  

To introduce these notions, pick two vertices $u , v \in \bbZ^2$ with $u \le v$ and a boundary parameter $z \in (-\smin{a}_{(u \cdot e_1):(v \cdot e_1)}, \smin{b}_{(u \cdot e_2):(v \cdot e_2)})$ as before. Consider the weights on the rectangle $\Rect_{u}^{v+e_1+e_2}$ given by 
\begin{align}
\label{Ewt}
\begin{split}
\wt{\w}^{u, v, z}_x &= \frac{\tau_x}{a_{x \cdot e_1}+z} \cdot \one_{\{x \cdot e_1  \le v \cdot e_1, x \cdot e_2 = v \cdot e_2 + 1\}} + \frac{\tau_x}{b_{x \cdot e_2}-z} \cdot \one_{\{x \cdot e_1  = v \cdot e_1 + 1, x \cdot e_2 \le v \cdot e_2\}}\\
&+ \frac{\tau_x}{a_{x \cdot e_1} + b_{x \cdot e_2}} \cdot \one_{\{x \le v\}} \qquad \text{ for } x \in \Rect_{u}^{v+e_1+e_2}. 
\end{split}
\end{align}
As in \eqref{Ewh}, these weights agree with the $\w$ weights in the bulk, 
\begin{align}
\wt{\w}^{u, v, z}_x = \w_x \quad \text{ for } x \in \Rect_{u}^v, \label{Ecpl}
\end{align}
but now the boundary weights are placed on the north and east sides of $\Rect_u^{v+e_1+e_2}$. %See Figure \ref{fig:wt}. 

Define the last-passage times corresponding to the weights $\wt{\w}^{u, v, z}$ by 
\begin{align}
\wt{\G}^{u, v, z}_{x, y} = \Lpt_{x, y}(\wt{\w}^{u, v, z}) \quad \text{ for } x, y \in \Rect_{u}^{v+e_1+e_2},  \label{EGt}
\end{align}
and denote the corresponding increments with respect to initial points by 
\begin{align}
\label{EIJt}
\wt{\I}^{u, v, z}_{x, y} = \init{\I}_{x, y}(\wt{\w}^{u, v, z}) \quad \text{ and } \wt{\J}^{u, v, z}_{x, y} = \init{\J}_{x, y}(\wt{\w}^{u, v, z}) \quad \text{ for } x, y \in \Rect_{u}^{v+e_1+e_2} \text{ with } x \le y. 
\end{align}

To connect the $\wt{\G}^{u, v, z}$-process to a process of the form \eqref{ELpph}, let $\wh{\w}^{u, v, z, \leftarrow}$ denote the weights in \eqref{Ewh} computed with the reversed parameter sequences $a_{u, v}^{\leftarrow} = \{a_{v \cdot e_1}, a_{v \cdot e_1-1}, \dotsc, a_{u \cdot e_1}\}$ and $b_{u, v}^{\leftarrow} = \{b_{v \cdot e_2}, b_{v \cdot e_2-1}, \dotsc, b_{u \cdot e_2}\}$ in place of the parameters $a_{(u \cdot e_1):(v \cdot e_1)}$ and $b_{(u \cdot e_2):(v \cdot e_2)}$, respectively. More explicitly, 
\begin{align}
\label{Ewhrev}
\begin{split}
&\wh{\w}^{u, v, z, \leftarrow}_x  \\ 
&= \frac{\tau_x}{a_{x \cdot e_1}^{\leftarrow} + z}\cdot \one_{\{x \cdot e_1 \ge u \cdot e_1,\, x \cdot e_2 = u \cdot e_2-1\}} + \frac{\tau_x}{b_{x \cdot e_2}^{\leftarrow} - z} \cdot \one_{\{x \cdot e_2 \ge u,\, x \cdot e_1 = u \cdot e_1-1\}} \\
&+\frac{\tau_x}{a_{x \cdot e_1}^{\leftarrow} + b_{x \cdot e_2}^{\leftarrow}} \cdot \one_{\{x \ge u\}} \\ 
&=\frac{\tau_x}{a_{(u+v-x) \cdot e_1} + z}\cdot \one_{\{x \cdot e_1 \ge u \cdot e_1,\, x \cdot e_2 = u \cdot e_2-1\}} + \frac{\tau_x}{b_{(u+v-x) \cdot e_2} - z} \cdot \one_{\{x \cdot e_2 \ge u,\, x \cdot e_1 = u \cdot e_1-1\}} \\
&+\frac{\tau_x}{a_{(u+v-x) \cdot e_1} + b_{(u+v-x) \cdot e_2}} \cdot \one_{\{x \ge u\}} \quad \text{ for } x \in \Rect_{u-e_1-e_2}^v. 
\end{split}
\end{align}
Define 
$\wh{\G}^{u, v, z, \leftarrow}$ as in \eqref{ELpph}, and $\wh{\I}^{u, v, z, \leftarrow}$ and $\wh{\J}^{u, v, z, \leftarrow}$ as in \eqref{EInch} using the weights $\wh{\w}^{u, v, z, \leftarrow}$. 

\begin{lemma}
\label{lem:NEDisId}
The following distributional identities hold. 
\begin{enumerate}[label={\rm(\alph*)}] 
\itemsep=3pt
\item $\{\wt{w}_x^{u, v, z}: x \in \Rect_{u}^{v+e_1+e_2}\} \stackrel{\rm{dist.}}{=} \{\wh{w}_{u+v-x}^{u, v, z, \leftarrow}: x \in \Rect_u^{v+e_1+e_2}\}$.  
\item $\{\wt{\G}_{x, y}^{u, v, z}: x, y \in \Rect_u^{v+e_1+e_2}\} \stackrel{\rm{dist.}}{=} \{\wh{\G}^{u, v, z, \leftarrow}_{u+v-y, u+v-x}: x, y \in \Rect_{u}^{v+e_1+e_2}\}$. 
\item $
\begin{aligned}[t]
&\{\wt{\I}^{u, v, z}_{x, y}, \wt{\J}^{u, v, z}_{x, y}: x, y \in \Rect_{u}^{v+e_1+e_2} \text{ and } x \le y\} \\ 
&\stackrel{\rm{dist.}}{=} \{\wh{\I}^{u, v, z, \leftarrow}_{u+v-y, u+v-x}, \wh{\J}^{u, v, z, \leftarrow}_{u+v-y, u+v-x}: x, y \in \Rect_u^{v+e_1+e_2} \text{ and } x \le y\}. 
\end{aligned}
$
\end{enumerate}
\end{lemma}
\begin{proof}
Applying the reflection map $\leftarrow$ (on $\Rect_{u-e_1-e_2}^{v}$) from \eqref{eq:wref} to the $\wh{\w}^{u, v, z, \leftarrow}$-weights and then a shift by $e_1+e_2$ produces the following weights on $\Rect_{u}^{v+e_1+e_2}$. 
\begin{align}
\label{EwtRef}
\begin{split}
&(\wh{\w}^{u, v, z, \leftarrow})^{\leftarrow}_{x-e_1-e_2} = \wh{\w}^{u, v, z, \leftarrow}_{u-e_1-e_2+v-(x-e_1-e_2)} = \wh{\w}^{u, v, z, \leftarrow}_{u+v-x}\\ 
&= \frac{\tau_{u+v-x}}{a_{x \cdot e_1}+z} \cdot \one_{\{x \cdot e_1  \le v \cdot e_1, x \cdot e_2 = v \cdot e_2 + 1\}} + \frac{\tau_{u+v-x}}{b_{x \cdot e_2}-z} \cdot \one_{\{x \cdot e_1  = v \cdot e_1 + 1, x \cdot e_2 \le v \cdot e_2\}}\\
&+ \frac{\tau_{u+v-x}}{a_{x \cdot e_1} + b_{x \cdot e_2}} \cdot \one_{\{x \le v\}} \qquad \text{ for } x \in \Rect_{u}^{v+e_1+e_2}.
\end{split}
\end{align}
Since the $\tau$-variables are i.i.d., a comparison of \eqref{EwtRef} with \ref{Ewt} proves (a). 
% shows that 
% \begin{align}
% \label{Ewht}
% \{\wt{w}_x^{u, v, z}: x \in \Rect_{u}^{v+e_1+e_2}\} \stackrel{\rm{dist.}}{=} \{\wh{w}_{u+v-x}^{u, v, z, \leftarrow}: x \in \Rect_u^{v+e_1+e_2}\}.  
% \end{align}
Using the first line of \eqref{EwtRef} together with Lemma \ref{lem:LppShf} and identity \eqref{eq:Lref} also gives 
\begin{align}
\label{eq:Grev}
\begin{split}
\Lpt_{x, y}(\wh{\w}^{u, v, z, \leftarrow}_{u+v-\bullet}) &= \Lpt_{x, y}((\wh{\w}^{u, v, z, \leftarrow})^{\leftarrow}_{\bullet-e_1-e_2}) = \Lpt_{x-e_1-e_2, y-e_1-e_2}((\wh{\w}^{u, v, z, \leftarrow})^{\leftarrow}) \\
&= \Lpt_{u+v-y, u+v-x}(\wh{\w}^{u, v, z, \leftarrow}) 
= \wh{\G}^{u, v, z, \leftarrow}_{u+v-y, u+v-x} \quad \text{ for } x, y \in \Rect_{u}^{v+e_1+e_2}. 
\end{split}
\end{align}
The last equality holds by definition. Combining definition \eqref{EGt} with part (a) and \eqref{eq:Grev}, one reaches part (b). 
% the distributional identity
% \begin{align}
% \label{EGht}
% \{\wt{\G}_{x, y}^{u, v, z}: x, y \in \Rect_u^{v+e_1+e_2}\} \stackrel{\rm{dist.}}{=} \{\wh{\G}^{u, v, z, \leftarrow}_{u+v-y, u+v-x}: x, y \in \Rect_{u}^{v+e_1+e_2}\}. 
% \end{align}
%A relation similar to \eqref{EGht} can be stated for the corresponding increments as well. Indeed, 
Arguing as in \eqref{eq:Grev} and invoking Lemma \ref{lem:incid-2} along with definition \eqref{EInch}, one also finds that  
\begin{align}
\label{eq:GrevInc}
\begin{split}
\init{\I}_{x, y}(\wh{\w}_{u+v-\bullet}^{u, v, z, \leftarrow}) &= \ter{\I}_{u+v-y, u+v-x}(\wh{\w}^{u, v, z, \leftarrow}) = \wh{\I}_{u+v-y, u+v-x}^{u, v, z, \leftarrow}, \\ 
\init{\J}_{x, y}(\wh{\w}_{u+v-\bullet}^{u, v, z, \leftarrow}) &= \ter{\J}_{u+v-y, u+v-x}(\wh{\w}^{u, v, z, \leftarrow}) = \wh{\J}_{u+v-y, u+v-x}^{u, v, z, \leftarrow}
\end{split}
\end{align}
for $x, y \in \Rect_u^{v+e_1+e_2}$ with $x \le y$. On account of part (a) again and definition \eqref{EIJt}, part (b) follows from \eqref{eq:GrevInc}. 
% that 
% \begin{align}
% \label{EIht}
% \begin{split}
% &\{\wt{\I}^{u, v, z}_{x, y}, \wt{\J}^{u, v, z}_{x, y}: x, y \in \Rect_{u}^{v+e_1+e_2} \text{ and } x \le y\} \\ 
% &\stackrel{\rm{dist.}}{=} \{\wh{\I}^{u, v, z, \leftarrow}_{u+v-y, u+v-x}, \wh{\J}^{u, v, z, \leftarrow}_{u+v-y, u+v-x}: x, y \in \Rect_u^{v+e_1+e_2} \text{ and } x \le y\}. 
% \end{split}
% \end{align}
\end{proof}

Because the $\wh{\G}^{u, v, z, \leftarrow}$-process is precisely of the form in \eqref{ELpph}, one can now infer the following from Proposition \ref{prop:Burke}  and Lemma \ref{lem:NEDisId}.
%for the $\wt{\G}^{u, v, z}$-process. %See Figure \ref{fig:BurkNE} below. 

\begin{proposition}
\label{prop:BurkeNE}    
The following statements hold.
\begin{enumerate}[label={\rm(\alph*)}] 
\itemsep=3pt
\item $\wt{\I}_{x, v+e_1+e_2}^{u, v, z} \sim \Exp\{a_{x \cdot e_1} + z\}$ for $x \in \Rect_{u}^{v+e_2}$. 
\item $\wt{\J}_{x, v+e_1+e_2}^{u, v, z} \sim \Exp\{b_{x \cdot e_2}-z\}$ for $x \in \Rect_{u}^{v+e_1}$.  
\item $\wt{\I}^{u, v, z}_{x-e_1, v+e_1+e_2} \wedge \wt{\J}^{u, v, z}_{x-e_2, v+e_1+e_2} \stackrel{\rm{dist.}}{=} (\wh{\w}^{u, v, z, \leftarrow})^*_{u+v-x} \sim \Exp\{a_{x \cdot e_1-1} + b_{x \cdot e_2-1}\}$ for $x \in \Rect_{u+e_1+e_2}^{v+e_1+e_2}$. 
\item \label{prop:BurkeNE:indp} For any down-right path $\pi$ from $(u \cdot e_1, v \cdot e_2+1)$ to $(v \cdot e_1+1, u \cdot e_2)$, the collection 
\begin{align*}
&\{\w_x: x \in \sG_{u, v+e_1+e_2, \pi}^-\} \cup \{\wt{\I}_{x, v+e_1+e_2}^{u, v, z}: x, x+e_1 \in \pi\} \\ 
&\cup \{\wt{\J}_{x, v+e_1+e_2}^{u, v, z}: x, x+e_2 \in \pi\} \cup \{\wt{\I}^{u, v, z}_{x-e_1, v+e_1+e_2} \wedge \wt{\J}^{u, v, z}_{x-e_2, v+e_1+e_2}: x \in \sG_{u, v+e_1+e_2, \pi}^+\}
\end{align*}
is independent.
\end{enumerate}
\end{proposition}

\subsection{Limits of LPP increments in strictly concave regions}

Recall %from the discussion following \eqref{eq:crit} 
that the strictly concave region associated with $x \in \bbZ^2$ is the nonempty open interval $]\mfc_1^x, \mfc_2^x[\subset [e_2,e_1]$. We now examine directional limits of the $\G$-increments given by \eqref{eq:LPPInc} for directions in this interval. Our next result establishes the a.s.\ existence of these limits along with their recursive and distributional structure. 
\begin{lemma}
\label{lem:BuseSc}
Fix $x \in \bbZ^2$, $\xi \in ]\mfc_1^x, \mfc_2^x[$ and $(u_n)$  satisfying $n^{-1}u_n \to \xi$ as $n \to \infty$. Then there exist random real numbers $\B_y^{\xi, \hor} = \B_y^{x, \xi, \hor}$ and $\B_y^{\xi, \ver} = \B_y^{x, \xi, \ver}$ 
for $y \in \bbZ^2_{\ge x}$ such that the following statements hold. %$\bfP$ almost surely.  
\begin{enumerate}[label={\rm(\alph*)}, ref={\rm\alph*}] \itemsep=3pt
\item \label{lem:BuseSc:lim}
$\B_y^{\xi,\hor} \stackrel{\rm{a.s.}}{=} \lim \limits_{n \to \infty} \I_{y, u_n}$ and $\B_y^{\xi,\ver} \stackrel{\rm{a.s.}}{=} \lim \limits_{n \to \infty} \J_{y, u_n}$ for $y \in \bbZ_{\ge x}^2$. 
\item \label{lem:BuseSc:rec} $\B_y^{\xi, \hor} \stackrel{\rm{a.s.}}{=} \w_y + (\B_{y+e_2}^{\xi, \hor}-\B_{y+e_1}^{\xi, \ver})^+$ and $\B_y^{\xi, \ver} \stackrel{\rm{a.s.}}{=} \w_y + (\B_{y+e_1}^{\xi, \ver}-\B_{y+e_2}^{\xi, \hor})^+$ for $y \in \bbZ_{\ge x}^2$. 
% $
% \begin{aligned}[t]
% &\B_y^{\xi, \hor} \stackrel{\rm{a.s.}}{=} \w_y + (\B_{y+e_1}^{\xi, \hor}-\B_{y+e_2}^{\xi, \ver})^+, \quad \B_y^{\xi, \hor} \stackrel{\rm{a.s.}}{=} \w_y + (\B_{y+e_1}^{\xi, \hor}-\B_{y+e_2}^{\xi, \ver})^+, \\ 
% &\text{ and } \w_y \stackrel{\rm{a.s.}}{=} \B_y^{\xi, \hor} \wedge \B_y^{\xi, \ver} \text{ for } y \in \bbZ_{\ge x}^2. 
% \end{aligned}
% $
\item \label{lem:BuseSc:marg} $\B_y^{\xi, \hor} \sim \Exp(a_{y \cdot e_1} + \chi^{x}(\xi))$ and $\B_{y}^{\xi, \ver} \sim \Exp(b_{y \cdot e_2}-\chi^x(\xi))$ for $y \in \bbZ_{\ge x}^2$. 
\item \label{lem:BuseSc:indp} %Let $\w_y^\xi = \w_y^{x, \xi} = \B_{y-e_1}^{\xi,\hor}  \wedge \B_{y-e_2}^{\xi,\ver}$ for $y \in \bbZ_{\ge x+e_1+e_2}^2$. 
For any $v \in \bbZ_{\ge x}$ and down-right path $\pi$ from $(x \cdot e_1, v \cdot e_2)$ to $(v \cdot e_1, x \cdot e_2)$, the collection 
\begin{align*}
&\{\w_y: y \in \sG_{x, v, \pi}^-\} \cup \{\B_{y-e_1}^{\xi,\hor}  \wedge \B_{y-e_2}^{\xi,\ver}: y \in \sG_{x, v, \pi}^+\} \\ 
&\cup \{\B_y^{\xi, \hor}: y, y+e_1 \in \pi\} \cup \{\B_y^{\xi, \ver}: y, y+e_2 \in \pi\} 
\end{align*}
is independent. 
% \item 
% Let $v \in \bbZ_{\ge x + e_1 + e_2}^2$. Calling $\w_v^{\xi} = \B_{{\color{red}v-e_1}}^{\xi,\hor}  \wedge \B_{{\color{red}v-e_2}}^{\xi,\ver}$, the collection 
% \begin{align*}
% \{(\w_y^\xi: y \in \Rect_{x+e_1+e_2}^v), (\B_y^{\xi, \hor}: x \in \Rect_{x+e_1}^v), (\B_x^{\xi, \ver}: y \in \Rect_{x+e_2}^v), (\w_y: \Rect_{x}^{v-e_1-e_2})\}
% \end{align*}
% is an $(a_{x \cdot e_1, v \cdot e_1-1}, b_{x \cdot e_2, v \cdot e_2-1}, \zmin{x}(\xi))$ exponential system.
%
% Let $v \in \bbZ_{\ge x + e_1 + e_2}^2$. Calling $\w_v^{\xi} = \B_v^{\xi,\hor}  \wedge \B_v^{\xi,\ver} $, the collection 
% \begin{align*}
% \{(\w_y^\xi: y \in \Rect_{x+e_1+e_2}^v), (\B_y^{\xi, \hor}: x \in \Rect_{x+e_1}^v), (\B_x^{\xi, \ver}: y \in \Rect_{x+e_2}^v), (\w_y: \Rect_{x}^{v-e_1-e_2})\}
% \end{align*}
% is an $(a_{x \cdot e_1, v \cdot e_1-1}, b_{x \cdot e_2, v \cdot e_2-1}, \zmin{x}(\xi))$ exponential system.
%
\end{enumerate}
\end{lemma}

\begin{remark} (Consistency)
In the setting of the preceding lemma, pick $y, z \in \bbZ_{\ge x}^2$ with $y \ge z$. Due to \eqref{eq:critineq}, $]\mfc_1^x, \mfc_2^x[ \subset ]\mfc_1^z, \mfc_2^z[$. Therefore, $\mfc_1^z \prec \xi \prec \mfc_2^z$ as well. Then part (a) implies that $\B_y^{x, \xi, \hor} \stackrel{\rm{a.s.}}{=} \B_y^{z, \xi, \hor}$ while part (b) gives $\B_y^{z, \xi, \hor} \sim \Exp\{a_{y \cdot e_1} + \zmin{z}(\xi)\}$. For these to be consistent, we need $\zmin{z}(\xi) = \zmin{x}(\xi)$, which holds on $]\mfc_1^x, \mfc_2^x[$ by \eqref{eq:zetaext}. 
\end{remark}

%In particular, part (b) of the lemma shows that 
%\begin{align}
%\B_x^{\xi, \hor} \sim \Exp\{a_{x \cdot e_1} + \zmin{x}(\xi)\} \quad \text{ and } \quad \B_x^{\xi, \ver} \sim \Exp\{b_{x \cdot e_2}-\zmin{x}{\xi}\} \quad \text{ for } x \in \bbZ^2 \text{ and } \xi \in ]\mfc_1^x, \mfc_2^x[. \label{EBuseScMar}
%\end{align}

The proof of Lemma \ref{lem:BuseSc} is deferred to the end of this subsection. For now, we proceed to record some implications of it. 

\begin{lemma}
\label{lem:BuseScMon}
Let $x \in \bbZ^2$ be as in Lemma \ref{lem:BuseSc} and $\xi, \zeta \in ]\mfc_1^x, \mfc_2^x[$ with $\xi \prec \zeta$. Then, a.s.,  
\begin{align*}
\B_y^{\xi, \hor} \ge \B_y^{\zeta, \hor} \quad \text{ and } \quad \B_y^{\xi, \ver} \le \B_y^{\zeta, \ver} \quad \text{ for } y \in \bbZ^2_{\ge x}. 
\end{align*}
\end{lemma}
\begin{proof}
Pick sequences $(u_n^\xi)_{n \in \bbZ_{>0}}$ and $(u_n^\zeta)_{n \in \bbZ_{>0}}$ as in Lemma \ref{lem:BuseSc} and let $y \in \bbZ_{\ge x}^2$. Since $\xi, \zeta \in ]e_2, e_1[$ with $\xi \prec \zeta$, for sufficiently large $n_0 = n_0^y \in \bbZ_{>0}$, we have $y \cdot e_1 \le u_n^\xi \cdot e_1 < u_n^\zeta \cdot e_1$ and $y \cdot e_2 \le u_n^\zeta \cdot e_2 < u_n^\xi \cdot e_2$ for $n \ge n_0$. Therefore, Lemma \ref{lem:incineq} implies that 
\begin{align*}
\I_{y, u_n^\xi} \ge \I_{y, u_n^\zeta} \quad \text{ and } \quad \J_{y, u_n^\xi} \le \J_{y, u_n^\zeta} \quad \text{ for } n \ge n_0. 
\end{align*} 
Since $\xi, \zeta \in ]\mfc_1^x, \mfc_2^x[$, we send $n \to \infty$ and appeal to Lemma \ref{lem:BuseSc}(a) 
 to complete the proof. 
\end{proof}

Now fix a countable dense subset $\sU_0$ of $]e_2,e_1[$. Let $x \in \bbZ^2$ and $\sV_0^x = \sU_0 \cap ]\mfc_1^x, \mfc_2^x[$. By virtue of Lemmas \ref{lem:BuseSc}\eqref{lem:BuseSc:lim} and \ref{lem:BuseScMon}, there exists an a.s.\ event $\Omega_1 = \Omega_1^x$ such that the limits 
\begin{align}
\B_x^{\xi, \hor} = \lim \limits_{n \to \infty} \I_{x, u_n^\xi} \quad \text{ and } \quad \B_x^{\xi, \ver} = \lim \limits_{n \to \infty} \J_{x, u_n^\xi}, \label{EBuseLim}
\end{align}
and the inequalities 
\begin{align}
\B_x^{\xi, \hor} \ge \B_x^{\zeta, \hor} \quad \text{ and } \quad \B_x^{\xi, \ver} \le \B_x^{\zeta, \ver}\label{EBuseMon}
\end{align}
hold whenever $\xi, \zeta \in \sV_0^x$ with $\xi \preceq \zeta$, and $\omega \in \Omega_1$. Then define
\begin{align}
\label{EBusehvpm}
\begin{split}
\B_x^{\xi+, \hor} &= \sup \limits_{\substack{\zeta \in \sV_0^x \\ \zeta \succeq \xi}} \B_x^{\zeta, \hor} = \lim_{\substack{\zeta \in \sV_0^x \\ \zeta \downarrow \xi}} \B_x^{\zeta, \hor}, 
\qquad \B_x^{\xi+, \ver} = \inf \limits_{\substack{\zeta \in \sV_0^x \\ \zeta \succeq \xi}} \B_x^{\zeta, \ver} = \lim_{\substack{\zeta \in \sV_0^x \\ \zeta \downarrow \xi}} \B_x^{\zeta, \ver} \\
\B_x^{\xi-, \hor} &= \inf \limits_{\substack{\zeta \in \sV_0^x \\ \zeta \preceq \xi}} \B_x^{\zeta, \hor} = \lim_{\substack{\zeta \in \sV_0^x \\ \zeta \uparrow \xi}} \B_x^{\zeta, \hor}, 
\qquad \B_x^{\xi-, \ver} = \sup \limits_{\substack{\zeta \in \sV_0^x \\ \zeta \preceq \xi}} \B_x^{\zeta, \ver} = \lim_{\substack{\zeta \in \sV_0^x \\ \zeta \uparrow \xi}} \B_x^{\zeta, \ver}
\end{split}
\end{align}
for each $\xi \in ]\mfc_1^x, \mfc_2^x[$ and $\omega \in \Omega_1$. 
In particular, $\B_x^{\xi\pm, \hor}$ and $\B_x^{\xi\pm, \ver}$ coincide with $\B_x^{\xi, \hor}$ and $\B_x^{\xi, \ver}$, respectively, whenever $\xi \in \sV_0^x$.
%Note that, a.s., $\B_x^{\xi+, \hor}$ and $\B_x^{\xi+, \ver}$ are right continuous with left limits while $\B_x^{\xi-, \hor}$ and $\B_x^{\xi-, \ver}$ are left continuous with right limits in $\xi \in ]\mfc_1^x, \mfc_2^x[$. Moreover, a.s., $\B_x^{\xi\pm, \hor}$ are nonincreasing while $\B_x^{\xi\pm, \ver}$ are nondecreasing in $\xi$. The next lemma verifies also that the processes $(\B_x^{\xi\pm, \hor})$ and $(\B_x^{\xi\pm, \ver})$ are versions of the processes $(\B_x^{\xi, \hor})$ and $(\B_x^{\xi, \ver})$, respectively.  

% \begin{proof}
% By considering positive and negative parts and truncating, we may assume that $\bfP(0 \leq X \leq Y \leq C)=1$ for some constant $C$. In this case, the result follows from the fact that $\bfE[Y-X] =0$.
% \end{proof}

The next lemma shows that the former are versions of the latter with some path regularity. 

\begin{lemma}
\label{lem:pmbusdef}
Let $x \in \bbZ^2$ and $\Omega_1$ denote the event on which \eqref{EBuseLim} and \eqref{EBuseMon} hold. 
\begin{enumerate}[label={\rm(\alph*)}, ref={\rm\alph*}] \itemsep=3pt
\item\label{lem:pmbusdef-1} On $\Omega_1$, $\B_x^{\zeta+, \hor}$ and $\B_x^{\zeta+, \ver}$ are right continuous with left limits while $\B_x^{\zeta-, \hor}$ and $\B_x^{\zeta-, \ver}$ are left continuous with right limits in $\zeta \in ]\mfc_1^x, \mfc_2^x[$. 
\item\label{lem:pmbusdef-2} On $\Omega_1$, $\B_x^{\zeta\pm, \hor}$ are nonincreasing while $\B_x^{\zeta\pm, \ver}$ are nondecreasing in $\zeta \in ]\mfc_1^x, \mfc_2^x[$. 
\item\label{lem:pmbusdef-3} If $\xi \in ]\mfc_1^x, \mfc_2^x[$, then
$\bfP(
\B_x^{\xi^-, \hor} = \B_x^{\xi, \hor} = \B_x^{\xi^+, \hor})= \bfP(\B_x^{\xi^-, \ver} = \B_x^{\xi, \ver} = \B_x^{\xi+, \ver})=1$. 
\end{enumerate}
\end{lemma}
Before the proof, we record an easy fact about real random variables. 
\begin{lemma} \label{lem:disteq}
If $X$ and $Y$ satisfy $\bfP(X \leq Y) = 1$ and $X \stackrel{\tiny{d}}{=}Y$, then $\bfP(X=Y)=1.$
\end{lemma}

\begin{proof}[Proof of Lemma \ref{lem:pmbusdef}]
\eqref{lem:pmbusdef-1} and \eqref{lem:pmbusdef-2} are immediate from \eqref{EBuseMon} and \eqref{EBusehvpm}. 
%%%
%OMITTED DETAILS. 
% \underline{A few more details}. Part (b) is clear from definition \eqref{EBusehvpm}. Consequently, one-sided limits exist. We check that $\zeta \mapsto \B_x^{\zeta+, \hor}$ is right continuous at any $\xi \in ]\mfc_1^x, \mfc_2^x[$. Pick any $\eta_k \in ]\mfc_1^x, \mfc_2^x[$ for $k \in \bbZ_{>0}$ such that $\eta_k \downarrow \xi$. Since $\sV_0^x$ is dense, one can find $\zeta_k \prec \eta_k \prec \vartheta_k$ such that $\zeta_k, \vartheta_k \downarrow \xi$. Then $\B_x^{\vartheta_k, \hor} \le \B_{x}^{\eta_k+, \hor} \le \B_x^{\zeta_k, \hor}$ and $\B_x^{\xi+, \hor} = \lim \limits_{k \to \infty} \B_x^{\zeta_k, \hor} = \lim \limits_{k \to \infty} \B_x^{\vartheta_k, \hor}$. Consequently, one obtains that $\B_x^{\xi+, \hor} = \lim \limits_{k \to \infty} \B_x^{\eta_k, \hor}$ by squeezing. 
%END OF OMITTED DETAILS. 
%%%
Lemma \ref{lem:BuseScMon} and definition \eqref{EBusehvpm} imply the a.s.\ inequalities $\B_x^{\xi^-, \hor} \ge \B_x^{\xi, \hor} \ge \B_x^{\xi+, \hor}$. From the limits in \eqref{EBusehvpm}, Lemma \ref{lem:BuseSc} and continuity of $\zeta \mapsto \zmin{x}(\zeta)$, one concludes that 
$\B_x^{\xi^\pm, \hor} \sim \Exp\{a_{x \cdot e_1} + \zmin{x}(\xi)\}$. Therefore, the first probability in part (\ref{lem:pmbusdef-3}) indeed equals $1$ by Lemma \ref{lem:disteq}.
%Because the two extremal terms in the inequality have the same distribution, the first claim in \eqref{lem:pmbusdef-3} now follows. 
This is also true of the second probability in that expression via a similar argument. %The proof of the second is similar. 
\end{proof}

%\begin{lemma}
%\label{lem:pmbusdef}
%Let $x \in \bbZ^2$ and $\xi \in ]\mfc_1^x, \mfc_2^x[$. Then, a.s., 
%\begin{align*}
%\B_x^{\xi^-, \hor} = \B_x^{\xi, \hor} = \B_x^{\xi^+, \hor} \qquad \text{ and } \qquad \B_x^{\xi^-, \ver} = \B_x^{\xi, \ver} = \B_x^{\xi+, \ver}. 
%\end{align*}
%\end{lemma}
%\begin{proof}
%Lemma \ref{lem:BuseScMon} and definition \eqref{EBusehvpm} imply the a.s.\ inequalities $\B_x^{\xi^-, \hor} \ge \B_x^{\xi, \hor} \ge \B_x^{\xi+, \hor}$. From the limits in \eqref{EBusehvpm}, Lemma \ref{lem:BuseSc} and the continuity of the function $\zeta \mapsto \zmin{x}(\zeta)$, one concludes that 
%$\B_x^{\xi^\pm, \hor} \sim \Exp\{a_{x \cdot e_1} + \zmin{x}(\xi)\}$. The first two of the claimed equalities now follow from Lemma \ref{lem:disteq}. The proof of the last two is similar. 
%\end{proof}

We next extend Lemma \ref{lem:BuseSc}(a) in two ways. Part (a) of the next lemma shows that a suitable weakening of the limits in Lemma \ref{lem:BuseSc}(a) a.s.\ holds simultaneously for all directions in $]\mfc_1^x, \mfc_2^x[$. Part (b) strengthens Lemma \ref{lem:BuseSc}(a) by allowing any $\xi$-directed sequence. 
\begin{lemma}
\label{lem:BuseScLim}
Let $x \in \bbZ^2$ and $\xi \in ]\mfc_1^x, \mfc_2^x[$. The following statements hold. 
\begin{enumerate}[label={\rm(\alph*)}, ref={\rm\alph*}] \itemsep=3pt
\item \label{lem:BuseScLim-1} Let $\Omega_1$ denote the a.s.\ event on which \eqref{EBuseLim} and \eqref{EBuseMon} are in force. Then  for any $\omega \in \Omega_1$ and $(v_n)$ satisfying  $n^{-1}v_n \to \xi$ as $n \to \infty$,
\begin{alignat*}{2}
\varliminf_{n \to \infty} \I_{x, v_n} &\ge \B_x^{\xi+, \hor}, \quad \quad \quad &&\varlimsup_{n \to \infty} \I_{x, v_n} \le \B_x^{\xi-, \hor}, \\
\varliminf_{n \to \infty} \J_{x, v_n} &\ge \B_x^{\xi-, \ver}, \quad \text{ and } \quad &&\varlimsup_{n \to \infty} \J_{x, v_n} \le \B_x^{\xi+, \hor}.
\end{alignat*}
 
\item \label{lem:BuseScLim-2} There exists an a.s.\ event $\Omega_2 = \Omega_2^{x, \xi}$ such that  for any $\omega \in \Omega_2$ and any $(v_n)$ satisfying  $n^{-1}v_n \to \xi$ as $n \to \infty$,
\begin{align*}
\lim_{n \to \infty} \I_{x, v_n} = \B_x^{\xi\pm, \hor} \quad \text{ and } \quad \lim_{n \to \infty} \J_{x, v_n} = \B_x^{\xi\pm, \ver}.
\end{align*}
\end{enumerate}
\end{lemma}
\begin{proof}
Let $(v_n)$ satisfy  $n^{-1}v_n \to \xi$ as $n \to \infty$. Pick $\zeta, \eta \in \sV_0^x$ (defined in the paragraph of \eqref{EBuseLim}) such that $\zeta \prec \xi \prec \eta$. Then, as in the proof of Lemma \ref{lem:BuseScMon}, 
$
\I_{x, u_n^{\zeta}} \ge \I_{x, v_n} \ge \I_{x, u_n^{\eta}} 
$
for $n \in \bbZ_{\ge n_0}$ for some sufficiently large $n_0 \in \bbZ_{>0}$. Passing to the limit as $n \to \infty$ gives 
\begin{align*}
\varliminf_{n \to \infty} \I_{x, v_n} &\ge \B_x^{\eta, \hor} \quad \text{ and } \quad \varlimsup_{n \to \infty} \I_{x, v_n} \le \B_x^{\zeta, \hor} \quad \text{ for } \omega \in \Omega_1
\end{align*}
in view of \eqref{EBuseLim}. Now take $\zeta \uparrow \xi$ and $\eta \downarrow \xi$ in $\sV_0^x$ and recall \eqref{EBusehvpm} to obtain the first line of inequalities in \eqref{lem:BuseScLim-1}. The proof of the second line is similar. Finally, \eqref{lem:BuseScLim-2} follows from part (a) and Lemma \ref{lem:pmbusdef}\eqref{lem:pmbusdef-3}. 
\end{proof}

We now begin working towards the proof of Lemma \ref{lem:BuseSc}. %Our argument relies on certain couplings with an auxiliary LPP process constructed as follows. 
%, which is based on certain couplings with increment-stationary LPP processes constructed below. %Before reading this section, it is worth revisiting the discussion in Section \ref{sec:Burke} and looking at Figure \ref{fig:corner}.
%For some of our arguments, it will be convenient to index the dual weights in a manner slightly different from the one described in Sections \ref{sec:Burke} and \ref{sec:dualw}. 
Let $u, v \in \bbZ$ with $u \le v$ and $z \in (-\smin{a}_{(u \cdot e_1): (v \cdot e_1)}, \smin{b}_{(u \cdot e_2): (v \cdot e_1)})$. Recall from \eqref{ELpph} the increment-stationary LPP process $\wh{\G}^{u, v, z}_{u-e_1-e_2, \cdot}$ defined on the rectangle $\Rect_{u-e_1-e_2}^v$. Using the increments of this process, introduce new weights $\wc{\w}^{u, v, z} = \{\wc{\w}^{u, v, z}_x: x \in \Rect_{u}^{v+e_1+e_2}\}$ by 
\begin{align}
\label{Edw}
\begin{split}
\wc{\w}_x^{u, v, z} &= \wh{\I}^{u, v, z}_{u-e_1-e_2, x-e_2} \cdot \one_{\{x \cdot e_1 \le v \cdot e_1, x \cdot e_2 = v \cdot e_2 +1\}} + \wh{\J}^{u, v, z}_{u-e_1-e_2, x-e_1} \cdot \one_{\{x \cdot e_1 = v \cdot e_1 +1, x \cdot e_2 \le v \cdot e_2\}}\\
&+(\wh{\I}^{u, v, z}_{u-e_1-e_2, x-e_2} \wedge \wh{\J}^{u, v, z}_{u-e_1-e_2, x-e_1}) \cdot \one_{\{x \le v\}}  
\quad \text{ for } x \in \Rect_u^{v+e_1+e_2}. 
% \wc{\w}_x^{u, v, z} &= \wh{\I}^{u, v, z}_{x-e_2} \cdot \one_{\{x \cdot e_1 \le v \cdot e_1, x \cdot e_2 = v \cdot e_2 +1\}} + \wh{\J}^{u, v, z}_{x-e_1} \cdot \one_{\{x \cdot e_1 = v \cdot e_1 +1, x \cdot e_2 \le v \cdot e_2\}}\\
% &+(\wh{\I}^{u, v, z}_{x-e_2} \wedge \wh{\J}^{u, v, z}_{x-e_1}) \cdot \one_{\{x \le v\}}  
% \quad \text{ for } x \in \Rect_u^{v+e_1+e_2}. 
\end{split}
\end{align}
One can rewrite the preceding definition as 
% To connect back to the previous discussions, through definitions \eqref{EInch}, \eqref{Edw} and \eqref{Edwop}, one can write  
% \begin{align}
% \label{Eshift}
% \begin{split}
% \wc{\w}_{x+e_1+e_2}^{u, v, z} &= \wh{\I}^{u, v, z}_{u-e_1-e_2, x+e_1} \cdot \one_{\{x \cdot e_1 +1 \le v \cdot e_1, x \cdot e_2 = v \cdot e_2\}} + \wh{\J}^{u, v, z}_{u-e_1-e_2, x+e_2} \cdot \one_{\{x \cdot e_1 = v \cdot e_1, x \cdot e_2 + 1 \le v \cdot e_2\}}\\
% &+(\wh{\I}^{u, v, z}_{u-e_1-e_2, x+e_1} \wedge \wh{\J}^{u, v, z}_{u-e_1-e_2, x+e_2}) \cdot \one_{\{x+e_1+e_2 \le v\}} \\ 
% &= (\wh{\w}^{u, v, z})^*_x \quad \text{ for } x \in \Rect_{u-e_1-e_2}^{v}  
% % \wc{\w}_{x+e_1+e_2}^{u, v, z} &= \wh{\I}^{u, v, z}_{x+e_1} \cdot \one_{\{x \cdot e_1 +1 \le v \cdot e_1, x \cdot e_2 = v \cdot e_2\}} + \wh{\J}^{u, v, z}_{x+e_2} \cdot \one_{\{x \cdot e_1 = v \cdot e_1, x \cdot e_2 + 1 \le v \cdot e_2\}}\\
% % &+(\wh{\I}^{u, v, z}_{x+e_1} \wedge \wh{\J}^{u, v, z}_{x+e_2}) \cdot \one_{\{x+e_1+e_2 \le v\}} \\ 
% % &= (\wh{\w}^{u, v, z})^*_x \quad \text{ for } x \in \Rect_{u-e_1-e_2}^{v}  
% \end{split}
% \end{align}
% where $*$ denotes the operator (on the weight space $\sW_{u-e_1-e_2}^{v}$) defined at \eqref{Edwop}. Equivalently, 
\begin{align}
\label{Edw2}
\wc{\w}_x^{u, v, z} = (\wh{\w}^{u, v, z})^*_{x-e_1-e_2} \quad \text{ for } x \in \Rect_{u}^{v+e_1+e_2}
\end{align}
using the $*$-operator (on the weight space $\bbR^{\Rect_{u-e_1-e_2}^v}$) given by \eqref{Edwop}. %Incorporating the shift by $e_1+e_2$ into definition \eqref{Edw} will make some statements ahead slightly more concise. %; for example, the distributional identity \eqref{EwcDisId} below then appears without any shift. 
The next lemma is immediate from Proposition \ref{prop:Burke} and identity \eqref{Edw2}. %describes the joint distribution of the $\wc{\w}^{u, v, z}$-weights. 
\begin{lemma}
\label{lem:wcDis}
The weights $\wc{\w}^{u, v, z}$ are independent with $\wc{\w}^{u, v, z}_{v+e_1+e_2} = 0$ and 
\begin{align*}
\wc{\w}^{u, v, z}_x \sim 
\begin{cases}
\Exp(a_{x \cdot e_1} + z) &\quad \text{ if } x \cdot e_2 = v \cdot e_2 + 1 \text{ and } x \cdot e_1 \le v \cdot e_1, \\
\Exp(b_{x \cdot e_2} - z) &\quad \text{ if } x \cdot e_1 = v \cdot e_1 + 1 \text{ and } x \cdot e_2 \le v \cdot e_2, \\
\Exp(a_{x \cdot e_1} + b_{x \cdot e_2}) &\quad \text{ if } x \le v
\end{cases}
\end{align*}
% \begin{enumerate}[label={\rm(\alph*)}, ref={\rm\alph*}] \itemsep=3pt
% \item $\wc{\w}_x^{u, v, z} \sim $
% \item 
% \item 
% \end{enumerate}
\end{lemma}
%\begin{proof}
%Immediate from Proposition \ref{prop:Burke} and identity \eqref{Edw2}. 
%\end{proof}

We next extend the $\wc{\w}^{u, v, z}$-weights to $\bbZ^2_{\ge u}$ by observing a consistency property. Pick any $v' \in \bbZ^2$ with $v' \ge v$, and consider the weights $\wc{\w}^{u, v', z}$ defined according to \eqref{Edw} assuming further that $z \in (-\smin{a}_{(u \cdot e_1): (v' \cdot e_1)}, \smin{b}_{(u \cdot e_2): (v' \cdot e_2)})$. Then, for any $x \in \Rect_{u}^v$,  
\begin{align}
\label{Ewc-cons}
\begin{split}
\wc{\w}_x^{u, v', z} &= \wh{\I}^{u, v', z}_{u-e_1-e_2, x-e_2} \wedge \wh{\J}^{u, v', z}_{u-e_1-e_2, x-e_1} = \wh{\I}^{u, v, z}_{u-e_1-e_2, x-e_2} \wedge \wh{\J}^{u, v, z}_{u-e_1-e_2, x-e_1} = \wc{\w}_x^{u, v, z}. 
\end{split}
\end{align}
The second equality above holds because the $\wh{\w}^{u, v', z}$-weights restricted to $\Rect_{u-e_1-e_2}^v$ coincide with $\wh{\w}^{u, v, z}$. In view of identity \eqref{Ewc-cons}, for each boundary parameter $z \in (-\sinf{a}_{(u \cdot e_1): \infty}, \sinf{b}_{(u \cdot e_2): \infty})$, one can now define the weights $\wc{\w}^{u, z} = \{\wc{\w}^{u, z}_x: x \in \bbZ^2_{\ge u}\}$ consistently through 
\begin{align}
\label{Edw3}
\wc{\w}_{x}^{u, z} = \wc{\w}_{x}^{u, v, z} \quad \text{ for } x \in \bbZ_{\ge u}^2   
\end{align}
using any $v \in \bbZ^2$ with $v \ge x$. By Lemma \ref{lem:wcDis}, the $\wc{\w}^{u, z}$-weights are independent with marginals $\wc{\w}^{u, z}_x \sim \Exp(a_{x \cdot e_1} + b_{x \cdot e_2})$ for $x \in \bbZ^2_{\ge u}$. 
%Naturally, the distributional identity in \eqref{EwcDisId} is inherited by these new weights. 
In particular, these weights have the same joint distribution as the bulk weights in \eqref{eq:wdef}: %one has the distributional equality
\begin{align}
\label{EwcDisId}
\{\wc{\w}_x^{u, z}: x \in \bbZ_{\ge u}^2\} \stackrel{\text{dist}}{=} \{\w_x: x \in \bbZ_{\ge u}^2\}. %\label{Edwdist}
\end{align}

%A useful implication of \eqref{EwcDisId} is that the weights $\wc{\w}^{u, z}$ are identical in distribution to the weights $\w$ on $\bbZ^2_{\ge u}$: 
%\begin{align}
%\{\wc{\w}_x^{u, z}: x \in \bbZ_{\ge u}^2\} \stackrel{\text{dist}}{=} \{\w_x: x \in %\bbZ_{\ge u}^2\}. \label{Ewcdist}
%\end{align}

The last-passage times associated to the $\wc{\w}^{u, v, z}$-weights are 
\begin{align}
\wc{\G}^{u, v, z}_{x, y} = \Lpt_{x, y}(\wc{\omega}^{u, v, z}) \quad \text{ for } x, y \in \Rect_{u}^{v+e_1+e_2}. \label{ELppNW}
\end{align}
Denote the increments of this process with respect to the initial points by 
\begin{align}
\label{EIncc}
\begin{split}
\wc{\I}^{u, v, z}_{x, y} &= \init{\I}_{x, y}(\wc{\w}^{u, v, z}) = \wc{\G}^{u, v, z}_{x, y}-\wc{\G}^{u, v, z}_{x+e_1, y}\quad \text{ and } \\ 
\wc{\J}^{u, v, z}_{x, y} &= \init{\J}_{x, y}(\wc{\w}^{u, v, z}) = \wc{\G}^{u, v, z}_{x, y}-\wc{\G}^{u, v, z}_{x+e_2, y} \quad \text{ for } x, y \in \Rect_{u}^{v+e_1+e_2} \text{ with } x \le y. 
\end{split}
\end{align}
%The weights \eqref{Edw} can be written as 
%\begin{align}
%\wc{\w}^{u, v, z}_x = (\wh{\w}^{u, v, z})^*_{x-e_1-e_2} \quad %\text{ for } x \in \Rect_{u}^{v+e_1+e_2},  \label{Eshift}
%\end{align}
A key point will be that the preceding increments relate to the increments in \eqref{EInch} as follows. 

\begin{lemma}
\label{lem:incid-3}   
The following identities hold. 
\begin{enumerate}[label={\rm(\alph*)}, ref={\rm\alph*}] \itemsep=3pt
\item $\wc{\I}^{u, v, z}_{x, v+e_1+e_2} = \wh{\I}^{u, v, z}_{u-e_1-e_2, x-e_2}$ for $x \in \Rect_{u}^{v+e_2}$. 
\item $\wc{\J}^{u, v, z}_{x, v+e_1+e_2} = \wh{\J}^{u, v, z}_{u-e_1-e_2, x-e_1}$ for $x \in \Rect_u^{v+e_1}$. 
\end{enumerate}
%\begin{align}
% \label{EIncId}
% \begin{split}
% \wc{\I}^{u, v, z}_{x, v+e_1+e_2} &= \wh{\I}^{u, v, z}_{u-e_1-e_2, x-e_2} \quad \text{ for } x \in \Rect_{u}^{v+e_2}, \\ 
% \wc{\J}^{u, v, z}_{x, v+e_1+e_2} &= \wh{\J}^{u, v, z}_{u-e_1-e_2, x-e_1} \quad \text{ for } x \in \Rect_u^{v+e_1}. 
% \end{split}
% \end{align}
\end{lemma}
\begin{proof}
Part (a) comes from definitions \eqref{EInch} and \eqref{EIncc}, identity \eqref{Edw2} and Lemmas \ref{lem:LppShf}(b) and \ref{lem:incid}: %as follows. 
\begin{align}
\label{EIncIdPf}
\begin{split}
\wc{\I}^{u, v, z}_{x, v+e_1+e_2} &= \init{\I}_{x, v+e_1+e_2}(\wc{\w}^{u, v, z}) = \init{\I}_{x, v+e_1+e_2}((\wh{\w}^{u, v, z})^*_{\bullet-e_1-e_2}) \\ 
&= \init{\I}_{x-e_1-e_2, v}((\wh{\w}^{u, v, z})^*) \\ 
&= \ter{\I}_{u-e_1-e_2, x-e_2}(\wh{\w}^{u, v, z}) = \wh{\I}_{u-e_1-e_2, x-e_2}^{u, v, z} \quad \text{ for } x \in \Rect_{u}^{v+e_2}. 
\end{split}
\end{align}
% Writing out increments in terms of the last-passage map in \eqref{eq:Lpt}, it becomes clear that the shift of the environment can be transferred to the initial and terminal points, as done in the second line of \eqref{EIncIdPf}. 
The requirement $x \le v+e_2$ comes in when passing to the third line of \eqref{EIncIdPf}, and is imposed by part (a) of Lemma \ref{lem:incid}. Part (b) can be verified similarly.     
\end{proof}

The LPP process defined in \eqref{ELppNW} satisfies  
\begin{align}
\wc{\G}_{u, v+e_1+e_2}^{u, v, z} = \wc{\G}_{u, v+e_1}^{u, v, z} \vee \wc{\G}_{u, v+e_2}^{u, v, z} \label{E17}
\end{align}
due to the recursion \eqref{eq:LptRec} and the fact that $\wc{\w}_{v+e_1+e_2}^{u, v, z} = 0$. The next lemma determines which of the terms on the right-hand side attains the maximum in a certain asymptotic regime. See \cite[Lemma 4.8]{Sep-18} and \cite[Lemma 6.5]{Geo-Ras-Sep-17-ptrf-1} for analogous statements in i.i.d.~settings.
%The maximum in \eqref{E17} is attained by the first and second terms when the geodesic from $u$ to $v+e_1+e_2$ exits the northeast boundary from the east and north sides, respectively. The next lemma establishes which of these possibilities occurs in a certain asymptotic regime. 
%The homogeneous case of the lemma corresponds to \cite[Lemma 4.8]{Sep-18}. A version of the statement for i.i.d.\ setting appears in \cite[Lemma 6.5]{Geo-Ras-Sep-17-ptrf-1}.% In the proof, we appeal to a result from our previous paper \cite{Emr-Jan-Sep-21}.

\begin{lemma} \label{lem:halfstatLLN}
Let $x \in \bbZ^2$, $\xi \in ]\mfc_1^x, \mfc_2^x[$ and $z \in (-\sinf{a}_{(x \cdot e_1):\infty},  \sinf{b}_{(x \cdot e_2):\infty})$. Let $(v_n)_{n \in \bbZ_{>0}}$ be a sequence in $\bbZ^2$ such that $v_n/n \to \xi$. The following statements hold for each $y \in \bbZ_{\ge x}^2$. 
\begin{enumerate}[label={\rm(\alph*)}, ref={\rm\alph*}] \itemsep=3pt
\item \label{lem:halfstatLLN:z<}
If $z < \zmin{x}(\xi)$ then, a.s., $\wc{\G}_{y, v_n+e_1+e_2}^{x, v_n, z} = \wc{\G}_{y, v_n+e_2}^{x, v_n, z} > \wc{\G}_{y, v_n+e_1}^{x, v_n, z}$ for $n \ge N_0$ for some {\rm(}random{\rm)} $N_0 \in \bbZ_{>0}$. 
%Earlier version had $x$ in the subscripts instead of $y$: 
%If $z < \zmin{x}(\xi)$ then, a.s., $\wc{\G}_{x, v_n+e_1+e_2}^{x, v_n, z} = \wc{\G}_{x, v_n+e_2}^{x, v_n, z} > \wc{\G}_{x, v_n+e_1}^{x, v_n, z}$ for $n \ge N_0$ for some {\rm(}random{\rm)} $N_0 \in \bbZ_{>0}$. 
\item \label{lem:halfstatLLN:z>}
If $z > \zmin{x}(\xi)$ then, a.s., $\wc{\G}_{y, v_n+e_1+e_2}^{x, v_n, z} = \wc{\G}_{y, v_n+e_1}^{x, v_n, z} > \wc{\G}_{y, v_n+e_2}$ for $n \ge N_0$ for some {\rm(}random{\rm)} $N_0 \in \bbZ_{>0}$.
%Earlier version had $x$ in the subscripts instead of $y$: 
% If $z > \zmin{x}(\xi)$ then, a.s., $\wc{\G}_{x, v_n+e_1+e_2}^{x, v_n, z} = \wc{\G}_{x, v_n+e_1}^{x, v_n, z} > \wc{\G}_{x, v_n+e_2}$ for $n \ge N_0$ for some {\rm(}random{\rm)} $N_0 \in \bbZ_{>0}$. 
\end{enumerate}
\end{lemma}
\begin{proof}
Let $y \in \bbZ_{\ge x}^2$. Since $]\mfc_1^x, \mfc_2^x[ \subset ]\mfc_1^y, \mfc_2^y[$ by \eqref{eq:critineq}, definition \eqref{eq:zetaext} implies that $\chi^y(\xi) = \chi^x(\xi)$. %On account of the distributional structure recorded in
By Lemma \ref{lem:wcDis}, one can apply \cite[Theorem 3.6]{Emr-Jan-Sep-21} with the $\wc{\w}^{x, v_n, z}$-weights to obtain %that 
% \footnote{Further details can be found in Corollary 3.16 and Remark 3.16.2 of the first arXiv version of article \cite{Emr-Jan-Sep-21}.}
\begin{align*}
%n^{-1}\wc{\G}_{x, v_n+e_2}^{x, v_n, z} \to \inf_{s \in (-\inf a_{i, \infty}, z)} \shp_z (\xi), \quad 
n^{-1}\wc{\G}_{y, v_n+e_1}^{x, v_n, z} \stackrel{\rm{a.s.}}{\to} \inf_{w \in (z, b_{(y \cdot e_2), \infty}^{\inf})} \shp_w(\xi) \quad \text{ and } \quad n^{-1}\wc{\G}_{y, v_n+e_1+e_2}^{x, v_n, z} \stackrel{\rm{a.s.}}{\to} \shp_z(\xi) \quad \text{ as } n \to \infty.  
\end{align*}
If $z < \zmin{y}(\xi) = \zmin{x}(\xi)$ then the first limit above equals $\shp^y(\xi)$, which is strictly less than $\shp_z(\xi)$. This together with \eqref{E17} gives \eqref{lem:halfstatLLN:z<}. The proof of \eqref{lem:halfstatLLN:z>} is similar. 
\end{proof}

\begin{proof}[Proof of Lemma \ref{lem:BuseSc}]
Let $v \in \bbZ_{\ge x}^2$, and pick $N \in \bbZ_{>0}$ large enough that $u_n \ge v+e_1+e_2$ for $n \ge N$. Let $k = (v-x) \cdot e_1$ and $\ell = (v-x) \cdot e_2$, and pick any $s_i, t_j \in \bbR$ for $i \in [k]$ and $j \in [\ell]$. The core of our argument is to establish the following two inequalities  
\begin{align}
\label{E49}
\begin{split}
&\P\bigg\{\varlimsup_{n \to \infty} \I_{v-ie_1, u_n} > s_i \text{ and } \varliminf_{n \to \infty}\J_{v-je_2, u_n} < t_j \text{ for } i \in [k], j \in [\ell]\bigg\} \\
&\le \prod_{i \in [k]} \exp\{-(a_{v \cdot e_1-i}+\zmin{x}(\xi))s_i^+\} \prod_{j \in [\ell]} (1-\exp\{-(b_{v \cdot e_2 - j}-\zmin{x}(\xi)) t_j^+\}) \\
&\le \P\bigg\{\varliminf_{n \to \infty} \I_{v-ie_1, u_n} > s_i \text{ and } \varlimsup_{n \to \infty}\J_{v-je_2, u_n} < t_j \text{ for } i \in [k], j \in [\ell]\bigg\}. 
\end{split}
\end{align}
We include the details of the first inequality, with the second being similar.\\
%Recall from \eqref{eq:LPPInc} that the increments in \eqref{E49} are given by 
%\begin{align}
%\label{E26}
%\begin{split}
%\I_{v-ie_1, u_n} = \init{\I}_{v-ie_1, u_n}(\w) \quad \text{ and } \quad \J_{v-je_2, u_n} = \init{\J}_{v-je_2, u_n}(\w)
%\end{split}
%\end{align}
%for $i \in [k]$, $j \in [\ell]$ and $n \ge N$ where $\w$ refers to the bulk weights in \eqref{eq:wdef}. 
%
\underline{Deducing the lemma from \eqref{E49}}. If \eqref{E49} holds, the inequalities must be equalities because the first probability there is greater than or equal to the last probability. Combining this with Lemma \ref{lem:disteq} and the arbitrariness of the parameters $s_i$ and $t_j$, implies the limits% one concludes that the a.s.\ limits 
\begin{align}
\label{E51}
\B_{v-ie_1}^{\xi, \hor} = \lim_{n \to \infty} \I_{v-ie_1, u_n} \quad \text{ and } \quad \B_{v-je_2}^{\xi, \ver} = \lim_{n \to \infty} \J_{v-je_2, u_n}
\end{align}
exist a.s.\ for $i \in [k]$ and $j \in [\ell]$. Moreover, it implies that 
\begin{align}
\label{E52}
\begin{split}
&\{\B_{v-ie_1}^{\xi, \hor}: i \in [k]\} \cup \{\B_{v-je_2}^{\xi, \ver}: j \in [\ell]\} \text{ is independent with marginals } \\ 
&\B_{v - i e_1}^{\xi, \hor} \sim \Exp[a_{v \cdot e_1-i}+\chi^x(\xi)] \quad \text{ and } \quad \B_{v-je_2}^{\xi, \ver} \sim \Exp[b_{v \cdot e_2-j}-\chi^x(\xi)]. 
\end{split}
\end{align}
Because $v \in \bbZ_{\ge x}^2$ is arbitrary, \eqref{E51} and \eqref{E52} imply parts \eqref{lem:BuseSc:lim} and \eqref{lem:BuseSc:marg}, respectively. Also, part \eqref{lem:BuseSc:rec} follows from part \eqref{lem:BuseSc:lim} combined with the recursion  in \eqref{eq:inrec}. %and the recovery property \eqref{eq:wrec} 
%for the increments $\I_{y, u_n}$ and $\J_{y, u_n}$. 
% that the limits 
% \begin{align}
% \label{E53}
% \B^{\xi, \hor}_{p} = \lim_{n \to \infty} \I_{p, u_n} \quad \text{ and } \quad \B_{p}^{\xi, \ver} = \lim_{n \to \infty} \J_{p, u_n}
% \end{align}
% exist a.s.\ for any $p \in \bbZ^2_{\ge x}$, which is part (a) of the lemma. 
%(We omitted the $x$-dependence from the notation above). 

To derive part \eqref{lem:BuseSc:indp}, one may assume that $v \ge x + e_1 + e_2$ because the complementary case is already contained in \eqref{E52}. Consider the weights $\wb{\w}^{x, v, \xi}$ on $\Rect_x^v$ given by 
\begin{align}
\label{eq:wb}
\begin{split}
\wb{\w}^{x, v, \xi}_y &= \B_{y}^{\xi, \hor} \cdot \one_{\{y \cdot e_1 < v \cdot e_1, y \cdot e_2 = v \cdot e_2\}} + \B_{y}^{\xi, \ver} \cdot \one_{\{y \cdot e_1 = v \cdot e_1, y \cdot e_2 < v \cdot e_2}\} \\ 
&+ \w_y \cdot \one_{\{y \le v-e_1-e_2\}} \qquad \text{ for } y \in \Rect_x^v. 
\end{split}
\end{align}
Let $\wb{\G}^{x, v, \xi} = \Lpt(\wb{\w}^{x, v, \xi})$ denote the corresponding LPP process. We denote the increments of this process by $\wb{\I}^{x, v, \xi} = \init{\I}(\wb{\w}^{x, v, \xi})$ and $\wb{\J}^{x, v, \xi} = \init{\I}(\wb{\w}^{x, v, \xi})$. We claim that   
\begin{align}
\label{eq:Incwb}
\begin{split}
\wb{\I}^{x, v, \xi}_{y, v} = \wb{\G}^{x, v, \xi}_{y, v}-\wb{\G}^{x, v, \xi}_{y+e_1, v} &\stackrel{\rm{a.s.}}{=} \B_y^{\xi, \hor} \quad \text{ for } y \in \Rect_{x}^{v-e_1}, \\
\wb{\J}^{x, v, \xi}_{y, v} = \wb{\G}^{x, v, \xi}_{y, v}-\wb{\G}^{x, v, \xi}_{y+e_2, v} &\stackrel{\rm{a.s.}}{=} \B_y^{\xi, \ver} \quad \ \text{ for } y \in \Rect_{x}^{v-e_2}.
\end{split}
\end{align}
%where the first equalities are due to definition \eqref{eq:incint}. 
By \eqref{eq:wb}, the claimed identities hold when $y \cdot e_2 = v \cdot e_2$ and $y \cdot e_1 = v \cdot e_1$ (on the north and east boundaries), respectively. By \eqref{eq:inrec}, these satisfy the same recursion as $\B^{\xi, \hor}$ and $\B^{\xi, \ver}$ coming from part \eqref{lem:BuseSc:rec}, which implies the claim. %Furthermore, by \eqref{eq:inrec}, we have %the increments in \eqref{eq:Incwb} satisfy the recursion 
%\begin{align}
%\label{E107}
%\wb{\I}_{y, v}^{x, v, \xi} = \w_y + (\wb{\I}_{y+e_2, v}^{x, v, \xi}-\wb{\J}_{y+e_1, v}^{x, v, \xi})^+ \quad \text{ and } \quad \wb{\J}_{y, v}^{x, v, \xi} = \w_y + (\wb{\J}_{y+e_1, v}^{x, v, \xi}-\wb{\I}_{y+e_2, v}^{x, v, \xi})^+
%\end{align}
%for $y \le v-e_1-e_2$, which is the same recursion in part (b) for $\B^{\xi, \hor}$ and $\B^{\xi, \ver}$. %Hence, the identities in \eqref{eq:Incwb} indeed hold. 

Next compare definitions \eqref{Ewt} and \eqref{eq:wb}. Using \eqref{Ecpl} and \eqref{E52} along with the independence of the $\wb{\w}^{x, v, \xi}$-weights, one obtains the distributional identity
\begin{align}
\label{E108}
\begin{split}
\{\wb{\w}_y^{x, v, \xi}: y \in \Rect_x^v\} \stackrel{\rm{dist.}}{=} \{\wt{\w}_y^{x, v-e_1-e_2, \chi^x(\xi)}: y \in \Rect_x^v\}. 
\end{split}
\end{align}
%The increments of associated LPP processes $\wb{\G}^{x, v, \xi}$ and $\wt{\G}^{x, v-e_1-e_2, \chi^x(\xi)}$ are both defined with respect to the initial points (see \eqref{EIJt}), so 
Part \eqref{lem:BuseSc:indp} now follows from Proposition \ref{prop:BurkeNE}\ref{prop:BurkeNE:indp}. %\eqref{eq:Incwb} and \eqref{E109} 
To finish the proof, it remains now to derive the inequalities in \eqref{E49}. %We include the details of the first, with the second being similar.

\noindent\underline{Proof of the first bound in \eqref{E49}}.
%To obtain the first inequality, we first bound the probability 
%\begin{align}
%\label{E25}
%\P\bigg\{\sup_{n \ge N}\I_{v-ie_1, u_n} > s_i \text{ and } \inf_{n \ge N}\J_{v-je_2, u_n} < t_j \text{ for } i \in [k], j \in [\ell]\bigg\}  
%\end{align}
%from above.  
By \eqref{EwcDisId}, $\{\wc{\w}_p^{x, z}: p \in \bbZ_{\ge x}^2\}$ and $\{\w_p: p \in \bbZ_{\ge x}^2\}$ have the same distribution
%\begin{align}
%\label{E27}
%\{\wc{\w}_p^{x, z}: p \in \bbZ_{\ge x}^2\} \stackrel{\text{dist}}{=} \{\w_p: p \in \bbZ_{\ge x}^2\}
%\end{align}
for any %boundary parameter 
$z \in (-\sinf{a}_{(x \cdot e_1):\infty}, \sinf{b}_{(x \cdot e_2):\infty})$. Using the definitional fact (recorded as \eqref{Edw3}) that $\wc{\w}^{x, z}_p = \wc{\w}^{x, u_n, z}_p$ for $p \in \Rect_x^{u_n}$, it follows that $\{\wc{\I}_{y-ie_1, u_n}^{x, u_n, z}, \wc{\J}_{v-je_2, u_n}^{x, u_n, z} : i \in [k], j \in [\ell], n \geq N\}$ and $\{\I_{v-ie_1, u_n}, \J_{v-je_2, u_n} : i \in [k], j \in [\ell], n \geq N\}$ have the same distribution. This observation combined with the bounded convergence theorem implies that the first probability in \eqref{E49} is equal to the limit as $N\to\infty$ of
%Because the increments in \eqref{E26} depend only on the weights $\{\w_p: p \in \bbZ_{\ge x}^2\}$, it follows from \eqref{E27} that the collection \eqref{E26} is identical in distribution to following collection of increments:  
%\begin{align}
%\label{E28}
%\begin{split}
%\init{\I}_{v-ie_1, u_n}(\wc{\w}^{x, z}) &= \init{\I}_{v-ie_1, u_n}(\wc{\w}^{x, u_n, z}) = \wc{\I}_{y-ie_1, u_n}^{x, u_n, z}, \quad \text{ and } \\ 
%\init{\J}_{v-je_2, u_n}(\wc{\w}^{x, z}) &= \init{\J}_{v-je_2, u_n}(\wc{\w}^{x, u_n, z}) = \wc{\J}_{v-je_2, u_n}^{x, u_n, z}
%\end{split}
%\end{align}
%indexed by $i \in [k]$, $j \in [\ell]$ and $n \ge N$. 
%The first equalities in \eqref{E28} hold by \eqref{Edw3} because the preceding increments only use the weights 
%\begin{align}
%\label{E29}
%\wc{\w}^{x, z}_p = \wc{\w}^{x, u_n, z}_p %\quad \text{ for } p \in \Rect_x^{u_n}. 
%\end{align}
%The second equalities in \eqref{E28} hold by definition \eqref{EIncc}. 
%Now, due to the distributional equivalence of the collections \eqref{E26} and \eqref{E28}, the probability in \eqref{E25} is the same as  
\begin{align}
\label{E30}
\P\bigg\{\sup_{n \ge N}\wc{\I}^{x, u_n, z}_{v-ie_1, u_n} > s_i \text{ and } \inf_{n \ge N} \wc{\J}^{x, u_n, z}_{v-je_2, u_n} < t_j \text{ for } i \in [k], j \in [\ell]\bigg\}. 
\end{align}

Because $\xi \in ]\mfc_1^x, \mfc_2^x[$, \eqref{eq:zetaext} implies that $\zmin{x}(\xi) \in (-a_{(x \cdot e_1): \infty}^{\inf}, b_{(x \cdot e_2): \infty}^{\inf})$. Thus, one can pick $z \in (-a_{(x \cdot e_1):\infty}^{\inf}, \zmin{x}(\xi))$ and $w \in (\zmin{x}(\xi), b_{(x \cdot e_2):\infty}^{^{\inf}})$ arbitrarily close to $\zmin{x}(\xi)$. We will work with such $z$ for the upper bound on the probability \eqref{E30}, with $w$ playing a similar role for omitted proof of the lower bound in \eqref{E49}.  

Recall the following deterministic inequalities coming from Lemma \ref{lem:incineq}:
\begin{align}
\label{E31}
\begin{split}
\wc{\I}^{x, u_n, z}_{v-ie_1, u_n} \le \wc{\I}^{x, u_n, z}_{v-ie_1, u_n+e_2} \text{ and } \wc{\J}^{x, u_n, z}_{v-je_2, u_n} \ge \wc{\J}^{x, u_n, z}_{v-je_2, u_n+e_2} \text{ for } i \in [k], j \in [\ell], n \geq N.
\end{split}
\end{align}
%for $i \in [k]$, $j \in [\ell]$ and $n \ge N$. 
As a consequence of \eqref{E31}, the probability in \eqref{E30} is at most 
\begin{align}
\label{E32}
\begin{split}
\P\bigg\{\sup_{n \ge N}\wc{\I}^{x, u_n, z}_{v-ie_1, u_n+e_2} > s_i \text{ and } \inf_{n \ge N} \wc{\J}^{x, u_n, z}_{v-je_2, u_n+e_2} < t_j \text{ for } i \in [k], j \in [\ell]\bigg\}. 
\end{split}
\end{align}

Recall \eqref{E17} and introduce the event 
\begin{align}
\label{E41}
E_n^z &= \{\wc{\G}^{x, u_n, z}_{v-\ell e_2, u_n+e_1} = \wc{\G}^{x, u_n, z}_{v-\ell e_2, u_n+e_1+e_2}\}.
\end{align}
%for each $n \ge N$ using the $\wc{\G}$-process defined in \eqref{ELppNW}. Because  
%\begin{align}
%\label{E33}
%\wc{\G}^{x, u_n, z}_{v-\ell e_2, u_n+e_1+e_2} = \wc{\G}^{x, u_n, z}_{v-%\ell e_2, u_n+e_1} \vee \wc{\G}^{x, u_n, z}_{v-\ell e_2, u_n+e_2}
%\end{align}
%By \eqref{E17}, 
On the complement of the union $\bigcup_{n \ge N} E_n^z$, one has 
\begin{align}
\label{E34}
 \wc{\G}^{x, u_n, z}_{v-\ell e_2, u_n+e_1+e_2} =\wc{\G}^{x, u_n, z}_{v-\ell e_2, u_n+e_2} \quad \text{ for } n \ge N. 
\end{align}
Then Lemma \ref{lem:plan}\eqref{lem:plan:e1} implies that, on the complement of $\bigcup_{n \ge N} E_n^z$, 
\begin{align}
\label{E35}
\wc{\G}^{x, u_n, z}_{p, u_n+e_1+e_2} = \wc{\G}^{x, u_n, z}_{p, u_n+e_2} \quad \text{ for } p \in \Rect_x^v \text{ and } n \ge N,  
\end{align}
which in turn implies that for $i \in [k]$, $j \in [\ell]$, and $n \ge N$,
\begin{align}
\label{E36}
\wc{\I}^{x, u_n, z}_{v-ie_1, u_n+e_2} = \wc{\I}^{x, u_n, z}_{v-ie_1, u_n+e_1+e_2} \quad \text{ and } \quad \wc{\J}^{x, u_n, z}_{v-je_2, u_n+e_2} = \wc{\J}^{x, u_n, z}_{v-je_2, u_n+e_1+e_2}.
\end{align}
Then it follows from \eqref{E36} and a union bound that the probability in \eqref{E32} is at most 
\begin{align}
\label{E37}
\begin{split}
&\P\bigg\{\sup_{n \ge N}\wc{\I}^{x, u_n, z}_{v-ie_1, u_n+e_1+e_2} > s_i \text{ and } \inf_{n \ge N} \wc{\J}^{x, u_n, z}_{v-je_2, u_n+e_1+e_2} < t_j \text{ for } i \in [k], j \in [\ell]\bigg\} \\ 
&+ \P\bigg\{\bigcup_{n \ge N} E_n^z\bigg\}. 
\end{split}
\end{align}
By Lemma \ref{lem:incid-3}, we have the following identities for each $n \ge N$:
\begin{align}
\label{E38}
\begin{split}
\wc{\I}^{x, u_n, z}_{v-ie_1, u_n+e_1+e_2} &= \wh{\I}^{x, u_n, z}_{x-e_1-e_2, v-ie_1-e_2} = \wh{\I}^{x, v, z}_{x-e_1-e_2, v-ie_1-e_2} \quad \text{ for } i \in [k], \\ 
\wc{\J}^{x, u_n, z}_{v-je_2, u_n+e_1+e_2} &= \wh{\J}^{x, u_n, z}_{x-e_1-e_2, v-je_2-e_1} = \wh{\J}^{x, v, z}_{x-e_1-e_2, v-je_2-e_1} \quad \text{ for } j \in [\ell].
\end{split}
\end{align}
The second equalities in \eqref{E38} hold because the dependence on the weights $\wh{\w}^{x, u_n, z}$ above is only through their restriction to the rectangle $\Rect_{x-e_1-e_2}^{v}$. Since $n$ does not feature on the far right-hand sides in \eqref{E38}, Proposition \ref{prop:Burke} implies the first probability in \eqref{E37} is   
\begin{align}
\label{E39}
\begin{split}
&\P\{\wh{\I}^{x, v, z}_{x-e_1-e_2, v-ie_1-e_2} > s_i \text{ and } \wh{\J}^{x, v, z}_{x-e_1-e_2, v-je_2-e_1} < t_j \text{ for } i \in [k], j \in [\ell]\} \\ 
&= \prod_{i \in [k]} \exp\{-(a_{v \cdot e_1-i}+z)s_i^+\} \prod_{j \in [\ell]} (1-\exp\{-(b_{v \cdot e_2 - j}-z) t_j^+\}). 
\end{split}
\end{align}
%The exact value of the preceding probability comes from . 
Lemma \ref{lem:halfstatLLN}\eqref{lem:halfstatLLN:z<} implies that 
%To deal with the second probability in \eqref{E37}, one can apply 
%Lemma \ref{lem:halfstatLLN}(a), which implies the existence of a random $N_0 \in \bbZ_{>0}$ %(depending on $\xi$, $x$, $v$ and $z$) 
%such that 
%\begin{align}
%\label{E43}
%\wc{\G}^{x, u_n, z}_{v-\ell e_2, u_n+e_1} < \wc{\G}^{x, u_n, z}_{v-\ell e_2, u_n+e_1+e_2} \quad \text{ for } n \ge N_0 \quad \text{ a.s.}
%\end{align}
%It follows from \eqref{E43} that %and bounded convergence theorem that 
\begin{align}
\label{E44}
\lim_{N \to \infty} \P\bigg\{\bigcup_{n \ge N} E_n^z\bigg\} = \P\bigg\{\bigcap_{N} \bigcup_{n \ge N} E_n^z\bigg\}%\lim_{N \to \infty} \E\bigg[\prod_{n \ge N}(1-(E_n^z)^c)\bigg] = 
= 0. 
\end{align}
Sending $N\to\infty$ and  $z \nearrow \chi^x(\xi)$ in \eqref{E37} now implies the first inequality in \eqref{E49}. The second inequality is similar.
\end{proof}

\subsection{Limits of LPP increments in thin rectangles}
We turn to the Busemann functions associated with thin rectangles. Existence is immediate: for $x \in \bbZ^2$ and $(k, l) \in \bbZ_{\ge x}^2$, the monotonicity in Lemma \ref{lem:incineq} implies that
\begin{align}
\label{EBuseThin}
\begin{split}
\B_{x}^{(k, \infty), \hor} &= \sup_{n \ge x \cdot e_2} \I_{x, (k, n)} = \lim_{n \to \infty} \I_{x, (k, n)}, \qquad \B_{x}^{(k, \infty), \ver} = \inf_{n \ge x \cdot e_2} \J_{x, (k, n)} = \lim_{n \to \infty} \J_{x, (k, n)} \\
\B_{x}^{(\infty, \ell), \hor} &= \inf_{n \ge x \cdot e_1} \I_{x, (n, \ell)} = \lim_{n \to \infty} \I_{x, (n, \ell)}, \qquad \B_{x}^{(\infty, \ell), \ver} = \sup_{n \ge x \cdot e_1} \J_{x, (n, \ell)} = \lim_{n \to \infty} \J_{x, (n, \ell)}, 
\end{split}
\end{align}
where the first equalities are definitions. The preceding limits are readily identified from the definition \eqref{eq:incint} of the increments in the extreme cases below. 
\begin{align}
\label{EBuseThin-2}
\begin{split}
\B_x^{(k, \infty), \hor} &= \infty \quad \text{ and } \quad \B_x^{(k, \infty), \ver} = \w_x \quad \text{ if } k = x \cdot e_1, \\ 
\B_x^{(\infty, \ell), \hor} &= \w_x \quad \text{ and } \quad \B_x^{(\infty, \ell), \ver} = \infty \quad \text{ if } \ell = x \cdot e_2. 
\end{split}
\end{align}
%When $k = x \cdot e_1$ and $\ell = x \cdot e_2$, respectively, the first and last limits above equal $\infty$. 

We continue with the following recursion, which is the analogue of Lemma \ref{lem:BuseSc}(b) for the thin rectangle Busemann functions. This result follows from \eqref{eq:inrec} and \eqref{EBuseThin}.
\begin{lemma}
\label{lem:BuseThinRec}
Let $x \in \bbZ^2$, $(k, \ell) \in \bbZ^2_{\ge x+e_1+e_2}$ and $\square \in \{(k, \infty), (\infty, \ell)\}$. Then 
\begin{align*}
\B_x^{\square, \hor} = \w_x + (\B_{x+e_2}^{\square, \hor}-\B_{x+e_1}^{\square, \ver})^+ \quad \text{ and } \quad \B_x^{\square, \ver} = \w_x + (\B_{x+e_1}^{\square, \ver}-\B_{x+e_2}^{\square, \hor})^+. 
\end{align*}
\end{lemma}
%\begin{proof}
%These identities are inherited from the increment recursion \eqref{eq:inrec} through the limits in \eqref{EBuseThin}.     
%\end{proof}

Our next result records the monotonicity which is inherited from  Lemma \ref{lem:incineq}. 
\begin{lemma}
\label{lem:BuseThinMon}
Let $x  \in \bbZ^2$, $(k, \ell) \in \bbZ^2_{\ge x}$. %and $\square \in \{(k, \infty), (\infty, \ell)\}$. 
The following statements hold for $k' \in \bbZ_{\ge k}$ and $\ell' \in \bbZ_{\ge \ell}$. 
\begin{alignat*}{2}
&\B_x^{(k, \infty), \hor} \ge \B_x^{(k', \infty), \hor}, \quad  \quad \quad &&\B_x^{(k, \infty), \ver} \le \B_x^{(k', \infty), \ver}, \\
&\B_x^{(\infty, \ell), \hor} \le \B_x^{(\infty, \ell'), \hor}, \quad  \text{ and } \quad &&\B_x^{(\infty, \ell), \ver} \ge \B_x^{(\infty, \ell'), \ver}. 
\end{alignat*}
\end{lemma}
%\begin{proof}
%Iterate the first inequality in Lemma \ref{lem:incineq} to obtain 
%$\I_{x, (k, n)} \ge \I_{x, (k', n)}$ for %$n \in \bbZ$ with 
%$n \ge x \cdot e_2$, then send $n \to \infty$ %and using the first convergence in \eqref{EBuseThin} then yields 
%to obtain $\B_x^{(k, \infty), \hor} \ge \B_x^{(k', \infty), \hor}$. The remaining inequalities are similar. 
%\end{proof}

We now turn to the distributional structure of the limits in \eqref{EBuseThin}. 
% \begin{lemma}
% Let $x = (i, j) \in \bbZ^2$ and $(k, \ell) \in \bbZ_{\ge x}^2$. The following statements hold. 
% \begin{enumerate}[label={\rm(\alph*)}, ref={\rm\alph*}] \itemsep=3pt
% \item If $i = \ir{k}{x}$ then $\B_x^{(k, \infty), \ver} = \w_x$. 
% \item If $j = \jr{\ell}{x}$ then $\B_x^{(\infty, \ell), \hor} = \w_x$. 
% \end{enumerate}
% \end{lemma}
% \begin{proof}
% It follows from the recovery property \eqref{eq:wrec} and the second limit in \eqref{EBuseThin} that $\B_x^{(k, \infty), \ver} \ge \w_x$... 
% \end{proof}
For part (b) below, recall from \eqref{eq:rec} that $\ir{k}{x} \in \{x \cdot e_1, \dotsc, k\}$ is the first index where the minimum of the sequence $a_{(x \cdot e_1): k}$ is attained. Likewise, for $\jr{\ell}{x}$ and the sequence $b_{(x \cdot e_2): \ell}$. 

\begin{lemma}
\label{lem:nwBuseDis}
Let $x = (i, j) \in \bbZ^2$ and $(k, \ell) \in \bbZ^2_{\ge x}$. Let $\square \in \{(k, \infty), (\infty, \ell)\}$, and $v = (\ir{k}{x}, \ell)$ if $\square = (k, \infty)$, and $v = (k, \jr{\ell}{x})$ if $\square = (\infty, \ell)$. The following statements hold. 
\begin{enumerate}[label={\rm(\alph*)}, ref={\rm\alph*}] \itemsep=3pt
\item $
\begin{aligned}[t]
\B_{x}^{(k, \infty), \hor} &\sim \Exp(a_{i}-\smin{a}_{i: k}), \qquad \B_{x}^{(k, \infty), \ver} \sim \Exp(b_{j}+\smin{a}_{i: k}), \\
\B_{x}^{(\infty, \ell), \hor} &\sim \Exp(a_{i} + \smin{b}_{j: \ell}), \qquad \ \B_{x}^{(\infty, \ell), \ver} \sim \Exp(b_{j}-\smin{b}_{j: \ell}). 
\end{aligned}
$
\item 
$
\begin{aligned}[t]
\B_{x}^{(k, \infty), \hor} &\stackrel{\rm{a.s.}}{=} \B_{x}^{(\ir{k}{x}, \infty), \hor}, \qquad \B_{x}^{(k, \infty), \ver} \stackrel{\rm{a.s.}}{=} \B_{x}^{(\ir{k}{x}, \infty), \ver}, \\
\B_{x}^{(\infty, \ell), \hor} &\stackrel{\rm{a.s.}}{=} \B_{x}^{(\infty, \jr{\ell}{x}), \hor}, \qquad \B_{x}^{(\infty, \ell), \ver} \stackrel{\rm{a.s.}}{=} \B_x^{(\infty, \jr{\ell}{x}), \ver}. 
\end{aligned}
$\\
\item For any down-right path $\pi$ from $(i, v \cdot e_2)$ to $(v \cdot e_1, j)$, the collection 
\begin{align*}
&\{\w_y: y \in \sG_{x, v, \pi}^-\} \cup \{\B_y^{\square, \hor}: y, y+e_1 \in \pi\} \\
&\cup \{\B_y^{\square, \ver}: y, y+e_2 \in \pi\} \cup \{\B_{y-e_1}^{\square, \hor} \wedge \B_{y-e_2}^{\square, \ver}: y \in \sG_{x, v, \pi}^+\}
\end{align*}
is independent.  
\end{enumerate}
\end{lemma}

To prove the preceding lemma, we need a thin rectangle version of the exit point lemma (Lemma \ref{lem:halfstatLLN}) for the LPP process in \eqref{EGt}. 
\begin{lemma}
\label{lem:halfstatLLN2}
Let $x = (i, j) \in \bbZ^2$ and $(k, \ell) \in \bbZ^2_{\ge x}$. The following statements hold. 
\begin{enumerate}[label={\rm(\alph*)}, ref={\rm\alph*}] \itemsep=3pt
\item Let $z \in (-\smin{a}_{i: k},  \sinf{b}_{j: \infty})$. Then, a.s., $\wt{\G}_{x, (k+1, n)}^{x, (k, n), z} = \wt{\G}_{x, (k+1, n+1)}^{x, (k, n), z} > \wt{\G}_{x, (k, n+1)}^{x, (k, n), z}$ for $n \ge N$ for some {\rm(}random{\rm)} $N \in \bbZ_{> j}$. 
\item Let $z \in (- \sinf{a}_{i: \infty}, \smin{b}_{j: \ell})$. Then, a.s., $\wt{\G}_{x, (m, \ell+1)}^{x, (m, \ell), z} = \wt{\G}_{x, (m+1, \ell+1)}^{x, (m, \ell), z} > \wt{\G}_{x, (m+1, \ell)}^{x, (m, \ell), z}$ for $m \ge M$ for some {\rm(}random{\rm)} $M \in \bbZ_{>i}$. 
\end{enumerate}
\end{lemma}
\begin{proof}
Let $z \in (-\smin{a}_{i: k},  \sinf{b}_{j: \infty})$. It follows from \cite[Theorem 3.7(b)]{Emr-Jan-Sep-21} that, a.s., 
\begin{align*}
n^{-1}\wt{\G}_{x, (k, n+1)}^{x, (k, n), z} \to \int \frac{\beta(\dd b)}{b+\smin{a}_{i: k}} \quad \text{ and } \quad n^{-1}\wt{\G}_{x, (k+1, n+1)}^{x, (k, n), z} \to \int \frac{\beta(\dd b)}{b-z} \quad \text{ as } n \to \infty. 
\end{align*} 
Since $\beta$ is assumed nonzero, the second limit is strictly larger by the assumption $z > -\smin{a}_{i: k}$. This implies part (a) as in proof of Lemma \ref{lem:halfstatLLN}. The proof of part (b) is similar.
\end{proof}

A useful special case of Lemma \ref{lem:halfstatLLN2} is recorded as the following lemma. The idea (in part (a)) is that if $k = \ir{k}{x} > i$ then one can regard the weights along column $k$ as the east boundary weights with boundary parameter $z = -a_k = - \smin{a}_{i:k}$. 

\begin{lemma}
\label{lem:halfstatLLN3}
Let $x = (i, j) \in \bbZ^2$, and $(k, \ell) \in \bbZ^2_{\ge x}$. The following statements hold. 
\begin{enumerate}[label={\rm(\alph*)}, ref={\rm\alph*}] \itemsep=3pt
\item If $k = \ir{k}{x} > i$ then, a.s., $\wt{\G}_{x, (k, n)}^{x, (k-1, n), -a_k} = \wt{\G}_{x, (k, n+1)}^{x, (k-1, n), -a_k}$ for $n \ge N$ for some {\rm(}random{\rm)} $N \in \bbZ_{>j}$.   
\item If $\ell = \jr{\ell}{x} > j$ then, a.s., $\wt{\G}_{x, (m, \ell)}^{x, (m, \ell-1), b_{\ell}} = \wt{\G}_{x, (m+1, \ell)}^{x, (m, \ell-1), b_{\ell}}$ for $m \ge M$ for some {\rm(}random{\rm)} $M \in \bbZ_{>i}$. 
\end{enumerate}
\end{lemma}
\begin{proof}
To obtain part (a), apply Lemma \ref{lem:halfstatLLN2} with $(k-1, n)$ in place of $(k, n)$ and with $z = -a_k$. The proof of part (b) is similar. 
\end{proof}

We are now ready to prove the main lemma for the current subsection. 

\begin{proof}[Proof of Lemma \ref{lem:nwBuseDis}]
Recall that $x = (i, j) \in \bbZ^2$ and $(k, \ell) \in \bbZ^2_{\ge x}$. By symmetry, it suffices to prove the assertions of the lemma related to the $(k, \infty)$ Busemann functions. Hence, the vertex $v = (m, \ell)$ where $m = \ir{k}{x}$ and (without loss of generality) $\ell > j$. Somewhat similarly to the proof of Lemma \ref{lem:BuseSc}, the main part of our argument is to derive suitable bounds for the joint CDF of the Busemann functions along the northeast boundary of the rectangle $\Rect_{x}^{v}$. To this end, write $p = m-i$ and $q = \ell-j > 0$ for the side lengths of $\Rect_x^v$, and pick any $x_r, y_s \in \bbR$ for $r \in [p] \cup \{0\}$ and $s \in [q]$. 

% Let $p = k-i$ and $q = \ell-j$. If $p = q = 0$ then parts (b) and (c) are also trivially true. Part (a) holds as well because definition \eqref{EBuseThin} shows that $\B_{x}^{(k, \infty), \hor} = \infty$ and $\B_{x}^{(k, \infty), \ver} = \w_x$ in this case. Hence, assume that $p+q > 0$ below. 

\underline{Lower bound}. %We implement a variation of the argument for the lower bound in \eqref{E75} using the $\wt{\G}$-process defined at \eqref{EGt}. Recall from definition \eqref{eq:rec} that $a_m = \smin{a}_{i: k}$. 
Let $z \in (-a_m, \inf b_{j, \infty})$. By \eqref{Ecpl} and Lemma \ref{lem:incineq}, for any $n \in \bbZ_{>\ell}$,
\begin{align}
\label{E78}
\begin{split}
\I_{(m-r, \ell), (k, n)} &= \wt{\I}_{(m-r, \ell), (k, n)}^{x, (k, n), z} \ge \wt{\I}_{(m-r, \ell), (k+1, n)}^{x, (k, n), z} \ \quad \text{ for } r \in [p] \cup \{0\}, \\ 
\J_{(m, \ell-s), (k, n)} &= \wt{\J}_{(m, \ell-s), (k, n)}^{x, (k, n), z} \le \wt{\J}_{(m, \ell-s), (k+1, n)}^{x, (k, n), z} \quad  \text{ for } s \in [q]. 
\end{split}
\end{align}
Next consider the event 
\begin{align}
\label{E76}
E_n^z = \{\wt{\G}_{(i, \ell), (k, n+1)}^{x, (k, n), z} = \wt{\G}_{(i, \ell), (k+1, n+1)}^{x, (k, n), z}\} = \{\wt{\G}_{(i, \ell), (k, n+1)}^{(i, \ell), (k, n), z} = \wt{\G}_{(i, \ell), (k+1, n+1)}^{(i, \ell), (k, n), z}\}.
\end{align}
The second equality in \eqref{E76} is due the event $E_n^z$ depending only on the smaller collection of weights $\wt{\w}^{(i, \ell), (k, n), z}$. As a consequence of Lemma \ref{lem:plan}(a), on the complement of $E_n^z$, %one has 
\begin{align}
\label{E77}
\begin{split}
\wt{\I}_{(m-r, \ell), (k+1, n)}^{x, (k, n), z} &= \wt{\I}^{x, (k, n), z}_{(m-r, \ell), (k+1, n+1)} \quad \text{ for } r \in [p] \cup \{0\}, \\ 
\wt{\J}_{(m, \ell-s), (k+1, n)}^{x, (k, n), z} &= \wt{\J}^{x, (k, n), z}_{(m, \ell-s), (k+1, n+1)} \quad \text{ for } s \in [q]. 
\end{split}
\end{align}
It follows from \eqref{E78}, \eqref{E77} and a union bound that  
\begin{align}
\label{E11}
\begin{split}
&\P\bigl\{\I_{(m-r, \ell), (k, n)} > x_r \text{ for } r \in [p] \cup \{0\} \text{ and } \J_{(m, \ell-s), (k, n)} < y_s \text{ for } s \in [q]\bigr\} \\
%&=\P\{\wt{\I}_{(i-r, j), (k, n)}^{u, (k, n), z} > x_r \text{ for } r \in [p] \text{ and } \wt{\J}_{(i, j-s), (k, n)}^{u, (k, n), z} < y_s \text{ for } s \in [q]\} \\
&\ge \P\bigl\{\,\wt{\I}_{(m-r, \ell), (k+1, n)}^{x, (k, n), z} > x_r \text{ for } r \in [p] \cup \{0\} \text{ and } \wt{\J}_{(m, \ell-s), (k+1, n)}^{x, (k, n), z} < y_s \text{ for } s \in [q]\bigr\}\\ 
&\ge \P\bigl\{\,\wt{\I}_{(m-r, \ell), (k+1, n+1)}^{x, (k, n), z} > x_r \text{ for } r \in [p] \cup \{0\} \text{ and } \wt{\J}_{(m, \ell-s), (k+1, n+1)}^{x, (k, n), z} < y_s \text{ for } s \in [q]\bigr\}\\ 
&-\P\{E_n^z\}\\
&= \prod_{r \in [p] \cup \{0\}} \exp\{-(a_{m-r}+z)x_r^+\}\prod_{s \in [q]} \exp\{-(b_{\ell-s}-z)y_s^+\}-\P\{E_n^z\}.
\end{split}
\end{align}
%for any $x_r, y_s \in \bbR$ for $r \in [p]$ and $s \in [q]$. 
The exact expression in the last step of \eqref{E11} is due to Proposition \ref{prop:BurkeNE}. 
% comes from the distributional structure of the increments (Burke property) recorded in \eqref{EIncmar-2} and \eqref{EIncIndDR-2}. 
Via the second representation of the event $E_n^z$ in \eqref{E76} and Lemma \ref{lem:halfstatLLN2}(a), one has $\P\{E_n^z\} \to 0$ as $n \to \infty$. 
Therefore, letting $n \to \infty$ and then $z \downarrow -a_m$ in \eqref{E11} yields  
\begin{align}
\label{E12}
\begin{split}
&\P\bigl\{\B_{(m-r, \ell)}^{(k, \infty), \hor} > x_r \text{ for } r \in [p] \cup \{0\} \text{ and } \B_{(m, \ell-s)}^{(k, \infty), \ver} < y_s \text{ for } s \in [q]\bigr\} \\
&\ge \prod_{r \in [p]} \exp \{-(a_{m-r}-a_m)x_r^+\} \prod_{s \in [q]} (1-\exp\{-(b_{\ell-s}+a_m)y_s^+\}). 
\end{split}
\end{align}
In particular, one obtains from \eqref{E12} that  
\begin{align}
\label{E116}
\B_{(m, \ell)}^{(k, \infty), \hor} \stackrel{\rm{a.s.}}{=} \infty. 
\end{align}

\underline{Upper bound}. We next develop an upper bound matching \eqref{E12}. %The argument here is also of a similar flavor to the one for the upper bound in \eqref{E47}.
The key new observation is that one can profitably interpret $-a_m$ as a boundary parameter in this setting.

Since $m \le k$, repeated use of the first inequality in Lemma \ref{lem:incineq}(a) gives    
\begin{align}
\label{E79}
\begin{split}
\I_{(m-r, \ell), (k, n)} &\le \I_{(m-r, \ell), (m, n)} \quad \text{ for } r \in [p], \\ 
\J_{(m, \ell-s), (k, n)} &\ge \J_{(m, \ell-s), (m, n)} = \w_{(m, \ell-s)} %= \B_{(m, \ell-s)}^{(m, \infty), \ver}
\quad \text{ for } s \in [q]. 
\end{split}
\end{align}
The last equality in \eqref{E79} comes from the definition in \eqref{eq:incint}. %and \eqref{EBuseThin-2}. 
By \eqref{E79}, %our initial upper bound
\begin{align}
\label{E22}
\begin{split}
&\P\{\I_{(m-r, \ell), (k, n)} > x_r \text{ for } r \in [p] \text{ and } \J_{(m, \ell-s), (k, n)} < y_s \text{ for } s \in [q]\} \\
&\le \P\{\I_{(m-r, \ell), (m, n)} > x_r \text{ for } r \in [p] \text{ and } \w_{(m, \ell-s)} < y_s \text{ for } s \in [q]\} \\ 
&= \P\{\I_{(m-r, \ell), (m, n)} > x_r \text{ for } r \in [p]\} \cdot \P\{\w_{(m, \ell-s)} < y_s \text{ for } s \in [q]\} \\
&= \P\{\I_{(m-r, \ell), (m, n)} > x_r \text{ for } r \in [p]\} \cdot \prod_{s \in [q]} (1-\exp\{-(b_{\ell-s}+a_m)y_s^+\}). 
\end{split}
\end{align}
For the third line in \eqref{E22}, note that the increments $\I_{(m-r, \ell), (m, n)}$ for $r \in [p]$ do not use the bulk weights below the horizontal level $\ell$. %The final step in \eqref{E22} recalls definition \eqref{eq:wdef}. 

To continue developing the bound in \eqref{E22}, consider the case $m > i$. Since $a_{m-r} < a_m$ for $r \in [p]$, the weights $\wt{\w}^{x, (m-1, n), -a_m}$ defined on $\Rect_{x}^{(m, n+1)}$ by \eqref{Ewt} make sense. In fact, these weights coincide with the bulk weights on $\Rect_{x}^{(m, n)}$: 
\begin{align}
\label{E82}
\wt{\w}^{x, (m-1, n), -a_m}_y = \w_y \quad \text{ for } y \in \Rect_{x}^{(m, n)}. 
\end{align}
The preceding equality holds on $\Rect_x^{(m-1, n)}$ due to \eqref{Ecpl}. Thus, the new content of \eqref{E82} is that the equality holds also along the column $\{m\} \times [n]$, which acts as the east boundary. On account of \eqref{E82}, one has 
\begin{align}
\label{E83}
\begin{split}
\I_{(m-r, \ell), (m, n)} &= \wt{\I}_{(m-r, \ell), (m, n)}^{x, (m-1, n), -a_m} \quad \text{ for } r \in [p].%, %\\
%\J_{(m, \ell-s), (m, n)} &= \wt{\J}_{(m, \ell-s), (m, n)}^{x, (m-1, n), -a_m} \quad \text{ for } s \in [q]. 
\end{split}
\end{align}
Now introduce the event 
\begin{align}
\label{E80}
\begin{split}
F_n &= \{\wt{\G}_{(i, \ell), (m-1, n+1)}^{x, (m-1, n), -a_m} = \wt{\G}_{(i, \ell), (m, n+1)}^{x, (m-1, n), -a_m}\}= \{\wt{\G}_{(i, \ell), (m-1, n+1)}^{(i, \ell), (m-1, n), -a_m} = \wt{\G}_{(i, \ell), (m, n+1)}^{(i, \ell), (m-1, n), -a_m}\}. 
\end{split}
\end{align}
% Once again, in plain terms, $F_n$ is the event on which the geodesic (with respect to the weights $\wt{\w}^{u, (m-1, n), -a_m}$) from $(i-p, j)$ to $(m, n+1)$ visits $(m-1, n+1)$. Also,
The second equality in \eqref{E80} holds because the event $F_n$ depends only on the weights $\wt{\w}^{(i, \ell), (m-1, n), -a_m}$. %Furthermore, 
As a consequence of Lemma \ref{lem:plan}(b), one can switch the terminal points from $(m, n)$ to $(m, n+1)$ below on the complement of $F_n$: 
\begin{align}
\label{E81}
\begin{split}
\wt{\I}_{(m-r, \ell), (m, n)}^{x, (m-1, n), -a_m} &= \wt{\I}_{(m-r, \ell), (m, n+1)}^{x, (m-1, n), -a_m} \quad \text{ for } r \in [p].%, \\ 
%\wt{\J}_{(m, \ell-s), (m, n)}^{x, (m-1, n), -a_m} &= \wt{\J}_{(m, \ell-s), (m, n+1)}^{x, (m-1, n), -a_m} \quad \text{ for } s \in [q].  
\end{split}
\end{align}
By \eqref{E83}, \eqref{E81} and a union bound, 
%Putting together \eqref{E83}, \eqref{E81} and a union bound yields
\begin{align}
\label{E84}
\begin{split}
&\P\{\I_{(m-r, \ell), (m, n)} > x_r \text{ for } r \in [p] \} 
%\text{ and } \J_{(m, \ell-s), (m, n)} < y_s \text{ for } s \in [q]\} \\ 
%&=  
= \P\{\wt{\I}_{(m-r, \ell), (m, n)}^{x, (m-1, n), -a_m} > x_r \text{ for } r \in [p]\} \\%\text{ and } \wt{\J}_{(m, \ell-s), (m, n)}^{x, (m-1, n), -a_m} < y_s \text{ for } s \in [q]\} \\ 
&\le \P\{\wt{\I}_{(m-r, \ell), (m, n+1)}^{x, (m-1, n), -a_m} > x_r \text{ for } r \in [p]\} %\text{ and } \wt{\J}_{(m, \ell-s), (m, n+1)}^{x, (m-1, n), -a_m} < y_s \text{ for } s \in [q]\} 
+ \P\{F_n\} \\ 
&= \prod_{r \in [p]} \exp \{-(a_{m-r}-a_m)x_r^+\} %\prod_{s \in [q]} (1-\exp\{-(b_{\ell-s}+a_m)y_s^+\}) 
+\P\{F_n\}. 
\end{split}
\end{align}
The first term at the end of \eqref{E84} is again computed through Proposition \ref{prop:BurkeNE}. 
%the knowledge of the distributional structure observed in \eqref{EIncmar-2} and \eqref{EIncIndDR-2}. 
%In view of \eqref{E19}, 
One can conclude from the second form of $F_n$ in \eqref{E80} and Lemma \ref{lem:halfstatLLN3}(a) (the condition $k = \ir{k}{x} > i$ in the lemma is precisely that $m > i$ holds) that $\P\{F_n\} \to 0$ as $n \to \infty$. Consequently, combining \eqref{E22} and \eqref{E84} and then passing to the limit as $n \to \infty$ gives 
\begin{align}
\label{E21}
\begin{split}
&\P\{\B_{(m-r, \ell)}^{(k, \infty), \hor} > x_r \text{ for } r \in [p] \text{ and } \B_{(m, \ell-s)}^{(k, \infty), \ver} < y_s \text{ for } s \in [q]\} \\ 
% &\le \P\{\B_{(m-r, \ell)}^{(m, \infty), \hor} > x_r \text{ for } r \in [p] \text{ and } \B_{(m, \ell-s)}^{(m, \infty), \ver} < y_s \text{ for } s \in [q]\} \\
&\le \prod_{r \in [p]} \exp \{-(a_{m-r}-a_m)x_r^+\}  \prod_{s \in [q]} (1-\exp\{-(b_{\ell-s}+a_m)y_s^+\}), 
\end{split}
\end{align}
which provides the sought upper bound for the case $m > i$. The remaining case $m = i$ is already contained in \eqref{E22}. 

\underline{Completing the proof}. The matching bounds in \eqref{E12} and \eqref{E21} together with Lemma \ref{lem:disteq} imply that 
\begin{align}
\label{E23}
\begin{split}
&\B_{(m-r, \ell)}^{(k, \infty), \hor} \sim \Exp\{a_{m-r}-a_m\} \quad \text{ for } r \in [p] \cup \{0\}, \\
&\B_{(m, \ell-s)}^{(k, \infty), \ver} \stackrel{\rm{a.s.}}{=} \w_{(m, \ell-s)} 
\sim \Exp\{b_{\ell-s}+a_m\} \quad \ \ \text{ for } s \in [q], \text{ and }\\
&\{\B_{(m-r, \ell)}^{(k, \infty), \hor}: r \in [p]\} \cup \{\B_{(m, \ell-s)}^{(k, \infty), \ver}: s \in [q]\} \text{ are mutually independent. }
\end{split}
\end{align}
Because $x$ and $\ell$ are arbitrary, the first line of \eqref{E23} yields  
\begin{align}
\label{E117}
\B_y^{(k, \infty), \hor} \sim \Exp\{a_{y \cdot e_1}-\smin{a}_{y \cdot e_1, k}\} \quad \text{ for } y \in \bbZ^2 \text{ with } y \cdot e_1 \le k, 
\end{align}
proving the first statement in part (a). 

Next, considering the case $p = m-i > 0$, introduce a variation of the weights in \eqref{eq:wb} on the rectangle $\Rect_x^{v}$ (where $v = (m, \ell)$) as follows.  
\begin{align}
\label{eq:wb-thin}
\begin{split}
\wb{\w}_y^{x, v, (k, \infty)} &= 
\B_y^{(k, \infty), \hor} \cdot \one_{\{y \cdot e_1 < m, y \cdot e_2 = \ell\}} + \w_y \cdot \one_{\{y \cdot e_2 < \ell\}} \\ 
&\stackrel{\rm{a.s.}}{=} \B_y^{(k, \infty), \hor} \cdot \one_{\{y \cdot e_1 < m, y \cdot e_2 = \ell\}} + \B_y^{(k, \infty), \ver} \cdot \one_{\{y \cdot e_1 = m, y \cdot e_2 < \ell\}}\\ 
&+ \w_y \cdot \one_{\{y \cdot e_1 < m, y \cdot e_2 < \ell\}} \quad \text{ for } y \in \Rect_{x}^v. 
\end{split}
\end{align}
%Considering $\ell \in [L-1]$ such that $m_{\ell-1}+1 < m_\ell$, introduce a variation of the weights in \eqref{eq:wb} on $R_\ell$ by 
% \begin{align}
% \label{eq:wb-thin}
% \begin{split}
% \wb{\w}_y^{\ell, (k, \infty)} &= 
% \B_y^{(k, \infty), \hor} \cdot \one_{\{y \cdot e_1 < m_\ell, y \cdot e_2 = j\}} + \w_y \cdot \one_{\{y \cdot e_2 < j\}} \\ 
% &\stackrel{\rm{a.s.}}{=} \B_y^{(k, \infty), \hor} \cdot \one_{\{y \cdot e_1 < m_\ell, y \cdot e_2 = j\}} + \B_y^{(k, \infty), \ver} \cdot \one_{\{y \cdot e_1 = m_\ell, y \cdot e_2 < j\}}\\ 
% &+ \w_y \cdot \one_{\{y \cdot e_1 < m_\ell, y \cdot e_2 < j\}} \quad \text{ for } y \in R_\ell. 
% \end{split}
% % \end{align}
% The second equality in \eqref{eq:wb-thin} is due to Lemma \ref{lem:disteq} because $\B_{(m_\ell, j-s)}^{(k, \infty), \hor} \ge \w_{(m_\ell, j-s)}$ and both of these variables are $\Exp(b_{j-s} + a_{m_\ell})$-distributed by \eqref{eq:wdef} and \eqref{E105}.
Comparing \eqref{eq:wb-thin} with the $\wt{w}$-weights defined at \eqref{Ewt} implies the following:  
\begin{align}
\label{E111}
\wb{\w}^{x, v, (k, \infty)} \stackrel{\rm{dist.}}{=} \wt{\w}^{x, (m-1, \ell-1), -a_{m}}. 
\end{align}
Similarly to \eqref{eq:Incwb}, one also obtains the identities 
\begin{align}
\label{E110}
\begin{split}
\wb{\I}^{x, v, (k, \infty)}_{y, (m, \ell)} &= \underline{\I}_{y, (m, \ell)}(\wb{\w}^{x, v, (k, \infty)}) \stackrel{\rm{a.s.}}{=} \B_y^{(k, \infty), \hor} \quad \text{ if } y \cdot e_1 < m, \\
\wb{\J}^{x, v, (k, \infty)}_{y, (m, \ell)} &= \underline{\J}_{y, (m, \ell)}(\wb{\w}^{x, v, (k, \infty)}) \stackrel{\rm{a.s.}}{=} \B_y^{(k, \infty), \ver} \quad \text{ if } y \cdot e_2 < \ell
\end{split}
\end{align}
for $y \in \Rect_x^v$ through the agreement of the recursions in \eqref{eq:inrec} and Lemma \ref{lem:BuseThinRec} and their boundary values on the north and east sides. In view of \eqref{E111} and \eqref{E110}, it follows from Proposition \ref{prop:BurkeNE}(b) that 
\begin{align}
\label{E115}
\begin{split}
% \B_x^{(k, \infty), \hor} &\stackrel{\rm{dist.}}{=} \underline{\I}_{y, (m, \ell)}(\wt{\w}^{x, (m-1, \ell-1), -a_{m}}) = \wt{\I}_{x, (m, \ell)}^{x, (m-1, \ell-1), -a_m} \sim \Exp(a_{i}-a_m), \\ 
\B_y^{(k, \infty), \ver} &\stackrel{\rm{dist.}}{=} \underline{\J}_{y, (m, \ell)}(\wt{\w}^{x, (m-1, \ell-1), -a_{m}}) = \wt{\J}_{y, (m, \ell)}^{x, (m-1, \ell-1), -a_m} \sim \Exp(b_{y \cdot e_2}+a_m) 
\end{split}
\end{align}
for $y \in \Rect_x^v$ with $y \cdot e_2 < \ell$. With \eqref{E115}, the second statement in part (a) is also proved.

Part (b) follows from part (a), the monotonicities in Lemma \ref{lem:BuseThinMon} and Lemma \ref{lem:disteq}. 

To finish the proof, pick any down-right path $\pi$ from $(i, \ell)$ to $(m, j)$. If $m > i$ then, by \eqref{E111}, \eqref{E110} and Proposition \ref{prop:BurkeNE}(d), the collection 
\begin{align}
\label{E112}
 \begin{split}
&\{\w_y: y \in \sG_{x, v, \pi}^-\} \cup \{\B_y^{(k, \infty), \hor}: y, y+e_1 \in \pi\} \\ 
&\cup \{\B_{y}^{(k, \infty), \ver}: y, y+e_2 \in \pi\} \cup \{\B_{y-e_1}^{(k, \infty), \hor} \wedge \B_{y-e_2}^{(k, \infty), \ver}: y \in \sG_{x, v, \pi}^+\},  
\end{split}
\end{align}
which is distributionally equivalent to 
\begin{align*}
\begin{split}
&\{\w_y: y \in \sG_{x, v, \pi}^-\} \cup \{\wt{\I}_{y, (m, \ell)}^{x, (m-1, \ell-1), -a_m}: y, y+e_1 \in \pi\} \cup \{\wt{\J}_{y, (m, \ell)}^{x, (m-1, \ell-1), -a_m}: y, y+e_2 \in \pi\} \\ 
&\cup \{\wt{\I}_{y-e_1, (m, \ell)}^{x, (m-1, \ell-1), -a_m} \wedge \wt{\J}_{y-e_2, (m, \ell)}^{x, (m-1, \ell-1), -a_m}: y \in \sG_{x, v, \pi}^+\}, 
\end{split}
\end{align*}
is independent. In the remaining case $m = i$, the preceding independence also holds by the second line of \eqref{E23}. 
\end{proof}

\subsection{Limits of LPP increments in flat regions}
We continue with the study of Busemann functions in the flat regions. Recall from Section \ref{sec:LimShp} that for each $x \in \bbZ^2$, there are two (possibly empty) flat regions: one between the vertical axis and direction $\mfc_1^x$ and another between the horizontal axis and direction $\mfc_2^x$. Define the Busemann functions along the critical directions by  
\begin{align}
\label{EBuseCrit}
\begin{split}
&\B_x^{\mfc_1^x, \hor} = \inf \limits_{\substack{k \in \bbZ \\ k \ge x \cdot e_1}} \B_x^{(k, \infty), \hor} = \lim_{k \to \infty} \B_x^{(k, \infty), \hor}, \ \B_x^{\mfc_1^x, \ver} = \sup \limits_{\substack{k \in \bbZ \\ k \ge x \cdot e_1}} \B_x^{(k, \infty), \ver} = \lim_{k \to \infty} \B_x^{(k, \infty), \ver}, \\
&\B_x^{\mfc_2^x, \hor} = \sup \limits_{\substack{\ell \in \bbZ \\ \ell \ge x \cdot e_2}} \B_x^{(\infty, \ell), \hor} = \lim_{\ell \to \infty} \B_x^{(\infty, \ell), \hor}, \ \ \B_x^{\mfc_2^x, \ver} = \inf \limits_{\substack{\ell \in \bbZ \\ \ell \ge x \cdot e_2}} \B_x^{(\infty, \ell), \ver} = \lim_{\ell \to \infty} \B_x^{(\infty, \ell), \ver}. 
\end{split}
\end{align}
The second equalities above are due to Lemma \ref{lem:BuseThinMon}. 
Part \eqref{BuseFlat.a} of the next lemma shows that the definitions \eqref{EBuseCrit} from outside the concave region are matched by limits from inside the concave region. 
Part \eqref{BuseFlat.b} then shows that these quantities capture limits of LPP increments in flat regions.

\begin{lemma}
\label{lem:BuseFlat}
The following statements hold for each $x \in \bbZ^2$ $\P$-almost surely. 
\begin{enumerate}[label={\rm(\alph*)}, ref={\rm\alph*}] \itemsep=3pt
\item\label{BuseFlat.a} Let  $\square \in \{+, -\}$.
\begin{align*}
\B_x^{\mfc_1^x, \hor} &= \sup \limits_{\substack{\zeta \in ]\mfc_1^x, \mfc_2^x[}} \B_x^{\zeta\square, \hor} = \lim \limits_{\zeta \downarrow \mfc_1^x} \B_x^{\zeta\square, \hor}, 
\quad 
\B_x^{\mfc_1^x, \ver} = \inf \limits_{\substack{\zeta \in ]\mfc_1^x, \mfc_2^x[}} \B_x^{\zeta\square, \ver} = \lim \limits_{\zeta \downarrow \mfc_1^x} \B_x^{\zeta\square, \ver} \\
\B_x^{\mfc_2^x, \hor} &= \inf \limits_{\substack{\zeta \in ]\mfc_1^x, \mfc_2^x[}} \B_x^{\zeta\square, \hor} = \lim \limits_{\zeta \uparrow \mfc_2^x} \B_x^{\zeta\square, \hor}, 
\quad 
\B_x^{\mfc_2^x, \ver} = \sup \limits_{\substack{\zeta \in ]\mfc_1^x, \mfc_2^x[}} \B_x^{\zeta\square, \ver} = \lim \limits_{\zeta \uparrow \mfc_2^x} \B_x^{\zeta\square, \ver}. 
\end{align*}
\item\label{BuseFlat.b} If $\xi \in [e_2, \mfc_1^x]$ and $(u_n)_{n \in \bbZ_{>0}}$ satisfies $n^{-1}u_n \to \xi$ and  $u_n \cdot e_1 \to \infty$, then 
\begin{align*}
\B_x^{\mfc_1^x, \hor} \stackrel{\rm{a.s.}}{=} \lim_{n \to \infty}\I_{x, u_n}\quad \text{ and } \quad \B_x^{\mfc_1^x, \ver} \stackrel{\rm{a.s.}}{=} \lim_{n \to \infty}\J_{x, u_n}. 
\end{align*}
If $\xi \in [\mfc_2^x, e_1]$  and $(u_n)_{n \in \bbZ_{>0}}$ satisfies $n^{-1}u_n \to \xi$ and $u_n \cdot e_2 \to \infty$, then 
\begin{align*}
\B_x^{\mfc_2^x, \hor} \stackrel{\rm{a.s.}}{=} \lim_{n \to \infty}\I_{x, u_n}\quad \text{ and } \quad \B_x^{\mfc_2^x, \ver} \stackrel{\rm{a.s.}}{=} \lim_{n \to \infty}\J_{x, u_n}. 
\end{align*}
%\item $
%\begin{aligned}[t]
%\B_x^{\mfc_1^x, \hor} &\sim \Exp\{a_{x \cdot e_1}-\inf a_{x \cdot e_1, \infty}\}, \quad \B_x^{\mfc_1^x, \ver} \sim \Exp\{b_{x \cdot e_2}+\inf a_{x \cdot e_1, \infty}\}\\
%\B_x^{\mfc_2^x, \hor} &\sim \Exp\{a_{x \cdot e_1}+\inf b_{x \cdot e_2, \infty}\}, \quad \B_x^{\mfc_2^x, \ver} \sim \Exp\{b_{x \cdot e_2}-\inf b_{x \cdot e_2, \infty}\}. 
%\end{aligned}
%$ 
\end{enumerate}
\end{lemma}
\begin{proof}
Let $(k, n) \in \bbZ_{\ge x}^2$, $\zeta \in ]\mfc_1^x, \mfc_2^x[$, $\xi \in [e_2,\mfc_1^x]$, $\square \in \{+,-\}$, $n^{-1}v_n \to \zeta$, $n^{-1}u_n \to \xi$ and $u_n \cdot e_1 \to \infty$. %Assume also that $\xi \in [e_2, \mfc_1^x]$ and $u_n \cdot e_1 \to \infty$. 
Then $k \le u_n \cdot e_1 \le v_n \cdot e_1$ and $v_n \cdot e_2 \le u_n \cdot e_2 \le n$ for all sufficiently large $n$. %whenever $n \in \bbZ_{\ge n_0}$ for some sufficiently large $n_0 \in \bbZ_{>0}$. 
Hence, Lemma \ref{lem:incineq} implies that $\I_{x, (k, n)} \ge \I_{x, u_n} \ge \I_{x, v_n}$ for $n \in \bbZ_{\ge n_0}$. Letting $n \to \infty$ yields 
\begin{align*}\B_x^{(k, \infty)} \stackrel{\rm{a.s.}}{\ge} \varlimsup_{n \to \infty} \I_{x, u_n} \ge \varliminf_{n \to \infty} \I_{x, u_n} \stackrel{\rm{a.s.}}{\ge} \B_x^{\zeta\square, \hor}\end{align*}
in view of \eqref{EBuseThin} and Lemma \ref{lem:BuseScLim}. Then by definition \eqref{EBuseCrit} and since $\B_x^{\zeta\square, \hor}$ is a.s.\ nonincreasing in $\zeta$ by Lemma \ref{lem:pmbusdef}(\ref{lem:pmbusdef-2}), one obtains that 
\begin{align*}
\B_x^{\mfc_1^x, \hor} \stackrel{\rm{a.s.}}{\ge} \varlimsup_{n \to \infty} \I_{x, u_n} \ge \varliminf_{n \to \infty} \I_{x, u_n} \stackrel{\rm{a.s.}}{\ge} \sup \limits_{\substack{\zeta \in ]\mfc_1^x, \mfc_2^x[}} \B_x^{\zeta\square, \hor} \stackrel{\rm{a.s.}}{=} \lim \limits_{\zeta \downarrow \mfc_1^x} \B_x^{\zeta\square, \hor}. 
\end{align*}
Recalling the definition of $\zmin{x}$ in \eqref{eq:zetaext}, Lemmas \ref{lem:BuseSc}(c) %\ref{lem:pmbusdef}, 
and \ref{lem:nwBuseDis}(a) imply that 
\begin{align}
\label{E118}
\B_x^{(k, \infty)} \sim \Exp\{a_{x \cdot e_1}-\smin{a}_{(x \cdot e_1): k}\} \quad \text{ and } \quad \B_x^{\zeta\square, \hor} \sim \Exp\{a_{x \cdot e_1}+\zmin{x}(\zeta)\}. 
\end{align}
These distributions both converge to $\Exp\{a_{x \cdot e_1} - a_{(x \cdot e_1): \infty}^{\inf}\}$ as $k \to \infty$ and $\zeta \downarrow \mfc_1^x$, respectively. %The latter convergence uses the continuity of the function \eqref{eq:zetaext}. 
The first two equalities in (a) and the first equality in (b) follow from Lemma \ref{lem:disteq}. The remaining statements are similar. 
\end{proof}

On account of Lemma \ref{lem:BuseFlat}(b), it makes sense to define 
\begin{align}
\label{eq:BuseFlat}
\begin{split}
\B_x^{\xi^\pm, \hor} &= \B_x^{\mfc_1^x, \hor} \quad \text{ and } \quad \B_x^{\xi^\pm, \ver} = \B_x^{\mfc_1^x, \ver} \quad \text{ for } \xi \in [e_2, \mfc_1^x], \\
\B_x^{\xi^\pm, \hor} &= \B_x^{\mfc_2^x, \hor} \quad \text{ and } \quad \B_x^{\xi^\pm, \ver} = \B_x^{\mfc_2^x, \ver} \quad \text{ for } \xi \in [\mfc_2^x, e_1]. 
\end{split}
\end{align}
%With this extension, 
We have the following version of Lemma \ref{lem:BuseSc}(b)--(d) in the linear segments. 

\begin{lemma}
\label{lem:BuseFlatBurke}
Let $x = (i, j) \in \bbZ^2$ and $\xi \in [e_2, \mfc_1^x] \cup [\mfc_2^x, e_1]$. The following statements hold. 
\begin{enumerate}[label={\rm(\alph*)}, ref={\rm\alph*}] \itemsep=3pt
% \item $\B_y^{\xi, \hor} \stackrel{\rm{a.s.}}{=} \w_y + (\B_{y+e_1}^{\xi, \hor}-\B_{y+e_2}^{\xi, \ver})^+$ and $\B_y^{\xi, \hor} \stackrel{\rm{a.s.}}{=} \w_y + (\B_{y+e_1}^{\xi, \hor}-\B_{y+e_2}^{\xi, \ver})^+$ for $y \in \bbZ_{\ge x}^2$. 
\item \label{lem:BuseFlatBurke:c1marg} If $\xi \preceq \mfc_1^x$ then $\B_x^{\xi, \hor} \sim \Exp(a_{i} - \sinf{a}_{i:\infty})$  and 
$\B_x^{\xi, \ver} \sim \Exp(b_{j} + \sinf{a}_{i:\infty})$. 
\item \label{lem:BuseFlatBurke:c2marg} If $\xi \succeq \mfc_2^x$ then $\B_x^{\xi, \hor} \sim \Exp(a_{i} + \sinf{b}_{j: \infty})$  and 
$\B_x^{\xi, \ver} \sim \Exp(b_{j} - \sinf{b}_{j: \infty})$.
% \item $\B_y^{\xi, \hor} \sim \Exp(a_{y \cdot e_1} + \chi^{x}(\xi))$ and $\B_{y}^{\xi, \ver} \sim \Exp(b_{y \cdot e_2}-\chi^x(\xi))$ for $y \in \bbZ_{\ge x}^2$.
\item\label{lem:BuseFlatBurke:rec} $\B_x^{\xi, \hor} \stackrel{\rm{a.s.}}{=} \w_x + (\B_{x+e_2}^{\xi, \hor}-\B_{x+e_1}^{\xi, \ver})^+$ and $\B_x^{\xi, \ver} \stackrel{\rm{a.s.}}{=} \w_x + (\B_{x+e_1}^{\xi, \ver}-\B_{x+e_2}^{\xi, \hor})^+$. 
\item\label{lem:BuseFlatBurke:indp} 
For $v \in \bbZ_{\ge x}$ such that $\ir{\infty}{x} = \ir{\infty}{v}$ if $\xi \preceq \mfc_1^x$ and $\jr{\infty}{x} = \jr{\infty}{v}$ if $\mfc_2^x \preceq \xi$, the collection 
\begin{align*}
&\{\B_y^{\xi, \hor}: y, y+e_1 \in \pi\} \cup \{\B_y^{\xi, \ver}: y, y+e_2 \in \pi\} \\ 
&\cup \{\w_y: y \in \sG_{x, v, \pi}^-\} \cup \{\B_{y-e_1}^{\xi,\hor}  \wedge \B_{y-e_2}^{\xi,\ver}: y \in \sG_{x, v, \pi}^+\}
\end{align*}
is independent for any down-right path $\pi$ from $(x \cdot e_1, v \cdot e_2)$ to $(v \cdot e_1, x \cdot e_2)$.    
\end{enumerate}
\end{lemma}
\begin{proof}
The first property in part (a) was already noted after \eqref{E118} for the direction $\mfc_1^x$ and the remaining parts of (a) and (b) are similar. %The definition in \eqref{eq:BuseFlat} trivially extends the property to the flat segment $[e_2, \mfc_1^x]$. The second half of part (a) and part (b) are obtained similarly. 

Turning to part (c), for $k \in \bbZ_{>i}$, Lemma \ref{lem:BuseThinRec} gives the recursion 
\begin{align}
\label{E123}
\B_x^{(k, \infty), \hor} = \w_x + (\B_{x+e_2}^{(k, \infty), \hor}-\B_{x+e_1}^{(k, \infty), \ver})^+. 
\end{align}
By \eqref{EBuseCrit}, letting $k \to \infty$ turns \eqref{E123} into 
\begin{align}
\label{E124}
\B_x^{\mfc_1^x, \hor} = \w_x + (\B_{x+e_2}^{\mfc_1^{x+e_2}, \hor}-\B_{x+e_1}^{\mfc_1^{x+e_1}, \ver})^+. 
\end{align}
By part (a), $\B_{x+e_1}^{\mfc_1^{x+e_1}, \ver}$ is necessarily finite so the right-hand side makes sense. From the definition of $\mfc_1^x$ in \eqref{eq:crit}, one sees that $\mfc_1^{x+e_2} = \mfc_1^{x}$. Moreover, $\mfc_1^{x+e_1} = \mfc_1^{x}$ unless $a_i < \sinf{a}_{i+1: \infty}$. Hence, in the case $a_i \ge \sinf{a}_{i+1: \infty}$, it follows from \eqref{eq:BuseFlat}  and \eqref{E124} that 
\begin{align}
\label{E125}
\B_x^{\xi, \hor} = \w_x + (\B_{x+e_2}^{\xi, \hor}-\B_{x+e_1}^{\xi, \ver})^+
\end{align}
for $\xi \in [e_2, \mfc_1^x]$. If $a_i < \sinf{a}_{i+1: \infty}$, \eqref{E125} then both sides are $+\infty$ due to part (a), so the result follows. This gives the first equation in part (c) when $\xi \preceq \mfc_1^x$. The case of $\xi \succeq \mfc_2^x$ is similar. %The other case $\xi \succeq \mfc_2^x$ is handled similarly. 
 
To verify (d) when $\xi \preceq \mfc_1^x$, assume that $v \in \bbZ_{\ge x}^2$ satisfies $\ir{\infty}{x} = \ir{\infty}{v}$. Then definition \eqref{eq:rec} implies the existence of $k_0 \in \bbZ_{\ge v \cdot e_1}$ such that $a_{k_0} < a_{r}$ for $r \in \bbZ$ with $i \le r < v \cdot e_1$. Consequently, $\ir{k}{x} = \ir{k}{v} \ge v \cdot e_1$ for $k \in \bbZ_{\ge k_0}$. Now pick any down-right path $\pi$ from $(x \cdot e_1, v \cdot e_2)$ to $(v \cdot e_1, x \cdot e_2)$, which (uniquely) extends to a down-right path $\pi^{(k)}$ from $(x \cdot e_1, v \cdot e_2)$ to $(\ir{k}{x}, x \cdot e_2)$ by appending horizontal steps. Applying Lemma \ref{lem:nwBuseDis}(c) with $\pi^{(k)}$ yields the independence of the collection 
\begin{align*}
&\{\w_y: y \in \sG_{x, (\ir{k}{x}, v \cdot e_2), \pi^{(k)}}^-\} \cup \{\B_y^{(k, \infty), \hor}: y, y+e_1 \in \pi^{(k)}\} \\
&\cup \{\B_y^{(k, \infty), \ver}: y, y+e_2 \in \pi^{(k)}\} \cup \{\B_{y-e_1}^{(k, \infty), \hor} \wedge \B_{y-e_2}^{(k, \infty), \ver}: y \in \sG_{x, (\ir{k}{x}, v \cdot e_2), \pi^{(k)}}^+\}, 
\end{align*}
which contains 
\begin{align*}
&\{\w_y: y \in \sG_{x, v, \pi}^-\} \cup \{\B_y^{(k, \infty), \hor}: y, y+e_1 \in \pi\} \\
&\cup \{\B_y^{(k, \infty), \ver}: y, y+e_2 \in \pi\} \cup \{\B_{y-e_1}^{(k, \infty), \hor} \wedge \B_{y-e_2}^{(k, \infty), \ver}: y \in \sG_{x, v, \pi}^+\}, 
\end{align*}
as a subcollection.  
Passing to the limit as $k \to \infty$ and using \eqref{EBuseCrit}, one then obtains that %the limiting collection 
\begin{align}
\label{E119}
\begin{split}
&\{\w_y: y \in \sG_{x, v, \pi}^-\} \cup \{\B_y^{\mfc_1^y, \hor}: y, y+e_1 \in \pi\} \\
&\cup \{\B_y^{\mfc_1^y, \ver}: y, y+e_2 \in \pi\} \cup \{\B_{y-e_1}^{\mfc_1^y, \hor} \wedge \B_{y-e_2}^{\mfc_1^y, \ver}: y \in \sG_{x, v, \pi}^+\}, 
\end{split}
\end{align}
is independent. The assumption $\ir{\infty}{x} = \ir{\infty}{v}$ implies that $a_{(y \cdot e_1): \infty}^{\inf} = a_{(x \cdot e_1): \infty}^{\inf}$ for $y \in \Rect_x^v$.  It follows that $\mfc_1^y = \mfc_1^x$ for $y \in \Rect_x^v$. %Therefore and in view of definition \eqref{eq:BuseFlat}, 
By definition, replacing %the superscripts 
$\mfc_1^y$ with $\xi$ in \eqref{E119} does not alter the collection. Hence, part (c) holds with $\xi \preceq \mfc_1^x$. The remaining case $\xi \succeq \mfc_2^x$ is similar. 
\end{proof}

% \begin{lemma}
% \label{lem:ExpSys}
% Let $u, v \in \bbZ^2$ with $u + e_1 +e_2 \le v$. Let $\xi \in [e_2, e_1]$ and $\square \in \{+, -\}$. Then the collection 
% \begin{align*}
% \{(\w_x^\xi: x \in \Rect_{u+e_1+e_2}^v), (\B_x^{\xi\square, \hor}: x \in \Rect_{u+e_1}^v), (\B_x^{\xi\square, \ver}: x \in \Rect_{u+e_2}^v), (\w_x: \Rect_{u}^{v-e_1-e_2})\}
% \end{align*}
% is an $(a_{u \cdot e_1, v \cdot e_1-1}, b_{u \cdot e_2, v \cdot e_2-1}, Z)$ exponential system where $Z_x = \zmin{x}(\xi)$ for $x \in \Rect_{u}^v \smallsetminus \{u\}$. 
% \end{lemma}
% \begin{proof}
% {\color{red} Add the proof.}
% \end{proof}

\subsection{Construction of Busemann process}

We now introduce the Busemann process and verify the properties described in Theorem \ref{thm:buslim}. Let $\xi \in [e_2, e_1]$, $k, \ell \in \bbZ$ and $\square \in \{\xi+, \xi-, (k, \infty), (\infty, \ell)\}$. Our first task is to define the random variable $\B^{\square}_{x, y}$ for each $x, y \in \bbZ^2$ such that  
$(x \vee y) \cdot e_1 \le k$ when $\square = (k, \infty)$ and $(x \vee y) \cdot e_2 \le \ell$ when $\square = (\infty, \ell)$. The definition will use the single-step Busemann functions $\B^{\square, \hor}$ and $\B^{\square, \ver}$ studied above. %given piecewise by \eqref{EBusehvpm}, \eqref{EBuseThin}, \eqref{EBuseCrit} and \eqref{eq:BuseFlat}. 
Being limits of non-negative $\G$-increments, these are necessarily nonnegative and possibly infinite. The infinite values occur in precisely the following situations as can be seen from the marginal distributions in Lemmas \ref{lem:BuseSc}(c), \ref{lem:nwBuseDis}(a) and \ref{lem:BuseFlatBurke}(a)-(b).
\begin{align}
\label{eq:BuseInft}
\begin{split}
\B_x^{\square, \hor} = \infty \quad &\text{ if } \square = (k, \infty) \text{ and } x \cdot e_1 = \ir{k}{x}, \text{ or } \\ 
\quad &\text{ if } \square \in \{\xi+, \xi-\}, \ \xi \in [e_2, \mfc_1^x] \text{ and } x \cdot e_1 = \ir{\infty}{x}, \\ 
\B_x^{\square, \ver} = \infty \quad &\text{ if } \square = (\infty, \ell) \text{ and } x \cdot e_2 = \jr{\ell}{x}, \text{ or } \\ 
\quad &\text{ if } \square \in \{\xi+, \xi-\}, \xi \in [\mfc_2^x, e_1] \text{ and } x \cdot e_2 = \jr{\infty}{x}. 
\end{split}
\end{align}

We proceed to the definition of the Busemann process. If $x \le y$ then define   
\begin{align}
\label{EBuse}
\B_{x, y}^{\square} = \sum_{p \in \pi: p+e_1 \in \pi} \B_{p}^{\square, \hor} + \sum_{p \in \pi: p+e_2 \in \pi} \B_p^{\square, \ver}
\end{align} 
using some up-right path $\pi=\pi_{x, y} \in \Pi_x^y$. We later show the definition is independent of the chosen path. %for example, the one that visits the southeast corner $(y \cdot e_1, x \cdot e_2)$ of the rectangle $\Rect_x^y$. 
In the particular case of $x = y$, one has 
%of \eqref{EBuse}, one has 
%\begin{align}
%\label{EBuse-eq}
$\B_{x, x}^{\square} = 0$
%\end{align}
due to the sums being empty. Also, taking $y = x + e_i$ for $i \in \{1, 2\}$, one recovers the single-step Busemann functions: 
\begin{align}
\label{EBuse-ss}
\B_{x, x+e_1}^{\square} = \B_x^{\square, \hor} \quad \text{ and } \quad \B_{x, x+e_2}^{\square} = \B_x^{\square, \ver}. 
\end{align}
Being a sum of exponentially-distributed terms, 
\begin{align}
\label{EBuse-pos}
\B_{x, y}^\square > 0 \quad \text{ when } y > x. 
\end{align}
The cases of infinities can be determined from \eqref{eq:BuseInft} as follows. 
\begin{lemma}
\label{lem:BuseInf}
If $x \le y$ then $\B_{x, y}^{\square} = \infty$ if and only if one of the following conditions holds. 
\begin{enumerate}[label={\rm(\roman*)}, ref={\rm \roman*}] \itemsep=3pt
\item $\square = (k, \infty)$ and $\ir{k}{x} < y \cdot e_1$. 
\item $\square = (\infty, \ell)$ and $\jr{\ell}{x} < y \cdot e_2$. 
\item $\square \in \{\xi+, \xi-\}$, $\xi \in [e_2, \mfc_1^x]$ and $\ir{\infty}{x} < y \cdot e_1$. 
\item $\square \in \{\xi+, \xi-\}$, $\xi \in [\mfc_2^x, e_1]$ and $\jr{\infty}{x} < y \cdot e_2$. 
\end{enumerate}
\end{lemma}
\begin{proof}
Consider $p \in \pi$ such that $p+e_1 \in \pi$ (which requires that $x \cdot e_1 < y \cdot e_1$). Note also that as $p$ varies on $\pi$, the first coordinate $p \cdot e_1$ traces the set $\{x \cdot e_1, \dotsc, y \cdot e_1 -1\}$. 
By \eqref{eq:BuseInft}, $\B_p^{\square, \hor} = \infty$ if and only if $\square = (k, \infty)$ and $p \cdot e_1 = \ir{k}{p}$, or $\square \in \{\xi+, \xi-\}$, $\xi \in [e_2, \mfc_1^p]$ and $p \cdot e_1 = \ir{\infty}{p}$. It can be seen from definition \eqref{eq:rec} that $p \cdot e_1 = \ir{k}{p}$ for some $p$ as above if and only if $\ir{k}{x} < y \cdot e_1$. The same equivalence also holds after replacing $k$ with $\infty$. Since also $\mfc_1^{p} \le \mfc_1^x$ by \eqref{eq:critineq}, the statement 
\begin{align*}
p \cdot e_1 = \ir{\infty}{p} \ \text{ and } \ \xi \in [e_2, \mfc_1^p] \ \text{ for some } p \in \pi \text{ with } p+e_1 \in \pi
\end{align*}
holds if and only if 
\begin{align*}
\ir{\infty}{x} < y \cdot e_1 \ \text{ and } \ \xi \in [e_2, \mfc_1^p] \ \text{ for some } p \in \pi \text{ with } p+e_1 \in \pi, 
\end{align*}
which in turn holds if and only if 
\begin{align*}
\ir{\infty}{x} < y \cdot e_1 \ \text{ and } \ \xi \in [e_2, \mfc_1^x]. 
\end{align*}
Putting the preceding equivalences together, one obtains that 
\begin{align}
\label{E121}
\sum_{p \in \pi: p+e_1 \in \pi} \B_{p}^{\square, \hor} = \infty \quad \text{ if and only if } \quad \text{ condition (\romannumeral1) or (\romannumeral3) holds.}
\end{align}
Similarly, one can see that 
\begin{align}
\label{E122}
\sum_{p \in \pi: p+e_2 \in \pi} \B_{p}^{\square, \ver} = \infty \quad \text{ if and only if } \quad \text{ condition (\romannumeral2) or (\romannumeral4) holds.}
\end{align}
The result then follows from \eqref{EBuse} and the equivalences \eqref{E121} and \eqref{E122}. 
\end{proof}

Now dropping the requirement that $x \le y$, define 
\begin{align}
\label{EBuse-Ext}
\B_{x, y}^{\square} = \B_{x \wedge y, y}^{\square}-\B_{x \wedge y, x}^{\square}.   
\end{align}
If $x \le y$ then \eqref{EBuse-Ext} recovers the previous definition \eqref{EBuse} because $\B_{x, x}^{\square}=\B_{y, y}^{\square}=0$. Next consider the case when $x \cdot e_1 \le y \cdot e_1$ and $x \cdot e_2 \ge y \cdot e_2$. Then \eqref{EBuse-Ext} can be written out as  
\begin{align*}
\begin{split}
\B_{x, y}^\square &= \B_{(x \cdot e_1, y \cdot e_2), y}^{\square}-\B_{(x \cdot e_1, y \cdot e_2), x}^{\square} = \sum_{i \in [(y-x) \cdot e_1]} \B_{(x \cdot e_1 + i-1, y \cdot e_2)}^{\square, \hor} - \sum_{j \in [(x-y) \cdot e_2]} \B_{(x \cdot e_1, y \cdot e_2 + j-1)}^{\square, \ver}.  
\end{split}
\end{align*}
Lemma \ref{lem:BuseInf} shows that the two sums cannot both be infinite and, therefore, $\B_{x, y}^\square$ is well-defined. This is also true if $y \cdot e_2 > x \cdot e_2$ by the anti-symmetry  $\B_{x, y}^{\square} = - \B_{y, x}^{\square}.$
%\begin{align}
%\label{EBuse-AntSym}
%\B_{x, y}^{\square} = - \B_{y, x}^{\square}. 
%\end{align}
%
%Having introduced the Busemann process, it now remains to show that the properties listed in Theorem \ref{thm:buslim} hold. 

\begin{proof}[Proof of Theorem \ref{thm:buslim}]
Let $\xi \in [e_2, e_1]$, $k, \ell \in \bbZ$ and $\square \in \{\xi+, \xi-, (k, \infty), (\infty, \ell)\}$. Let $x, y \in \bbZ^2$ be such that $(x \vee y) \cdot e_1 \le k$ when $\square = (k, \infty)$ and $(x \vee y) \cdot e_2 \le \ell$ when $\square = (\infty, \ell)$. Let $(v_n)_{n \in \bbZ}$ be a sequence on $\bbZ^2$ such that $v_n/n \stackrel{n \to \infty}{\to} \xi$ and $\min\{v_n \cdot e_1, v_n \cdot e_2\} \stackrel{n \to \infty}{\to} \infty$ in the case $\square \in \{\xi\pm\}$. %(The second condition adds content only in the case $\xi \in \{e_1, e_2\}$ of axis directions). 
Also, set $v_n = (k, n)$ if $\square = (k, \infty)$ and $v_n = (n, \ell)$ if $\square = (\infty, \ell)$. We work below with sufficiently large $n$ to ensure that $v_n > x \vee y$. The following limits 
\begin{align}
\label{E120}
\lim_{n \to \infty} \I_{x, v_n} \stackrel{\rm{a.s.}}{=} \B_x^{\square, \hor} \quad \text{ and } \quad \lim_{n \to \infty} \J_{x, v_n} \stackrel{\rm{a.s.}}{=} \B_x^{\square, \ver}
\end{align}
have already been established case by case via Lemma \ref{lem:BuseScLim}\eqref{lem:BuseScLim-2} (when $\square = \xi\pm$ and $\xi \in ]\mfc_1^x, \mfc_2^x[$), \eqref{EBuseThin} (when $\square \in \{(k, \infty), (\infty, \ell)\}$), and Lemma \ref{lem:BuseFlat}\eqref{BuseFlat.b} and \eqref{eq:BuseFlat} (when $\square \in \{\xi+, \xi-\}$ and $\xi \in [e_2, \mfc_1^x] \cup [\mfc_2^x, e_1]$). 

We turn to properties (\ref{thm:buslim:pos})--(\ref{thm:buslim:indp}). Part (\ref{thm:buslim:pos}) follows from \eqref{EBuse-pos} and Lemma \ref{lem:BuseInf}.

Properties \eqref{thm:buslim:decomp}\eqref{thm:buslim:decomp:ord} and \eqref{thm:buslim:decomp}\eqref{thm:buslim:decomp:sym} are immediate from the construction above.  Property \eqref{thm:buslim:decomp}\eqref{thm:buslim:decomp:upright} will be derived shortly as a consequence of parts \eqref{thm:buslim:buslim} and \eqref{thm:buslim:thinbuslim}. 

By the recovery property \eqref{eq:wrec}, $\I_{x, v_n} \wedge \J_{x, v_n} = \w_x$. Passing to the limit as $n \to \infty$ and using \eqref{E120} along with \eqref{EBuse-ss} yields part \eqref{thm:buslim:recovery}.  

Lemmas \ref{lem:BuseSc}\eqref{lem:BuseSc:rec}, \ref{lem:BuseThinRec} and \ref{lem:BuseFlatBurke}\eqref{lem:BuseFlatBurke:rec} together give part \eqref{thm:buslim:recursion}.  

We turn to the properties of the exceptional set from part (\ref{thm:buslim:exdir}), which can be written as 
\begin{align*}
\Lambda_x = \{\eta \in [e_2, e_1]: \B_{x}^{\eta+, \hor} \neq \B_x^{\eta-, \hor} \text{ or } \B_x^{\eta+, \ver} \neq \B_x^{\eta-, \ver}\}
\end{align*}
using \eqref{EBuse-ss}. By definition \eqref{eq:BuseFlat}, $\Lambda_x \subset ]\mfc_1^x, \mfc_2^x[$. Let $D^{\hor}$ and $D^{\ver}$ denote the discontinuity sets of the function $\B_x^{\eta+, \hor}$ and $\B_x^{\eta+, \ver}$ in direction $\eta \in [e_2, e_1]$. It follows from Lemmas \ref{lem:pmbusdef}\eqref{lem:pmbusdef-2} and \ref{lem:BuseFlat}\eqref{BuseFlat.a} along with \eqref{eq:BuseFlat} that the union $D^{\hor} \cup D^{\ver}$ is countable and a.s.\ contained in $]\mfc_1^x, \mfc_2^x[$. Now pick any direction $\eta \in ]\mfc_1^x, \mfc_2^x[ \smallsetminus (D^{\hor} \cup D^{\ver})$, and recall the countable dense set $\sV_0^x \subset ]\mfc_1^x, \mfc_2^x[$ used in definition \eqref{EBusehvpm}. Then, by continuity along with the monotonicity noted in \eqref{EBuseMon}, 
\begin{align*}
\B_x^{\eta+, \hor} = \lim \limits_{\substack{\zeta \in \sV_0^x \\ \zeta \uparrow \eta}} \B_x^{\zeta, \hor} = \inf \limits_{\substack{\zeta \in \sV_0^x \\ \zeta \preceq \eta}} \B_x^{\zeta, \hor} = \B_x^{\eta-, \hor}.
\end{align*}
%where the last step appeals again to definition \eqref{EBusehvpm}. 
Similarly, $\B_x^{\eta+, \ver} = \B_x^{\eta-, \ver}$. It follows that %$\Lambda_x \subset D^{\hor} \cup D^{\ver}$ and, therefore,
$\Lambda_x$ is countable. Also, $\P\{\eta \in \Lambda_x\} = 0$ for each $\eta \in [e_2, e_1]$ due to Lemma \ref{lem:pmbusdef}\eqref{lem:pmbusdef-3}.% and the containment $\Lambda_x \subset ]\mfc_1^x, \mfc_2^x]$. 
Part \eqref{thm:buslim:exdir} follows. 

We check the first statement in part (\ref{thm:buslim:linconst}) the second one being similar. Let $\xi \in [e_2, \mfc_1^x]$. By part (\ref{thm:buslim:pos})(\ref{thm:buslim:pos:Bxy=infty}), one has $\B_{x, y}^\xi = \infty$ if and only if $\ir{\infty}{x} < y \cdot e_1$, which also applies in particular to the direction $\mfc_1^x$. Hence, if $\ir{\infty}{x} < y \cdot e_1$ then $\B_{x, y}^\xi = \infty = \B_{x, y}^{\mfc_1^x}$, and the claim holds. In the complementary case $\ir{\infty}{x} \ge y \cdot e_1$, one has $a_{(p \cdot e_1):\infty}^{\inf} = a_{(x \cdot e_1):\infty}^{\inf}$ for any $p \in \Rect_x^y$. Then definition \eqref{eq:crit} implies that $\mfc_1^p = \mfc_1^x$ for $p \in \Rect_x^y$. Therefore, $\B_p^{\xi, \hor} = \B_p^{\mfc_1^x, \hor}$ and $\B_p^{\xi, \ver} = \B_p^{\mfc_1^x, \ver}$ for $p \in \Rect_x^y$ by \eqref{eq:BuseFlat}. Combining these identities with \eqref{EBuse}, one obtains that $\B_{x, y}^{\xi} = \B_{x, y}^{\mfc_1^x}$.% as claimed.   

% Part (\ref{thm:buslim:linconst}) is due to definition \eqref{eq:BuseFlat}. {\color{red} Currently proved for single-step Busemann functions.}

The first half of (\ref{thm:buslim:buslim}) holds by Lemmas \ref{lem:BuseScLim}(a) and \ref{lem:BuseFlat}(b) and definition \eqref{eq:BuseFlat}. We present the argument for the second half of (\ref{thm:buslim:buslim}) and part (\ref{thm:buslim:thinbuslim}) together. First consider the case $x \le y$. Then, for any up-right path $\pi \in \Pi_x^y$, one can write the telescoping sum 
\begin{align}
\label{E126}
\begin{split}
\G_{x, v_n}-\G_{y, v_n} &= \sum_{p \in \pi: p+e_1 \in \pi} \{\G_{p, v_n}-\G_{p+e_1, v_n}\} + \sum_{p \in \pi: p+e_2 \in \pi} \{\G_{p, v_n}-\G_{p+e_2, v_n}\} \\ 
&= \sum_{p \in \pi: p+e_1 \in \pi} \I_{p, v_n} + \sum_{p \in \pi: p+e_2 \in \pi} \J_{p, v_n}. 
\end{split}
\end{align}
Sending $n \to \infty$ in \eqref{E126} and using \eqref{E120} yields 
\begin{align}
\label{E127}
\lim_{n \to \infty}\{\G_{x, v_n}-\G_{y, v_n}\} = \sum_{p \in \pi: p+e_1 \in \pi} \B_{p}^{\square, \hor} + \sum_{p \in \pi: p+e_2 \in \pi} \B_{p}^{\square, \ver}. 
\end{align}
Note that in the case $\square \in \{\xi+, \xi-\}$, assuming that $\xi \not \in \Lambda_p$ for each $p \in \Rect_x^y$ guarantees the convergences of the summands in \eqref{E126} due to part (\ref{thm:buslim:buslim})(\romannumeral1). Choosing $\pi = \pi_{x, y}$ (the arbitrary path chosen in \eqref{EBuse}) %and recalling definition \eqref{EBuse}, 
one then obtains that 
\begin{align}
\label{E128}
\lim_{n \to \infty}\{\G_{x, v_n}-\G_{y, v_n}\} = \buse{x}{y}{\square}. 
\end{align}
Equating the right-hand sides of \eqref{E127} and \eqref{E128} completes the proof of part (\ref{thm:buslim:decomp}). With \eqref{E128}, we have also obtained the second half of part (\ref{thm:buslim:buslim}) as well as part (\ref{thm:buslim:thinbuslim}) for the case $x \le y$. The general case can be reduced to the case $x \le y$ by writing 
\begin{align*}
\lim_{n \to \infty} \{\G_{x, v_n}-\G_{y, v_n}\} &= \lim_{n \to \infty} \{\G_{x, v_n}-\G_{x \wedge y, v_n} + \G_{x \wedge y, v_n}-\G_{y, v_n}\} \\ 
&= -\lim_{n \to \infty}\{\G_{x \wedge y, v_n}-\G_{x, v_n}\} + \lim_{n \to \infty} \{\G_{x \wedge y, v_n}-\G_{y, v_n}\} \\ 
&= -\B_{x \wedge y, x}^{\square} + \B_{x \wedge y, y}^{\square} = \B_{x, y}^{\square}. 
\end{align*}
%The last step uses definition \eqref{EBuse-Ext}. 

The inequalities stated in part (\ref{thm:buslim:mono}) are immediate consequences of \eqref{EBusehvpm}, Lemmas \ref{lem:pmbusdef}(b) and \ref{lem:BuseThinMon}, \eqref{EBuseCrit}  and \eqref{eq:BuseFlat}. 

Lemma \ref{lem:nwBuseDis}(b) implies part (\ref{thm:buslim:thinconst}). 

In view of properties (\ref{thm:buslim:decomp:upright}) and (\ref{thm:buslim:decomp:ord}) of part (\ref{thm:buslim:decomp}), it suffices to verify property (\ref{thm:buslim:cont})(\romannumeral1) only when $y = x+e_i$ for $i \in \{1, 2\}$. We obtain the first claim for the case $y = x+e_1$, the others being similar. For any $\eta, \zeta \in [e_2, e_1]$ such that $\eta \prec \zeta \prec \xi$, part (\ref{thm:buslim:mono}) gives 
\begin{align}
\label{E140}
\B_{x}^{\eta-, \hor} \ge \B_{x}^{\zeta+, \hor} \ge \B_{x}^{\xi-, \hor}. 
\end{align}
It follows from \eqref{E140}, Lemmas \ref{lem:pmbusdef}(\ref{lem:pmbusdef-1}) and \ref{lem:BuseFlat}(a) and definition \eqref{eq:BuseFlat} that 
\begin{align}
\label{E141}
\B_{x}^{\xi-, \hor} = \lim_{\eta \uparrow \xi} \B_x^{\eta-, \hor} \ge \lim_{\zeta \uparrow \xi} \B_x^{\zeta+, \hor} \ge \B_x^{\xi-, \hor}. 
\end{align}
Since the first and last terms in \eqref{E141} are the same, the claim is proved. To obtain (\ref{thm:buslim:cont})(\romannumeral2), assume now that $x \le y$. First, consider the case $\ir{\infty}{x} < y \cdot e_1$. Then $\ir{k}{x} < y \cdot e_1$ as well. Therefore, by part (\ref{thm:buslim:pos})(\ref{thm:buslim:pos:Bxy=infty}), $\B_{x, y}^{(k, \infty), \hor} = \infty = \B_{x, y}^{\mfc_1^x, \hor}$ and the first limit in part (\ref{thm:buslim:cont})(\romannumeral2) trivially holds. Assume now that $\ir{\infty}{x} \ge y \cdot e_1$. Then $\mfc_1^p = \mfc_1^x$ for $p \in \Rect_x^y$ as noted in the proof of part (\ref{thm:buslim:linconst}). Recall that  
\begin{align}
\label{E142}
\B_{x, y}^{(k, \infty)} &= \sum_{p \in \pi: p+e_1 \in \pi} \B_{p}^{(k, \infty), \hor} + \sum_{p \in \pi: p+e_2 \in \pi} \B_p^{(k, \infty), \ver}
\end{align}
where $\pi = \pi^{x, y}$ is the path used in definition \eqref{EBuse}. Letting $k \to \infty$ in \eqref{E142} yields 
\begin{align}
\label{E143}
\begin{split}
\lim_{k \to \infty}\B_{x, y}^{(k, \infty)} &\stackrel{\rm{a.s.}}{=} \sum_{p \in \pi: p+e_1 \in \pi} \B_{p}^{\mfc_1^p, \hor} + \sum_{p \in \pi: p+e_2 \in \pi} \B_p^{\mfc_1^p, \ver} \\ 
&= \sum_{p \in \pi: p+e_1 \in \pi} \B_{p}^{\mfc_1^x, \hor} + \sum_{p \in \pi: p+e_2 \in \pi} \B_p^{\mfc_1^x, \ver} = \B_{x, y}^{\mfc_1^x}
\end{split}
\end{align}
by definitions \eqref{EBuseCrit} and \eqref{EBuse}. With \eqref{E143}, the proof of the first limit in part (\ref{thm:buslim:cont})(\romannumeral2) is complete. The second limit is derived similarly. 

For part (\ref{thm:buslim:marg}), combine Lemmas \ref{lem:BuseSc}(c), \ref{lem:nwBuseDis}(a) and \ref{lem:BuseFlatBurke}(a)-(b). 

Finally, part (\ref{thm:buslim:indp}) follows from Lemmas \ref{lem:BuseSc}(d), \ref{lem:nwBuseDis}(c) and \ref{lem:BuseFlatBurke}(d). 
\end{proof}

\section{Semi-infinite geodesics}\label{sec:geo}
 With the Busemann functions constructed, we next study the structure of semi-infinite geodesics through the Busemann geodesics defined in \eqref{eq:busgeodef}. Most of the basic properties of the geodesics are immediate consequences of the properties of Busemann functions that we have just proven. The goals of this section are to prove Theorem \ref{thm:geo} and Theorem \ref{thm:lindir}.
\subsection{Monotonicity and continuity of Busemann geodesics}
Monotonicity and continuity of the Busemann geodesics are immediate consequences of the corresponding properties of Busemann functions which play an important role in the arguments which follow. The following is immediate from Theorem \ref{thm:buslim}\eqref{thm:buslim:mono} and the local rule defining the Busemann geodesics in \eqref{eq:busgeodef}. 
\begin{lemma}\label{lem:BusGeoOrder}
The following holds 
$\bfP$-almost surely. For all $x = (i,j) \in \bbZ^2$, all $k', k, \ell', \ell \in \bbZ$ and $\zeta,\eta \in [e_2,e_1]$ satisfying $k' \geq k \geq i$, $\ell' \geq \ell \geq j$, all $\zeta \preceq \eta$, and all $n$,
\begin{align*}
\pi_n^{x,(k,\infty)} &\preceq \pi_n^{x,(k',\infty)}\preceq \pi_n^{x,\zeta-} \preceq \pi_n^{x,\zeta+} \preceq \pi_n^{x,\eta-}\preceq \pi_n^{x,\eta+} \preceq\pi_n^{x,(\infty,\ell')} \preceq \pi_n^{x,(\infty,\ell)}.
\end{align*}
\end{lemma}

Turning to continuity, convergence of paths in the next lemma is in the sense of convergence of finite length segments. This result similarly follows immediately from definitions, the choice of the tie-breaking rule in \eqref{eq:busgeodef}, and Theorem \ref{thm:buslim} \eqref{thm:buslim:mono} and \eqref{thm:buslim:cont}.
\begin{lemma} \label{lem:BusGeoCont}
The following holds $\bfP$-almost surely. For all $x = (i,j) \in \bbZ^2$ and for all $\xi \in [e_2,e_1],$
\begin{align*}
\lim_{\substack{\zeta \nearrow  \xi}} \geo{n}{x}{\zeta\pm} = \geo{n}{x}{\xi-}, \qquad \lim_{\substack{\zeta \searrow  \xi}} \geo{n}{x}{\zeta\pm} = \geo{n}{x}{\xi+}
\end{align*}
and
\begin{align*}
\lim_{k\to \infty} \geo{n}{x}{(k,\infty)} = \geo{n}{x}{\mfc_1^x}, \qquad \lim_{\ell\to \infty} \geo{n}{x}{(\infty, \ell)} = \geo{n}{x}{\mfc_2^x}
\end{align*}
\end{lemma}
Note that in the statement of this result, $\geo{}{x}{\xi\pm} = \geo{}{x}{\mfc_1^x}$ for all $\xi \in [e_2,\mfc_1^x]$ and $\geo{}{x}{\xi\pm} = \geo{}{x}{\mfc_2^x}$ for all $\xi \in [\mfc_2^x,e_1]$ by Theorem \ref{thm:buslim}\eqref{thm:buslim:linconst} and the definition in \eqref{eq:busgeodef}.
\subsection{Directedness of Busemann geodesics}
We next turn to the asymptotic directions of Busemann geodesics, starting with the boundary constrained cases. Recall the notation $\ir{k}{x}$ and $\jr{\ell}{x}$  introduced in \eqref{eq:rec} for the first time the running minimum of a parameter sequence is encountered between $x$ and column $k$ or row $\ell$.
\begin{lemma} \label{lem:HorVerLines}
The following holds 
$\bfP$-almost surely for each  $x = (i,j) \in \bbZ^2$.
\begin{enumerate}[label={\rm(\alph*)}, ref={\rm\alph*}] \itemsep=3pt
\item \label{lem:HorVerLines:1} If $i \leq k \leq k'$ and $\smin{a}_{i:k} = \smin{a}_{i:k'}$ then $\geo{}{x}{(k,\infty)} = \geo{}{x}{(k',\infty)}$.
\item \label{lem:HorVerLines:2} For each $k \geq i$, there exists $N_{1,k}$ $<\infty$ so that for all $n > N_{1,k}$
\begin{align*}
\geo{n}{x}{(k,\infty)}\cdot e_1 = \ir{k}{x}. 
\end{align*}
\item \label{lem:HorVerLines:3} If $j \leq \ell \leq \ell'$ and $\smin{b}_{j:\ell} = \smin{b}_{j:\ell'}$ then $\geo{}{x}{(\infty,\ell)} = \geo{}{x}{(\infty,\ell')}$.
\item \label{lem:HorVerLines:4} For each $k \geq j$ there exists  $N_{2,k}$ $<\infty$ so that for all $n > N_{2,k}$,
\begin{align*}
\geo{n}{x}{(\infty,k)}\cdot e_2 = \jr{k}{x}.
\end{align*}
%where $j_k^x = \min\{u : j \leq u \leq k, j_{u} = b_{j,k}^{\min}\}.$
\end{enumerate}
\end{lemma}
\begin{proof}
We prove \eqref{lem:HorVerLines:1} and \eqref{lem:HorVerLines:2}, with the proofs of \eqref{lem:HorVerLines:3} and \eqref{lem:HorVerLines:4} being similar. Suppose that there exists $k' \geq k$ with $\smin{a}_{i:k'} = \smin{a}_{i:k}$. Let $y = (m,n)$ satisfy $i \leq m \leq \ir{k}{x}$. Then by Theorem \ref{thm:buslim}\eqref{thm:buslim:thinconst}, for $\ir{k}{x} \leq k < k'$ and for each $p \in \{1,2\}$ we have $\buse{y}{y+e_p}{(i_{k}^x,\infty)} = \buse{y}{y+e_p}{(k,\infty)} = \buse{y}{y+e_p}{(k',\infty)}$. The geodesics $\geo{}{x}{(\ir{k}{x},\infty)}$, $\geo{}{x}{(k,\infty)}$, and $\geo{}{x}{(k',\infty)}$ are constructed according to the local rules in \eqref{eq:busgeodef}. Therefore, the two geodesics remain the same at least until they cross the column with index $\ir{k}{x}$. But by Theorem \ref{thm:buslim}\eqref{thm:buslim:pos}\eqref{thm:buslim:pos:Bxy=infty}, $\buse{y}{y+e_1}{(k,\infty)} = \infty$ for any $y$ with $y \cdot e_1 = \ir{k}{x}$ and so this never happens. Part \eqref{lem:HorVerLines:1} follows.

By part \eqref{lem:HorVerLines:1}, we have that $\geo{}{x}{(k,\infty)} = \geo{}{x}{(\ir{k}{x},\infty)}$. To prove \eqref{lem:HorVerLines:2}, it only remains to be shown that $\geo{}{x}{(k,\infty)}$ eventually reaches column $\ir{k}{x}$. Call $v_n =\geo{n}{x}{(k,\infty)}$, so that we have $\buse{x}{v_n}{(k,\infty)} = \LPP{x}{v_n}$. Note that $v_n$ must eventually become trapped on some column with index $i \leq \ir{k}{x}$. Appealing to Theorem \ref{thm:buslim}\eqref{thm:buslim:marg} and \eqref{thm:buslim:indp}, we may average the vertical Busemann increments along all columns with index $\leq k$ to obtain that no matter which column $v_n$ becomes trapped on, we must have
\begin{align*}
\lim_{n\to\infty} \frac{1}{n}\buse{x}{\geo{n}{x}{(k,\infty)}}{(k,\infty)} &= \int_{0}^{\infty} \frac{\beta(db)}{b+\smin{a}_{i:k}}.
\end{align*}
But this must also be equal to the limit of $\LPP{x}{v_n}/n$. As $\beta$ is a non-zero sub-probability measure and $1/(b+a_\ell) < 1/(b+\smin{a}_{i:k})$ for $\ell < \ir{k}{x}$, this can only occur if $v_n$ eventually reaches the column with index $\ir{k}{x}$.
\end{proof}

Our next lemma describes when Busemann geodesics cross vertical or horizontal lines. Recall once again the notation $\ir{k}{x}$ and $\jr{k}{x}$ defined in \eqref{eq:rec}.
\begin{lemma} \label{lem:bulkgeodir}
The following holds $\bfP$-almost surely for each $x = (i,j) \in \bbZ^2$. 
\begin{enumerate}[label={\rm(\alph*)}, ref={\rm\alph*}] \itemsep=3pt
\item \label{lem:bgd-1} For each $\xi \in ]\mfc_1^x,\mfc_2^x[$, and each $k,\ell \in \bbZ$, there exists $N$ so that for all $n\geq N$, 
\[\geo{n}{x}{\xi\pm} \cdot e_1 > k \text{ and  }\geo{n}{x}{\xi\pm} \cdot e_2 > \ell.\]
\item \label{lem:bgd-2} If $\ir{\infty}{x} \in \bbZ$, then there exists $N$ so that for all $n \geq N,  \geo{n}{x}{\mfc_1^x} \cdot e_1 = \ir{\infty}{x}$. Otherwise, $\lim_{n \to \infty}  \geo{n}{x}{\mfc_1^x} \cdot e_1  = \lim_{n \to \infty}  \geo{n}{x}{\mfc_1^x} \cdot e_2 = \infty.$ 
\item \label{lem:bgd-3} If $j_\infty^x \in \bbZ$, then there exists $N$ so that for all $n \geq N,  \geo{n}{x}{\mfc_2^x} \cdot e_2 = \ell$. Otherwise, $\lim_{n \to \infty}  \geo{n}{x}{\mfc_2^x} \cdot e_1  = \lim_{n \to \infty}  \geo{n}{x}{\mfc_2^x} \cdot e_2 = \infty.$
\end{enumerate}
\end{lemma}
\begin{proof}
We begin with the $e_1$ claim in \eqref{lem:bgd-1}, with the $e_2$ claim being similar. By Lemma \ref{lem:BusGeoOrder} it suffices to prove the result for $\xi$ in a fixed countable dense subset of $]\mfc_1^x,\mfc_2^x[$, with the general result following by taking limits from within that set.

Fix $\xi \in ]\mfc_1^x,\mfc_2^x[$ and suppose that $\geo{n}{x}{\xi} \cdot e_1$ is bounded. By the path structure, it must be the case that $\geo{n}{x}{\xi} \cdot e_1$ is eventually constant. To show that this is impossible, fix $k \geq i$ and let $y_n = (k,n-k)$. Note that $\geo{n}{x}{\xi} \cdot e_1=k$ for $n \geq N$ if and only if  $\omega_{y_n}= \buse{y_n}{y_n+e_2}{\xi}$ for all $n\geq N$. In particular, for some $N$, we must have
\begin{align} \label{eq:buswgteq}
\bfP\left(\omega_{y_n}= \buse{y_n}{y_n+e_2}{\xi} \quad \forall n \geq N \right)>0
\end{align}
$\{\omega_{y_n} : n \geq N\}$ are independent with $\omega_{y_n} \sim$ Exp$(a_k+ b_{n-k})$.  By Theorem \ref{thm:buslim}\eqref{thm:buslim:marg} and \eqref{thm:buslim:indp}, $\{ \buse{y_n}{y_n+e_2}{\xi}: n \geq N\}$ are independent with $\buse{y_n}{y_n+e_2}{\xi}\sim$Exp $(b_{n-k} - \zmin{y_n}(\xi))$. Since $\xi \in ]\mfc_1^x,\mfc_2^x[ \subseteq ]\mfc_1^{y_n}, \mfc_2^{y_n}[$, we have $\zmin{x}(\xi) = \zmin{y_n}(\xi)$. Therefore, $\bfP$-almost surely,
\begin{align*}
\lim_{n\to\infty} \frac{1}{n}\sum_{\ell=N}^{n} \omega_{\ell} = \int \frac{\beta(\dd b)}{b+a_k}, \qquad \text{ and } \qquad \lim_{n\to\infty} \frac{1}{n} \sum_{\ell=N}^n \buse{y_n}{y_n+e_2}{\xi} = \int \frac{\beta(db)}{b - \zmin{x}(\xi)}.
\end{align*} 
Moreover, $\zmin{x}(\xi) > -\sinf{a}_{i:\infty} \ge -a_k$. As $\beta$ is a non-zero subprobability measure, these two quantities are different and so \eqref{eq:buswgteq} cannot hold. 

Next, we turn to claim \eqref{lem:bgd-2}, with claim \eqref{lem:bgd-3} being similar. First, consider the case where $\ir{\infty}{x}\in\bbZ$ and call $\ir{\infty}{x}=k$. It now follows from Lemmas \ref{lem:HorVerLines} that for each $\ell \geq k$, $\geo{}{x}{(k,\infty)} = \geo{}{x}{(\ell,\infty)}$. By Lemma \ref{lem:BusGeoCont}, $\geo{}{x}{(\ell,\infty)} \to \geo{}{x}{\mfc_1^x}$ as $\ell\to\infty$, which implies the claim. 

If $\ir{\infty}{x}\notin\bbZ$, then $\sinf{a}_{i:\infty} < a_k$ for all $k$ and so the value of $\smin{a}_{i:k}$ changes infinitely often as we send $k \to \infty$. Fix $m$ and let $k$ be sufficiently large that $\smin{a}_{i:k} < \smin{a}_{i:m}$. Lemma \ref{lem:HorVerLines} then implies that there exists $N_k$ so that for all $n \geq N_k$, $\geo{n}{x}{(k,\infty)} \cdot e_1 > m$. Moreover, if $\ell \geq k$, then for each such $n$, we have $m < \geo{n}{x}{(k,\infty)} \cdot e_1 \leq \geo{n}{x}{(\ell,\infty)} \cdot e_1$ by Lemma \ref{lem:BusGeoOrder}. Sending $\ell \to \infty$ gives $m < \geo{n}{x}{\mfc_1^x} \cdot e_1$ for each $n \geq N_k$. %Now, fix $\xi \in]\mfc_1^x,\mfc_2^x[$. By Lemmas \ref{lem:BusGeoOrder} and \ref{lem:BusGeoCont}, we have $\pi_n^{x,\mfc_1^x} \cdot e_1 \leq \pi_n^{x,\xi} \cdot e_1$  for all $n$. 
It now follows from claim \eqref{lem:bgd-1} that $\lim_{n\to\infty} \pi_n^{x,\mfc_1^x} \cdot e_1  = \infty$. \qedhere
\end{proof}

With the previous results in mind, we can now complete the proof of Theorem \ref{thm:geo}\eqref{thm:geo:ex}.
\begin{proof}[Proof of Theorem \ref{thm:geo}\eqref{thm:geo:ex}]
We prove Theorem \ref{thm:geo}\eqref{thm:geo:ex}\eqref{thm:geo:ex:con} by first considering a fixed countable set of directions and then squeezing. Take $\xi \in ]\mfc_1^x,\mfc_2^x[$ and recall that we have $\buse{x}{\geo{n}{x}{\xi}}{\xi}= \LPP{x}{\geo{n}{x}{\xi}}$.
Define
\begin{align*}% \label{eq:r1def}
\zeta &= \varlimsup_{n\to\infty} \frac{\geo{n}{x}{\xi}}{n}. 
\end{align*}
Let $n_k$ be a subsequence along which we have the convergence $\frac{\geo{n_k}{x}{\xi}}{n_k} \to \zeta$. 
It follows from Theorem \ref{thm:buslim}\eqref{thm:buslim:marg} and \eqref{thm:buslim:indp}, standard concentration of estimates for independent exponentials (such as \cite[Lemma A.2]{Emr-Jan-Sep-21}), the Borel-Cantelli lemma and assumption \eqref{as:weakcon} that 
\begin{align*}
\lim_{k\to\infty} \frac{1}{n_k}\buse{x}{\geo{n_k}{x}{\xi}}{\xi} =  \zeta \cdot e_1 \int_0^\infty \frac{\alpha(\dd a)}{a+\zmin{x}(\xi)} + \zeta \cdot  e_2 \int_0^\infty \frac{\beta(\dd b)}{b-\zmin{x}(\xi)}.
\end{align*}

By Lemma \ref{lem:bulkgeodir} we have $\geo{n_k}{x}{\xi} \cdot e_1 \to \infty$ and $\geo{n_k}{x}{\xi} \cdot e_2 \to \infty$. Using Lemma \ref{lem:BusRGeo}, by Proposition \ref{prop:shape}, $\zeta$ satisfies 
\begin{align*}
\gamma^x(\zeta) &= \zeta \cdot e_1 \int_0^\infty \frac{\alpha(\dd a)}{a+\zmin{x}(\xi)} + \zeta \cdot  e_2 \int_0^\infty \frac{\beta(\dd b)}{b-\zmin{x}(\xi)}
\end{align*}
By strict concavity of $\gamma^x$ on $]\mfc_1^x,\mfc_2^x[$ and concavity on $[e_2,e_1]$, this holds if and only if $\zeta = \xi$. A similar argument with a subsequence corresponding to $\varliminf  \geo{n}{x}{\xi}/n$ completes the proof of the case of a fixed $\xi \in ]\mfc_1^x,\mfc_2^x[$. Theorem \ref{thm:geo}\eqref{thm:geo:ex}\eqref{thm:geo:ex:con} then follows from Lemma \ref{lem:BusGeoOrder} and Lemma \ref{lem:BusGeoCont} by considering a countable dense set of fixed directions in $]\mfc_1^x,\mfc_2^x[$. 

Part \eqref{thm:geo:ex:bdy} and all of the claims in part \eqref{thm:geo:ex:lin} except \eqref{eq:thm:geo:ex:lin-1} are contained in Lemma \ref{lem:bulkgeodir}. One of the two inequalities is trivial since all geodesics are contained in $[e_2,e_1].$ The other inequality follows from part \eqref{thm:geo:ex:con} and Lemmas \ref{lem:BusGeoOrder} and \ref{lem:BusGeoCont} by taking a sequence $\xi_k \in ]\mfc_1^x,\mfc_2^x[$ with $\xi_k \searrow \mfc_1^x$ and using the limit $\geo{}{x}{\xi_k\pm} \to \geo{}{x}{\mfc_1^x}$. 
\end{proof}

We next turn to the proofs of Theorem \ref{thm:geo}\eqref{thm:geo:dir} and \eqref{thm:geo:uniq}.
\begin{proof}[Proof of Theorem \ref{thm:geo}\eqref{thm:geo:dir} and \eqref{thm:geo:uniq}]
Let $x = (i,j)$ and let $\pi$ be a semi-infinite geodesic containing $x$. Suppose first that $\pi_n \cdot e_1$ remains bounded. By the path structure, there exists $k$ so that for all sufficiently large $n$, $\pi_n \cdot e_1 = k$. Call $N$ the index at which $\pi$ first satisfies $\pi_N \cdot e_1 = k$, so that for all $n \geq N$, $\pi_n \cdot e_1 = k$. We claim that $k = \ir{k}{x}$ and $\pi_n = \geo{n}{x}{(\ir{k}{x},\infty)}$ for $n \geq i+j$. 

Let   $\geo{}{x}{(k,n)}$ denote the unique geodesic between $x$ and $(k,n)$. For $n \geq N$, uniqueness of finite geodesics forces that $\geod{(i+j):(k+n)}=\geo{(i+j):(k+n)}{x}{(k,n)}$ But $\geo{}{x}{(k,n)}$ evolves according to the local rule \eqref{eq:geoloc}. Combining this observation with Theorem \ref{thm:buslim}\eqref{thm:buslim:buslim} and the local rule defining $\geo{}{x}{(k,\infty)}$ in \eqref{eq:busgeodef}, we see that as $n\to\infty$, $\geo{}{x}{(k,n)}$ converges to $\geo{}{x}{(k,\infty)}=\geo{}{x}{(\ir{k}{x},\infty)}$. It now follows from Theorem \ref{thm:geo}\eqref{thm:geo:ex}\eqref{thm:geo:ex:bdy} that $k=\ir{k}{x}$ and for all $n \geq i+j,$ $\geo{n}{x}{(\ir{k}{x},\infty)}=\geo{n}{x}{(k,\infty)}=\geod{n}$. The case where $\geod{n}\cdot e_2$ remains bounded is similar. 

Suppose now that there exists a subsequence $n_k$ with the property that $\geod{n_k}/n_k \to \xi \in ]\mfc_1^x, \mfc_2^x[$. Fix $\zeta,\eta \in ]\mfc_1^x,\mfc_2^x[$ with $\zeta \prec \xi \prec \eta$. By Theorem \ref{thm:geo}\eqref{thm:geo:ex}\eqref{thm:geo:ex:con}, we know that $\geo{}{x}{\zeta+}$ and $\geo{}{x}{\eta-}$ are $\zeta$ and $\eta$ directed, respectively. Uniqueness of finite geodesics now forces that for all $n \geq i+j$, we must have $\geo{n}{x}{\zeta+} \preceq \geod{n} \preceq \geo{n}{x}{\eta-}$. Sending $\zeta \nearrow \xi$ and $\eta \searrow \xi$ and appealing to Lemma \ref{lem:BusGeoCont}, we have for all $n \geq i+j,$ $\geo{n}{x}{\xi-} \preceq \geod{n} \preceq \geo{n}{x}{\xi+}$ and consequently, by Theorem \ref{thm:geo}\eqref{thm:geo:ex}\eqref{thm:geo:ex:con}, $\geod{n}/n\to\xi$.

The only remaining possibility is that $\geod{n}\cdot e_1 \to \infty, \geod{n}\cdot e_2 \to \infty$, and all limit points of $\geod{n}/n$ are contained in one of $[e_2,\mfc_1^x]$ or $[\mfc_2^x,e_1]$. We consider the case of $[e_2,\mfc_1^x]$, with the other case being similar. Arguing as above, uniqueness of finite geodesics implies that for each $k \geq i$ and for each $\xi \in ]\mfc_1^x,\mfc_2^x[$, we must have that for all $n \geq i+j$, $\geo{n}{x}{(k,\infty)}\preceq \geod{n}\preceq \geo{n}{x}{\xi-}$. Sending $k \to \infty$ and $\xi \searrow \mfc_1^x$ and appealing to Lemma \ref{lem:BusGeoCont}, we conclude that for all $n \geq i+j$, $\geod{n}=\geo{n}{x}{\mfc_1^x}$.
\end{proof}

Before turning to the proof of Theorem \ref{thm:geo}\eqref{thm:geo:coal} (which appears in Section \ref{sec:dualcoal}), we make a detour to complete our discussion about the asymptotic directions of geodesics by proving Theorem \ref{thm:lindir}.
\subsection{Asymptotic direction of linear segment Busemann geodesics}
The next result, recorded as Theorem \ref{thm:lindir}, concerns possible behaviors of geodesics which correspond to the linear segments, but which do not become trapped on rows or columns. The basic idea is a classical (though possibly counterintuitive) one: we use the curvature of an appropriate shape function to control the geodesic. The reason this works, despite the limit shape having a flat segment in the directions of interest, is that the natural centering for any point-to-point passage time is not the asymptotic limit shape defined in \eqref{eq:shape}, but rather a limit shape that only sees the parameters which are involved in the computation of the passage time. This is the limit shape that would have been seen if the parameter sequences had been periodic with a finite period. The reason curvature estimates can be used to study the behavior in the linear region is that shape functions for periodic parameter sequences are always strictly concave.

\begin{proof}[Proof of Theorem \ref{thm:lindir}]
It suffices to consider $x=(1,1) \leq (m,n)=y$ and the result concerning $[e_2,\mfc_1^x]$. We introduce notation for the shape function which would have arisen if the parameter sequences had been the periodic extensions of $a_{1:m}$ and $b_{1:n}$: for $\xi=(\xi_1,\xi_2) \in \bbR_{\ge 0}^2$, call
\be\begin{aligned}
\shp_z^{x,y}(\xi) &= \frac{\xi_1}{m} \sum _{k=1}^m \frac{1}{a_k+z} + \frac{\xi_2}{n} \sum_{\ell=1}^n \frac{1}{b_\ell-z} = \xi_1 \shor^{x,y}(z) + \xi_2\sver^{x,y}(z), \text{ where } \\
\shor^{x,y}(z) &=  \shp_z^{x,y}(e_1) \qquad  \text{ and } \qquad \sver^{x,(m,n)}(z) =  \shp_z^{x,(m,n)}(e_2). 
\end{aligned}\ee

Set
\be\begin{aligned}
\shp^{x,y}(\xi) &= \inf_{- \smin{a}_{1:m}< z <  \smin{b}_{1:n}} \{\shp_z^{x,y}(\xi)\} = \shp_{\zmin{x,y}(\xi)}^{x,y}(\xi),
\end{aligned}\ee
where $\zmin{x,y}(\xi) \in (-\smin{a}_{1:m},\smin{b}_{1:n})$ is the unique minimizer of the infimum. We also record the derivatives which go into the main estimates:
\begin{align}
\partial_z\shp_z^{x,y}(\xi) = \xi_1 \partial_z \shor^{x,y}(z) + \xi_2\partial_z \sver^{x,y}(z) = - \frac{\xi_1}{m}\sum _{k=1}^m \frac{1}{(a_k+z)^2} + \frac{\xi_2}{n} \sum_{\ell=1}^n \frac{1}{(b_\ell-z)^2}.\label{eq:derivz}
\end{align}

By the cocycle property of Busemann functions, Theorem \ref{thm:buslim}\eqref{thm:buslim:decomp}\eqref{thm:buslim:decomp:upright}, we may write $\buse{x}{y}{\mfc_1^x}$ as a sum of nearest neighbor horizontal increments from $x=(1,1)$ to $(m,1)$ followed by a sum of vertical nearest-neighbor increments from $(m,1)$ to $(m,n)=y$:
\begin{align*}
\buse{x}{y}{\mfc_1^x} = \sum_{k=0}^{m-2} \buse{(k+1,1)}{(k+2,1)}{\mfc_1^x} + \sum_{\ell=0}^{n-2}\buse{(m,\ell+1)}{(m,\ell+2)}{\mfc_1^x}
\end{align*}
By parts \eqref{thm:buslim:indp} and \eqref{thm:buslim:marg} of the same theorem, these two sums each consist of jointly independent exponential random variables. Note that although the terms in each sum are independent, the two sums are not independent. 

Condition \eqref{eq:lincond-1} implies that for $y \geq x$, $\mfc_1^y = \mfc_1^x$. Abbreviate $\LPP{x}{y}=G(y)$, $\buse{x}{y}{\mfc_1^x} = B(y)$, $\pi(k) =(\pi(k)_1, \pi(k)_2) = \geo{k}{x}{\mfc_1^x}$, and $\chi_k=\zmin{x,\pi(k)}(\pi(k))$. By Lemma \ref{lem:BusRGeo}, for all $k\geq 2$, \[
G(\pi(k)) = B(\pi(k)).
\]

The key estimate needed to prove the result is to show that $k^{-1}\partial_z\shp_z^{x,\pi(k)}(\pi(k))\big|_{z=-\sinf{a}_{1:\infty}}$ converges to zero almost surely as $k\to\infty$ under our hypotheses. To prove this, we consider cases based on how close $\chi_k$ is to $-\sinf{a}_{1:\infty}$.

Let $\eta \in (0,\epsilon)$ and $\delta \in (0,\epsilon-\eta)$, where $\epsilon$ is as in Condition \ref{cond:lincond}. Suppose first that $|\chi_k+\sinf{a}_{1:\infty}| \leq k^{-\frac{1}{2}+\epsilon-\delta}$. By \eqref{eq:lincond-1} and using that $\pi(k)_1,\pi(k)_2 \leq k$, it follows that for some absolute constant $c>0$ and for $i = 1,\dots, \pi(k)_1$ and $j =1,\dots,\pi(k)_2$, 
\be\label{eq:parest}\begin{aligned}
(a_i + \chi_k) &= (a_i-\sinf{a}_{1:\infty})\bigg(1+\frac{\sinf{a}_{1:\infty}+\chi_k}{a_i-\sinf{a}_{1:\infty}}\bigg) \geq (a_i -\sinf{a}_{1:\infty})(1-c k^{-\delta}) \qquad \text{ and }\\
(b_j - \chi_k) &= (b_j+\sinf{a}_{1:\infty})\bigg(1-\frac{\chi_k+\sinf{a}_{1:\infty}}{b_j+\sinf{a}_{1:\infty}}\bigg) \leq (b_j + \sinf{a}_{1:\infty})\bigg(1+ck^{-\frac{1}{2}+\epsilon-\delta} \bigg).
\end{aligned}\ee

Recall that
\begin{align}\label{eq:minFOC}
0 = \partial_z\shp_z^{x,\pi(k)}(\pi(k))\big|_{z=\chi_k}= -\sum_{i=1}^{\pi(k)_1} \frac{1}{(a_i+\chi_k)^2} + \sum_{j=1}^{\pi(k)_2} \frac{1}{(b_j-\chi_k)^2} .
\end{align}
Using these observations, and again adjusting $c$ several times, we have
\begin{align*}
\partial_z\shp_z^{x,\pi(k)}(\pi(k))\big|_{z=-\sinf{a}_{1:\infty}} &= -\sum_{j=1}^{\pi(k)_1} \frac{1}{(a_j-\sinf{a}_{1:\infty})^2} + \sum_{j=1}^{\pi(k)_2} \frac{1}{(b_j+\sinf{a}_{1:\infty})^2} \\
   &\leq -(1-ck^{-\delta}) \sum_{i=1}^{\pi(k)_1} \frac{1}{(a_i+\chi_k)^2} + (1+ck^{-\delta})\sum_{j=1}^{\pi(k)_2} \frac{1}{(b_j-\chi_k)^2} \\
   &= c k^{-\delta}\bigg( \sum_{i=1}^{\pi(k)_1} \frac{1}{(a_i+\chi_k)^2} + \sum_{j=1}^{\pi(k)_2} \frac{1}{(b_j-\chi_k)^2} \bigg) \leq c k^{1-\delta}.
\end{align*}
The last inequality comes from equation (5.2) in \cite{Emr-Jan-Sep-21}. The corresponding lower bound can be argued similarly, so we conclude that there is an absolute constant $c>0$ so that
\begin{align}
\bigg|k^{-1}\partial_z\shp_z^{x,\pi(k)}(\pi(k))\big|_{z=-\sinf{a}_{1:\infty}} \bigg| \leq c k^{-\delta}.%\stackrel{\tiny{k\to\infty}}{\to}0.
\end{align}
Next, we consider the more difficult case where $|\chi_k+\sinf{a}_{1:\infty}|\geq k^{-\frac{1}{2}+\epsilon-\delta}$. We consider the subcase of $-\sinf{a}_{1:\infty}<\chi_k$, with the subcase of $-\sinf{a}_{1:\infty}>\chi_k$ being similar.  

The independence of Busemann increments in Theorem \ref{thm:buslim}\eqref{thm:buslim:indp} combined with the marginal distributions recorded in \eqref{eq:busmar} and straightforward concentration bounds for sums of independent exponential random variables (recorded as Lemma A.2 in \cite{Emr-Jan-Sep-21}) implies that for each $p>0$, there exists $C$ so that for all $y=(m,n) \geq (1,1)=x$ and all $s>0$,
\begin{align}
\bfP\bigg(\big|B(y) - \shp_{-\sinf{a}_{1:\infty}}^{x,y}(y)\big|  \geq s\bigg(\sqrt{-\partial_z \shor^{x,y}(-\sinf{a}_{1:\infty})}+\sqrt{\partial_z \sver^{x,y}(-\sinf{a}_{1:\infty})}\bigg)\bigg) \leq \frac{C}{s^p}.\label{eq:lin-conc-1}
\end{align}
Lemma 4.2 in \cite{Emr-Jan-Sep-21} shows that we also have, under the same hypotheses,
\begin{align}
\bfP\bigg(G(y) - \shp^{x,y}(y)  \geq s\bigg(\sqrt{-\partial_z \shor^{x,y}(\zmin{x,y}(y))}+\sqrt{\partial_z \sver^{x,y}(\zmin{x,y}(y))}\bigg)\bigg) \leq \frac{C}{s^p}.\label{eq:lin-conc-2}
\end{align}

Using \eqref{as:inf} and \eqref{as:weakcon}, we may adjust $C>0$ so that 
\[%\partial_z \sver^{x,(m,n)}(-\sinf{a}_{i:m})\leq 
\partial_z \sver^{x,y}(-\sinf{a}_{1:\infty}) = \sum_{\ell=1}^n \frac{1}{(b_\ell + \sinf{a}_{1:\infty})^2} \leq Cn\] 
for all $n \geq 1$. The hypothesis that $\overline{\mfa}_x<\infty$ implies that (possibly again adjusting $C$), we also have that for $m \geq 1$,
\begin{align}
-\partial_z \shor^{x,y}(-\sinf{a}_{1:\infty}) = \sum_{k=1}^m \frac{1}{(a_k - \sinf{a}_{1:\infty})^2} \leq C m. \label{eq:hypbd}
\end{align}

Applying the previous four displays and Borel-Cantelli, we may conclude that there exists a random $L$ so that whenever $|y|_1\geq L$, we have 
\be\begin{aligned}
&|B(y) - \shp_{-\sinf{a}_{1:\infty}}^{x,y}(y)\big|< |y|_1^{\frac{1}{2}+\eta} \qquad \text{ and } \qquad G(y) - \shp^{x,y}(y) \leq |y|^{\frac{1}{2}+\eta}.
\end{aligned}\label{eq:lin-conc-bds}
\ee
Recalling that we always have $\shp^{x,y}_{-\sinf{a}_{1:\infty}}(y) \geq \shp^{x,y}(y)$,  it follows that there is a random $L$ so that whenever $k \geq L$, we have
\begin{align}
\bigg|\shp_{-\sinf{a}_{1:\infty}}^{x,\pi(k)}(\pi(k)) - \shp^{x,\pi(k)}(\pi(k))\bigg| < k^{\frac{1}{2}+\eta}. \label{eq:BCbd}
\end{align}
So long as $k>L$ from \eqref{eq:BCbd}, we have the following by convexity of $z \mapsto \shp_z^{x,\pi(k)}(\pi(k))$: 
\begin{align*}
k^{\frac{1}{2}+\eta} &> \shp_{-\sinf{a}_{1:\infty}}^{x,\pi(k)}(\pi(k)) - \shp^{x,\pi(k)}(\pi(k))  = \shp_{-\sinf{a}_{1:\infty}}^{x,\pi(k)}(\pi(k)) -  \shp_{\chi_k}^{x,\pi(k)}(\pi(k)) \\
&\geq \shp_{-\sinf{a}_{1:\infty}}^{x,\pi(k)}(\pi(k)) -  \shp_{-\sinf{a}_{1:\infty}+k^{-\frac{1}{2}+\epsilon-\delta}}^{x,\pi(k)}(\pi(k)) \\
&\geq -k^{-\frac{1}{2}+\epsilon-\delta}\partial_z\shp_z^{x,\pi(k)}(\pi(k))\big|_{z=-\sinf{a}_{1:\infty}+k^{-\frac{1}{2}+\epsilon-\delta}} \\
&= k^{-\frac{1}{2}+\epsilon-\delta}\bigg(\sum _{\ell=1}^{\pi(k)_1} \frac{1}{(a_\ell-\sinf{a}_{1:\infty}+k^{-\frac{1}{2}+\epsilon-\delta})^2} - \sum_{\ell=1}^{\pi(k)_2} \frac{1}{(b_\ell+\sinf{a}_{1:\infty}-k^{-\frac{1}{2}+\epsilon-\delta})^2}\bigg)\\
&\geq k^{-\frac{1}{2}+\epsilon-\delta}\bigg(\sum _{\ell=1}^{\pi(k)_1} \frac{1}{(a_\ell-\sinf{a}_{1:\infty})^2}(1-ck^{-\delta}) - \sum_{\ell=1}^{\pi(k)_2} \frac{1}{(b_\ell+\sinf{a}_{1:\infty})^2}(1+ck^{-\delta})\bigg)\\
&\geq k^{-\frac{1}{2}+\epsilon-\delta}\bigg(\sum _{\ell=1}^{\pi(k)_1} \frac{1}{(a_\ell-\sinf{a}_{1:\infty})^2} - \sum_{\ell=1}^{\pi(k)_2} \frac{1}{(b_\ell+\sinf{a}_{1:\infty})^2}-ck^{1-\delta}\bigg)
\end{align*}
In the last step, we have used \eqref{eq:hypbd}, which required the assumption that $\overline{\mfa}_x<\infty$. Using that $-\sinf{a}_{1:\infty}<\chi_k$ implies that the derivative at $-a_{1:\infty}^{\inf}$ is negative (by convexity), we have
\begin{align*}
0 > \partial_z\shp_z^{x,\pi(k)}(\pi(k))\bigg|_{z=-\sinf{a}_{1:\infty}} = -\sum _{\ell=1}^{\pi(k)_1} \frac{1}{(a_\ell-\sinf{a}_{1:\infty})^2} + \sum_{\ell=1}^{\pi(k)_2} \frac{1}{(b_\ell+\sinf{a}_{1:\infty})^2} \geq -k(k^{\eta+\delta-\epsilon} + ck^{-\delta}).
\end{align*}
The case of $\chi_k < \sinf{a}_{1:\infty}$ is similar and so we conclude that
\[
\bigg| k^{-1}\partial_z\shp_z^{x,\pi(k)}(\pi(k))\bigg|_{z=-\sinf{a}_{1:\infty}}\bigg| \leq (k^{\eta+\delta-\epsilon} + k^{-\delta}) \stackrel{\tiny{k\to\infty}}{\to} 0.
\]

Now, let $(k_j : j \geq 1)$ be any sequence of distinct natural numbers along which $\pi(k_j)/k_j$ converges to a vector $\xi=(\xi_1,1-\xi_1)\in[e_2,e_1]$ and the limit
\begin{align*}
\lim_{j\to\infty} \frac{1}{\pi(k_j)_1}\sum _{\ell=1}^{\pi(k_j)_1} \frac{1}{(a_\ell-\sinf{a}_{1:\infty})^2} := \mfa
\end{align*} 
exists. Note that we have proven in Theorem \ref{thm:geo}\eqref{thm:geo:ex}\eqref{thm:geo:ex:lin} that under Condition \ref{cond:lincond}, $\pi(k_j)_1\to\infty$ and $\pi(k_j)_2\to\infty$ as $j\to\infty$.  It follows from \eqref{as:weakcon} that for any such sequence, we have 
\begin{align*}
\lim_{j\to\infty} \frac{1}{\pi(k_j)_2} \sum_{\ell=1}^{\pi(k_j)_2} \frac{1}{(b_\ell+\sinf{a}_{1:\infty})^2}  = \int \frac{1}{(b+\sinf{a}_{1:\infty})^2}\beta(db) = \bfB_x.
\end{align*}
Recalling that $k^{-1}\partial_z\shp_z^{x,\pi(k)}(\pi(k))\big|_{z=-\sinf{a}_{1:\infty}}\to 0$, we have 
\be\label{eq:limitpt}\begin{aligned}
0 &= \lim_{j\to\infty} - \frac{\pi(k_j)_1}{k_j} \frac{1}{\pi(k_j)_1}\sum _{\ell=1}^{\pi(k_j)_1} \frac{1}{(a_\ell-\sinf{a}_{1:\infty})^2} + \frac{\pi(k_j)_2}{k_j} \frac{1}{\pi(k_j)_2}\sum_{\ell=1}^{\pi(k_j)_2} \frac{1}{(b_\ell+\sinf{a}_{1:\infty})^2} \\
&= -\xi_1 \mfa + (1-\xi_1) \bfB_x
\end{aligned}\ee
and consequently, we have $\xi_1 = \dfrac{\bfB_x}{\mfa+\bfB_x}$. By definition, $\mfa \in [\underline{\mfa}_x,\overline{\mfa}_x]$ and therefore \begin{align}\xi_1 \in \bigg[\frac{\bfB_x}{\overline{\mfa}_x + \bfB_x}, \frac{\bfB_x}{\underline{\mfa}_x + \bfB_x}\bigg]. \label{eq:limitpts}\end{align}
Because the set of asymptotic directions of $\pi$ must be connected, it remains to show that the extreme points of this interval of directions are both attained along some subsequence. Recall that \eqref{eq:limitpt} holds for each sequence $k_j$ for which $\pi(k_j)_1/k_j$ converges to some $\xi_1$. We now construct subsequences realizing the extreme points. Let $n_j$ be a sequence of distinct integers along which we have
\begin{align*}
\lim_{j\to\infty} \frac{1}{n_j}\sum _{\ell=1}^{n_j} \frac{1}{(a_\ell-\sinf{a}_{1:\infty})^2} = \overline{\mfa}_x.
\end{align*} 
Let $k_j$ be the smallest index satisfying $\pi(k_j)_1 = n_j$ and then pass to a sub-sequence $k_{j_\ell}$ along which $\pi(k_{j_\ell})/k_{j_\ell}$ converges. By \eqref{eq:limitpt}, we see that the limit is $\xi\cdot e_1 = \bfB_x/(\overline{\mfa}_x+\bfB_x)$. Obtaining $\xi \cdot e_1=\bfB_x/(\underline{\mfa}_x+\bfB_x)$ is similar. It follows that the set of limit points of $\pi(k)/k$ is given precisely by the vectors in $[e_2,e_1]$ with first coordinate in the interval in \eqref{eq:limitpts}.
\end{proof}

\subsection{Dual paths and coalesence}\label{sec:dualcoal}
We next prove %the coalescence of geodesics in 
Theorem \ref{thm:geo}\eqref{thm:geo:coal} by adapting an argument introduced by the third author in \cite[Theorem 4.12]{Sep-18} and \cite[Theorem 3.6]{Sep-20}. 

Fix $x,y \in \bbZ^2$; with reference to Theorem \ref{thm:geo}\eqref{thm:geo:coal}, our goal is to prove that for $\xi \in ]\mfc_1^{x \wedge y},\mfc_2^{x \wedge y}[$, $\bfP(\geo{}{x}{\xi}\text{ and }\geo{}{y}{\xi}\text{ coalesce})=1.$  Without loss of generality (by re-indexing), we prove the claim for the case of $x \wedge y = (0,0)$, in which case the event in the probability only depends on $\{\w_v : v \geq (0,0)\}$. Noting that these weights are not impacted by the choice of parameter sequences $a_{-\infty:-1}$ and $b_{-\infty:-1}$, it will be convenient to assume without loss of generality that the parameter sequences are symmetric about zero.  
\begin{condition} \label{cond:sym}
    For all $k \in \bbZ$, $a_{-k}=a_k$ and $b_{-k}=b_k$.
\end{condition}
Throughout this section, we will work under Condition \ref{cond:sym} and for some fixed $\xi \in ]\mfc_1^{(0,0)},\mfc_2^{(0,0)}[$. We remain on the full probability event on which for all $x \in \bbZ^2$ and all $i\in\{1,2\}$, we have both $\buse{x}{x+e_i}{\xi+}=\buse{x}{x+e_i}{\xi-}:= \buse{x}{x+e_i}{\xi}$ and $\buse{x}{x+e_1}{\xi} \neq \buse{x}{x+e_2}{\xi}$.

We view the Busemann geodesics defined according to \eqref{eq:busgeodef} as consisting of directed edges and consider the graph $\sT^{\xi}$ obtained by taking the union of all of these edges on what we will call the primal lattice, $\bbZ^2$. We also consider the graph $\sT^{\xi,*}$ obtained by taking the union of their dual edges on the dual lattice $\bbZ^2 + (1/2,1/2)$, as illustrated in Figure \ref{fig:geodual}. 

For each site $x$ in the primal lattice $\bbZ^2$, exactly one of the oriented edges $(x,x+e_1)$ and $(x,x+e_2)$ is in $\sT^{\xi}$. The rule determining which of these two edges is included is as follows:
\be\begin{aligned}
&\begin{cases}(x,x+e_1) \in \sT^{\xi} \text{ if } \buse{x}{x+e_1}{\xi} < \buse{x}{x+e_2}{\xi}\\ (x,x+e_2) \in \sT^{\xi}  \text{ if } \buse{x}{x+e_1}{\xi} > \buse{x}{x+e_2}{\xi}\end{cases} 
\end{aligned} \label{eq:treedef}\ee

The dual graph (on the dual lattice $\bbZ^{2*}= \bbZ^2 + (1/2,1/2)$) is denoted by $\sT^{\xi,*}$. We associate to each $x \in \bbZ^2$ a unique point $x^*\in \bbZ^2$ via $x^* = x + (1/2,1/2)$. The dual graph is defined by including $(x^*, x^*-e_i) \in \sT^{\xi,*}$  if and only if $(x,x+e_i) \in \sT^{\xi}$. Note that the orientation of dual edges is reversed in the dual graph. See Figure \ref{fig:geodual} for an illustration. In particular, we have the following rule generating $\sT^{\xi,*}$:
\be\begin{aligned}
&\begin{cases} (x^*,x^*-e_1) \in \sT^{\xi,*} \text{ if } \buse{x}{x+e_1}{\xi} < \buse{x}{x+e_2}{\xi}, \\  (x^*,x^*-e_2) \in \sT^{\xi,*}  \text{ if } \buse{x}{x+e_1}{\xi} > \buse{x}{x+e_2}{\xi}\end{cases} \label{eq:dualtreedef}
\end{aligned}\ee
Given a site of the dual lattice $x^* \in \bbZ^{2*}$, we denote by $\geo{}{x^*}{\xi*}$ the unique south-west directed semi-infinite dual vertex path obtained by following the edges in $\sT^{\xi,*}$ originating from $x^*$.
\begin{figure}[h!]
\begin{subfigure}{.5\textwidth}
\begin{tikzpicture}[scale = 0.8]
%Grid
\draw[gray,thin] (0,0) grid (5,5);
%Geodesics
\draw[very thick] (0,0) -- (1,0);
\draw[very thick] (1,0) -- (1,1);
\draw[very thick] (2,0) -- (2,1);
\draw[very thick] (3,0) -- (3,1);
\draw[very thick] (4,0) -- (5,0);
\draw[very thick] (5,0) -- (5,1);
\draw[very thick] (0,1) -- (1,1);
\draw[very thick] (1,1) -- (2,1);
\draw[very thick] (2,1) -- (3,1);
\draw[very thick] (3,1) -- (3,2);
\draw[very thick] (4,1) -- (4,2);
\draw[very thick] (5,1) -- (5,2);
\draw[very thick] (0,2) -- (1,2);
\draw[very thick] (1,2) -- (1,3);
\draw[very thick] (2,2) -- (3,2);
\draw[very thick] (3,2) -- (3,3);
\draw[very thick] (4,2) -- (4,3);
\draw[very thick] (5,2) -- (5,3);
\draw[very thick] (0,3) -- (1,3);
\draw[very thick] (1,3) -- (1,4);
\draw[very thick] (2,3) -- (2,4);
\draw[very thick] (3,3) -- (4,3);
\draw[very thick] (4,3) -- (5,3);
\draw[very thick,->] (5,3) -- (6,3);
\draw[very thick] (0,4) -- (0,5);
\draw[very thick] (1,4) -- (2,4);
\draw[very thick] (2,4) -- (3,4);
\draw[very thick] (3,4) -- (3,5);
\draw[very thick] (4,4) -- (4,5);
\draw[very thick] (5,4) -- (5,5);
\draw[very thick] (0,5) -- (1,5);
\draw[very thick] (1,5) -- (2,5);
\draw[very thick] (2,5) -- (3,5);
\draw[very thick,->] (3,5) -- (3,6);
\draw[very thick,-> ] (4,5) -- (4,6);
\draw[very thick, -> ] (5,5) -- (5,6);
%Dual paths
\draw[very thick, dashed,->] (2.5,4.5) -- (.5,4.5) -- (.5,3.5) -- (-.5,3.5);
\draw[very thick, dashed,->] (3.5,5.5) -- (3.5,3.5) -- (2.5,3.5) -- (2.5, 2.5) -- (1.5, 2.5) -- (1.5,1.5) -- (-.5,1.5);
\draw[very thick, dashed] (1.5,3.5) -- (1.5,2.5);
\draw[very thick, dashed] (2.5, 1.5) -- (1.5,1.5); 
\draw[very thick, dashed] (4.5,5.5) -- (4.5,3.5) -- (3.5,3.5);
\draw[very thick, dashed] (5.5, 3.5) -- (4.5, 3.5);
\draw[very thick, dashed,->] (3.5,2.5) -- (3.5,-.5);
\draw[very thick, dashed,->] (.5,.5) --  (-.5, .5);
\draw[very thick, dashed,->] (1.5,.5) -- (1.5,-.5);
\draw[very thick, dashed,->] (2.5,.5) --(2.5,-.5); 
\draw[very thick, dashed,->] (.5,2.5) --  (-.5,2.5);
\draw[very thick, dashed] (4.5, 2.5) -- (4.5,.5) -- (3.5,.5);
\end{tikzpicture}

\end{subfigure}
\begin{subfigure}{.3\textwidth}
\begin{tikzpicture}[scale = 0.8]
\draw[very thick,->] (1,0) -- (3,0);
\draw[gray, thin] (1,0) -- (1,2);
\draw[very thick, dashed,<-] (0,1) -- (2,1);
\draw[very thick,->] (5,0) -- (5,2);
\draw[gray,thin] (5,0) -- (7,0);
\draw[very thick, dashed,->] (6,1) -- (6,-1);
\draw (5,0)node[below]{$x$};
\draw (1,0)node[below]{$x$};
\draw (3,0)node[below]{$x+e_1$};
\draw (7,0)node[below]{$x+e_1$};
\draw (5,2)node[above]{$x+e_2$};
\draw (1,2)node[above]{$x+e_2$};
%\draw (4,0)node[below]{$x+e_1$};
%\draw (1,0)node[below]{$x+e_1$};
\end{tikzpicture}
\end{subfigure}
\caption{Left: Semi-infinite geodesics (solid) in $\sT^{\xi}$ are separated by dual paths (dashed) in $\sT^{\xi, *}$. Directions in the graph are indicated by arrows.  \\
Right: Possible configurations of geodesic (solid) and dual edges (dashed) at a given site. }  \label{fig:geodual}
\end{figure}
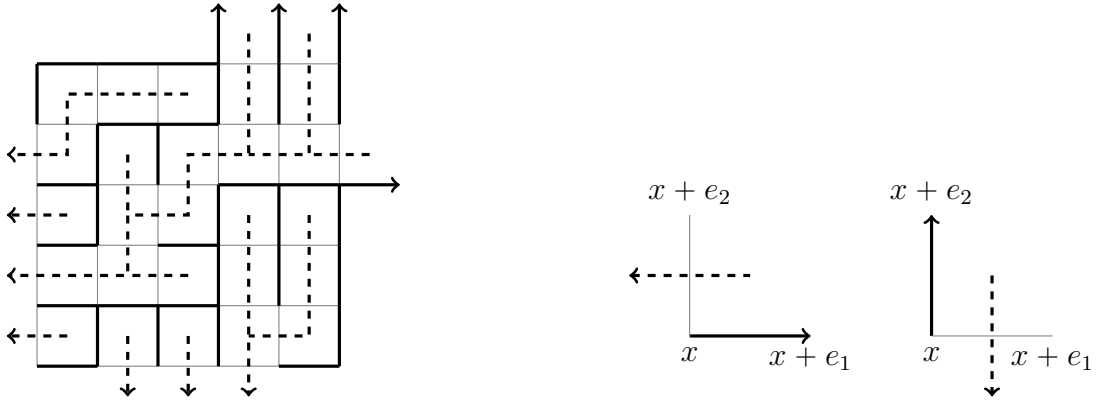

Condition \ref{cond:sym} implies that the distribution of $(\w_x)_{x\in\bbZ^2}$ is invariant under reflection about the coordinate axes: $(\w_x)_{x\in\bbZ^2} \stackrel{d}{=} (\w_{-x})_{x\in\bbZ^2}$. Under Condition \ref{cond:sym}, it then follows from Proposition \ref{prop:shape} for all $x=(i,j) \in \bbZ^2$ and all sequences $u_n\in\bbZ^2$ with $-u_n/n \to \xi \in ]e_2,e_1[$:
\begin{align}
\lim_{n\to\infty} \frac{\LPP{u_n}{x}}{n} &= \shp^{-x}(\xi) = \inf_{- \sinf{a}_{-i:\infty} < z <  \sinf{b}_{-j:\infty}} \left\{\shp_z(\xi)\right\} = \inf_{- \sinf{a}_{-\infty:i} < z <  \sinf{b}_{-\infty:j}} \left\{\shp_z(\xi)\right\}. 
\label{eq:bwdshape}
\end{align}
We have the following lemma about the structure of the strictly concave regions, which will be of use in the arguments that follow.
\begin{lemma} \label{lem:xiin} If Condition \ref{cond:sym} holds and $\xi \in ]\mfc_1^{(0,0)},\mfc_2^{(0,0)}[$,  then \[\xi\in\displaystyle \bigcap_{x\in\bbZ^2}]\mfc_1^x,\mfc_2^x[.\]
\end{lemma}
\begin{proof}
    If $x \geq (0,0)$, then the inclusion $]\mfc_1^{(0,0)}, \mfc_2^{(0,0)} [ \subseteq ]\mfc_1^x, \mfc_2^x[$ follows from \eqref{eq:critineq}. The case of $x \leq (0,0)$ is similar by Condition \ref{cond:sym}. If $x=(i,j)$ for $i<0$ and $j \geq 0$, then by Condition \ref{cond:sym}, $\sinf{a}_{i:\infty} =\sinf{a}_{0:\infty}$; thus, $\mfc_1^x= \mfc_1^{(0,j)}$. The case of $i \geq 0$ and $j<0$ is similar.
\end{proof}

Under Condition \ref{cond:sym}, symmetry implies that a version of Theorem \ref{thm:buslim} holds with initial points $u_n$ tending to infinity in the southwest direction. In particular, for $\xi \in ]\mfc_1^{(0,0)},\mfc_2^{(0,0)}[$, we may define a south-west directed Busemann function via
\begin{align*}
\lim_{n\to\infty} G_{u_n,y} -  G_{u_n,x}&= \buse{x}{y}{\xi,\text{sw}}
\end{align*}
for all $u_n \in \bbZ^2$ with $-u_n/n \to \xi.$

For the statement of our main estimate in this section, define the following passage time with the initial point removed.
\begin{align}
\LPP{x}{y}^o &= \max_{\geod{} \in \Path{x}{y}}\left\{\sum_{p \in \pi \smallsetminus \{x\}} \w_{p}\right\}. \label{eq:oLPP}
\end{align}
\begin{proposition}\label{prop:nobi}
Suppose that Condition \ref{cond:sym} holds and fix $\xi \in ]\mfc_1^{(0,0)},\mfc_2^{(0,0)}[$. The following holds $\bfP$ almost surely. For each $y=(i,j) \in \bbZ^2$ and all sequences $v_n, u_n$ with  $u_n \leq y \leq v_n$, $|u_n| \to \infty$, $|v_n| \to \infty$, and
\begin{align*}
\lim_{n\to\infty} \frac{v_n}{n} = \xi = \lim_{n\to\infty}\frac{- u_n}{n},
\end{align*}
for all sufficiently large $n$,
\begin{align*}
\LPP{u_n}{y} + \LPP{y}{v_n}^o < \LPP{u_n}{v_n}.
\end{align*}
\end{proposition}
\begin{proof}
It suffices to prove the result for $y \in \bbZ^2$ fixed. Call $\ell = i+j$ and recall that $\bbV_\ell=\{x \in \bbR^2 : x \cdot (e_1+e_2)=\ell\}$. We augment the probability space by adding an extra family of weights $(\w'_{x} : x\in\bbV_{\ell} \cap \bbZ^2)$, independent  of $\w$ with the same distribution as $(\w_{x} : x \in \bbV_\ell\cap \bbZ^2)$ under $\bfP$. For notational convenience, we will continue to denote the measure on this extended space by $\bfP$. Extend to all of $\bbZ^2$ by setting $\w_x' = \w_x$ if $x \notin \bbV_\ell \cap \bbZ^2$.
%\begin{align}
%\w'_x &= \begin{cases}
%\w_x & x \notin \bbV_\ell \cap \bbZ^2 \\
%\w_{x}' & x \in \bbV_\ell \cap \bbZ^2
%\end{cases}\label{eq:twobusproc}
%\end{align}
For $x \leq y$, set $\LPP{x}{y}' = \Lpt_{x,y}(\w')$. 

We work on a $\bfP$ almost sure event where the following limits exist for all sequences $u_n$,$v_n$ as in the statement and all $x,y \in \bbZ^2$,
\begin{align*}
\lim_{n\to\infty} \LPP{u_n}{x} -  \LPP{u_n}{y}&= \buse{x}{y}{\xi,\text{sw}}, \qquad \lim_{n\to\infty} \LPP{x}{v_n}' -  \LPP{y}{v_n}'= \buse{x}{y}{'\xi}.
\end{align*}
 By reflection symmetry, $(\buse{x}{y}{'\xi} : x,y \in \bbZ^2) \stackrel{d}{=} (\buse{-x}{-y}{\xi,\textup{sw}} : x,y \in \bbZ^2)$. Moreover, we have that $(\buse{x}{y}{'\xi} : x,y \in \bbV_\ell\cap\bbZ^2)$ and $(\buse{x}{y}{\xi,\textup{sw}} : x,y \in \bbV_\ell\cap\bbZ^2)$ are independent as they are functions of disjoint collections of independent weights.  
 
 Notice that if for infinitely many values of $n$ we have $\LPP{u_n}{y} + \LPP{y}{v_n}^o = \LPP{u_n}{v_n}$ then it must be the case that along that sequence in $n$, we must have for all $k \in \bbZ_{>0}$, 
\begin{align*}
\LPP{u_n}{y} -  \LPP{u_n}{y+(-k,k)} + \LPP{y}{v_n}^o  -  \LPP{y+(-k,k)}{v_n}^o >  0 .
\end{align*}
Sending $n \to \infty$ along this subsequence, it therefore suffices to show that we cannot have
\begin{align}
\buse{y}{y+(-k,k)}{\xi,\text{sw}} + \buse{y}{y+(-k,k)}{'\xi}  >  \w_y' - \w_{y+(-k,k)}'. \label{eq:imposs}
\end{align}
for all $k \in \bbZ_{>0}$. Using the cocycle property, Theorem \ref{thm:buslim}\eqref{thm:buslim:decomp}\eqref{thm:buslim:decomp:upright}, we may write 
\begin{align*}
&\buse{y}{y+(-k,k)}{'\xi} + \buse{y}{y+(-k,k)}{\xi,\text{sw}} = \sum_{m=0}^{k-1} \bigg[\buse{y+(-m,m)}{y+(-m-1,m+1)} {\xi,\text{sw}}+\buse{y+(-m,m)}{y+(-m-1,m+1)} {'\xi}\bigg] \end{align*}
%\\
%&= \sum_{i=0}^k \bigg[\buse{y+(-i,i)}{y+(-i-1,-i+1)}{'\xi} - \buse{y+(-i+1,i-1) }{y+(-i+1,i)}{'\xi}  + \buse{y+(-i+1,i-1) }{y+(-i+1,i)}{\xi,\text{sw}} - \buse{y+(-i,i)}{y+(-i+1,i)}{\xi,\text{sw}}\bigg].
  and similarly, 
\begin{align*}
&\buse{y+(-m,m)}{y+(-m-1,m+1)}{'\xi} = \buse{y+(-m,m)}{y+(-m,m+1)}{'\xi} + \buse{y+(-m,m+1)}{y+(-m-1,m+1)}{'\xi}, \\
&\buse{y+(-m,m)}{y+(-m-1,m+1)}{\xi,\text{sw}} =\buse{y+(-m,m)}{y+(-m,m+1)} {\xi,\text{sw}} + \buse{y+(-m,m+1)}{y+(-m-1,m+1)}{\xi,\text{sw}}.
\end{align*} By Lemma \ref{lem:xiin}, the value of $\zmin{x}(\xi) := \chi 
\in (-\sinf{a}_{0:\infty},\sinf{b}_{0:\infty})$ does not depend on $x\in\bbZ^2$. By Theorem \ref{thm:buslim} \eqref{thm:buslim:indp} and \eqref{thm:buslim:marg}, the summands in these expressions are independent with
\begin{align*}
\buse{y+(-m,m)}{y+(-m,m+1)}{'\xi} \sim \Exp(b_{j+m}-\chi),\,  -\buse{y+(-m,m+1)}{y+(-m-1,m+1)}{'\xi} \sim \Exp(a_{i-m-1} + \chi),\\
-\buse{y+(-m,m)}{y+(-m,m+1)} {\xi,\text{sw}} \sim \Exp(b_{j+m+1}-\chi), \, \buse{y+(-m,m+1)}{y+(-m-1,m+1)} {\xi,\text{sw}}\sim \Exp(a_{m-i}+\chi).
\end{align*}
By the invariance principle \cite[Theorem 7.1.4]{Eth-Kur-86}, 
\begin{align*}
\left(\frac{1}{\sqrt{k}} \sum_{i=1}^{\lf k t \rf} \left[\buse{y+(-i,i)}{y+(-i+1,i-1)}{'\xi} - \buse{y+(-i,i)}{y+(-i+1,i-1)}{\xi,\text{sw}} \right]\right)_{t\geq 0} \Longrightarrow (W(C t))_{t\geq 0},
\end{align*}
where $W$ is standard Brownian motion, which we take for notational simplicity to be defined on $(\Omega,\sF,\bfP)$ and $C = 2\bigg[\int (a + \chi)^{-2} \alpha(da) + \int(b - \chi)^{-2} \beta(db)\bigg].$
%\begin{align*}
%C &= 2\bigg[\int \frac{1}{(a + \chi)^2} \alpha(da) + \int\frac{1}{(b - \chi)^2} \beta(db)\bigg].
%\end{align*}

By \eqref{as:inf}, $\sinf{a}_{0:\infty}+\sinf{b}_{0:\infty}>0$ and so there exists $c>0$ so that for $N \in \bbZ_{>0}$,
\begin{align*}
\bfP\left(\max_{k \in [N]}\{ \w_{y+(-k,k)}' \} \geq N^{1/4} \right) &\leq N e^{-c N^{1/4}}.
\end{align*}
We have
\begin{align*}
&\bfP\left(\buse{y+(-k,k)}{y}{\xi,\text{sw}} + \buse{y}{y+(-k,k)}{'\xi} > \w_{y}' - \w_{y+(-k,k)}'\qquad \forall k \in [N]\right) \\
&\qquad \leq N e^{-cN^{1/4}}  + \bfP\left(\buse{y+(-k,k)}{y}{\xi,\text{sw}} + \buse{y}{y+(-k,k)}{'\xi} > - N^{1/4}\, \forall k \in [N]\right) 
\end{align*}
As $N \to \infty$, the last probability converges to $\bfP\left(\inf_{0 \leq t \leq 1} \left\{ W(Ct)\right\} \geq 0 \right) = 0$.
%\[
%\bfP\left(\inf_{0 \leq t \leq 1} \left\{ W(Ct)\right\} \geq 0 \right) = 0. \qedhere
%\]
\end{proof}
%\begin{lemma} \label{lem:concaveint}
%Let $x = (i,j)$ and $y = (m,n)$ and call $z = (\min(i,m),\min(j,n))$. Suppose that $\xi \in ]\mfc_1^x,\mfc_2^x[ \cap ]\mfc_1^y,\mfc_2^y[$. Then $\xi \in ]\mfc_1^z,\mfc_2^z[$.
%\end{lemma}
%\begin{proof}
%Notice that $\xi \in ]\mfc_1^x,\mfc_2^x[$ if and only if $- \inf a_{i,\infty} < \chi^x(\xi) < \inf b_{j,\infty}$ and similarly, $\xi \in ]\mfc_1^y, \mfc_2^y[$ if and only if $- \inf a_{m,\infty} < \chi^y(\xi) < \inf b_{n,\infty}.$ Inspection of the variational problem in \eqref{eq:shape} shows that if this holds, then $\chi^x(\xi) = \chi^y(\xi) := \chi(\xi)$. Moreover, because $-\inf a_{\min(i,m),\infty} < \chi(\xi) < \inf b_{\min(j,n),\infty}$, this value is feasible in the variational problem defining $\shp^z(\xi)$ and therefore $\chi(\xi) = \chi^z(\xi)$ as well.
%\end{proof}

Define a family of weights $(\w_x^{\xi} : x \in \bbZ^2)$ via $\w_x^\xi = \buse{x-e_1}{x}{\xi} \wedge \buse{x-e_2}{x}{\xi}$. Theorem \ref{thm:buslim}\eqref{thm:buslim:indp} implies that this family is independent under $\bfP$ and by the distributional properties in Theorem \ref{thm:buslim}\eqref{thm:buslim:marg}, we see that
\be(\w_x)_{x \in \bbZ^2} \stackrel{d}{=}(\w_{x+e_1+e_2}^\xi)_{x\in\bbZ^2}.\label{eq:dualdist}\ee
Define passage times $\LPP{x}{y}^{\xi} = \Lpt_{x, y}(\w^\xi)$ according to \eqref{eq:Lpt}.  The next lemma states that paths in the dual graphs $\sT^{\xi,*}$ define geodesics in the environment $\w^{\xi}$ after re-centering. The proof is identical to the proof of Lemmas 4.1(i) and 4.3(i) in \cite{Sep-20} in the i.i.d.~case, which only depends on the cocycle and recovery properties of the Busemann functions.

\begin{lemma} \label{lem:dualgeo}
Suppose Condition \ref{cond:sym} holds and $\xi \in ]\mfc_1^{(0,0)},\mfc_2^{(0,0)}[$. If $y^*,z^* \in\bbZ^{2*}$ satisfy $\geo{m}{x^*}{\xi*} = y^*$ and $\geo{n}{x^*}{\xi*}=z^*$ and $m \leq n$, then
\begin{align*}
\buse{y}{z}{\xi} + \w_y^{\xi} &= \LPP{y}{z}^{\xi}
\end{align*}
where $y=y^*-(1/2,1/2)$ and $z=z^*-(1/2,1/2)$. In particular, the primal lattice sites $\geo{m:n}{x^*}{\xi*} - (1/2,1/2)$ are geodesics in the environment $\w^{\xi}$.
\end{lemma}
%\begin{proof}
%For each $w \in \bbZ^2$, \eqref{eq:buswgt} immediately gives $\w_w^{\xi\pm} \leq \buse{w-e_i}{w}{\xi\pm}$ for each $i \in \{1,2\}$. As a consequence, for any path $\geod{m,n}$ from $x$ to $y$, we have
%\begin{align*}
%\sum_{i=m}^n \w_{\geod{i}} \leq \sum_{i=m}^{n-1} \buse{\geod{i+1}}{\geod{i}}{\xi\pm} + \w_{x}^{\xi\pm} = \buse{x}{y}{\xi\pm}
%\end{align*}
%consequently, $\LPP{x}{y}^{\xi\pm} \leq \buse{x}{y}{\xi\pm} + \w_x^{\xi\pm}$. On the other hand, along the geodesic, equality holds.
%\end{proof}

We have the following directedness result concerning dual paths.
\begin{lemma}\label{lem:backLLN}
Suppose that Condition \ref{cond:sym} holds and $\xi \in ]\mfc_1^{(0,0)},\mfc_2^{(0,0)}[$. Then $\bfP$-almost surely, for all $x^* \in \bbZ^{2*}$,
\begin{align*}
\lim_{n\to-\infty} \frac{\geo{n}{x^*}{\xi*}}{n} &= \xi.
\end{align*}
\end{lemma}
\begin{proof}
This proof is similar to that of Theorem \ref{thm:geo}\eqref{thm:geo:ex}. Recall the local rule defining $\sT^{\xi,*}$ in \eqref{eq:dualtreedef} and the recovery property of Busemann functions in Theorem \ref{thm:buslim}\eqref{thm:buslim:recovery}.

Fix $x^*\in \bbZ^{2*}.$ Our first claim is that for each $k \in \bbZ$, we have that $\geo{n}{x^*}{\xi*}\cdot e_p \leq k$ for all $n$ sufficiently large and $p \in \{1,2\}$.  Suppose this fails with positive probability for some $k$. Then calling $y_n = (k,n-k)$, there must exist $N \in \bbZ$ and $p \in \{1,2\}$ so that 
\[
\bfP(\w_{y_n} = \buse{y_n}{y_{n}+e_p}{\xi}\text{ for all } n \leq N)>0
\]
By Condition \ref{cond:sym}, the estimates following \eqref{eq:buswgteq} rule out this possibility.

Now call $x = x^*-(1/2,1/2)=(i,j)$ and $\geo{n}{x}{\xi**} := \geo{n}{x^*}{\xi*}-(1/2,1/2).$ Suppose that $n_k$ is a subsequence with $n_k \to -\infty$ along which we have for some $\zeta \in [e_2,e_1]$,
\[
\lim_{k\to\infty}\frac{\geo{n_k}{x}{\xi**} }{n_k} = \zeta
\]
As in the proof of Theorem \ref{thm:geo}\eqref{thm:geo:ex}\eqref{thm:geo:ex:con}, standard concentration estimates for sums of exponential random variables imply that
\begin{align*}
\lim_{k\to\infty} -\frac{1}{n_k}\buse{\geo{n_k}{x}{\xi**}}{x}{\xi} =  \zeta \cdot e_1 \int_0^\infty \frac{\alpha(\dd a)}{a+\zmin{x}(\xi)} + \zeta \cdot  e_2 \int_0^\infty \frac{\beta(\dd b)}{b-\zmin{x}(\xi)}.
\end{align*}
By Lemma \ref{lem:dualgeo} for each $k$, we have that
\begin{align*}
\frac{\buse{\geo{n_k}{x}{\xi**}}{x}{\xi}}{n_k} + \frac{\w_{\geo{n_k}{x}{\xi**}}^{\xi}}{n_k} = \frac{\LPP{\geo{n_k}{x}{\xi**}}{x}^{\xi}}{n_k}.
\end{align*}
By the distributional identity in \eqref{eq:dualdist}, Lemma \ref{lem:xiin}, and Proposition \ref{prop:shape}, the right-hand side converges to $-\shp^{x}(\zeta)$ almost surely. Because $\sinf{a}_{0:\infty}+\sinf{b}_{0:\infty}>0$, the middle term can be seen to converge to zero in probability. This implies that
\[
\shp^{x}(\zeta) =  \zeta \cdot e_1 \int_0^\infty \frac{\alpha(\dd a)}{a+\zmin{x}(\xi)} + \zeta \cdot  e_2 \int_0^\infty \frac{\beta(\dd b)}{b-\zmin{x}(\xi)}.
\]
By strict concavity of $\gamma^x$ on $]\mfc_1^x,\mfc_2^x[$, concavity on $[e_2,e_1]$, and the assumption that $\xi \in ]\mfc_1^{(0,0)},\mfc_2^{(0,0)}[\subseteq ]\mfc_1^x,\mfc_2^x[$, this holds if and only if $\zeta = \xi$.
 \end{proof}

With reference to Figure \ref{fig:geodual}, note that if a bi-infinite path $\pi^*$ exists in $\sT^{\xi,*}$, then it partitions $\sT^{\xi}$ into two disjoint forests. We say that $\pi^*$ separates two semi-infinite paths $\pi$ and $\nu$ in $\sT^{\xi}$ if one of the paths $\pi$ and $\nu$ lies strictly above $\pi^*$ and one lies strictly below. The next lemma is a deterministic fact coming from the construction of the graphs $\sT^{\xi}$ and their duals $\sT^{\xi,*}$. The proof is verbatim identical to that of Step 2 of Lemma 4.6 in \cite{Sep-20}.
 
\begin{lemma}  \label{lem:dualbi}
Take $\xi \in [e_2,e_1]$ and $x,y\in \bbZ^2$. If $\geo{}{x}{\xi} \cap \geo{}{y}{\xi} = \varnothing$ if and only if there is a bi-infinite path in $\sT^{\xi,*}$ which separates them. 
\end{lemma}
%\begin{proof}
%This is a deterministic fact coming from the construction of the graphs $\sT^{\xi}$ and their duals $\sT^{\xi,*}$.  If a bi-infinite path in $\sT^{\xi,*}$ separates $\geo{}{x}{\xi}$ and $\geo{}{y}{\xi}$, then it follows immediately that $\geo{}{x}{\xi} \cap \geo{}{y}{\xi} = \varnothing$. 
%
%Conversely, suppose that $\geo{}{x}{\xi} \cap \geo{}{y}{\xi} = \varnothing$ and let $z \in \bbZ^2 \cap \bbV_n$ be such that $z^*$ lies between $\geo{}{x}{\xi}$ and $\geo{}{y}{\xi}$. Without loss of generality, we may assume that $x$ and $y$ lie on the same level $\bbV_n$ and $x \prec y$.  At least one of $z^* + e_1$ and $z^* + e_2$ must also lie between $\geo{}{x}{\xi}$ and $\geo{}{y}{\xi}$, or else we would have to have $x+e_1 = y+e_2 \in \geo{}{x}{\xi} \cap \geo{}{y}{\xi}$. We may thus inductively construct a path $z_m^* \in \bbZ^{2*}$ with $z_m^* - z_{m-1}^* \in \{e_1,e_2\}$ which lies  between $\geo{}{x}{\xi}$ and $\geo{}{y}{\xi}$.  Paths in $\sT^{\xi,*}$ are always infinite in the south-west direction. By construction, such paths cannot cross paths in $\sT^{\xi}$. See the picture on the right in Figure \ref{fig:geodual} for an illustration.  For each $m \geq n$, we have that $\geo{}{z_m}{\xi*}$ does not cross $\geo{}{x}{\xi}$ or $\geo{}{y}{\xi}$. Consider arbitrary indices $k \leq \ell$. We can extract a subsequence $m_n \to \infty$ along which $\geo{k:\ell}{z_{m_n}}{\xi*}$ is constant for all sufficiently large $n$. Diagonalizing then constructs a bi-infinite path in $\sT^{\xi*}$ which separates $\geo{}{x}{\xi}$ and $\geo{}{y}{\xi}$.
%\end{proof}

\begin{proof}[Proof of Theorem \ref{thm:geo} \eqref{thm:geo:coal}]
We begin by noting that the event $\{\geo{}{x}{\xi} \cap \geo{}{y}{\xi} = \varnothing\}$ is measurable with respect to $\sigma(\w_w : w \geq  x \wedge y)$. By re-indexing the lattice, we may assume that $x \wedge y=(0,0)$. By coupling, we may alter the parameter sequences which factor into the distribution of sites which do not satisfy $w \geq (0,0)$ without changing this event and therefore may assume without loss of generality that Condition \ref{cond:sym} holds. In this new environment, by Lemma \ref{lem:xiin}, we have $\xi \in ]\mfc_1^{(0,0)},\mfc_2^{(0,0)}[$ and therefore, because of our assumption on the parameter sequences, $\xi \in \bigcap_{w \in \bbZ} ]\mfc_1^w,\mfc_2^w[$. 

By Lemma \ref{lem:dualbi}, $\geo{}{x}{\xi} \cap \geo{}{y}{\xi} = \varnothing$ if and only if there is a point $w^* \in \bbZ^{2,*}$ and a bi-infinite path $\geo{}{w*}{\xi*}  \in \sT^{\xi,*}$ containing $w^*$ which separates them. By Lemma \ref{lem:dualgeo}, such a path is a bi-infinite geodesic in the environment $\w^{\xi}$.  Lemma \ref{lem:backLLN} ensures that $\geo{-n}{w}{\xi*}/n \to - \xi$.  If such a path which separates $\geo{}{x}{\xi}$ and $\geo{}{y}{\xi}$ exists in $\sT^{\xi,*}$, the facts that $\geo{n}{x}{\xi}/n \to \xi$ and $\geo{n}{y}{\xi}/n \to \xi$ force $\geo{n}{w*}{\xi*}/n \to \xi$. Combining \eqref{eq:dualdist} with Proposition \ref{prop:nobi}  rules out this possibility and so we conclude that with probability one $\geo{}{x}{\xi} \cap \geo{}{y}{\xi} \neq \varnothing$. By definition of Busemann geodesics in \eqref{eq:busgeodef}, this implies coalescence.
\end{proof}

\section{Competition interfaces} \label{sec:cif} 
%For the proof of Theorem \ref{th:cif1} we take  the base point   to be $(k,l)=(1,1)$ without loss of generality.  
Recall the locations $\Cif^x(n)$ and $\Cifs^x(m)$, which denote the locations where the competition interface rooted at $x$ pass the horizontal and vertical levels $n$ and $m$, respectively. We now prove Theorem \ref{thm:cif1}, which records the distribution  of $\Cif^{x}(\infty)$, $\Cifs^{x}(\infty)$ (defined in \eqref{eq:cifslope} and \eqref{def:Cifs}) and that of the limit of $\Cifa_n^x/n$ (defined in \eqref{eq:cifdef}).
\begin{proof}[Proof of Theorem \ref{thm:cif1}]
We begin by showing part \eqref{thm:cif1:1}, with the proof of \eqref{thm:cif1:2} being similar. We can read off  the distribution of $\Cif^{x}(\infty)$ from Theorem \ref{thm:buslim} \eqref{thm:buslim:indp} and \eqref{thm:buslim:marg}. For 
%finite integers 
 $x = (i,j) \in \bbZ^2$ and $m \in \bbZ_{\ge i}$, 
\begin{align*}
&\bfP\left(\Cif^{x}(\infty) = m\right) = \bfP\left(\Cif^{x}(\infty) \ge  m\right) - \bfP\left(\Cif^{x}(\infty) \ge m+1\right) \\
&\qquad
=\bfP\left(\Buse{x}{x+e_1}{(m,\infty)} > \Buse{x}{x+e_2}{(m,\infty)}\right)
-  \bfP\left(\Buse{x}{x+e_1}{(m+1,\infty)} > \Buse{x}{x+e_2}{(m+1,\infty)}\right)\\
&\qquad
=\frac{\smin{a}_{i:m}+b_j}{a_i+b_j} -  \frac{\smin{a}_{i:m+1}+b_j}{a_i+b_j} =  \frac{\smin{a}_{i:m}-\smin{a}_{i:m+1}}{a_i+b_j}.
\end{align*}
%and
It follows that
\begin{align*}
\bfP\left(\Cif^{x}(\infty) = \infty\right) =\lim_{m\to\infty} \bfP\left(\Cif^{x}(\infty) \ge  m\right)
=  \frac{\sinf{a}_{i:\infty} + b_j}{a_i+b_j} . 
\end{align*}
Next, we turn to part \eqref{thm:cif1:3}. 
Define  for $x \in \bbZ^2$, 
\be\label{def:cif}  
\cif^{x}=\sup\{\xi \in [e_2,e_1] :  \buse{x}{x+e_2}{\xi+} \leq\buse{x}{x+e_1}{\xi+}\},  \ee
where the supremum is taken with respect to the total ordering $\preceq$ on $[e_2,e_1]$, with the understanding that if the set above is empty, the supremum is $e_2$. Note that, as above, $\buse{x}{x+e_1}{\xi+}$ is non-increasing and $\buse{x}{x+e_2}{\xi+}$ is non-decreasing in $\xi$. For fixed $\xi$ and $x\in \bbZ^2$,  $\buse{x}{x+e_1}{\xi}$ and $\buse{x}{x+e_2}{\xi}$ are independent exponential variables with marginal distributions recorded in \eqref{eq:busmar}. The distributional claims in \eqref{cif5} follow immediately. 

It remains to show that $\cif^x = \lim_{n\to\infty} \Cifa_n^x/n$, $\bfP$ almost surely. First, we note that $\bfP(\cif^x \in ]e_2,\mfc_1^x[ \, \cup \, ]\mfc_2^x,e_1[) = 0$ and therefore we may assume without loss of generality that $\cif^x \in \{e_1\} \cup \{e_2\} \cup\, [\mfc_1^x,\mfc_2^x]$.

Consider the case $\xi \notin \{e_1,e_2\}$. Take $\zeta,\eta \in ]e_2,e_1[\, \cap\, \sU_0$, where $\sU_0$ is any fixed countable dense subset of $[e_2,e_1]$ containing $e_1$  and $e_2$, with $\zeta \prec \cif^x \prec \eta$. Consider sequences $v_{n,\zeta},v_{n,\zeta}\in \bbZ^2$ with $v_{n,\zeta}/n\to\zeta$ and $v_{n,\eta}/n\to\eta$. By \eqref{def:cif} and Theorem \ref{thm:buslim}\eqref{thm:buslim:indp} and \eqref{thm:buslim:marg} (to rule out ties), we have $\buse{x}{x+e_1}{\zeta} > \buse{x}{x+e_2}{\zeta}$ and $\buse{x}{x+e_1}{\eta} < \buse{x}{x+e_2}{\eta}$. For all sufficiently large $n$, by Theorem \ref{thm:buslim} \eqref{thm:buslim:buslim}, 
\begin{align*}
\init{\J}_{x,v_{n,\zeta}} < \init{\I}_{x,v_{n,\zeta}} \text{ and } \init{\J}_{x,v_{n,\eta}} > \init{\I}_{x,v_{n,\eta}}
\end{align*}
It follows then that 
\begin{align*}
\zeta \preceq \varliminf \frac{\Cifa_n^x}{n} \preceq \varlimsup \frac{\Cifa_n^x}{n} \preceq \eta.
\end{align*}
Taking $\zeta\nearrow\cif^x$ and $\eta\searrow\cif^x$ gives  $\lim \Cifa_n^x/n = \cif^x$. 

We give the details of the case $\Cifa^x = e_2$, with the $e_1$ case being similar. Combining the hypothesis that $\cif^x = e_2$ with the observation that $\buse{x}{x+e_1}{e_2} \neq \buse{x}{x+e_2}{e_2}$,  (which follows from Theorem \ref{thm:buslim} \eqref{thm:buslim:indp} and \eqref{thm:buslim:marg}), we have $\buse{x}{x+e_1}{e_2} > \buse{x}{x+e_2}{e_2}.$  It then follows from Theorem \ref{thm:buslim} \eqref{thm:buslim:buslim} that if we take any sequence $v_n$ with $v_n\cdot e_1, v_n \cdot e_2 \to \infty$ and with $v_n/n \to e_2$, then we must have for all sufficiently large $n$, $\init{\I}_{x,v_n} > \init{\J}_{x,v_n}$. It then follows that $\varlimsup \Cifa_n^x/n \preceq e_2$, from which we see that $\lim \Cifa_n^x/n = e_2$.
\end{proof}

\section{Inhomogeneous TAZRP}\label{sec:particles}
Denote $\Cifa=\Cifa^{(1,1)}$, $\Cifa^* = \Cifa - (1/2,1/2)$, $\cif = \cif^{(1,1)}$, $\shp^{(1,1)} = \shp$, and recall $\Cifb_t$, which was introduced in \eqref{cif-456}.  %From the limit in Theorem \ref{thm:cif1} \eqref{thm:cif1:3} and Proposition \ref{prop:shape}, we deduce the following:
\begin{proposition}
 On the event where $\Cifa_n \cdot e_1 \to \infty$ and $\Cifa_n \cdot e_2 \to \infty$,
\begin{align*}
 \lim_{t\to\infty} \frac{\Cifb_t}{t}&=  \frac{\cif}{\shp(\cif)}.
\end{align*} 
\end{proposition}
 On the event where $\Cifa_n \cdot e_1\to k < \infty$ or $\Cifa_n \cdot e_2 \to \ell < \infty$, the limit exists and is given by
 \begin{align*}
  \lim_{t\to\infty} \frac{\Cifb_t}{t} = \frac{e_1}{\int (b+ \smin{a}_{1:k})^{-1}\beta(db)} \qquad \text{ or }\qquad    \lim_{t\to\infty} \frac{\Cifb_t}{t}=\frac{e_2}{\int (a+ \smin{b}_{1:k})^{-1}\alpha(da)}
 \end{align*}
 respectively. With this observation in mind, we can now prove Theorem \ref{th:2cl}.

\begin{proof}
On the event in the first part of the statement, by  the limit in Theorem \ref{thm:cif1} \eqref{thm:cif1:3}  and Proposition \ref{prop:shape}, we have
\be\label{cif78}  \begin{aligned}
 \lim_{t\to\infty} \frac{\Cifb_t}{t}&=\lim_{n\to\infty}\frac{\Cifa_n^*}{\tau_n} =\lim_{n\to\infty}\frac{\Cifa_n^*}{ \LPP{(1,1)}{\Cifa_n^*}-\w_{(1,1)}} =\lim_{n\to\infty}\frac{\Cifa_n^*}n \cdot \frac{n}{\LPP{(1,1)}{\Cifa_n^*}-\w_{(1,1)}}. \\[4pt] 
\end{aligned}\ee 
The cases correspond to the possible limits in Proposition \ref{prop:shape}. 
\end{proof}

\begin{proof}[Proof of Theorem \ref{th:2cl}] 
The location of a customer can only increase, so  $\ddd Z(\infty)=\lim_{t\to\infty} Z(t) \in \Z_{\ge2}\cup\{\infty\}$ exists by monotonicity.   
By   Lemma \ref{2cl-star-lm},  Lemma \ref{fp-lem},  \eqref{cif-456}, and   \eqref{def:Cifs}, 
\begin{align*}
   Z(\infty)&= \lim_{t\to\infty} J(t) +1 \stackrel{d}{=} \lim_{t\to\infty} \Cifb(t)\cdot e_2 +1 = \lim_{n\to\infty} \Cifa(n)\cdot e_2 +1 %= \lim_{m\to\infty}\Cifs^{(1,1)}(m)+1\\
   =\Cifs^{(1,1)}(\infty)+1.  \end{align*}  
This proves Theorem \ref{th:2cl} \eqref{th:2cl:1}.

Recalling that $a_i = 0$ for all $i$, we have $\sinf{b}_{1:\infty}>0$ (because of \eqref{as:inf}), so
\begin{align*}
\shp(e_2) = \int_0^\infty \frac{\beta(db)}{b} <\infty \qquad \text{ and }\qquad \shp(e_1) = \frac{1}{\inf b_{1:\infty}} < \infty. %\int_0^\infty \frac{\beta(db)}{b - a_{i,m}^{\min}}.
\end{align*}
%Similarly, $.$ 
It then follows from Proposition \ref{prop:shape} %and considering the possible cases of $\Cifa_n \cdot e_1$ and $\Cifa_n \cdot e_2$
that
\begin{align*}
  \vela&= \lim_{t\to\infty} \frac{Z(t)}t= \lim_{t\to\infty} \frac{J(t)}t =\lim_{t\to\infty} \frac{\Cifb(t)\cdot e_2}{t} = \frac{\cif \cdot e_2}{\shp(\cif)}.
 \end{align*}
We can write each $\xi \in [e_2,e_1]$ as $(1-t,t)$ for a unique $t = t(\xi) \in [0,1]$. With this identification, for $t \neq 0$, we have $\xi \cdot e_2/\shp(\xi) = t/\shp(1-t,t) = 1/\shp(1/t-1,1)$ by homogeneity. From \eqref{eq:shape},  $x \mapsto \shp(x,1)$ is strictly increasing on $(0,\infty)$. It follows that $\vela \in \left[0, \left(\int_0^\infty b^{-1}\beta(db)\right)^{-1}\right]$. %, with $\vela = 0$ corresponding to $\cif = e_1$ and $\vela = \left(\int_0^\infty b^{-1}\beta(db)\right)^{-1}$ corresponding to $\cif = e_2$. 
By Theorem \ref{thm:cif1} \eqref{thm:cif1:3},
\begin{align*}
\bfP\left(\vela=0\right) = \bfP\left(\cif = e_1\right) =  1 - \frac{\sinf{b}_{1:\infty}}{b_1}.
\end{align*}
 
  Denote by $\gamma^{-1}(x,1)$ the inverse function of the function $x \mapsto \gamma(x,1)$. For $0< s< \left(\int_0^\infty b^{-1} \beta(db)\right)^{-1}$, call
\begin{align*}
\zeta(s) &= \left(\frac{\shp^{-1}(1/s,1)}{1 + \shp^{-1}(1/s,1)}, \frac{1}{1+\shp^{-1}(1/s,1)}\right).
\end{align*}
Again using the notation $\cif=(1-t(\cif),t(\cif))$ as above, apply Theorem \ref{thm:cif1}\eqref{thm:cif1:3}  to obtain 
\begin{align*}
\bfP\left(\vela\le s\right)&=\bfP\left(  \frac{\cif \cdot e_2}{\gamma(\cif)}\le s \right) =\bfP(1/s \leq \shp(1/t(\cif)-1,1) =\bfP(\shp^{-1}(1/s,1) \leq 1/t(\cif)-1) \\
&= 1- \bfP\left(\cif \preceq  \zeta(s) \right) 
=1-\frac{\zmin{}(\zeta(s))}{b_{1}}.
\end{align*}
Differentiating, we see that for $0 < x < (\int_0^\infty b^{-1} \beta(db))^{-1}$, $\shp'(x,1) = 1/\zmin{}(x,1)$. Using this observation and homogeneity, it follows that
$\zmin{}(\zeta(s)) = 1/\shp'(\shp^{-1}(1/s,1), 1) = (\shp^{-1})'(1/s,1)$
where $\shp'(x,1)$ is the derivative of $x \mapsto \shp(x,1)$, $\shp^{-1}(x,1)$ is the inverse of the same map, and $(\shp^{-1})'(x,1)$ is the derivative of this inverse function.
\end{proof} 

\section*{Statements and declarations}
\subsection*{Data availability statement} Data sharing is not applicable to this article as no datasets were generated or analyzed during the current study.

\subsection*{Conflicts of interest/Competing interests} The authors have no conflicts of interest to declare that are relevant to the content of this article.

\bigskip 

\small 
\bibliographystyle{habbrv}
\bibliography{refs} 

\begin{thebibliography}{10}

\bibitem{Ahl-Dam-Sid-16}
D.~Ahlberg, M.~Damron, and V.~Sidoravicius.
\newblock Inhomogeneous first-passage percolation.
\newblock {\em Electron. J. Probab.}, 21:Paper No. 4, 19, 2016.

\bibitem{Ale-Ber-18}
K.~S. Alexander and Q.~Berger.
\newblock Geodesics toward corners in first passage percolation.
\newblock {\em J. Stat. Phys.}, 172(4):1029--1056, 2018.

\bibitem{And-etal-00}
E.~D. Andjel, P.~A. Ferrari, H.~Guiol, and C.~Landim.
\newblock Convergence to the maximal invariant measure for a zero-range process
  with random rates.
\newblock {\em Stochastic Process. Appl.}, 90(1):67--81, 2000.

\bibitem{Auf-Dam-13}
A.~Auffinger and M.~Damron.
\newblock Differentiability at the edge of the percolation cone and related
  results in first-passage percolation.
\newblock {\em Probab. Theory Related Fields}, 156(1-2):193--227, 2013.

\bibitem{Bah-Bod-18}
C.~Bahadoran and T.~Bodineau.
\newblock Quantitative estimates for the flux of {TASEP} with dilute site
  disorder.
\newblock {\em Electron. J. Probab.}, 23:Paper No. 44, 44, 2018.

\bibitem{Bal-Cat-Sep-06}
M.~Bal{{\'a}}zs, E.~Cator, and T.~Sepp{{\"a}}l{{\"a}}inen.
\newblock Cube root fluctuations for the corner growth model associated to the
  exclusion process.
\newblock {\em Electron. J. Probab.}, 11:no. 42, 1094--1132 (electronic), 2006.

\bibitem{Bat-etal-25-}
E.~Bates, E.~Emrah, J.~Martin, T.~Sepp\"al\"ainen, and E.~Sorensen.
\newblock Permutation invariance in last-passage percolation and the
  distribution of the busemann process.
\newblock 2025.
\newblock Preprint (\href{https://arxiv.org/abs/2506.12641}{arXiv:2506.12641}).

\bibitem{Bis-etal-23}
E.~Bisi, Y.~Liao, A.~Saenz, and N.~Zygouras.
\newblock Non-intersecting path constructions for {TASEP} with inhomogeneous
  rates and the {KPZ} fixed point.
\newblock {\em Comm. Math. Phys.}, 402(1):285--333, 2023.

\bibitem{Bor-Pec-08}
A.~Borodin and S.~P\'ech\'e.
\newblock Airy kernel with two sets of parameters in directed percolation and
  random matrix theory.
\newblock {\em J. Stat. Phys.}, 132(2):275--290, 2008.

\bibitem{Bri-Hof-21}
G.~Brito and C.~Hoffman.
\newblock Geodesic rays and exponents in ergodic planar first passage
  percolation.
\newblock In {\em In and out of equilibrium 3. {C}elebrating {V}ladas
  {S}idoravicius}, volume~77 of {\em Progr. Probab.}, pages 163--186.
  Birkh\"{a}user/Springer, Cham, [2021] \copyright 2021.

\bibitem{Cat-Gro-05}
E.~Cator and P.~Groeneboom.
\newblock Hammersley's process with sources and sinks.
\newblock {\em Ann. Probab.}, 33(3):879--903, 2005.

\bibitem{Cat-Gro-06}
E.~Cator and P.~Groeneboom.
\newblock Second class particles and cube root asymptotics for {H}ammersley's
  process.
\newblock {\em Ann. Probab.}, 34(4):1273--1295, 2006.

\bibitem{Cat-Pim-12}
E.~Cator and L.~P.~R. Pimentel.
\newblock Busemann functions and equilibrium measures in last passage
  percolation models.
\newblock {\em Probab. Theory Related Fields}, 154(1-2):89--125, 2012.

\bibitem{Cat-Pim-13}
E.~Cator and L.~P.~R. Pimentel.
\newblock Busemann functions and the speed of a second class particle in the
  rarefaction fan.
\newblock {\em Ann. Probab.}, 41(4):2401--2425, 2013.

\bibitem{Cor-etal-14}
I.~Corwin, N.~O'Connell, T.~Sepp{{\"a}}l{{\"a}}inen, and N.~Zygouras.
\newblock Tropical combinatorics and {W}hittaker functions.
\newblock {\em Duke Math. J.}, 163(3):513--563, 2014.

\bibitem{Dam-Han-14}
M.~Damron and J.~Hanson.
\newblock Busemann functions and infinite geodesics in two-dimensional
  first-passage percolation.
\newblock {\em Comm. Math. Phys.}, 325(3):917--963, 2014.

\bibitem{Dau-Ort-Vir-22-}
D.~Dauvergne, J.~Ortmann, and B.~Vir\'ag.
\newblock The directed landscape.
\newblock {\em Acta Math.}, 229(2):201--285, 2022.

\bibitem{Dau-Vir-21-}
D.~Dauvergne and B.~Vir{\'a}g.
\newblock The scaling limit of the longest increasing subsequence.
\newblock Preprint (\href{https://arxiv.org/abs/2104.08210}{\tt arXiv
  2104.08210}).

\bibitem{Dim-24-}
E.~Dimitrov.
\newblock Airy wanderer line ensembles.
\newblock 2024.
\newblock {\tt arXiv:2408.08445}, 76 pages.

\bibitem{Dur-Lig-81}
R.~Durrett and T.~M. Liggett.
\newblock The shape of the limit set in {R}ichardson's growth model.
\newblock {\em Ann. Probab.}, 9(2):186--193, 1981.

\bibitem{Emr-16}
E.~Emrah.
\newblock Limit shapes for inhomogeneous corner growth models with exponential
  and geometric weights.
\newblock {\em Electron. Commun. Probab.}, 21:Paper No. 42, 16, 2016.

\bibitem{Emr-Jan-17}
E.~Emrah and C.~Janjigian.
\newblock Large deviations for some corner growth models with inhomogeneity.
\newblock {\em Markov Process. Related Fields}, 23(2):267--312, 2017.

\bibitem{Emr-Jan-Sep-21}
E.~Emrah, C.~Janjigian, and T.~Sepp\"{a}l\"{a}inen.
\newblock Flats, spikes and crevices: the evolving shape of the inhomogeneous
  corner growth model.
\newblock {\em Electron. J. Probab.}, 26:Paper No. 33, 45, 2021.

\bibitem{Emr-Jan-Sep-23}
E.~Emrah, C.~Janjigian, and T.~Sepp\"al\"ainen.
\newblock Optimal-order exit point bounds in exponential last-passage
  percolation via the coupling technique.
\newblock {\em Probab. Math. Phys.}, 4(3):609--666, 2023.

\bibitem{Eth-Kur-86}
S.~N. Ethier and T.~G. Kurtz.
\newblock {\em {M}arkov {P}rocesses: {C}haracterization and {C}onvergence}.
\newblock Wiley Series in Probability and Mathematical Statistics: Probability
  and Mathematical Statistics. John Wiley \& Sons, Inc., New York, 1986.

\bibitem{Fan-Sep-20}
W.-T.~L. Fan and T.~Sepp\"al\"ainen.
\newblock Joint distribution of {B}usemann functions in the exactly solvable
  corner growth model.
\newblock {\em Prob. Math. Phys.}, 1(1):55--100, 2020.

\bibitem{Fer-Mar-Pim-06}
P.~A. Ferrari, J.~B. Martin, and L.~P.~R. Pimentel.
\newblock Roughening and inclination of competition interfaces.
\newblock {\em Phys.\ Rev.\ E}, 73:031602(3), 2006.

\bibitem{Fer-Pim-05}
P.~A. Ferrari and L.~P.~R. Pimentel.
\newblock Competition interfaces and second class particles.
\newblock {\em Ann. Probab.}, 33(4):1235--1254, 2005.

\bibitem{Geo-Ras-Sep-17-ptrf-2}
N.~Georgiou, F.~Rassoul-Agha, and T.~Sepp{\"a}l{\"a}inen.
\newblock Geodesics and the competition interface for the corner growth model.
\newblock {\em Probab. Theory Related Fields}, 169(1-2):223--255, 2017.

\bibitem{Geo-Ras-Sep-17-ptrf-1}
N.~Georgiou, F.~Rassoul-Agha, and T.~Sepp{\"a}l{\"a}inen.
\newblock Stationary cocycles and {B}usemann functions for the corner growth
  model.
\newblock {\em Probab. Theory Related Fields}, 169(1-2):177--222, 2017.

\bibitem{Geo-etal-15}
N.~Georgiou, F.~Rassoul-Agha, T.~Sepp{{\"a}}l{{\"a}}inen, and A.~Yilmaz.
\newblock Ratios of partition functions for the log-gamma polymer.
\newblock {\em Ann. Probab.}, 43(5):2282--2331, 2015.

\bibitem{Gri-Kan-Sep-04}
I.~Grigorescu, M.~Kang, and T.~Sepp{{\"a}}l{{\"a}}inen.
\newblock Behavior dominated by slow particles in a disordered asymmetric
  exclusion process.
\newblock {\em Ann. Appl. Probab.}, 14(3):1577--1602, 2004.

\bibitem{Hag-Mee-95}
O.~H{{\"a}}ggstr{{\"o}}m and R.~Meester.
\newblock Asymptotic shapes for stationary first passage percolation.
\newblock {\em Ann. Probab.}, 23(4):1511--1522, 1995.

\bibitem{Hof-05}
C.~Hoffman.
\newblock Coexistence for {R}ichardson type competing spatial growth models.
\newblock {\em Ann. Appl. Probab.}, 15(1B):739--747, 2005.

\bibitem{Hof-08}
C.~Hoffman.
\newblock Geodesics in first passage percolation.
\newblock {\em Ann. Appl. Probab.}, 18(5):1944--1969, 2008.

\bibitem{Jan-Ras-Sep-22-}
C.~Janjigian, F.~Rassoul-Agha, and T.~Sepp\"al\"ainen.
\newblock Geometry of geodesics through {B}usemann measures in directed
  last-passage percolation.
\newblock {\em J. Eur. Math. Soc. (JEMS)}, 25(7):2573--2639, 2023.

\bibitem{Joh-06}
K.~Johansson.
\newblock Random matrices and determinantal processes.
\newblock In {\em Mathematical statistical physics}, pages 1--55. Elsevier B.
  V., Amsterdam, 2006.

\bibitem{Joh-Rah-22}
K.~Johansson and M.~Rahman.
\newblock On inhomogeneous polynuclear growth.
\newblock {\em Ann. Probab.}, 50(2):559--590, 2022.

\bibitem{Kni-Pet-Sae-19}
A.~Knizel, L.~Petrov, and A.~Saenz.
\newblock Generalizations of {TASEP} in discrete and continuous inhomogeneous
  space.
\newblock {\em Comm. Math. Phys.}, 372(3):797--864, 2019.

\bibitem{Kra-LeD-Oco-21}
A.~Krajenbrink, P.~Le~Doussal, and N.~O'Connell.
\newblock Tilted elastic lines with columnar and point disorder,
  non-{H}ermitian quantum mechanics, and spiked random matrices: pinning and
  localization.
\newblock {\em Phys. Rev. E}, 103(4):Paper No. 042120, 37, 2021.

\bibitem{Kru-Fer-96}
J.~Krug and P.~Ferrari.
\newblock Phase transitions in driven diffusive systems with random rates.
\newblock {\em J. Phys. A}, 29:L465--L471, 1996.

\bibitem{Lic-new-96}
C.~Licea and C.~M. Newman.
\newblock Geodesics in two-dimensional first-passage percolation.
\newblock {\em Ann. Probab.}, 24(1):399--410, 1996.

\bibitem{Mar-02}
R.~Marchand.
\newblock Strict inequalities for the time constant in first passage
  percolation.
\newblock {\em Ann. Appl. Probab.}, 12(3):1001--1038, 2002.

\bibitem{Mat-Qua-Rem-21}
K.~Matetski, J.~Quastel, and D.~Remenik.
\newblock The {KPZ} fixed point.
\newblock {\em Acta Math.}, 227(1):115--203, 2021.

\bibitem{New-95}
C.~M. Newman.
\newblock A surface view of first-passage percolation.
\newblock In {\em Proceedings of the {I}nternational {C}ongress of
  {M}athematicians, {V}ol.\ 1, 2 ({Z}{\"u}rich, 1994)}, pages 1017--1023,
  Basel, 1995. Birkh{\"a}user.

\bibitem{Rah-Vir-21-}
M.~Rahman and B.~Virag.
\newblock Infinite geodesics, competition interfaces and the second class
  particle in the scaling limit.
\newblock 2112.06849.
\newblock Preprint (\href{https://arxiv.org/abs/2112.06849}{\tt arXiv
  2112.06849}).

\bibitem{Ras-18}
F.~Rassoul-Agha.
\newblock Busemann functions, geodesics, and the competition interface for
  directed last-passage percolation.
\newblock In {\em Random growth models}, volume~75 of {\em Proc. Sympos. Appl.
  Math.}, pages 95--132. Amer. Math. Soc., Providence, RI, 2018.

\bibitem{Ros-81}
H.~Rost.
\newblock Nonequilibrium behaviour of a many particle process: density profile
  and local equilibria.
\newblock {\em Z. Wahrsch. Verw. Gebiete}, 58(1):41--53, 1981.

\bibitem{Sep-12-corr}
T.~Sepp{{\"a}}l{{\"a}}inen.
\newblock Scaling for a one-dimensional directed polymer with boundary
  conditions.
\newblock {\em Ann. Probab.}, 40(1):19--73, 2012.
\newblock Corrected version available at
  \href{https://arxiv.org/abs/0911.2446}{\tt arXiv:0911.2446}.

\bibitem{Sep-18}
T.~Sepp\"{a}l\"{a}inen.
\newblock The corner growth model with exponential weights.
\newblock In {\em Random growth models}, volume~75 of {\em Proc. Sympos. Appl.
  Math.}, pages 133--201. Amer. Math. Soc., Providence, RI, 2018.

\bibitem{Sep-20}
T.~Sepp\"{a}l\"{a}inen.
\newblock Existence, uniqueness and coalescence of directed planar geodesics:
  proof via the increment-stationary growth process.
\newblock {\em Ann. Inst. Henri Poincar\'{e} Probab. Stat.}, 56(3):1775--1791,
  2020.

\bibitem{Sep-Kru-99}
T.~Sepp{{\"a}}l{{\"a}}inen and J.~Krug.
\newblock Hydrodynamics and platoon formation for a totally asymmetric
  exclusion model with particlewise disorder.
\newblock {\em J. Statist. Phys.}, 95(3-4):525--567, 1999.

\bibitem{Sep-Sor-23}
T.~Sepp\"{a}l\"{a}inen and E.~Sorensen.
\newblock Busemann process and semi-infinite geodesics in {B}rownian
  last-passage percolation.
\newblock {\em Ann. Inst. Henri Poincar\'{e} Probab. Stat.}, 59(1):117--165,
  2023.

\bibitem{Sly-16-}
A.~Sly.
\newblock Note on the flux for {TASEP} with general disorder.
\newblock 2016.
\newblock {\tt arXiv:1609.06589}, 4 pages.

\end{thebibliography}
\end{document}